\documentclass[11pt]{article}
\usepackage{xcolor}
\usepackage{graphicx}
\usepackage{amsmath,amsfonts,amssymb,graphics,amsthm}
\usepackage{hyperref}
\usepackage{comment}
\usepackage{tabularx}
\usepackage{enumerate}
\usepackage{bbm}
\usepackage{mathrsfs}
\usepackage[margin=1in]{geometry}
\usepackage[shortlabels]{enumitem}

\newcommand{\gefei}[1]{{\color{brown}{#1}} }

\hypersetup{
    colorlinks=false,
    linktocpage,
    }

\numberwithin{equation}{section}

\newtheorem*{thm-n}{Theorem}
\newtheorem{theorem}{Theorem}[section]
\newtheorem{corollary}[theorem]{Corollary}
\newtheorem{lemma}[theorem]{Lemma}
\newtheorem{proposition}[theorem]{Proposition}
\newtheorem{prop}[theorem]{Proposition}

\newtheorem{remark}[theorem]{Remark}
\newtheorem{definition}[theorem]{Definition}
\newtheorem{fakedefinition}[theorem]{``Definition''}

\theoremstyle{remark}

\def\@rst #1 #2other{#1}
\newcommand\MR[1]{\relax\ifhmode\unskip\spacefactor3000 \space\fi
  \MRhref{\expandafter\@rst #1 other}{#1}}
\newcommand{\MRhref}[2]{\href{http://www.ams.org/mathscinet-getitem?mr=#1}{MR#2}}

\def\MR#1{\href{http://www.ams.org/mathscinet-getitem?mr=#1}{MR#1}}

\newcommand{\tv}{{\operatorname{tv}}}
\newcommand{\sep}{\operatorname{sep}}
\newcommand{\cont}{\operatorname{Cont}}
\newcommand{\bfn}{{\mathbf n}}
\newcommand{\bfb}{{\mathbf b}}
\newcommand{\bft}{{\mathbf t}}
\newcommand{\fC}{{\mathfrak C}}

\newcommand{\C}{\mathbbm{C}}
\newcommand{\D}{\mathbbm{D}}
\newcommand{\E}{\mathbbm{E}}
\newcommand{\N}{\mathbbm{N}}

\newcommand{\R}{\mathbbm{R}}
\renewcommand{\P}{\mathbbm{P}}
\newcommand{\bbH}{\mathbbm{H}}

\newcommand{\cM}{\mathcal{M}}
\newcommand{\conf}{\mathrm{Conf}}

\newcommand{\eps}{\varepsilon}
\newcommand{\1}{\mathbf{1}}

\newcommand{\sph}{\mathrm{sph}}
\newcommand{\disk}{\mathrm{disk}}

\newcommand{\sm}{\mathsf{m}}

\newcommand{\LF}{\mathrm{LF}}
\newcommand{\Leb}{\mathrm{Leb}}
\newcommand{\QD}{\mathrm{QD}}
\newcommand{\GQD}{\mathrm{GQD}}
\newcommand{\QS}{\mathrm{QS}}
\newcommand{\QA}{\mathrm{QA}}

\newcommand{\lexp}{{\beta}}
\newcommand{\CR}{\mathrm{CR}}
\renewcommand{\wp}{\eta}

\let\Re\undefined
\DeclareMathOperator{\Re}{Re}

\DeclareMathOperator{\Var}{Var}
\DeclareMathOperator{\SLE}{SLE}
\DeclareMathOperator{\CLE}{CLE}

\def\cX{\mathcal{X}}

\def\cU{\mathcal{U}}
\def\cT{\mathcal{T}}
\def\cS{\mathcal{S}}

\def\cP{\mathcal{P}}
\def\cO{\mathcal{O}}

\def\cM{\mathcal{M}}
\def\cL{\mathcal{L}}

\def\cI{\mathcal{I}}

\def\cF{\mathcal{F}}
\def\cE{\mathcal{E}}
\def\cD{\mathcal{D}}
\def\cC{\mathcal{C}}
\def\cB{\mathcal{B}}
\def\cA{\mathcal{A}}

\def\alb#1\ale{\begin{align*}#1\end{align*}}
\def\allb#1\alle{\begin{align}#1\end{align}}

\newcommand{\aryb}{\begin{eqnarray*}}
\newcommand{\arye}{\end{eqnarray*}}
\def\alb#1\ale{\begin{align*}#1\end{align*}}
\newcommand{\eqb}{\begin{equation}}
\newcommand{\eqe}{\end{equation}}
\newcommand{\eqbn}{\begin{equation*}}
\newcommand{\eqen}{\end{equation*}}

\newcommand{\ol}{\overline}
\newcommand{\ul}{\underline}
\newcommand{\op}{\operatorname}

\newcommand{\frk}{\mathfrak}

\newcommand{\rta}{\rightarrow}

\newcommand{\wt}{\widetilde}
\newcommand{\wh}{\widehat}

\newcommand{\bdy}{\partial}

\newcommand{\Loop}{\mathrm{Loop}}
\newcommand{\GA}{\mathrm{GA}}
\newcommand{\Weld}{\mathrm{Weld}}
\newcommand{\lp}{\mathrm{loop}}

\let\originalleft\left
\let\originalright\right
\renewcommand{\left}{\mathopen{}\mathclose\bgroup\originalleft}
\renewcommand{\right}{\aftergroup\egroup\originalright}

\DeclareMathAlphabet{\mathpzc}{OT1}{pzc}{m}{it}

\begin{document}

\title{SLE Loop Measure and Liouville Quantum Gravity}
\author{
\begin{tabular}{c}Morris Ang\\[-3pt]\small UC San Diego\\[-3pt]\end{tabular}\; 
\begin{tabular}{c}Gefei Cai\\[-3pt]\small Peking University\\[-3pt]\end{tabular}\; 
\begin{tabular}{c}Xin Sun\\[-3pt]\small Peking University\\[-3pt]\end{tabular}
\begin{tabular}{c}Baojun Wu\\[-3pt]\small Peking University\\[-3pt]\end{tabular}
}

\date{  }

\maketitle
\begin{abstract}
As recently shown by Holden and two of the authors, the conformal welding of two Liouville quantum gravity (LQG) disks produces a canonical variant of SLE curve whose law is called the  SLE loop measure. In this paper, we demonstrate how LQG can be used to study the SLE loop measure. Firstly, we show that for $\kappa\in (8/3,8)$, the loop intensity measure of the conformal loop ensemble agrees with the SLE loop measure as defined by Zhan (2021). The former was initially considered by Kemppainen and Werner (2016)  for $\kappa\in (8/3,4]$, and  the latter was constructed for $\kappa\in (0,8)$. Secondly, we establish a duality for the SLE loop measure between $\kappa$ and $16/\kappa$. Thirdly, we obtain the exact formula for the moment of the electrical thickness for the shape (probability) measure of the SLE loop, which in the regime $\kappa\in (8/3,8)$ was conjectured by Kenyon and Wilson (2004). This relies on the exact formulae for the reflection coefficient and the one-point disk correlation function in Liouville conformal field theory.  Finally, we compute several multiplicative constants associated with the SLE loop measure, which are not only of intrinsic interest but also used in our companion paper relating the conformal loop ensemble to the imaginary DOZZ formulae.
\end{abstract}


\section{Introduction}

Schramm-Loewner evolution (SLE) is a one-parameter family of planar curves, which are of fundamental importance in two-dimensional random geometry and statistical physics. There are several variants of the SLE curves, such as the chordal and radial versions considered originally by \cite{Schramm:1999rp}. For $\kappa\in (0,8)$, Zhan~\cite{zhan-loop-measures} constructed a variant called the SLE$_\kappa$ loop measure, 
an infinite measure on the space of non-crossing loops that locally look like SLE$_\kappa$.
In particular, the loop is simple for $\kappa\in (0,4]$ and non-simple for $\kappa\in (4,8)$, with Hausdorff dimension $1+\frac{\kappa}{8}$.  
{Zhan \cite{zhan-loop-measures} shows that for $\kappa\in (0,4]$, the SLE loop measure satisfies the conformal restriction property  proposed by Kontsevich and Sukov~\cite{kontsevich-suhov}. Recently,  Baverez and Jego~\cite{Baverez:2024drp} proved that this conformal restriction property uniquely characterizes the SLE loop measure for $\kappa\in (0,4]$  modulo multiplicative constants. Before~\cite{zhan-loop-measures,Baverez:2024drp}, the existence and uniqueness for $\kappa=8/3$ was established by Werner~\cite{werner-loops}. A construction for the SLE$_2$ loop measure was given by Benoist and Dubedat \cite{Benoist_2016} and Kassel and Kenyon \cite{Kassel_2017}. For $\kappa\in (8/3,4]$ a construction was given by Kemppainen and Werner \cite{werner-sphere-cle} based on the conformal loop ensemble (CLE); their construction in fact produces a SLE type loop measure for all $\kappa\in (8/3,8)$.}

With Holden, the first and third authors showed  the conformal welding of two Liouville quantum gravity
(LQG) disks produces Zhan's SLE loop measure with $\kappa\in (0,4)$ \cite{ahs-sle-loop}. 
In this paper, we first extend this result to $\kappa \in (4,8)$ using the natural LQG disk with non-simple boundaries from \cite{wedges,nonsimple-welding,hl-lqg-cle,msw-non-simple}.   We then clarify the equivalence between Zhan's construction of the SLE loop measure and 
the construction of Kemppainen and Werner \cite{werner-sphere-cle} based on CLE for $\kappa\in (8/3,8)$.
See Section~\ref{subsec:intro-def}.

After completing these two foundational
tasks, the bulk of our paper is to use LQG techniques to study the properties of the SLE loop measure. 
First, we establish a duality for the SLE loop measure between $\kappa$ and $16/\kappa$. The case when $\kappa=8/3$ had previously been treated  
by Werner in~\cite{werner-loops}. Second, we obtain the exact formula for the moment of the electrical thickness for the shape (probability) measure of the SLE loop, which in the regime $\kappa\in (8/3,8)$ had been conjectured by Kenyon and Wilson (2004). This relies on the exact formulae for the structure constants  in Liouville conformal field theory obtained in~\cite{krv-dozz,rz-boundary,ARS-FZZ}. 
See Sections~\ref{subsec:intro-duality} and~\ref{subsec:KW-intro} for the precise statements. 

Although the equivalence and duality results for the SLE loop may be derived without LQG, as discussed below Theorems~\ref{cor-loop-equiv-main} and~\ref{loop duality}, one of the main purposes of our paper is to demonstrate the methodology.
Several LQG objects and techniques from this paper are useful for our further study of SLE and CLE. For example, we introduce and study a natural LQG surface of annular topology, which we call the quantum annulus.  We also determine several constants related to the SLE loop measure that are interesting in their own right and useful in our companion paper relating CLE to the imaginary DOZZ formulae. We give an overview of our techniques in Section~\ref{subsec:lqg} and outline some further applications in Section~\ref{subsec:outlook}.

\subsection{On the definitions of the SLE loop measure}\label{subsec:intro-def}

The conformal loop ensemble (CLE) is a countable
collection of non-crossing random loops in a simply connected domain defined for $\kappa\in (\frac83,8)$~\cite{shef-cle,shef-werner-cle}. CLE has been proved or conjectured to be the scaling limit of many important loop models in two-dimensional statistical physics, such as percolation, the Ising model, the $O(n)$-loop model, and the random cluster model, see \cite{benoist-hongler-cle3,  camia-newman-sle6, kemp-smirnov-fk-full}. There is also a sphere version of CLE, called {\it full-plane $\CLE$}, which is constructed and proved to satisfy conformal invariance by Kemppainen-Werner \cite{werner-sphere-cle} for $\kappa\in (\frac{8}{3},4]$ and Gwynne-Miller-Qian \cite{gmq-cle-inversion} for $\kappa\in (4,8)$.

Let $\Gamma$ be a full-plane $\CLE_\kappa$ with  $\kappa\in (\frac{8}3, 8)$.
The \emph{loop intensity measure} $\wt\SLE_\kappa^{\mathrm{loop}}$ for $\CLE_\kappa$ is defined to be the distribution of a loop chosen from the counting measure over the set of loops in $\Gamma$.  This is an infinite measure on the space of non-crossing loops on the plane. By the conformal invariance of $\CLE_\kappa$, the measure $\wt\SLE_\kappa^{\mathrm{loop}}$ is  conformally invariant. For $\kappa\in(\frac{8}{3},4]$,
the measure $\wt\SLE_\kappa^{\mathrm{loop}}$ was first considered by Kemppainen and Werner \cite{werner-sphere-cle}, who showed it satisfies inversion symmetry.

Recently, Zhan \cite{zhan-loop-measures} constructed what he called the SLE loop measure for all $\kappa\in (0,8)$, which we denote by $\SLE_\kappa^{\rm loop}$. The construction is based on the Minkowski content of SLE and an un-rooting procedure. We will recall this in Section \ref{section sle conformal welding}. 
In this paper, we show that for $\kappa\in (8/3,8)$, Zhan's SLE$_\kappa$ loop measure agrees with the 
loop intensity measure  $\wt\SLE_\kappa^{\mathrm{loop}}$ modulo a multiplicative constant. 
\begin{theorem}\label{cor-loop-equiv-main}
  	For each $\kappa\in (\frac83,8)$, %
  	there exists  $C\in (0,\infty)$ such that $\wt\SLE_\kappa^{\mathrm{loop}}=C\SLE_\kappa^{\mathrm{loop}}$.
 \end{theorem}

For  $\kappa\in (8/3,4]$, it might be possible to show that the CLE loop intensity measure satisfies the conformal restriction property in~\cite{kontsevich-suhov,Baverez:2024drp} based on the Brownian loop soup construction of CLE from~\cite{shef-werner-cle}. Then Theorem~\ref{cor-loop-equiv-main} for  $\kappa\in (8/3,4]$ follows from the uniqueness result in~\cite{Baverez:2024drp}.  Alternatively, the equivalence in this range is also implicit from~\cite{werner-sphere-cle}, as explained to 
us by Werner. See Theorem 2.18 of the second arXiv  
version of \cite{ang2021integrability} for more explanation. That argument might extend to $\kappa\in (4,8)$ as well. Despite these alternative approaches, we believe that our LQG proof is interesting on its own and 
contains several useful technical intermediate results. 

\subsection{Duality for the SLE loop measure}\label{subsec:intro-duality}

Our next theorem is an instance of SLE duality, which says that for $\kappa<4$, the outer boundary of an $\SLE_{{16}/{\kappa}}$ curve is an $\SLE_\kappa$ curve. The duality of chordal SLE was established by Zhan \cite{zhan-duality1, zhan-duality2}
and Dub\'edat \cite{dubedat-duality} from martingale method. The chordal/radial/whole-plane SLE duality was systematically studied by Miller and Sheffield in the imaginary geometry framework in \cite{ig1,ig2,ig3,ig4}. 
The duality for the SLE bubble measure was proved in \cite{nolin2024backbone} via a limiting argument from chordal duality. Our Theorem~\ref{loop duality} concerns duality for the SLE loop measure. The case $\kappa=\frac{8}{3}$ was proved by Werner in \cite{werner-loops}, i.e. the shape of the scaling limit of percolation cluster outer perimeters is Brownian loops. 

\begin{theorem}\label{loop duality}
For $\kappa\in (2,4)$, sample a loop $\eta$ from $\SLE^{\rm loop}_{16/\kappa}$ on $\hat\C=\C\cup \{\infty\}$. Let $\cC^{\rm +}_{\eta}$ be the unbounded connected component of $\C\setminus\eta$ and let $\eta^{\rm out}$ be the boundary of $\cC^{\rm +}_{\eta}$. Then the law of $\eta^{\rm out}$ equals  $C\SLE^{\rm loop}_\kappa$ for some constant $C$.
\end{theorem}

Duality for loops in CLE has been extensively studied in the context of  CLE percolation~\cite{cle-percolations}, where the boundary of non-simple loops is described by the so-called boundary CLE  for $\kappa\in (2,4)$. We were informed by Werner that Theorem~\ref{loop duality} can be derived from this result and the imaginary geometry method~\cite{ig2}. We were also informed by Baverez and Jego that Theorem~\ref{loop duality} can be derived from conformal restriction arguments based on~\cite{Baverez:2024drp}.
Again, we stress that we prove Theorem~\ref{loop duality} using the LQG method, which we believe is interesting in its own right.

\subsection{Electrical thickness for the SLE loop measure}\label{subsec:KW-intro}

Suppose $\eta$ is sampled from $\SLE_\kappa^{\mathrm{loop}}$  restricted to the event that $\eta$ separates $0$ and $\infty$; we denote the law of $\eta$ under this restriction by $\SLE_\kappa^{\rm sep}$. Here we say $\eta$ separates $z$ and $\infty$ if $z \not \in \eta$ and, for an arbitrarily chosen orientation of $\eta$, its winding number around $z$ is nonzero. Note that this definition does not depend on the choice of orientation of $\eta$. When $\eta$ is simple, this is equivalent to $z$ lying in the bounded connected component of $\C \backslash \eta$, but in general, the statements ``$\eta$ separates $z$ and $\infty$'' and ``$z$ lies in a bounded connected component of $\C \backslash \eta$'' are not equivalent. Let $R_\eta=\inf\{ |z|:z\in \eta \}$ and $\hat\eta= \{z: R_\eta z\in \eta \}$.
By the conformal invariance of $\SLE_\kappa^{\mathrm{loop}}$, the law of  $\log R_\eta$ is translation invariant on $\R$, hence a constant multiple of the Lebesgue measure.
Also by conformal invariance, the conditional law of $\hat \eta$ given $R_\eta$ does not depend on the value of $R_\eta$. We denote this conditional law by $\cL_\kappa$, and call it the \emph{shape measure} of $\SLE_\kappa^{\mathrm{loop}}$. It is a probability measure on the space of loops that surround $\D$ (separates $\D$ and $\infty$) and touch $\bdy\D$. 

Let $D_\eta$ be the connected component of $\C\setminus \eta$ containing $0$.
Let $\D$ be the unit disk and $\psi :  \D \rta D_\eta$
be a conformal map such that $\psi(0)=0$. 
Let $\CR(\eta,0):=|\psi'(0)|$ be the \emph{conformal radius} of $D_\eta$ viewed from $0$. 
Let $\bar\eta$ be the image of $\eta$ under the inversion map $z\mapsto z^{-1}$.
Let \(\vartheta(\eta) := -\log \CR( \eta,0) - \log \CR(\bar\eta,0)\) be the \emph{electrical thickness} of $\eta$.
It is clear that $\vartheta(\eta)$ only depends on $\eta$ through its shape $\hat \eta$.
Our next result is an exact formula for the law of $\vartheta(\eta)$ under $\cL_\kappa$.

\begin{theorem}\label{thm-KW}
For $\kappa\in (0,8)$,
we have $\E[e^{\lambda \vartheta(\eta)}] <\infty$  if and only if $\lambda<1-\frac{\kappa}8$. Moreover, 
	\begin{equation}\label{eq-KW}
	\E[e^{\lambda \vartheta(\eta)}] = \frac{\sin(\pi (1-\kappa/4))}{\pi (1-\kappa/4)} 
	\frac{\pi \sqrt{(1-\kappa/4)^2+\lambda \kappa/2}}{\sin(\pi \sqrt{ (1-\kappa/4)^2+\lambda \kappa/2})} \quad \text{ for } \lambda < 1-\frac{\kappa}8.
	\end{equation}	
\end{theorem}

This concept of electrical thickness was introduced by Kenyon and Wilson~\cite{kw-loop-models} and was also recorded in~\cite[Eq (14)]{ssw-radii} as a way to describe how much the shape of a loop differs from a circle. It is nonnegative and equals $0$ if and only if the loop is a circle centered at the origin.  When viewing the plane as a homogeneous electrical material,  the electrical thickness
measures the net change in effective resistance between $0$ and $\infty$ when the loop $\eta$ becomes a perfect conductor. 
Recently, Baverez and Jego~\cite{Baverez:2024drp} initiated the study of the conformal field theory for the $\SLE$ loop measure, where they used the SLE loop measure to construct representations of the Virasoro algebra  with a bilinear form called the Shapovalov form.
As they point out, our Theorem~\ref{thm-KW} gives the explicit formula for the partition function of the Shapovalov form.

Theorem~\ref{thm-KW} settles a conjecture of Kenyon and Wilson on the electrical thickness of CLE loops for  $\kappa\in (\frac83,8)$.
Consider the shape measure $\wt \cL_\kappa$ of $\wt\SLE_\kappa^{\mathrm{loop}}$ defined in the same way as $\cL_\kappa$. 
Then Theorem~\ref{cor-loop-equiv-main} is equivalent to $\cL_\kappa =\wt\cL_\kappa$.
Let $(\eta_n)_{n\ge 1}$ be the sequence of loops of a $\CLE_\kappa$ on the unit disk which separate $0$ and $\infty$, ordered such that $\eta_n$ surrounds $\eta_{n+1}$. 
For $\kappa\in (\frac{8}{3},4]$, it is proved in \cite{werner-sphere-cle}  that  the law of the rescaled loop $\hat\eta_n$ converges weakly   to $\wt \cL_\kappa$, and for $\kappa\in (4,8)$, we prove the analogous result in Lemma~\ref{markov ns}.
Since $ \vartheta(\eta_n)=  \vartheta(\hat\eta_n)$, we see that  $\lim_{n\to\infty} \E[e^{\lambda \vartheta(\eta_n)}]$ equals $\E[e^{\lambda \vartheta(\eta)}]$ in Theorem~\ref{thm-KW}.  Kenyon and Wilson~\cite{kw-loop-models}  conjectured a formula for $\lim_{n\to\infty} \E[e^{\lambda \vartheta(\eta_n)}]$, as recorded in~\cite[Eq (14)]{ssw-radii}.
Their formula agrees with the right-hand of~\eqref{eq-KW} after $\kappa$ is replaced by $16/\kappa$.
Thus our Theorem~\ref{thm-KW} shows that the conjectural formula is false but only off by this flip\footnote{
As a sanity check, the formula should blow up as $\kappa\to 8$, but~\cite[Eq (14)]{ssw-radii} gives a finite limit and therefore cannot be correct.}. In the next subsection we give an overview of the proof of Theorem~\ref{thm-KW}.

\subsection{SLE  loop measure coupled with Liouville quantum gravity}\label{subsec:lqg}
We now explain our strategy for proving Theorem~\ref{thm-KW} to illustrate how  LQG can be used to study the SLE loop. 
LQG is the two-dimensional random geometry corresponding to the formal Riemannian metric tensor $e^{\gamma h} (dx^2+dy)$, where $\gamma \in (0,2)$ is a parameter and $h$ is a variant of the Gaussian free field (GFF).
See~\cite{shef-kpz,dddf-lfpp,gm-uniqueness} for its construction.
LQG describes the scaling limits of discrete random surfaces, namely random planar maps under conformal embeddings; see e.g.~\cite{hs-cardy-embedding,gms-tutte}.   We refer to the survey articles~\cite{gwynne-ams-survey,ghs-mating-survey, sheffield-icm} for more background on LQG and its relation to random planar maps. 
For natural LQG surfaces related to random planar maps, the GFF variant is produced by a quantum field theory called Liouville conformal field theory (CFT). This field theory was recently constructed and solved in ~\cite{dkrv-lqg-sphere,krv-dozz,gkrv-bootstrap,ARS-FZZ,ang2024derivation}.

A cornerstone in LQG is that when a pair of independent LQG surfaces are glued together along their boundaries in a conformal manner (i.e.\ conformal welding), the resulting interface is an SLE-type curve. The starting point of our arguments is the following way of producing the SLE loop measure via conformal welding of LQG surfaces. The most canonical random surface in LQG with the sphere (resp., disk) topology is the so-called $\gamma$-LQG  sphere (resp., disk). In~\cite{ahs-sle-loop}, Holden and the first and third authors  
proved that when two independent $\gamma$-LQG disks are glued together, the resulting surface is the $\gamma$-LQG sphere and the interface is an independent $\SLE_{\gamma^2}$ loop. In order to treat the $\kappa\in (4,8)$ regime where $\SLE_{\kappa}$ loops are non-simple, we first extend this basic conformal welding result, replacing the 
ordinary $\gamma$-LQG disks with $\gamma$-LQG disks having non-simple boundary considered in \cite{msw-non-simple,hl-lqg-cle,nonsimple-welding}. (These are also called generalized quantum disks.) This is our Theorem~\ref{loop weld ns}. Here we focus on explaining the $\kappa\in (0,4)$ case for simplicity, although dealing with non-simple boundary causes various technical difficulties here and there. The case of $\kappa=4$ can be treated by a limiting argument.

If we conformally weld a pair of $\gamma$-LQG disks, each having an interior marked point sampled from the quantum area measure, then the resulting surface is a $\gamma$-LQG sphere with two marked points, and the law of the interface becomes the SLE loop measure restricted to the event that these two points are separated. This is exactly $\SLE^{\rm sep}_{\kappa}$ with $\kappa=\gamma^2$. We will work with a variant of this result where the loop is a weighted variant of $\SLE^{\rm sep}_{\kappa}$. Liouville CFT allows us to define natural LQG surfaces with interior marked points that are not chosen from the quantum area measure. In particular, one can define a one-parameter family of $\gamma$-LQG surfaces of disk topology with one interior marked point, where the parameter $\alpha$ indexes the log-singularity of the field at the marked point. Similarly, Liouville CFT can define a one-parameter family of $\gamma$-LQG surfaces of sphere topology, with two interior marked points of equal log singularity.
Based on the previous conformal welding result and a method from~\cite{AHS-SLE-integrability}, we show in Proposition~\ref{prop-KW-weld} that once a pair of such LQG disks are conformal welded together, the resulting two-pointed sphere are those defined by Liouville CFT, and the law of interface is $\SLE^{\rm sep}_{\kappa}$ weighted by $e^{(2\Delta_\alpha-2)\theta(\eta)}$ where $\theta(\eta)$ is the electrical thickness and $\Delta_\alpha=\frac{\alpha}{2}\left(Q-\frac{\alpha}{2}\right)$. 

This new conformal welding result will ultimately allow us to express the moment-generating function of the electrical thickness in terms of the two-point sphere and one-point disk correlation functions of the Liouville CFT. The exact results needed are the formula for the reflection coefficient proved in~\cite{krv-dozz}, and the FZZ formula proved in~\cite{ARS-FZZ}.   However, there is an important difficulty, namely, 
the measure $\SLE^{\rm sep}_{\kappa}$ is infinite while we need to identify suitable observables to extract exact information. 
Overcoming this is the technical bulk of our proof of Theorem~\ref{thm-KW} in Section~\ref{sec:KW}.

\subsection{Outlook and perspectives}\label{subsec:outlook}

\begin{enumerate}
\item So far we have not paid much attention to various multiplicative constants concerning the SLE loop measure. For example, 
when we decompose the loop measure into its shape measure and the dilation, the law of the dilation is encoded by a multiple of the Lebesgue measure on $\R$.
There is also a multiplicative constant in the conformal welding result for the SLE loop measure. For the loop intensity measure $\wt\SLE_{\kappa}^{\rm loop}$,  the value of these constants turns out to play an important role in our study of the integrability of CLE in~\cite{ACSW24}.
We compute them explicitly in  Propositions~\ref{prop-SLE-lp-density},~\ref{weldingconstant simple} and~\ref{weldingconstant non-simple}. We do not know how to compute the corresponding constants for Zhan's loop measure $\SLE_{\kappa}^{\rm loop}$.   It would be interesting to compute the corresponding constants for the law of outer boundary of $\SLE^{\rm loop}_{{16}/{\kappa}}$ in Theorem ~\ref{loop duality}; we expect that they would also be useful for the quantitative study of CLE.

\item  In order to prove the equivalence loop measures in Theorem~\ref{cor-loop-equiv-main}, we introduce in Section~\ref{sec:annulus} an LQG surface with annular topology, which we call the quantum annulus. This surface also plays an important role in our study of the integrability of CLE in~\cite{ACSW24}. Moreover, Remy and two of the authors~\cite{ars-annuli} developed a method to derive the law of random moduli for LQG surfaces of annular topology, which can be applied to the quantum annulus defined here. Combined with ideas from bosonic string theory, the laws of such  random moduli encode important information about CLE on the annulus;
see the introduction of~\cite{ars-annuli} for more discussion. In Section~\ref{sec:annulus} we derive the joint length distribution of the two boundary components for the quantum annulus, which is interesting in its own right and useful for these further applications.

    \item  The conformal restriction property satisfied by the SLE loop measure when $\kappa\in (0,4]$ considered in~\cite{kontsevich-suhov,werner-loops,Baverez:2024drp}     is similar in spirit to the Weyl anomaly in Liouville conformal field theory; see~\cite[Theorem 3.5]{dkrv-lqg-sphere}.  It would be interesting to see if the conformal restriction property can be directly seen from the conformal welding picture. 
    By the recent uniqueness result of Baverez and Jego~\cite{Baverez:2024drp}, this would give an alternative proof for the conformal welding result for the SLE loop measure on the sphere when $\kappa\in (0,4]$.
    In \cite{zhan-loop-measures}, Zhan constructed the SLE loop measure on Riemann surfaces via the conformal restriction. It would be interesting to establish the corresponding conformal welding results using LQG surfaces of non-trivial topology. 
  
\end{enumerate}

\noindent{\bf Organization of the paper.}
In Section~\ref{sec:prelim}, we provide background on LQG, Liouville CFT, and conformal welding. 
In Section~\ref{section sle conformal welding}, we recall the conformal welding result for the SLE loop measure in the simple case and extend it to the non-simple case.
In Section~\ref{sec:app-CLE}, we provide background on the coupling between  CLE and LQG.
In Section~\ref{sec:annulus}, we introduce and study the quantum annulus. In Section~\ref{sec:msw-sphere}, we prove Theorems~\ref{cor-loop-equiv-main} and \ref{loop duality}. In Section~\ref{sec:KW} we prove Theorem~\ref{thm-KW}.
In Section~\ref{sec:constant} we evaluate the welding constant in the conformal welding result for the loop intensity measure.
\bigskip

\noindent\textbf{Acknowledgements.} 
We are grateful to Wendelin Werner for explaining to us how Theorems~\ref{cor-loop-equiv-main} and~\ref{loop duality} can be proved without LQG.
We thank  Nicolas  Curien,  Rick Kenyon, Scott Sheffield, Pu Yu and Dapeng Zhan for helpful discussions. 
M.A. was supported by the Simons Foundation as a Junior Fellow at the Simons Society of Fellows and partially supported by NSF grant DMS-1712862. G.C., X.S. and B.W.\ were supported by National Key R\&D Program of China (No. 2023YFA1010700). X.S. was also partially supported by the  NSF grant DMS-2027986, the NSF Career grant DMS-2046514, and a start-up grant from the University of Pennsylvania.

\section{Backgrounds and preliminary results on LQG}\label{sec:prelim}

In this section, we review the precise definition of some $\gamma$-LQG surfaces and Liouville fields.
For more background, we refer to~\cite{ghs-mating-survey,vargas-dozz-notes} and references therein, as well as the preliminary sections in \cite{AHS-SLE-integrability, ARS-FZZ}. We first provide some measure theoretic background that will be used throughout the paper.

We  will frequently consider infinite measures and extend the probability terminology to this setting. 
In particular, suppose $M$ is a $\sigma$-finite measure on a measurable space $(\Omega, \cF)$.
Suppose  $X:(\Omega,\cF)\rta (E,\cE)$ is an $\cF$-measurable function taking values in $(E,\cE)$.
Then we say that  $X$ is a random variable on $(\Omega,\cF)$
and call  the pushforward measure $M_X = X_*M$ on $(E,\sigma(X))$ the \emph{law} of $X$.
We say that $X$ is \emph{sampled} from $M_X$. We also write the integral $\int f(x) M_X(dx)$ as $M_X[f]$ or $M_X[f(x)]$ for simplicity.
For a finite measure $M$, we write $|M|$ as its total mass  
and write $M^{\#}=|M|^{-1}M$ as the probability measure proportional to $M$.

Given $(\Omega, \cF, M)$ as above, let $X:(\Omega, \cF)\rta (E,\cE)$ and $Y:(\Omega, \cF)\rta (E',\cE')$ be two  random variables.
A family of probability measures  $\{\P(\cdot |e): e\in E \}$ on $(E',\cE')$ is called the (regular) \emph{conditional law of $Y$ given $X$} if 
for each $A\in \sigma(Y) $, $\P(A |\cdot)$ is measurable  on $(E,\cE)$  and 
$$M[Y\in A,X\in B]=\int_B  \P(A |e) \, dM_X \textrm{ for each } A\in \sigma(Y) \textrm{ and } B\in \sigma(X) $$

We also need the concept of \emph{disintegration} in the case $(E,\cE)=\R^n$ for a positive integer $n$.

\begin{definition}[Disintegration] \label{def:disint}
	Let $M$ be a measure on  a measurable space $(\Omega, \cF)$. 
	Let $X:\Omega\rta \R^n$ be a measurable function with respect to $\cF$, where $\R^n$ is endowed with the Borel  $\sigma$-algebra and Lebesgue measure. 
	A family of measures $\{M_x: x\in \R^n \}$ on $(\Omega, \cF)$ 
	is called a disintegration of $M$ over $X$ if for each set $A\in\cF$, the function $x\mapsto M_x(A)$ is Borel measurable, and
	\begin{equation}\label{eq:disint}
	\int_{A} f(X) \, dM= \int_{\R^n}   f(x)  M_x(A) \,d^nx \textrm{ for each non-negative measurable function } f \textrm{on } \R^n. 
	\end{equation}
\end{definition}
When~\eqref{eq:disint} holds, we simply write $M=\int_{\R^n} M_x \, d^nx$.

\begin{lemma}\label{lem:disint}
	In the setting of Definition~\ref{def:disint}, suppose $M$ is $\sigma$-finite and $X$ satisfies $M[X\in B]=0$ for each Borel set $B\subset \R^n$ with zero Lebesgue measure. 
	Then the disintegration of $M$ over $X$ exists. Moreover if $\{M_x: x\in \R^n \}$ and $\{M'_x: x\in \R^n \}$ are two disintegrations  of $M$ over $X$,  then  $M_x=M'_x$ for almost every $x$.
\end{lemma}
\begin{proof}
	When $M$ is a probability measure, since the law $M_X$ of  $X$ is absolutely continuous with respect to the Lebesgue measure, 
	we can and must set $|M_x|$ to be the Radon-Nykodim derivative between the two measures. Since $\R^n$ is Polish, we can and must set 
	$\{M_x^\#\}$ to be the regular conditional probability of $M$ given $X$. This gives the desired existence and uniqueness of $M_x=|M_x|M_x^\#$.
	By scaling this gives Lemma~\ref{lem:disint} when $M$ is finite. If  $M$ is infinite, consider $\Omega_n\uparrow \Omega$ with $M(\Omega_n)<\infty$
	Applying Lemma~\ref{lem:disint} to $M|_{\Omega_n}$ and then sending $n\to\infty$ give the general result. 
\end{proof}

\subsection{Liouville field and Liouville quantum gravity}\label{subsubsec:GFF}
Let $\cX$ be the complex plane $\C$ or the upper half plane $\bbH$. Suppose $\cX$  is endowed with a smooth metric  $g$  
such that  the metric completion of $(\cX, g)$ is a compact Riemannian manifold. (We will not distinguish $\cX$ with its compactification for notional simplicity.)
Let $H^1(\cX)$ be the Sobolev space whose norm is the sum of the $L^2$-norm with respect to $(\cX,g)$ and the Dirichlet energy.
Let $H^{-1}(\cX)$ be the dual space of $H^1(\cX)$. Then $H^1(\cX)$ and $H^{-1}(\cX)$ do not depends on the choice of $g$.

We now recall two basic variants of the \emph{Gaussian free field} (GFF). Consider the two functions 
\begin{align}\label{eq:covariance}
G_\bbH(z,w) &= -\log |z-w| - \log|z-\ol w| + 2 \log|z|_+ + 2\log |w|_+.  \quad &&z,w\in \bbH\nonumber\\
G_\C(z,w) &= -\log|z-w| + \log|z|_+ + \log|w|_+ , \quad &&z,w\in \C.\nonumber
\end{align}Here $|z|_+:=\max\{|z|,1\}$ so that $\log|z|_+ =\max\{\log |z|,0\}$.
Let $h_\cX$  be a random function taking values in $H^{-1}(\cX)$ such that $(h_\cX,f)$ is a centered Gaussian with variance 
$\int f(z) G_\cX(z,w) f(w)d^2zd^2w$  for each  $f\in H^1(\cX)$. Then $h_\C$ is   a \emph{whole plane GFF} and $h_\bbH$ is   a \emph{free-boundary GFF} on $\bbH$, both of which are  normalized to have mean zero along $\{z\in \cX: |z|=1 \}$.
We denote the law of $h_\cX$ by $P_\cX$.

We now review the Liouville fields on $ \C$  and $\bbH$ following~\cite[Section 2.2]{AHS-SLE-integrability}.
\begin{definition}%
	\label{def-LF-sphere}
	Suppose $(h, \mathbf c)$ is sampled from $P_\C \times [e^{-2Qc}dc]$ and set $\phi =  h(z) -2Q \log |z|_+ +\mathbf c$. 
	Then we write  $\LF_{\C}$ as the law of $\phi$ and call a sample from  $\LF_{\C}$  a \emph{Liouville field on $\C$}.
	
	Suppose $(h, \mathbf c)$ is sampled from $P_\bbH \times [e^{-Qc}dc]$ and set $\phi =  h(z) -2Q \log |z|_+ +\mathbf c$. 
	Then we write  $\LF_{\bbH}$ as the law of $\phi$ and call a sample from  $\LF_{\bbH}$  a \emph{Liouville field on $\bbH$}.
\end{definition}
We also need Liouville fields on $\C$ with several insertions as the following.
\begin{definition}\label{def-RV-sph}
	Let $(\alpha_i,z_i) \in  \R \times \C$ for $i = 1, \dots, m$, where $m \ge 1$ and the $z_i$'s are distinct. 
	Let $(h, \mathbf c)$ be sampled from $ C_\C^{(\alpha_i,z_i)_i}  P_\C \times [e^{(\sum_i \alpha_i  - 2Q)c}dc]$ where
	\[C_{  \C}^{(\alpha_i,z_i)_i}=\prod_{i=1}^m |z_i|_+^{-\alpha_i(2Q -\alpha_i)} e^{\sum_{i < j} \alpha_i \alpha_j G_\C(z_i, z_j)}.\]
	Let \(\phi(z) = h(z) -2Q \log |z|_+  + \sum_{i=1}^m \alpha_i G_\C(z, z_i) + \mathbf c\).
	We write  $\LF_{ \C}^{(\alpha_i,z_i)_i}$ for the law of $\phi$ and call a sample from  $\LF_{ \C}^{(\alpha_i,z_i)_i}$ 
	a \emph{Liouville field on $ \C$ with insertions $(\alpha_i,z_i)_{1\le i\le n}$}. 
\end{definition}

The measure $\LF_{ \C}^{(\alpha_i,z_i)_{i}}$  formally equals $\prod_{i=1}^{m} e^{\alpha_i \phi(z_i)}\LF_{  \C}$.
In fact, after a regularization and limiting procedure, we will arrive at Definition~\ref{def-RV-sph}; see~\cite[Lemma 2.6]{AHS-SLE-integrability}.

Similarly, we define Liouville field with one bulk and several boundary insertions as the following:

\begin{definition}\label{def boundary LF}
    
    Let $(\alpha, u) \in \mathbb{R} \times \mathbb{H}$ and $(\beta_i,s_i) \in \mathbb{R} \times (\partial \mathbb{H} \cup \{\infty \})$ for $1 \leq i \leq m$ where $m \geq 0$ and all $s_i$'s are distinct. We also assume that $s_i \neq \infty$ for $i \geq 2$. Let the constant
    \begin{equation*}
    \begin{aligned}
    &\quad C_{\mathbb{H}}^{(\alpha,u),(\beta_i,s_i)_i} := \\
    & \quad \begin{cases}
        (2 {\rm Im} u)^{-\frac{\alpha^2}{2}} |u|_+^{-2\alpha(Q-\alpha)} \prod_{i=1}^m |s_i|_+^{-\beta_i(Q-\frac{\beta_i}{2})} \times e^{ \sum_{ 1 \leq i < j \leq m } \frac{\beta_i\beta_j}{4} G_{\mathbb{H}}(s_i,s_j) +  \sum_{i=1}^m \frac{\alpha\beta_i}{2} G_{\mathbb{H}}(u,s_i)}\quad  \mbox{  if }s_1 \neq \infty \\
        (2 {\rm Im} u)^{-\frac{\alpha^2}{2}} |u|_+^{-2\alpha(Q-\alpha)} \prod_{i=2}^m |s_i|_+^{-\beta_i(Q-\frac{\beta_i}{2}-\frac{\beta_1}{2})} \times e^{ \sum_{ 1 \leq i < j \leq m} \frac{\beta_i\beta_j}{4} G_{\mathbb{H}}(s_i,s_j) +  \sum_{i=1}^m \frac{\alpha\beta_i}{2} G_{\mathbb{H}}(u,s_i)}\quad  \mbox{  if }s_1 = \infty.
    \end{cases}
    \end{aligned}
    \end{equation*}
    Here, we use the convention that $G_{\mathbb{H}}(z,\infty) := \lim_{w \rightarrow \infty} G_{\mathbb{H}} (z,w) = 2 \log|z|_+$. 
    
    Sample $(h,  \mathbf c)$ from $C_{\mathbb{H}}^{(\alpha,u),(\beta_i,s_i)_i} P_{\mathbb{H}} \times [e^{(\frac{1}{2} \sum_{i=1}^m \beta_i + \alpha - Q)c} dc]$, and let $\phi(z) = h(z) - 2Q\log|z|_+ +\frac{1}{2} \sum_{i=1}^m \beta_i G_{\mathbb{H}} (z,s_i) + \alpha G_{\mathbb{H}}(z,u) +  \mathbf c$. Then we define ${\rm LF}_{\mathbb{H}}^{(\alpha,u),(\beta_i,s_i)_i}$ as the law of $\phi$. When $\alpha = 0$, we simply write it as ${\rm LF}_{\mathbb{H}}^{(\beta_i,s_i)_i}$.
\end{definition}
The Liouville field ${\rm LF}_{\mathbb{H}}^{(\alpha,u),(\beta_i,s_i)_i}$ can also be identified with $e^{\alpha\phi(u)}\prod_{i=1}^me^{\frac{\beta_i}{2}\phi(s_i)}{\rm LF}_{\mathbb{H}}$, see \cite[Lemma 2.8]{sun2023sle} for the precise statement and proof.

A \emph{quantum surface} is an equivalence class of pairs $(D, h)$ where $D$ is a planar domain and $h$ is a generalized function on $D$.  
For $\gamma\in (0,2)$, we say that 
$(D, h) \sim_\gamma (\wt D, \wt h)$ if there is a conformal map $\psi: \wt D \to D$ such that 
\eqb\label{eq-QS}
\wt h = h \circ \psi + Q \log |\psi'|. 
\eqe
We write $ (D, h)/{\sim_\gamma}$ as the quantum surface corresponding to $(D,h)$. 
An \emph{embedding} of a quantum surface is a choice of its representative.
Both quantum area and quantum length measures are intrinsic to the quantum surface thanks to~\eqref{eq-QS} as shown in~\cite{shef-kpz,shef-wang-lqg-coord}.

We can also consider quantum surfaces decorated with other structures. For example, 
let $n\in \N$ and $\cI$ be an at most countable index set, consider tuples $(D, h, (\eta_i)_{i\in \cI}, z_1,\cdots,z_n)$ such that $D$ is a domain, 
$h$ is a distribution on $D$, $\eta_i$ are loops on $D$ and
$z_i \in D\cup \bdy D$. We say that
$$(D, h, (\eta_i)_{i\in \cI}, z_1,\cdots,z_n )  \sim_\gamma (\wt D, \wt h,(\wt \eta_i)_{i\in \cI}, \wt z_1,\cdots,\wt z_n)$$
if there is a conformal map $\psi: \wt D \to D$ such that~\eqref{eq-QS} holds, $\psi(\wt z_i) = z_i$ for all $1\le i\le n$, and  $\psi\circ \wt \eta_i=\eta_i$ for all $i\in \cI$.
We call an equivalence class defined through ${\sim_\gamma}$ a \emph{decorated quantum surface}, and likewise an embedding of a decorated quantum surface is a choice of its representative.

Fix $\gamma\in (0,2)$, we now recall the quantum area and quantum length measure  in $\gamma$-LQG.
Suppose $h$ is  a GFF sampled from $P_\cX$ for $\cX=\C$ or $\bbH$.
For $\eps > 0$ and $z \in  \cX\cup \bdy \cX$, we write $h_\eps(z)$ for the average of $h$ on $\partial B_\eps(z)  \cap \cX$, and define the random measure $\mu_h^\eps:= \eps^{\gamma^2/2} e^{\gamma h_\eps(z)}d^2z$ on $\cX$, where $d^2z$ is Lebesgue measure on $\cX$. Almost surely, as $\eps \to 0$, the measures $\mu_h^\eps$ converge weakly to a limiting measure $\mu_h$ called the \emph{quantum area measure} \cite{shef-kpz,shef-wang-lqg-coord}. 
For $\cX=\bbH$, we define the \emph{quantum boundary length measure} $\nu_h:= \lim_{\eps \to 0} \eps^{\gamma^2/4}e^{\frac\gamma2 h_\eps(x)} dx$, where $h_\eps(x)$ is the average of $h$ on $\partial B_\eps(x) \cap \bbH$.
The definition of quantum area and boundary length can clearly be extended to other variants of GFF such as the Liouville fields, possibly with  insertions.

\subsection{Quantum sphere and Liouville reflection coefficient}\label{subsub:quantum-surface}
We now recall the two-pointed quantum sphere defined in~\cite{wedges} following the presentation of~\cite{ahs-disk-welding,AHS-SLE-integrability}. Consider the horizontal cylinder 
$\cC$ obtained from  $\R\times [0,2\pi]$ by identifying $(x,0) \sim (x, 2\pi)$. 
Let   $h_\cC(z)=h_\C (e^z)$ for $z\in \cC$ where $h_\C$ be  sampled from $P_\C$.
We call  $h_\cC$ the  GFF on $\cC$ normalized to have mean zero on the circle  $\{\Re z=0\}\cap \cC$.
The field $h_\cC$ can be written as $h_\cC=h^{\op 1}_\cC+h^{2}_\cC$, where 
$h^{\op 1}_\cC$ is constant on vertical  circles $\{\Re z=u\}\cap \cC$  for each $u\in \R$, and $h^{2}_\cC$  has mean zero on all such
circles. We call $h^{2}_\cC$ the \emph{lateral component} of the GFF on $\cC$.

\begin{definition}\label{def-sphere}
	For $\gamma\in (0,2)$, $W>0$ and $\alpha=Q-\frac{W}{2\gamma}$, let $(B_s)_{s \geq 0}$ be a standard Brownian motion  conditioned on $B_{s} - (Q-\alpha)s<0$ for all $s>0$, and $(\wt B_s)_{s \geq 0}$  an independent copy of $(B_s)_{s \geq 0}$. Let 
	\[Y_t =
	\left\{
	\begin{array}{ll}
	B_{t} - (Q -\alpha)t  & \mbox{if } t \geq 0 \\
	\wt B_{-t} +(Q-\alpha) t & \mbox{if } t < 0
	\end{array}
	\right. .\] 
	Let $h^1(z) = Y_{\Re z}$ for each $z \in \cC$.
	Let $h^2_\cC$ be independent of $h^1$  and have the law of the lateral component of the GFF on $\cC$. Let $\hat h=h^1+h^2_\cC$.
	Let  $\mathbf c\in \R$ be sampled from $ \frac\gamma2 e^{2(\alpha-Q)c}dc$ independent of $\hat h$ and set $h=\hat h+\mathbf c$.
	Let $\cM_2^\sph(W)$ be the infinite measure describing the law of the decorated quantum surface  $(\cC, h , -\infty, +\infty)/{\sim_\gamma}$.
\end{definition}
The un-pointed quantum sphere and $n$-pointed quantum sphere are defined from $\cM_2^\sph(4-\gamma^2)$ as follows:
\begin{definition}\label{def-QS}
Let $(\cC, h , -\infty, +\infty)/{\sim_\gamma}$ be a sample from  $\cM_2^\sph(4-\gamma^2)$. We let  $\QS$ be the law of the quantum surface $(\cC, h)/{\sim_\gamma}$ under the reweighted measure $\mu_h(\cC)^{-2}\cM_2^\sph(4-\gamma^2)$. 	
For $n\geq 0$, let $(\cC,h)$ be a sample from $\mu_h(\cC)^n\QS$, and then independently sample $z_1,..,z_n$ according to $\mu_h(\cC)^{\#}$. Let $\QS_n$ be the law of $(\cC,h,z_1,..,z_n)/\sim_\gamma$.
\end{definition}
It is proved in~\cite{wedges} that $\QS_2=\cM_2^\sph(4-\gamma^2)$. Namely  $\QS_2$ is invariant under re-sampling of its two marked points according to the quantum area.

Recall the unit-volume reflection coefficient for LCFT on the sphere~\cite{krv-dozz,rv-tail}
\eqb\label{eq-R}
\ol R(\alpha) := -\left(\frac{\pi\Gamma(\frac{\gamma^2}4)}{\Gamma(1-\frac{\gamma^2}4)}\right)^{\frac2\gamma(Q-\alpha)} \frac{1}{\frac2\gamma(Q-\alpha)} \frac{\Gamma(-\frac\gamma2(Q-\alpha))}{\Gamma(\frac\gamma2(Q-\alpha))\Gamma(\frac2\gamma(Q-\alpha))} .\eqe
\begin{lemma}\label{lem-sph-area-law}
	The law of the quantum area of a sample from $\cM_2^\sph(W)$
 is \[1_{a>0} \frac12 \ol R(\alpha) a^{\frac2\gamma(\alpha - Q) - 1} \, da \quad \textrm{where } \alpha=Q-\frac{W}{2\gamma}.\]
\end{lemma}
\begin{proof} 
	For $0< a < a'$ with $\wh h$ as in Definition~\ref{def-sphere},  we have 
	\[\cM_2^\sph(W)[ \mu_{\wh h + c} (\cC) \in (a, a') ] = \E \left[\int_{-\infty}^\infty \mathbf 1_{e^{\gamma c} \mu_{\wh h}(\cC) \in (a, a')} \frac\gamma2 e^{2(\alpha-Q)c} \, dc \right] = \E\left[\int_a^{a'} \frac\gamma2 \left(\frac y{\mu_{\wh h}(\cC)}\right)^{\frac2\gamma(\alpha - Q)} \frac{1}{\gamma y} \, dy \right]\]
	where we have used the change of variables $y = e^{\gamma c} \mu_{\wh h}(\cC)$. 
	By \cite[Theorem 3.5]{krv-dozz}  and \cite[(1.10)--(1.12)]{rv-tail}, 	
	for $\alpha \in (\frac\gamma2, Q)$ we have \(\E[\mu_{\wh h}(\cC)^{\frac2\gamma(Q-\alpha)}] = \ol R(\alpha)\).	Interchanging the expectation and integral  gives the result.
\end{proof}

It is shown in \cite{ahs-sphere,AHS-SLE-integrability} that the embedding of $\QS_3$ in $\C$ gives a Liouville field on $\C$ with three $\gamma$-insertions modulo an explicit multiplicative constant.
\begin{theorem}[{\cite[Proposition 2.26]{AHS-SLE-integrability}}] \label{thm-QS3-field}
Let $\phi$ be sampled from $\LF_\C^{(\gamma, u_1),(\gamma, u_2),(\gamma, u_3)}$
where $(u_1, u_2, u_3) = (0, 1, e^{i\pi/3})$.  Then the law of the decorated quantum surface $(\C, \phi,u_1,u_2,u_3)/\sim_\gamma$ is  
$\frac{2(Q-\gamma)^2}{\pi \gamma}\QS_3$.
\end{theorem}
Although $\QS_2$ has two marked points,  it also has a nice Liouville field description on the cylinder. 
\begin{definition}
	\label{def-LFC}
Let $P_\cC$ be the law of the GFF  on the cylinder $\cC$ defined above Definition~\ref{def-sphere}.
	Let $\alpha \in \R$. 
	Sample $(h, \mathbf c)$ from $P_\cC \times [e^{(2\alpha-2Q)c} \, dc]$ and let $\phi(z) = h(z) - (Q-\alpha) \left| \Re z \right| + \mathbf c$. We write $\LF_\cC^{(\alpha,\pm\infty)}$ as the law of $\phi$. 
\end{definition}
\begin{theorem}[{\cite[Theorem B.5]{AHS-SLE-integrability}}]\label{thm-sph-field}
Let $h$ be as in Def.~\ref{def-sphere}  so that the law of $(\cC, h, +\infty, -\infty)/{\sim_\gamma}$ is  $\QS_2$.
Let $T\in \R$ be sampled from the Lebesgue measure on $\R$ independently of 	$h$.  
Let $\phi(z)=h (z+T)$.  Then the law of $ \phi$ is given by 
$ \frac{\gamma}{4(Q-\gamma)^2} \LF_\cC^{(\gamma,\pm\infty)}$.
\end{theorem}

One can also use the \emph{uniform embedding} of quantum surfaces; which is defined as follows. Denote $\conf(\hat\C)$ for the conformal automorphism group on $\hat\C$. Then according to the definition of quantum surface, one can view $\QS$ as an (infinite) measure on $H^{-1}(\C)/\conf(\hat\C)$. Let $\sm_{\hat \C}$ be the left and right invariant Haar measure on $\conf(\hat\C)$. We first recall a basic property for $\sm_{\hat\C}$, which we will use frequently in the following sections.
\begin{lemma}[{\cite[Lemma 2.28]{AHS-SLE-integrability}}]\label{Haar-density}
Suppose $\frak f$ is a conformal automorphism sampled from the Haar measure $\sm_{\hat\C}$. Then for any three fixed points $z_1,z_2,z_3\in\hat\C$, the law of $(\frak f(z_1),\frak f(z_2),\frak f(z_3))$ equals $C|(p-q)(q-r)(r-p)|^{-2}\,d^2p\,d^2q\,d^2r$ for some constant $C\in(0,\infty)$.
\end{lemma}

In the following, we will always choose $\sm_{\hat \C}$ is such that the constant $C$ in Lemma \ref{Haar-density} equals $1$. The following theorem says the uniform embedding of $\QS$ gives the Liouville field $\LF_\C$.

\begin{theorem}[{\cite[Theorem 1.2, Proposition 2.32]{AHS-SLE-integrability}}]\label{prop:QS-field}
We have $\sm_{\hat\C}\ltimes\QS=\frac{\pi\gamma}{2(Q-\gamma)^2}\LF_{\C}$.
\end{theorem}

\subsection{Quantum disk and the FZZ formula}\label{subsub:integrable}

First, we follow \cite{ahs-disk-welding} and introduce the thick quantum disk with two boundary insertions. Consider the horizontal strip 
$\cS=\R\times (0,\pi)$. 
Let   $h_\cS(z)=h_\bbH(e^z)$ for $z\in \cS$ where $h_\bbH$ be  sampled from $P_\bbH$.
We call  $h_\cS$ the GFF on $\cS$ normalized to have mean zero on the vertical segment  $\{\Re z=0\}\cap \cS$.
The field $h_\cS$ can be written as $h_\cS=h^{\op 1}_\cS+h^{2}_\cS$, where 
$h^{\op 1}_\cS$ is constant on vertical lines $\{\Re z=u\}\cap \cS$  for each $u\in \R$, and $h^{\op 2}_\cS$  has mean zero on all such
circles. We call $h^{2}_\cS$ the \emph{lateral component} of the GFF on $\cS$. 
\begin{definition}\label{def-disk}
	For $\gamma\in (0,2)$, $W\geq\frac{\gamma^2}{2}$ and $\beta=Q+\frac{\gamma}{2}-\frac{W}{\gamma}$, let $(B_s)_{s \geq 0}$ be a standard Brownian motion  conditioned on $B_{s} - (Q-\beta)s<0$ for all $s>0$, and let $(\wt B_s)_{s \geq 0}$ be an independent copy of $(B_s)_{s \geq 0}$. Let 
	\[Y_t =
	\left\{
	\begin{array}{ll}
	B_{2t} - (Q -\beta)t  & \mbox{if } t \geq 0 \\
	\wt B_{-2t} +(Q-\beta) t & \mbox{if } t < 0
	\end{array}
	\right. .\] 
	Let $h^1(z) = Y_{\Re z}$ for each $z \in \cC$.
	Let $h^2_\cS$ be independent of $h^1$  and have the law of the lateral component of the GFF on $\cS$. Let $\hat h=h^1+h^2_\cS$.
	Let  $\mathbf c\in \R$ be sampled from $ \frac\gamma2 e^{2(\beta-Q)c}dc$ independently of $\hat h$ and set $h=\hat h+\mathbf c$.
	Let $\cM_{0,2}^{\rm disk}(W)$ be the infinite measure  describing the law of the decorated quantum surface  $(\cS, h , -\infty, +\infty)/{\sim_\gamma}$.
	We call a sample from $\cM^{\rm disk}_{0,2}(W)$ a two-pointed \emph{quantum disk} of weight $W$.
\end{definition}
\begin{definition}\label{def-qd}
    Let $(\cS,h, -\infty, +\infty)$ be the embedding of a sample from $\mathcal{M}_{0,2}^{\rm disk}(2)$ as in the above definition. We write $\mathrm{QD}$ for the law of $(\cS, h) /{\sim_\gamma}$ under the reweighted measure $\nu_h(\partial \cS)^{-2}\mathcal{M}_{0,2}^{\rm disk}(2)$. For $m,n\geq 0$, let $(\cS,h)$ be a sample from $\mu_h(\cS)^m\nu_h(\partial\cS)^n\QD$, and independently sample $z_1,..,z_m$ and $s_1,..,s_n$ according to $\mu^{\#}_h$ and $\nu^{\#}_h$ respectively. Let $\QD_{m,n}$ be the law of $(\cS, h,z_1,...,z_m,s_1,...,s_n) / {\sim_\gamma}$. We call a sample from $\QD_{m,n}$ a \emph{quantum disk} with $m$ bulk points and $n$ boundary points.
\end{definition}
By \cite{wedges}, $\QD_{0,2}=\cM_{0,2}^{\rm disk}(2)$, which means the marked points on $\cM_{0,2}^{\rm disk}$ is quantum typical.
\begin{proposition}[{\cite[Lemma 3.2]{ARS-FZZ}}]\label{prop-QD}  
The quantum boundary length law of $\QD$ is $R_\gamma\ell^{-\frac{4}{\gamma^2}-2}d\ell$,  where $$R_\gamma:=\frac{(2 \pi)^{\frac{4}{\gamma^2}-1}}{\left(1-\frac{\gamma^2}{4}\right) \Gamma\left(1-\frac{\gamma^2}{4}\right)^{\frac{4}{\gamma^2}}}$$
\end{proposition}
When $0<W<\frac{\gamma^2}{2}$, we define the thin quantum disk as the concatenation of weight $\gamma^2-W$ thick disks with two marked points as in  \cite[Section 2]{ahs-disk-welding}.

\begin{definition}\label{def-thin-disk}
 For $W \in\left(0, \frac{\gamma^2}{2}\right)$, we can define the infinite measure $\mathcal{M}_{0,2}^{\rm disk}(W)$ on the space of two-pointed beaded surfaces as follows. First, sample a random variable $T\sim\left(1-\frac{2}{\gamma^2} W\right)^{-2} \operatorname{Leb}_{\mathbb{R}_{+}}$; then sample a Poisson point process $\left\{\left(u, \mathcal{D}_u\right)\right\}$ from the intensity measure $\operatorname{Leb}_{[0,T]}\times \mathcal{M}_{0,2}^{\rm disk}\left(\gamma^2-W\right)$; and finally consider the ordered (according to the order induced by u) collection of doubly-marked thick quantum disks $\left\{\mathcal{D}_u\right\}$, called a thin quantum disk of weight $W$.
\end{definition}

Before discussing quantum disk with bulk insertions. We first recall the law of the quantum boundary length  under $\LF_\bbH^{(\alpha, i)}$ obtained in~\cite{remy-fb-formula}  following the presentation of~\cite[Proposition 2.8]{ARS-FZZ}.
\begin{proposition}[{\cite{remy-fb-formula}}]\label{prop-remy-U}
For $\alpha > \frac{\gamma}{2}$,  the law of the quantum length $\nu_\phi(\R)$ under $\LF_\bbH^{(\alpha, i)}$ is 
$$1_{\ell>0} \frac2\gamma 2^{-\frac{\alpha^2}2} \ol U(\alpha)\ell^{\frac2\gamma(\alpha-Q)-1} \, d\ell$$ 
where
	\eqb\label{eq:U0-explicit}
	\ol U(\alpha) = \left( \frac{2^{-\frac{\gamma\alpha}2} 2\pi}{\Gamma(1-\frac{\gamma^2}4)} \right)^{\frac2\gamma(Q-\alpha)} 
	\Gamma( \frac{\gamma\alpha}2-\frac{\gamma^2}4).
	\eqe
\end{proposition}
Let $\{\LF_\bbH^{(\alpha, i)}(\ell): \ell  \}$ be the disintegration of $\LF_\bbH^{(\alpha, i)}$  over $\nu_\phi(\R)$.
Namely for each non-negative measurable function $f$ on $(0,\infty)$ and $g$ on $H^{-1}(\bbH)$, 
\begin{equation}\label{eq:field-dis}
\LF_\bbH^{(\alpha, i)} [f(\nu_\phi(\R))g(\phi) ]  =   \int_0^\infty f(\ell) \LF_\bbH^{(\alpha, i)}(\ell) [g(\phi)] \, d\ell.
\end{equation} 
Although the general theory of disintegration  only defines $ \LF_\bbH^{(\alpha, i)}(\ell)$ for almost every $\ell\in (0,\infty)$, 
the following lemma describes a canonical version of $\LF_\bbH^{(\alpha, i)} (\ell)$ for every $\ell>0$.
\begin{lemma}[{\cite[Lemma 4.3]{ARS-FZZ}}]\label{lem:field-disk}
		Let $h$ be a sample from $P_\bbH$ and $\hat h(\cdot)=  h(\cdot) -2Q \log \left|\cdot\right|_+ +\alpha G_\bbH(\cdot,z)$.  
		The law of $\hat h+\frac{2}{\gamma}\log \frac\ell{\nu_{\wh h}(\R)}$ under the reweighted measure $2^{-\alpha^2/2} \frac2\gamma \ell^{-1} \left(\frac\ell{\nu_{\wh h}(\R)}\right)^{\frac2\gamma(\alpha-Q)}  P_\bbH$ is a version of the disintegration $\{\LF_\bbH^{(\alpha, i)}(\ell): \ell >0 \}$.
\end{lemma}

We now recall the quantum disk with one generic bulk insertion. 
\begin{definition}[{\cite{ARS-FZZ}}]\label{def-QD-alpha}
Fix $\alpha>\frac{\gamma}2$ and  $\ell>0$. We let $\cM_{1,0}^\disk(\alpha;\ell )$ be the law of $(\bbH, \phi, i)/{\sim_\gamma}$ 
where $\phi$ is sampled from $\LF_\bbH^{(\alpha, i)} (\ell)$. We let $\cM_{1,1}^\disk(\alpha;\ell )$ be the law of $(\bbH, \phi, i,s)/{\sim_\gamma}$ where $\phi$ is sampled from $\ell\LF_\bbH^{(\alpha, i)} (\ell)$ and $s$ is sampled from probability measure proportional to quantum length measure. 
Finally, write $\cM_{1,0}^\mathrm{disk}(\alpha) = \int_0^\infty \cM_{1,0}^\mathrm{disk}(\alpha; \ell)\, d\ell$ and $\cM_{1,1}^\mathrm{disk}(\alpha) = \int_0^\infty \cM_{1,1}^\mathrm{disk}(\alpha; \ell)\,d\ell$.
\end{definition}
One can fix the embedding of $\cM^{\disk}_{1,0}(\alpha;\ell)$ as the following:
\begin{proposition}\label{lem:har}
	For $\alpha > \frac\gamma2$ and $\ell > 0$, let $(D,h,z)$ be an embedding of a sample from $\cM^{\disk}_{1,0}(\alpha;\ell)$. Given $(D,h,z)$, let $p$ be a point sampled from the harmonic measure on $\bdy D$ viewed from $z$, then the law of  $(D,h,z,p)/{\sim_\gamma}$ equals that of $(\bbH, X, i,0)/{\sim_\gamma}$ where $X$ is sampled from $\LF_{\bbH}^{(\alpha,i)}(\ell)$.
\end{proposition}
\begin{proof}
	We assume that $(D,h,z)=(\bbH, h,i)$ where $h$ is a sample from   $\LF_{\bbH}^{(\alpha,i)}(\ell)$. 
	Let $\psi_p: \bbH\to \bbH$ be the conformal map  with $\psi_p(i) = i$ and $\psi_p(p) = 0$ and set $X=h \circ \psi_p^{-1} + Q \log |(\psi_p^{-1})'|$.
	Then by the coordinate change for Liouville fields on $\bbH$ (see e.g.\ \cite[Lemma 2.4]{ARS-FZZ}),  
	the law of $X$ is also $\LF_\bbH^{(\gamma, i)}(\ell)$. Since $(D,h,z,p)/{\sim_\gamma}=(\bbH, X, i,0)/{\sim_\gamma}$ we are done.
\end{proof}

\begin{theorem}[{\cite[Theorem 3.4]{ARS-FZZ}}]\label{thm:def-QD}
Fix $\ell>0$. Let  $\QD_{1,0}(\ell)$ be the law of the quantum disk with boundary length $
\ell$ and with one interior marked point defined in \ref{def-qd}. Then 
$$\QD_{1,0}(\ell) = \frac\gamma{2\pi (Q-\gamma)^2}\cM_{1,0}^\disk(\gamma;\ell ).$$
\end{theorem}

The FZZ formula is the analog of the DOZZ formula for  $\LF^{ (\alpha,z)}_\bbH$~proposed in~\cite{FZZ}  and proved in~\cite{ARS-FZZ}.	
We record the most convenient form  for our purpose, which uses the modified Bessel function of the second kind $K_\nu(x)$ \cite[Section 10.25]{NIST:DLMF}. 
One concrete representation of $K_\nu(x)$ in the range of our interest is the following \cite[(10.32.9)]{NIST:DLMF}:
\eqb\label{eq-Kv}
K_\nu(x) := \int_0^\infty e^{-x \cosh t} \cosh(\nu t) \, dt \quad \text{ for } x > 0 \text{ and } \nu \in \R.
\eqe
\begin{theorem}[{\cite[Theorem 1.2,  Proposition 4.19]{ARS-FZZ}}]\label{thm-FZZ}
	For $\alpha \in (\frac\gamma2, Q)$ and $\ell>0$,  let $A$ be the quantum area of a sample from $\cM_{1,0}^\disk(\alpha; \ell)$. 
	The law of $A$ under $\cM_{1,0}^\disk(\alpha; 1)^\#$ (i.e.\ the probability measure proportional to $\cM_1^\disk(\alpha; 1)$) is the inverse gamma distribution with density \[1_{x>0} \frac{(4 \sin \frac{\pi\gamma^2}4)^{-\frac2\gamma(Q-\alpha)}}{\Gamma(\frac2\gamma(Q-\alpha))} x^{-\frac2\gamma(Q-\alpha)-1} \exp\left(-\frac{1}{4x \sin\frac{\pi\gamma^2}4}\right). \]
	Moreover, recall  $\ol U(\alpha)$ from  in Proposition~\ref{prop-remy-U}.	Then for $\mu>0$ we have
	\[\cM_{1,0}^\disk(\alpha; \ell)[e^{-\mu A}] = \frac2\gamma 2^{-\frac{\alpha^2}2} \ol U(\alpha) \ell^{-1} \frac2{\Gamma(\frac2\gamma(Q-\alpha))} \left(\frac12 \sqrt{\frac{\mu}{\sin(\pi\gamma^2/4)}}  \right)^{\frac2\gamma(Q-\alpha)}K_{\frac2\gamma(Q-\alpha)} \left(\ell\sqrt{\frac{\mu}{\sin(\pi\gamma^2/4)}}  \right). \]
\end{theorem}

\subsection{Forested quantum surfaces and generalized quantum disks}

In this section, we review the notion of forested quantum surfaces, and introduce a particularly important example called the generalized quantum disk. Our presentation mainly follows \cite{nonsimple-welding}, and we refer to \cite{wedges,msw-non-simple,hl-lqg-cle,nonsimple-welding} for more backgrounds.

We first give a generalization of the notion of quantum surface defined in~\eqref{eq-QS}.
Consider pairs $(D,h)$ where $D \subset \C$ is now a closed set (not necessarily homeomorphic to a closed disk) such that each component of its interior together with its prime-end boundary is homeomorphic to the closed disk, and $h$ is only defined as a distribution on each of these components. We define the equivalence relation $\sim_\gamma$ such that $(D,h)\sim_\gamma(D',h')$ iff there is a homeomorphism $g:D\to D'$ that is conformal on each component of the interior of $D$, and $h'=h\circ g^{-1}+Q\log|(g^{-1})'|$. A {\it beaded quantum surface} $S$ is then defined to be an equivalence class of pairs $(D,h)$ under the above $\sim_\gamma$, and we say $(D, h)$ is an embedding of $S$ if $S = (D,h)/{\sim_\gamma}$. Generalized quantum surfaces with marked points and curve-decorated generalized quantum surfaces can be defined analogously. Later, in Section~\ref{subsec:MSW2}, we give a precise definition of the generalized surface with annulus topology (in some component), called the generalized quantum annulus.

\begin{definition}[Forested line]
    For $\gamma\in (\sqrt{2},2)$, let $(X_s)_{s\geq 0}$ be a stable L\'evy process of index $\frac{4}{\gamma^2}$ with only positive jumps satisfying $X_0=0$. For $t>0$, let $Y_t=\inf\{s>0:X_s\leq -t\}$. Fix the root $o=(0,0)$, and define the forested line as follows: First, construct the looptree corresponding to $X_t$. Then, for each loop of length $L$, we independently sample a quantum disk from $\QD_{0,1}(L)^\#$ and topologically identify its boundary with the loop, identifying its marked point with the root of the loop.

The closure of the collection of the points on the boundaries of quantum disks is referred as the forested boundary arc, while the set of the points corresponding to the running infimum of $X_t$ is called the line boundary arc. 
For a point $p_t$ on the line boundary arc corresponding to the point at which $X$ first takes the value $-t$, the \emph{quantum length} between $o$ and $p_t$ is defined to be $t$. The \emph{generalized boundary length} between two points on the forested boundary arc is defined to be the length of the corresponding time interval for $(X_s)_{s\ge 0}$. 
We denote the forested line by $\cL^o$.
\end{definition}
Define the truncation of $\cL^o$ at quantum length $t$ to be the union of the line boundary arc and the quantum disks on the forested boundary arc between $o$ and $p_t$.

\begin{definition}\label{def-forested-line-segment}
    Fix $\gamma\in (\sqrt{2},2)$. Define $\cM_2^{\rm f.l.}$ as the law of the surface obtained by first sampling $\bf{t}\sim {\rm Leb}_{\R^+}$ and truncating an independent forested line at quantum length ${\bf t}$. We define the disintegration of the measure $\cM_2^{\rm f.l.}$ : $$\cM_2^{\rm f.l.}=\int_{\R_+^2} \cM_2^{\rm f.l.}(t,\ell)dtd\ell$$
    where $\cM_2^{\rm f.l.}(t,\ell)$ is the measure on the forested line segments with quantum length $t$ for the line boundary arc and generalized boundary length $\ell$ for the forested boundary arc.
    We also define $\cM_2^\mathrm{f.l.}(t, \cdot) := \int_0^\infty \cM_2^\mathrm{f.l.}(t, \ell)\, d\ell$.
\end{definition}
Note that since the $\cM_2^\mathrm{f.l.}$-law of $\mathbf t$ is Lebesgue measure, and the disintegration above was with respect to Lebesgue measure, we have $|\cM_2^\mathrm{f.l.}(t, \cdot)| = 1$, i.e., the law of the forested line segment with quantum length $t$ is a probability measure.
\begin{definition}\label{def-forested-2disk}
Let $W>0$. 
    Consider a sample \[(\cL^1, \cD, \cL^2) \sim \int_{\R_+^2}\cM_2^{\rm f.l.}(t_1,\cdot)\times\cM^{\rm disk}_{0,2}(W; t_1, t_2)\times \cM_2^{\rm f.l.}(t_2,\cdot)dt_1dt_2.\] Glue the line boundary arcs of $\cL^1$ and $\cL^2$ to the left and right boundary arcs of $\cD$ according to quantum length. We call the resulting forested quantum surface a \emph{forested two-pointed quantum disk of weight $W$}, and denote its law by $\cM_{0,2}^\mathrm{f.d.}(W)$.
\end{definition}

The special weight $W = \gamma^2-2$ exhibits the following remarkable symmetry. A sample from $\cM_{0,2}^\mathrm{f.d.}(\gamma^2-2)$ is a generalized quantum surface with two marked boundary points. Given the surface, forget the two marked points and sample two new marked points independently from the generalized boundary length measure. Then the law of the resulting marked generalized quantum surface is still $\cM_{0,2}^\mathrm{f.d.}(\gamma^2-2)$ \cite[Proposition 3.16]{nonsimple-welding}. This motivates the following analog of the quantum disk (Definition~\ref{def-qd}).

\begin{definition}\label{def-GQD}
Let $\gamma \in (\sqrt2, 2)$. 
    Let $\GQD_{0,2}: = \cM_{0,2}^\mathrm{f.d.}(\gamma^2-2)$. Denoting the quantum area and generalized boundary length of a generalized quantum surface by $A$ and $L$, a \emph{generalized quantum disk} is a sample from the weighted measure $L^{-2} \GQD_{0,2}$ with boundary points forgotten; denote its law by $\GQD$. For $m,n \geq 0$, sample a generalized quantum surface from $A^m L^n \GQD$, and given this surface sample $m$ bulk points and $n$ boundary points independently from the probability measures proportional to the quantum area and generalized boundary length measures. The result is a \emph{generalized quantum disk with $m$ bulk points and $n$ boundary points}. Denote its law by $\GQD_{m,n}$.
\end{definition}
We define $\GQD_{m,n}(\ell)$ to be the disintegration of $\GQD_{m,n}$ over the generalized boundary length. That is, $\GQD(\ell) = \int_0^\infty \GQD_{m,n}(\ell) \, d\ell$ where each measure $\GQD_{m,n}(\ell)$ is supported on the space of generalized quantum surfaces with generalized boundary length $\ell$.

\begin{lemma}\label{length gqd}
    There exists a constant $R_\gamma'>0$ such that the law of the generalized boundary length of a sample from $\GQD$ is $R_\gamma' 1_{\ell > 0} \ell^{-2 - \frac{\gamma^2}4}\, d\ell$.
\end{lemma}
\begin{proof}
    By Definition~\ref{def-GQD} and \cite[Lemma 3.8]{nonsimple-welding} the generalized boundary length law of a sample from $\GQD_{0,2}$ is a multiple of $1_{\ell > 0} \ell^{-\frac{\gamma^2}4}\, d\ell$. Applying Definition~\ref{def-GQD} again gives the result for $\GQD$. 
\end{proof}

   As explained in \cite[Remark 3.12]{nonsimple-welding}, the definition of generalized quantum disk here agrees with that of \cite[Definition 5.8]{hl-lqg-cle}, in the sense that $\GQD(\ell)^\#$ is what they call the law of the length $\ell$ generalized quantum disk.   We now state some area statistics of $\GQD(\ell)^\#$, which are derived from \cite[Theorem 1.8]{hl-lqg-cle}. We write $\kappa'=\frac{16}{\gamma^2}$ and  define the constant $M':=2\left(\frac{\mu}{4\sin\frac{\pi\gamma^2}{4}}\right)^{\frac{\kappa'}{8}}$ for simplicity.
\begin{prop}{\cite[Theorem 1.8]{hl-lqg-cle}}\label{prop:area-law-gqd}
The quantum area $A$ of a sample from $\GQD(\ell)^\#$ satisfies
\begin{equation*}
\GQD(\ell)^\#[e^{-\mu A}]=\bar K_{4/\kappa'}\left(\ell M'\right), \qquad 
\GQD^\#(\ell)[Ae^{-\mu A}]=2\frac{\kappa'}{4\mu\Gamma(\frac{4}{\kappa'})}\left(\frac{M'\ell}{2}\right)^{4/\kappa'+1}K_{1-\frac{4}{\kappa'}}(M'\ell)
\end{equation*}
where here $\bar K_\nu(x):=\frac{2^{1-\nu}}{\Gamma(\nu)}x^\nu K_\nu(x)$ is the normalization of the   modified second Bessel function such that $\lim_{x\to 0}\bar K_\nu(x)=1$.  In particular, we have $\GQD(\ell)^\#[A]=\frac{\kappa'}{16\sin\frac{\pi\gamma^2}{4}}\frac{\Gamma(1-\frac{4}{\kappa'})}{\Gamma(\frac{4}{\kappa'})}\ell^{\frac{8}{\kappa'}}$.
\end{prop}
\begin{lemma}\label{bulk1pointgl}
The generalized boundary length law of $\GQD_{1,0}$ is $R_\gamma'\frac{\kappa'}{16\sin\frac{\pi\gamma^2}{4}}\frac{\Gamma(1-\frac{4}{\kappa'})}{\Gamma(\frac{4}{\kappa'})}\ell^{\frac{4}{\kappa'}-2}d\ell$, where $R_\gamma'$ is the constant in Lemma~\ref{length gqd}.
\end{lemma}
\begin{proof}
Denoting quantum area by $A$ and  generalized boundary length by $L$, 
\begin{equation*}
\GQD_{1,0}[1_{L\in(\ell_1,\ell_2)}]=\GQD[A1_{L\in(\ell_1,\ell_2)}]=\int_{\ell_1}^{\ell_2}|\GQD(\ell)|\GQD(\ell)^\#[A]d\ell
\end{equation*}
hence $|\GQD_{1,0}(\ell)|=|\GQD(\ell)|\GQD(\ell)^\#[A]$, and the result follows from Proposition~\ref{prop:area-law-gqd}. 
\end{proof}

In the remainder of this section, we will describe decompositions of forested quantum surfaces in terms of $\GQD$; informally, these are statements about self-similarity of fractals at different scales. 

The forested line can be described as a Poisson point process of generalized quantum disks. 

\begin{proposition}[{\cite[Proposition 3.11]{nonsimple-welding}}]\label{prop:fr-ppp}
    Sample a forested line, and consider the collection of pairs $(u, \cD_u^f)$ such that $\cD_u^f$ is a generalized quantum surface attached to the line boundary arc (with the root of $\cD_u^f$ defined to be the attachment point) and $u$ is the quantum length from $o$ to the root of $\cD_u^f$. Then the law of this collection is a Poisson point process with intensity measure $c 1_{u > 0} \, du \times \GQD_{0,1}$ for some constant $c$.
\end{proposition}
Thus, a sample from $\cM_{0,2}^\mathrm{f.d.}(W)$ can be obtained from a sample from $\cM_{0,2}^\mathrm{disk}(W)$ by adding a Poisson point process of generalized quantum disks to its boundary according to quantum length measure. More generally, given a quantum surface with boundary, one can \emph{forest its boundary} by adding a Poisson point process of generalized quantum disks to its boundary according to quantum length measure as in Proposition~\ref{prop:fr-ppp}. 

We now define some forested quantum surfaces obtained by foresting the boundary of a quantum disk with generic bulk point (Definition~\ref{def-QD-alpha}).

\begin{definition}\label{def:forested-disk}
Let $\alpha > \frac\gamma2$. 
Sample $\cD \sim \cM_{1,0}^\mathrm{disk}(\alpha)$ and let $\cD^f$ be the generalized quantum surface obtained by foresting the boundary of $\cD$. We denote the law of $\cD^f$ by $\cM_{1,0,0}^\mathrm{f.d.}(\alpha)$. Letting $L$ be the generalized boundary length of a random surface, consider a sample from $L\cM_{1,0}^\mathrm{f.d.}(\alpha)$, and conditioned on this generalized quantum surface sample a boundary point from the probability measure proportionate to generalized quantum length. Denote the law of the resulting two-pointed surface by $\cM_{1,0,1}^\mathrm{f.d.}
(\alpha)$. Finally, let $\cM_{1,1,0}^\mathrm{f.d.}(\alpha)$ be the law of a sample from $\cM_{1,1}^\mathrm{disk}(\alpha)$ with forested boundary. 
\end{definition}
The three subscripts of $\cM^{\rm f.d.}_{i,j,k}(\alpha)$ indicate the quantities of various kinds of marked points: a sample has $i$ marked bulk points, $j$ marked points on the boundaries of connected components of the generalized quantum surface, and $k$ marked points on the forested boundary.

\begin{proposition}\label{prop:gqd-fd}
Let $\alpha > \frac\gamma2$. 
Sample a pair of beaded quantum surfaces $(\cD_1, \cD_2)\sim\cM_{1,1,0}^{\rm f.d.}(\alpha) \times \GQD_{0,2}$, identify the marked boundary  point of $\cD_1$ with the first marked boundary point of $\cD_2$, then unmark this point. This gives a beaded  quantum surface with one marked bulk point and one marked boundary point, whose law is $C \cM_{1,0,1}^\mathrm{f.d.} (\alpha)$ for some constant $C$.  
\end{proposition}
\begin{proof}
    Recall in Proposition \ref{prop:fr-ppp} that the forested line is a Poisson point process with intensity measure $c 1_{u > 0} \, du \times \GQD_{0,1}$. Let $\mathcal{L}$ be a loop with length $L$, and $\Leb_{\mathcal{L}}$ be the Lebesgue measure on $\mathcal{L}$.
    Then, by Palm's formula \cite[Lemma 3.5]{hl-lqg-cle}, the following two procedures yield the same law: 
    \begin{enumerate}
        \item Sample a Poisson point process with intensity $c \Leb_{\mathcal{L}} \times \GQD_{0,1}$, then sample one point according to the generalized quantum length (this induces a weighting by the total generalized quantum length);
        \item Sample $x\sim c\Leb_{\mathcal{L}}$, and independently sample $\GQD_{0,1}$, identifying its marked point with $x$, and sample one point according to the generalized quantum length (i.e. we sample $\GQD_{0,2}$ and identifying one marked point with $x$). Then independently sample a forested line on $[0,L]$, with its two ends identified with $x$.
    \end{enumerate}
    Since $\cM_{1,0}^{\rm disk}(\gamma)=\int_0^\infty\cM_{1,0}^{\rm disk}(\gamma;\ell)d\ell$, the result follows by Definition \ref{def:forested-disk}.
\end{proof}

\begin{proposition}\label{prop:101=gqd11}

    We have $\cM_{1,0,1}^\mathrm{f.d.}(\gamma) = C^{\rm f.d.} \GQD_{1,1}$ for some constant $C^{\rm f.d.} = C^\mathrm{f.d.}(\gamma)$.
\end{proposition}
\begin{proof}
    By Definition~\ref{def-thin-disk} and the fact that $\cM_{0,2}^\mathrm{disk}(2) = \QD_{0,2}$ (Definition~\ref{def-qd}), a sample from $\cM_{0,2}^\mathrm{disk}(\gamma^2 -2)$ can be obtained by first sampling $T \sim (\frac4{\gamma^2}-1)^{-2} \mathrm{Leb}_{\R_+}$, then sampling a Poisson point process $\{(u, \cD_u)\}$ from the intensity measure $\mathrm{Leb}_{[0,T]} \times \QD_{0,2}$ and finally concatenating the collection $\{\cD_u\}$ according to the ordering induced by $u$. 

    Now, writing $A$ to denote the quantum area of a quantum surface, let $\cM_{1,2}^\mathrm{disk}(\gamma^2-2)$ be the law of a sample from the (area-weighted) measure $A\cM_{0,2}^\mathrm{disk}(\gamma^2-2)$ with a marked bulk point sampled according to the probability measure proportional to the quantum area measure. Using the Poissonian description of $\cM_{0,2}^\disk(\gamma^2-2)$ stated above, a sample from $\cM_{1,2}^\mathrm{disk}(\gamma^2-2)$ can be obtained as follows: sample $(\cD_1, \cD^\bullet, \cD_2) \sim (\frac4{\gamma^2}-1)^2 \cM_{0,2}^\mathrm{disk}(\gamma^2-2) \times \QD_{1,2} \times \cM_{0,2}^\mathrm{disk}(\gamma^2-2)$, identify the first marked points of $\cD_1$ and $\cD_2$ with the boundary marked points of $\cD^\bullet$, and forget these marked points. See \cite[Proposition 4.4]{ahs-disk-welding} for details; there, they add a boundary rather than a bulk point, but the argument is identical. 

Next, for a sample from $\GQD_{1,2}$, let $E$ be the event that the bulk point lies on the chain of disks connecting the two boundary points. By Definition~\ref{def-GQD}, the law of a sample from $\cM_{1,2}^\mathrm{disk}(\gamma^2-2)$ with boundary forested is $1_E \GQD_{1,2}$. Thus, a sample from $1_E\GQD_{1,2}$ can be obtained as follows: sample $(\cD_1, \cD^\bullet, \cD_2) \sim \frac2{\pi \gamma^2} \GQD_{0,2} \times \int_0^\infty \cM_{1,0}^\mathrm{disk}(\gamma; \ell) \ell^2 \, d\ell  \times \GQD_{0,2}$, forest the boundary of $\cD^\bullet$, and identify the first marked points of $\cD_1$ and $\cD_2$ with points on $\partial \cD^\bullet$ sampled independently from the probability measure proportional to boundary length measure. Here, we identified the foresting of $\QD_{1,0}$ with $\frac{\gamma}{2\pi (Q-\gamma)^2} \cM_{1,0,0}^\mathrm{f.d.}(\gamma)$ via Theorem~\ref{thm:def-QD}.
The forested boundary of $\cD^\bullet$ is a Poisson point process of generalized quantum disks (Proposition~\ref{prop:fr-ppp}), so using Palm's theorem for Poisson point processes, we can pass from our description of $1_E \GQD_{1,2}$ to the desired description of $\GQD_{1,1}$, with  $C^\mathrm{f.d.} = \frac{c \pi \gamma^2}2$ where $c$ is the constant in Proposition~\ref{prop:fr-ppp}. 
\end{proof}

\subsection{Conformal welding of quantum surfaces}\label{sec-conf-weld}

We first recall  the notion of \emph{conformal welding}. 
For concreteness,  suppose $S_1$ and $S_2$ are two oriented Riemann surfaces, both of which are 
conformally equivalent to a planar domain whose boundary consists of finite many disjoint circles.
For $i=1,2$, suppose $B_i$ is a boundary component of $S_i$ and $\nu_i$ is a finite length measure on $B_i$ with the same total length.  
Given an oriented Riemann surface $S$ and  a simple loop $\eta$ on $S$ with a length measure $\nu$, 
we call $(S,\eta,\nu)$ a \emph{conformal welding} of $(S_1,\nu_1)$ and $(S_2,\nu_2)$ if
the two  connected components of $S\setminus \eta$ with their  orientations inherited from $S$ are  conformally equivalent to $S_1$ and $S_2$, and moreover, 
both $\nu_1$ and $\nu_2$ agree with $\nu$.

We now introduce uniform conformal welding.
Suppose  $(S_1,\nu_1)$ and $(S_2,\nu_2)$ from the previous paragraph are such that for each $p_1\in B_1$ and $p_2\in B_2$, modulo conformal automorphism 
there exists a unique conformal welding identifying $p_1$ and $p_2$. Now, let $\mathbf p_1\in B_1$ and $\mathbf p_2\in B_1$ be independently sampled from the probability measures proportional to $\nu_1$ and  $\nu_2$, respectively. We call the conformal welding of $(S_1,\nu_1)$ and $(S_2,\nu_2)$ with $\mathbf p_1$ identified with $\mathbf p_2$ their  \emph{uniform conformal welding}.

Suppose that $\cD_1$ and $\cD_2$ are a pair of independent $\gamma$-LQG quantum surfaces, each having a distinguished boundary arc having the same boundary length; for example, they might be sampled from $\cM_{1,1}^\mathrm{disk}(\alpha_1; \ell) \times \cM_{1,1}^\mathrm{disk}(\alpha_2;\ell)$ for some fixed $\ell$. Viewed as oriented Riemann surfaces with the quantum length measure, there is almost surely a unique conformal welding of $\cD_1$ and $\cD_2$ identifying their marked boundary points. The existence is due to Sheffield's work~\cite{shef-zipper} which gives a conformal welding whose welding interface is locally absolutely continuous with respect to $\SLE_{\gamma^2}$, and the uniqueness follows from the conformal removability of $\SLE_\kappa$ for $\kappa \leq 4$ (see \cite{jones-smirnov-removability} for $\kappa < 4$ and \cite{kms-sle4-removability} for $\kappa=4$). In the same manner, the uniform conformal welding of independent quantum surfaces with the same quantum boundary length almost surely exists and is unique. 
In many cases, the law of the resulting curve-decorated quantum surface is exactly identified, see e.g.\ Theorem~\ref{thm-loop2}. 

Now consider $\gamma \in (\sqrt2, 2)$ and $\kappa' = 16/\gamma^2$ (so $\SLE_{\kappa'}$ is non-simple but not space-filling). 
The situation is similar for independent \emph{forested} quantum surfaces $\cD_1, \cD_2$ which have distinguished boundary arcs of the same generalized boundary length. By \cite{wedges}, there exists a conformal welding measurable with respect to $(\cD_1, \cD_2)$ where the welding interface is locally absolutely continuous with respect to $\SLE_{\kappa'}$. In this paper, we will always consider this conformal welding for forested quantum surfaces. 

At the present moment, the uniqueness of conformal welding of forested quantum surfaces is known for  $\gamma \geq 1.688$ (via conformal removability for $\SLE_{\kappa'}$ for $\kappa' \in (4, 5.6158...]$ \cite{kms-nonsimple-removability}), but for the restricted subset of conformal weldings such that the welding interface is locally absolutely continuous with respect to $\SLE_{\kappa'}$, uniqueness is known for all values of $\gamma \in (\sqrt2,2)$ \cite{mmq-welding}.

\section{The SLE loop via conformal welding of quantum disks}\label{section sle conformal welding}
In this section, we recall from \cite{ahs-sle-loop} that the conformal welding of quantum disks gives a quantum sphere decorated by an $\SLE_\kappa$ loop (Theorem~\ref{thm-loop2}), and prove an analogous result for $\kappa\in (4,8)$ (Theorem~\ref{loop weld ns})  where  we weld the generalized quantum disks.  This is the starting point of our  study of the SLE loop measure via LQG.

We first recall the definition of Zhan's loop measure $\SLE_\kappa^\mathrm{loop}$ for $\kappa \in (0,8)$ as in~\cite[Theorem~4.2]{zhan-loop-measures}. Given two distinct points $p,q\in\C$, \emph{two-sided whole plane} $\SLE_\kappa$ \emph{from $p$ to $q$} is a pair of curves $(\eta_1, \eta_2)$ sampled as follows. Let $\eta_1$ be a whole-plane $\SLE_\kappa(2)$ from $p$ to $q$. When $\kappa \in (0,4]$ the region $\wh \C \backslash \eta_1$ is simply connected; conditioned on $\eta_1$ let $\eta_2$ be chordal $\SLE_\kappa$ in $\wh \C \backslash \eta_1$ from $q$ to $p$. When $\kappa \in (4,8)$ there are countably many connected components of $\wh \C \backslash \eta_1$, which come in three types: those surrounded by $\eta_1$ on its left, those surrounded by $\eta_1$ on its right, and the remaining regions. For each region $D$ of the third type, independently sample chordal $\SLE_{\kappa}$ from the last point of $\partial D$ hit by $\eta$ to the first point of $\partial D$ hit by $\eta$, and let $\eta_2$ be the concatenation of these curves. We denote the law of $(\eta_1, \eta_2)$ by $\SLE^{p \rightleftharpoons q}_\kappa$. 

By concatenating the curves and forgetting the marked points, $\SLE^{p \rightleftharpoons q}_\kappa$ can be viewed as a measure on the space of loops in $\C$. Given a loop $\eta$ sampled from $\SLE^{p \rightleftharpoons q}_\kappa$, let
$$\cont(\eta):=\lim_{\varepsilon\to0}\varepsilon^{\frac{\kappa}{8}-1}{\rm Area}(\{z:{\rm dist}(z,\eta)<\varepsilon\})$$
be the $(1+\frac{\kappa}{8})$-dimensional Minkowski content of $\eta$.
By~\cite{lawler-rezai-nat}, $\cont(\eta)$ exists almost surely.
\begin{definition}\label{def:loop}
	For $\kappa\in(0,8)$, the \emph{SLE loop measure} $\SLE^{\mathrm {loop}}_\kappa$ is  
 given by
	\begin{equation}\label{eq:zhan-loop}
	\SLE_\kappa^\mathrm{loop}= \cont(\eta)^{-2} \iint_{\C\times \C}  |p-q|^{-2(1-\frac\kappa8)} \SLE^{p \rightleftharpoons q}_\kappa(d\eta)\,  d^2p\, d^2q.
	\end{equation}
\end{definition}

We now review the conformal welding result established in~\cite[Theorem 1.1]{ahs-sle-loop} for  $\SLE^{\mathrm {loop}}_\kappa$ with $\kappa\in (0,4)$.
Let $\gamma=\sqrt{\kappa}\in (0,2)$. 
Recall the measures $\QS$ and $\QD(\ell)$ for $\ell > 0$ from Section~\ref{sec:prelim} 
that correspond to variants of the quantum sphere and disk in $\gamma$-LQG. 
Let $\cD_1$ and $\cD_2$ be quantum surfaces sampled from $\QD(\ell) \times \QD(\ell)$. 
We write  $\mathrm{Weld}(\QD(\ell),\QD(\ell))$  as the law of the loop-decorated quantum surface obtained from the uniform conformal welding of  $\cD_1$ and $\cD_2$ (as described in Section~\ref{sec-conf-weld}).

\begin{theorem}[{\cite[Theorem 1.1]{ahs-sle-loop}}]\label{thm-loop2}
Let $\mathbb F$ be a measure on $H^{-1}(\C)$ such that the law of $(\C, h)/{\sim_\gamma}$ is $\QS$
if $h$ is sampled from $\mathbb F$. Let
$\QS\otimes \SLE^{\mathrm{loop}}_\kappa$ be the law of the decorated quantum surface $( \C, h,\eta)/{\sim_\gamma}$ when $(h,\eta)$ is sampled from $\mathbb F\times\SLE^{\mathrm{loop}}_\kappa$.
Then there exists a constant $C \in (0, \infty)$ such that  
	\[
	\QS\otimes \SLE^{\mathrm{loop}}_\kappa= C\int_0^\infty \ell\cdot  \mathrm{Weld}(\QD(\ell),\QD(\ell)) \, d\ell.
	\]
\end{theorem}

Our next theorem extends Theorem \ref{thm-loop2} to the non-simple case.
An analogous conformal welding can be defined for generalized quantum disks. Recall $\GQD(\ell)$ in Definition \ref{def-GQD}. Let $\Weld(\GQD(\ell),\GQD(\ell))$ denote the law of the (non-simple) loop-decorated quantum surface obtained from the uniform conformal welding of a sample from $\GQD(\ell)\times\GQD(\ell)$. 

\begin{theorem}\label{loop weld ns}
 For $\kappa'\in(4,8)$ and $\gamma = \frac4{\sqrt{\kappa'}}$, there is a constant $C\in(0,\infty)$ such that $$\QS\otimes\SLE_{\kappa'}^{\rm loop}=C\int_0^\infty\ell\Weld(\GQD(\ell),\GQD(\ell))d\ell.$$
Here, the measure $\QS\otimes\SLE_{\kappa'}^{\rm loop}$ on the space of loop-decorated quantum surfaces is defined in the same way as in Theorem \ref{thm-loop2}.
\end{theorem}

\subsection{Conformal welding results for two-pointed (forested) quantum disks}\label{non-simple cons}
The proof of Theorem~\ref{loop weld ns} follows the strategy of \cite{ahs-sle-loop}. 
First, we recall some basic conformal welding results for quantum disks and forested quantum disks. In the following, we write $\cM_2^{\rm disk}(W;\ell,\ell')\otimes\SLE_\kappa(\rho_1,\rho_2)$ to denote the law of the curve-decorated quantum surface obtained by taking an arbitrary embedding $(D,h,x,y)$ of a sample from $\cM_2^{\rm disk}(W;\ell,\ell')$,  sampling $\eta$ independently of $h$ to be $\SLE_\kappa(\rho_1,\rho_2)$ on $(D,x,y)$, and outputting $(D,h,x,y,\eta)/{\sim_\gamma}$.

\begin{proposition}[{\cite[Theorem 2.2]{ahs-disk-welding}}]\label{prop:disk-welding}
For $\kappa \in (0,4)$, $\gamma = \sqrt\kappa$ and $W_1,W_2>0$, there exists a constant $c \in (0,\infty)$ such that for all $\ell,\ell'>0$
\begin{equation*}
\cM_2^{\rm disk}(W_1+W_2;\ell,\ell')\otimes\SLE_\kappa(W_1-2,W_2-2)=c\int_0^\infty \mathrm{Weld}(\cM_2^{\rm disk}(W_1;\ell,\ell_1),\cM_2^{\rm disk}(W_2;\ell_1,\ell'))d\ell_1.
\end{equation*}
Here, $\mathrm{Weld}(\cM_2^{\rm disk}(W_1;\ell,\ell_1),\cM_2^{\rm disk}(W_2;\ell_1,\ell'))$ denotes the law of the conformal welding of a pair sampled from $\cM_2^{\rm disk}(W_1;\ell,\ell_1)\times\cM_2^{\rm disk}(W_2;\ell_1,\ell')$ according to quantum length.
\end{proposition}

One can also weld quantum disks to get a quantum sphere. Define $\cP^{\rm sph}(W_1,W_2)$ to be the law of the pair $(\eta_0,\eta_1)$ obtained as follows: sample $\eta_0$ as a whole-plane $\SLE_\kappa(W_1+W_2-2)$, and let $\eta_1$ be the concatenation of independent samples from chordal $\SLE_\kappa(W_1-2,W_2-2)$ in each connected component of $\C\backslash\eta_0$ from the first to the last boundary point hit by $\eta_0$. In the same way that we defined $\cM_2^{\rm disk}(W;\ell,\ell')\otimes\SLE_\kappa(\rho_1,\rho_2)$, we define measures $\mathcal{M}^{\rm sph}_2\left(W\right)\otimes \text{\rm whole-plane}\SLE_{\kappa}$ and $\cM_2^{\rm sph}(W)\otimes\cP^{\rm sph}(W_1,W_2)$ on the space of curve-decorated quantum surfaces.
\begin{prop}[{\cite[Theorem 2.4]{ahs-disk-welding}}]\label{prop:disk-welding-sphere} 
Let $\kappa \in (0,4)$ and $\gamma = \sqrt\kappa$.
For $W\ge 2-\frac{\gamma^2}{2}$, there is a constant $c\in(0,\infty)$ such that
\begin{equation*}
\mathcal{M}^{\rm sph}_2\left(W\right)\otimes \text{\rm whole-plane}\SLE_{\kappa}\left(W-2\right)=c\int_0^\infty\Weld(\mathcal{M}^{\rm disk}_2(W;\ell,\ell))d\ell.
\end{equation*}
Here, $\Weld(\mathcal{M}^{\rm disk}_2(W;\ell,\ell))$ denotes the law of the curve-decorated quantum surface with the sphere topology obtained by conformal welding of the left and right boundary arcs of a sample from $\mathcal{M}^{\rm disk}_2(W;\ell,\ell)$ according to quantum length.

For $W_1,W_2>0$, there is a constant $c\in(0,\infty)$ such that
\begin{equation*}
\cM_2^{\rm sph}(W_1+W_2)\otimes\cP^{\rm sph}(W_1,W_2)=c\iint_0^\infty \mathrm{Weld}(\cM_2^{\rm disk}(W_1;\ell_0,\ell_1),\cM_2^{\rm disk}(W_2;\ell_1,\ell_0))d\ell_0d\ell_1.
\end{equation*}

Here, $\mathrm{Weld}(\cM_2^{\rm disk}(W_1;\ell_0,\ell_1),\cM_2^{\rm disk}(W_2;\ell_1,\ell_0))$ denotes the law of the curve-decorated quantum surface with the sphere topology obtained by conformal welding of a pair sampled from \\$\cM_2^{\rm disk}(W_1;\ell_0,\ell_1)\times\cM_2^{\rm disk}(W_2;\ell_1,\ell_0)$ according to quantum length.
\end{prop}

We now state Propositions~\ref{sy} and~\ref{sy2}, which are the forested counterparts of Propositions~\ref{prop:disk-welding} and~\ref{prop:disk-welding-sphere} respectively. The $\mathrm{Weld}(-,-)$ notations are the same except that generalized boundary length is used rather than quantum length.

\begin{prop}[{\cite[Theorem 1.4]{nonsimple-welding}}]\label{sy}
For $\kappa' \in (4,8)$ and $\gamma= \frac{4}{\sqrt{\kappa'}}$, let $W_\pm>0$ and $\rho_\pm=\frac{4}{\gamma^2}(2-\gamma^2+W_\pm)$. Let $W=W_++W_-+2-\frac{\gamma^2}{2}$. Then for some constant $c\in(0,\infty)$,
\begin{equation}
\mathcal{M}^{\rm f.d.}_2(W)\otimes\SLE_{\kappa'}(\rho_-;\rho_+)=c\int_0^\infty\Weld(\mathcal{M}^{\rm f.d.}_2(W_-;\ell),\mathcal{M}^{\rm f.d.}_2(W_+,\ell))d\ell.
\end{equation}
Similarly, 
\begin{equation*}
\mathcal{M}^{\rm disk}_2\left(2-\frac{\gamma^2}{2}\right)\otimes\SLE_{\kappa'}\left(\frac{\kappa'}{2}-4;\frac{\kappa'}{2}-4\right)=c\int_0^\infty\Weld(\cM^{\rm f.l.}_2(\ell),\cM^{\rm f.l.}_2(\ell))d\ell.
\end{equation*}
\end{prop}
Note that the second claim above for forested line segments can be interpreted as a special case of the first claim with  $W_- = W_+ = 0$.

To state Proposition~\ref{sy2}, we first need to  define the $\kappa' \in (4,8)$ variant  of ${\mathcal{P}}^{\rm sph}(W_+,W_-)$. 
First sample $\eta_0$ as a whole-plane $\SLE_{\kappa'}(\frac{4(W_++W_-+4-\gamma^2)}{\gamma^2}-2)$. There are countably many connected components of $\wh \C \backslash \eta_0$, which come in three types: those surrounded by $\eta_0$ on its left, those surrounded by $\eta_0$ on its right, and the remaining regions. For each region $D$ of the third type, independently sample chordal $\SLE_{\kappa'}$ from the last point of $\partial D$ hit by $\eta$ to the first point of $\partial D$ hit by $\eta$, and let $\eta_1$ be the concatenation of these curves. Let $\Tilde{\mathcal{P}}^{\rm sph}(W_+,W_-)$ be the law of $(\eta_0, \eta_1)$. 
Note that for $W_+=W_-=\gamma^2-2$, we have $\tilde{\cP}^{\rm sph}(W_+,W_-)=\SLE^{p \rightleftharpoons q}_\kappa$ 
for $(p,q)=(0,\infty)$.

\begin{proposition}\label{sy2}
Let $\kappa' \in (4,8)$ and $\gamma = 4/\sqrt{\kappa'}$. 
For $W\geq 2-\frac{\gamma^2}{2}$, there is some constant $c\in(0,\infty)$ such that
\begin{equation}\label{eq:ns-welding-1}
\mathcal{M}^{\rm sph}_2\left(W\right)\otimes \text{\rm whole-plane}\SLE_{\kappa'}\left(\frac{4W}{\gamma^2}-2\right)=c\int_0^\infty\Weld(\mathcal{M}^{\rm f.d.}_2(W-(2-\frac{\gamma^2}{2});\ell,\ell))d\ell.
\end{equation}
For $W=W_1+W_2+4-\gamma^2$ and $W_1, W_2>0$, there is some constant $c\in(0,\infty)$ such that
\begin{equation}\label{eq:ns-welding-2}
\mathcal{M}^{\rm sph}_2(W)\otimes\Tilde{\mathcal{P}}^{\rm sph}(W_1,W_2)=c\int_0^\infty\int_0^\infty\Weld(\mathcal{M}^{\rm f.d.}_2(W_1;\ell_1,\ell_2),\mathcal{M}^{\rm f.d.}_2(W_2,\ell_1,\ell_2))d\ell_1d\ell_2.
\end{equation}
\end{proposition}
\begin{proof}

We first prove \eqref{eq:ns-welding-1}. According to Proposition \ref{sy}, the conformal welding of two independent forested lines gives a thin quantum disk with weight $2-\frac{\gamma^2}{2}$ decorated with an $\SLE_{\kappa'}\left(\frac{\kappa'}{2}-4;\frac{\kappa'}{2}-4\right)$. Therefore, the RHS of \eqref{eq:ns-welding-1} is equal to (up to a multiplicative constant)
\begin{equation*}
\int_0^\infty \Weld\left((\mathcal{M}^{\rm disk}_2(W-(2-\frac{\gamma^2}{2});\ell_1,\ell_2),\mathcal{M}^{\rm disk}_2\left(2-\frac{\gamma^2}{2};\ell_1,\ell_2\right)\otimes\SLE_{\kappa'}\left(\frac{\kappa'}{2}-4;\frac{\kappa'}{2}-4\right)\right) d\ell_1d\ell_2.
\end{equation*}
Therefore~\eqref{eq:ns-welding-1} follows from  Proposition~\ref{prop:disk-welding-sphere}  and {\cite[Theorem 1.17]{wedges}},  which we recall:
\begin{thm-n}[{\cite[Theorem 1.17]{wedges}}]
Let $\kappa = 16/\kappa' \in (2,4)$.
    Sample $\eta_0$ according to whole-plane $\SLE_{\kappa}(W-2)$ on $\hat\C$, and conditioned on $\eta_0$ let $\eta_1$ be chordal $\SLE_\kappa\left(W-4+\frac{\kappa}{2},-\frac{\kappa}{2}\right)$ on $\hat\C\setminus\eta_0$. Then $(\eta_0,\eta_1)$ is a pair of GFF flow lines with an angle gap of $\pi$. Conditioned on $(\eta_0,\eta_1)$, in each connected component of $\hat \C \backslash (\eta_0 \cap \eta_1)$ lying to the right of $\eta_0$ and left of $\eta_1$ sample an independent chordal $\SLE_{\kappa'}\left(\frac{\kappa'}{2}-4;\frac{\kappa'}{2}-4\right)$ between the two boundary points hit by both $\eta_0$ and $\eta_1$, and let $\eta$ be the concatenation of these curves. 
    Then the marginal law of $\eta$ is whole-plane $\SLE_{\kappa'}(\frac{4W}{\gamma^2}-2)$.
\end{thm-n}

Now we prove \eqref{eq:ns-welding-2}. By Proposition \ref{sy}, for a sample from $\int_0^\infty\mathcal{M}^{\rm f.d.}_2(W_1;\ell_1,\ell_2)\times \mathcal{M}^{\rm f.d.}_2(W_2,\ell_1,\ell_2))d\ell_1$, conformally welding the boundary arcs having length $\ell_1$ gives a curve-decorated quantum surface whose law is (up-to-constant) $\int_0^\infty\cM_2^{\rm f.d.}(W_1+W_2+2-\frac{\gamma^2}{2};\ell_2,\ell_2)\otimes\SLE_{\kappa'}(\rho_-;\rho_+) \, d\ell_2$ with $\rho_\pm=\frac{4}{\gamma^2}(2-\gamma^2+W_\pm)$. We then conformally weld the left and right boundary arcs and integrate over $\ell_2$. By \eqref{eq:ns-welding-1} and the definition of $\Tilde{\mathcal{P}}^{\rm sph}(W_1,W_2)$, the resulting curve-decorated quantum surface has law (up-to-constant) $\cM_2^\mathrm{sph}\otimes \tilde \cP^\mathrm{sph}(W_1, W_2)$, so \eqref{eq:ns-welding-2} holds.
\end{proof}
From Proposition \ref{sy2} we can obtain the following corollary.

\begin{corollary}\label{cor:weld-gqd}
Fix distinct $p, q \in \C$. 
Let $(\C,h,p,q)$ be an embedding of a sample from $\cM_2^{\rm sph}(\gamma^2)$ and independently sample $\eta$ from $\SLE^{p \rightleftharpoons q}_{\kappa'}$. Then the law of $(\C,h,\eta)$, viewed as a loop-decorated quantum surface, equals $C\int_0^\infty \ell^3\Weld(\GQD(\ell),\GQD(\ell))d\ell$ for some constant $C\in(0,\infty)$.
\end{corollary}
\begin{proof}
By Proposition~\ref{sy2} and the fact that $\GQD_{0,2}(\ell_1,\ell_2)=\mathcal{M}^{\rm f.d.}_2(\gamma^2-2;\ell_1,\ell_2)$ (see Definition~\ref{def-GQD}), the law of $(D,h,\eta,p,q)/{\sim_\gamma}$ equals $C\iint_0^\infty\Weld (\GQD_{0,2}(\ell_1,\ell_2),\GQD_{0,2}(\ell_1,\ell_2) )\,d\ell_1\, d\ell_2$.
Now the result follows by taking the pushforward $(\ell_1,\ell_2)\mapsto \ell=\ell_1+\ell_2$ (i.e. forgetting the two marked points), similarly in the proof of \cite[Lemma~3.2]{ahs-sle-loop}.
\end{proof}

\subsection{Proof of Theorem \ref{loop weld ns}}\label{non-simple cons}
The proof follows the argument of \cite{ahs-sle-loop}, where the inputs are instead conformal welding results for forested quantum surfaces (developed in Section~\ref{non-simple cons}). 
 We first need the following area-weighted variant of $\cM^{\rm sph}_2(W)$ in Definition \ref{def-QS}.
\begin{definition}
    Fix $W>0$ and let $(\mathcal{C}, \phi,-\infty,+\infty)$ be an embedding of a sample from the quantum-area-weighted measure $\mu_\phi(\mathcal{C}) \mathcal{M}_2^{\rm sph}(W)$. Given $\phi$, sample $z$ from the probability measure proportional to $\mu_\phi$. We write $\mathcal{M}_{2, \bullet}^{\rm sph}(W)$ for the law of the marked quantum surface $(\mathcal{C}, \phi,-\infty,+\infty, z) / {\sim_\gamma}$.
\end{definition}

 \begin{proof}[Proof of Theorem \ref{loop weld ns}]
Let $M=\QS\otimes\SLE_{\kappa'}^{\rm loop}$.  
 For a sample from $M$ weighted by the quantum area times the square of the generalized quantum length, pick an arbitrary embedding $(\C,h,\eta)$. Given $(h,\eta)$, we independently sample $p,q$ from the probability measure proportional to the generalized quantum length and $r$ from the probability measure proportional to the quantum area. Let $M_3$ be the  law of $(\C,h,\eta,p,q,r)/{\sim_\gamma}$. 

 {Next, for fixed distinct points $p, q, r$ in $\C$, let $(\C, h, p,q,r)$ be an embedding of a sample from $\cM_{2,\bullet}^\mathrm{sph} (\gamma^2)$ and let $\eta$ be independently sampled from $\SLE_{\kappa'}^{p\rightleftharpoons q}$.  Let $\wt M_3$ be the law of the decorated quantum surface $(\C, h, \eta, p,q,r)/{\sim_\gamma}$. Let $\widetilde M^{p,q,r}_3$ be the law of the field-curve pair $(h, \eta)$, i.e., $\wt M_3^{p,q,r}$ is the law of an embedding of a sample from $\wt M_3$ with marked points sent to $p,q,r$.}

To show Theorem \ref{loop weld ns}, it suffices to show $M_3=C\wt M_3$ for some constant $C>0$. Indeed, by Corollary \ref{cor:weld-gqd}, if we deweight the total quantum area from $\wt M_3$ and then forget its marked points, the resulting law is $C\int_0^\infty \ell^3\Weld(\GQD(\ell),\GQD(\ell))d\ell$. Further reweighting by $\ell^{-2}$ we get $C\int_0^\infty \ell\Weld(\GQD(\ell),\GQD(\ell))d\ell$. {On the other hand, applying the same weightings to $M_3$ and forgetting the marked points yields $M$ by definition. This gives $M = C\int_0^\infty \ell\Weld(\GQD(\ell),\GQD(\ell))d\ell$, concluding the proof of Theorem \ref{loop weld ns}.}

{We will now show $M_3=C\wt M_3$. Theorem \ref{prop:QS-field} states that $\mathbf m_{\wh \C}\ltimes\QS=C\LF_\C$,  and hence $\mathbf m_{\wh \C}\ltimes M=C\LF_\C\times\SLE_{\kappa'}^\lp$ by the conformal invariance of  $\SLE_{\kappa'}^\lp$. Building on this,
we have the following counterpart of \cite[Lemma 3.5]{ahs-sle-loop} (which considers $\kappa \in (0,4)$) via exactly the same argument:
\begin{equation}\label{eq:M3}
\mathbf m_{\wh \C}\ltimes M_3=C|p-q|^{\frac{4}{\gamma^2}-2}\LF_\C^{(\frac{2}{\gamma},p),(\frac{2}{\gamma},q),(\gamma,r)}(d\phi)\SLE_{\kappa'}^{p\rightleftharpoons q}(d\eta)d^2pd^2qd^2r.
\end{equation}
{Instead of repeating the argument, we just point out the two minor modifications needed. The right hand side of \cite[Lemma 3.5]{ahs-sle-loop}  is $C|p-q|^{\frac{\gamma^2}4-2}\LF_\C^{(\frac\gamma2,p),(\frac\gamma2,q),(\gamma,r)}(d\phi)\SLE_{\kappa'}^{p\rightleftharpoons q}(d\eta)d^2pd^2qd^2r$. Since the generalized quantum length of $\SLE_{\kappa'}$ is a $\frac2\gamma$-GMC (the quantum length of $\SLE_\kappa$ is a $\frac\gamma2$-GMC), the insertions of the Liouville field in~\eqref{eq:M3} are $\frac2\gamma$ rather than $\frac\gamma2$. Further, the polynomial term in~\eqref{eq:M3} is $|p-q|^{\frac4{\gamma^2}-2}$ rather than $|p-q|^{\frac{\gamma^2}4-2}$; this corresponds to having $|p-q|^{-2(1-\frac{\kappa'}8)}$ rather than $|p-q|^{-2(1-\frac{\kappa}8)}$ in~\eqref{def:loop}. Other than these, the proof of~\eqref{eq:M3} is identical to that of \cite[Lemma 3.5]{ahs-sle-loop}.}

On the other hand, following the proof of \cite[Lemma 3.8]{ahs-sle-loop}, we have
\begin{equation}\label{eq:wt-M3-pqr}
\widetilde M^{p,q,r}_3=C|p-q|^{\frac{4}{\gamma^2}-2}|(p-q)(q-r)(p-r)|^2\LF_\C^{(\frac{2}{\gamma},p),(\frac{2}{\gamma},q),(\gamma,r)}(d\phi)\SLE_{\kappa'}^{p\rightleftharpoons q}(d\eta).
\end{equation}
Indeed, from the definition of $\widetilde M^{p,q,r}_3$ one can first find that $\wt M_3^{0,1,-1}=C\LF_\C^{(\frac{2}{\gamma},0),(\frac{2}{\gamma},1),(\gamma,r)}(d\phi)\SLE_{\kappa'}^{0\rightleftharpoons 1}(d\eta)$. Then let $f\in{\rm conf}(\wh C)$ be that $(0,1,-1)\mapsto(p,q,r)$ and use the transformation law $\LF_\C^{(\alpha_i,f(z_i))_i}=\prod_{i=1}^m |f'(z_i)|^{-2\Delta_{\alpha_i}} f_*\LF_\C^{(\alpha_i,z_i)_i}$ with $\Delta_\alpha:=\frac{\alpha}{2}(Q-\frac{\alpha}{2})$ (see e.g. \cite[Proposition 2.9]{ahs-sle-loop} for more details), one obtains \eqref{eq:wt-M3-pqr}.

Finally recall that $\mathbf m_{\wh \C}\ltimes \widetilde M_3=C\widetilde M^{p,q,r}_3|(p-q)(q-r)(r-p)|^{-2}d^2p d^2q d^2r$.
Hence combining \eqref{eq:M3} and \eqref{eq:wt-M3-pqr} we conclude $\mathbf m_{\wh \C}\ltimes M_3=C\mathbf m_{\wh \C}\ltimes \widetilde M_3$ for some constant $C$; and then we deduce $M_3=C\widetilde M_3$ by disintegration (see \cite[Section 3.1]{ahs-sle-loop} for more details).}
\end{proof}

\section{Backgrounds and preliminary results on CLE coupled with LQG}\label{sec:app-MSW} \label{sec:app-CLE}
In this section, we review existing results on CLE coupled with independent quantum disks or generalized quantum disks, where the CLE and LQG parameters are related by $\kappa \in \{ \gamma^2, 16/\gamma^2\}$. 
Some of these results were implicitly obtained and others were sketched in prior literature; we will give details or alternative derivations.

\subsection{The independent coupling of LQG and simple CLE}\label{subsec:MSW}
\newcommand{\outleng}{a}

Let $\gamma \in (\sqrt{8/3}, 2)$ and $\kappa = \gamma^2$. 
Suppose that $(D,h)$ is an embedding of a sample from $\QD$, where $D$ is a bounded domain. Let $\Gamma$ be  a (non-nested) $\CLE_\kappa$ on $D$ that is independent of $h$.
Then we call  the decorated quantum surface	$(D,h,\Gamma)/{\sim_\gamma}$ a $\CLE_\kappa$ decorated quantum disk and 
denote its law  by $\QD\otimes \CLE_\kappa$.   By the conformal invariance of $\CLE_\kappa$,  the measure $\QD\otimes \CLE_\kappa$
does not depend on the choice of embedding of $(D,h)/{\sim_\gamma}$. 

Fix $\outleng>0$. Recall the probability measure  $\QD(\outleng)^{\#}=|\QD(a)|^{-1}  \QD(a)$ that corresponds to the quantum disk with boundary length $\outleng$.
We define the probability measure   $\QD(\outleng)^{\#}\otimes \CLE_\kappa$ in the same way  as $\QD\otimes \CLE_\kappa$ with $\QD(\outleng)^{\#}$ in place of $\QD$.
We can similarly define  measures such as $\QD_{1,0}\otimes \CLE_\kappa$ and $\QD_{1,0}(\outleng)^{\#}\otimes \CLE_\kappa$.%

Let $(D,h,\Gamma)$ be an embedding of a sample from $\QD(\outleng)^{\#}\otimes \CLE_\kappa$.
Given a loop $\wp$ in $\Gamma$,  let $D_\wp$ be the bounded component of $\C\setminus \wp$, namely the region encircled by $\wp$.
A loop $\wp\in \Gamma$ is called \emph{outermost} if it is not contained in any $D_{\eta'}$ for $\eta'\in \Gamma$.
Let  $(\ell_i)_{i\ge 1}$ be the collection of the quantum lengths of the outermost loops of $\Gamma$ listed in non-increasing order. 
Two crucial inputs to our paper are the  law of  $(\ell_i)_{i\ge 1}$ and  the conditional law of the quantum surfaces encircled by  the outermost  loops 
conditioned on $(\ell_i)_{i\ge 1}$. We summarize them as the two propositions below.
\begin{proposition}[\cite{msw-cle-lqg,bbck-growth-frag,ccm-perimeter-cascade}]\label{prop-ccm}
	Set $\lexp := \frac4{\kappa} + \frac12\in (\frac32,2)$. Let $(\zeta_t)_{t\geq 0}$ be a $\lexp$-stable L\'evy process  
	whose L\'evy measure is $1_{x>0} x^{-\beta-1} \, dx$, so that it has no downward jumps. We denote its law by $\P^\beta$.
	Let $\tau_{-\outleng}=\inf\{ t: \zeta_t=-\outleng  \}$. Let  $ (x_i)_{i \geq 1}$ be the sequence of the sizes of the upward jumps of $\zeta$ on $[0,\tau_{-\outleng}]$ sorted in decreasing order. 
	Then the law of $(\ell_i)_{i \geq 1}$ defined right above equals that of $ (x_i)_{i \geq 1}$ under the reweighted probability $\frac{\tau_{-\outleng}^{-1}\P^\beta}{\E[\tau_{-\outleng}^{-1}]}$.
\end{proposition}

It was pointed out at the end of \cite[Section 1]{ccm-perimeter-cascade} that Proposition~\ref{prop-ccm}  can be extracted  from \cite{msw-cle-lqg,bbck-growth-frag,ccm-perimeter-cascade}. The reason is that $(\ell_i)_{i \geq 1}$ and $ (x_i)_{i \geq 1}$ are 
two ways of  describing the scaling limit of the outermost loop lengths of an $O(n)$-loop-decorated planar map model. 
The former follows from~\cite{bbck-growth-frag,msw-cle-lqg} and the latter follows from~\cite{ccm-perimeter-cascade}. 
We explain this in more detail in Section~\ref{subsec:rpm}.

We also need the following quantum zipper result for a CLE outermost loop on the quantum disk.
Recall the measure $\QD_{1,0}(a)$  corresponds to the quantum disk with one interior marked point and boundary length $a$.
Recall  the probability measure $\QD_{1,0}(a)^\#\otimes\CLE_\kappa$  defined at the beginning of this subsection. 
For a $\CLE_\kappa$  sample $\Gamma$ on $D$ and a domain $U\subset D$, 
we write $\Gamma|_U$ for the subset of loops that lie in $U$.

\begin{proposition}\label{prop:single loop}
	For $a>0$, let $(D,h,\Gamma,z)$ be an embedding of a sample from $\QD_{1,0}(a)^{\#}\otimes\CLE_\kappa$. 
	Let $\eta$ be the outermost loop of $\Gamma$ surrounding $z$.
	Let  $D_\eta$ and $A_\eta$ be the two connected components of $D\setminus\wp$ where $z\in D_\eta$.
	Let $\ell_h$ be the quantum boundary length measure on  $\wp$.
	Conditioning on $(h,\Gamma, \eta,z)$, let $w$ be a point on $\wp$ sampled from the probability measure proportional to $\ell_h$. 
	Now consider the joint distribution of  $(D,h,\eta,z,w)$. 
	Then conditioning on $\ell_h(\eta)$,  the decorated quantum surfaces $(D_\eta, h,z,w)/{\sim_\gamma}$ and $(A_\eta,h, \Gamma|_{A_\eta},w)/{\sim_\gamma}$  
	are conditionally independent, and the conditional law of $(D_\eta, h,z,w)/{\sim_\gamma}$ is $\QD_{1,1}(\ell_h(\eta))^{\#}$.
\end{proposition}
Proposition~\ref{prop:single loop} is implicitly proved in \cite{msw-cle-lqg}, see Section~\ref{subsec:uniform-MSW} for more details.

\subsection{The independent coupling of  LQG and non-simple CLE}\label{subsec:MSW2}

Let $\gamma \in (\sqrt{2}, 2)$ and $\kappa = 16/\gamma^2$. 
Suppose that $(D,h)$ is an embedding of a sample from $\GQD$, where  $D$ is a bounded domain. Let $\Gamma$ be the union of independent $\CLE_{\kappa'}$'s in each connected component of $\mathrm{int}(D)$. 
Then we call the decorated quantum surface	$(D,h,\Gamma)/{\sim_\gamma}$ a $\CLE_{\kappa'}$-decorated quantum disk and 
denote its law  by $\GQD\otimes \CLE_{\kappa'}$. 

As the loops in $\Gamma$ are non-simple, we need to give the following definitions.
We say a loop $\eta$ \emph{surrounds} a point $z$ if $z \not \in \eta$ and $\eta$ has a nonzero winding number with respect to $z$. We call a loop $\eta \in \Gamma$  \emph{outermost} if no point of $\eta$ is surrounded by any loop in $\Gamma$.    
Let  $(\ell_i)_{i\ge 1}$ be the collection of the quantum natural time of the outermost loops of $\Gamma$ listed in non-increasing order. 

We first state the $\kappa \in (4,8)$ analog of Proposition~\ref{prop-ccm}. 

\begin{proposition}[\cite{msw-non-simple,bbck-growth-frag,ccm-perimeter-cascade}]\label{prop-ccm ns}
	Set $\lexp := \frac4{\kappa'} + \frac12\in (1,\frac32)$. Let $(\zeta_t)_{t\geq 0}$ be a $\lexp$-stable L\'evy process  
	whose L\'evy measure is $1_{x>0} x^{-\beta-1} \, dx$, so that it has no downward jumps. We denote its law by $\P^\beta$.
	Let $\tau_{-\outleng}=\inf\{ t: \zeta_t=-\outleng  \}$. Let  $ (x_i)_{i \geq 1}$ be the sequence of the sizes of the upward jumps of $\zeta$ on $[0,\tau_{-\outleng}]$ sorted in decreasing order. 
	Then the law of $(\ell_i)_{i \geq 1}$ defined right above equals that of $ (x_i)_{i \geq 1}$ under the reweighted probability $\frac{\tau_{-\outleng}^{-1}\P^\beta}{\E[\tau_{-\outleng}^{-1}]}$.
\end{proposition}

In Section~\ref{subsec:rpm} we will explain how Proposition~\ref{prop-ccm ns} follows from \cite{msw-non-simple,bbck-growth-frag,ccm-perimeter-cascade}. The argument is the same as that of Proposition~\ref{prop-ccm}.

Now, we state the $\kappa \in (4,8)$ analog of Proposition~\ref{prop:single loop}. The statement is more complicated due to the nontrivial topologies of the random surfaces involved.  
Let $\mathrm{Ann}$ denote the set of quantum surfaces with the annulus topology having distinguished inner and outer boundaries and decorated by a countable collection of loops. Concretely, $\mathrm{Ann}$ is the set of equivalence classes $(A, h, \Gamma, \partial_\mathrm{out} A, \partial_\mathrm{in} A)/{\sim_\gamma}$  where $A \subset \C$ has the annulus topology, $\partial_\mathrm{out} A$ and $\partial_\mathrm{in} A$ are the boundary components of $A$, $h$ is a distribution on $A$, $\Gamma$ is a collection of loops on $A$, and $\sim_\gamma$ identifies pairs of tuples if there exists a conformal map identifying the domains, boundaries and loops, which also relates the fields via the LQG coordinate change~\eqref{eq-QS}. Similarly, let $\mathrm{Ann}'$ denote the set of beaded quantum surfaces arising from a countable collection of loop-ensemble-decorated two-pointed disk-topology quantum surfaces endowed with a cyclic ordering, and say its  inner (resp.\ outer) boundary is the union of the left (resp.\ right) boundary arcs of the two-pointed quantum surfaces. The definition of $\mathrm{Ann}'$ is analogous to that of a thin quantum disk, except the ordering is cyclic.   See Figure~\ref{fig-ann-types}. 

\begin{figure}[ht!]
	\begin{center}
		\includegraphics[scale=0.45]{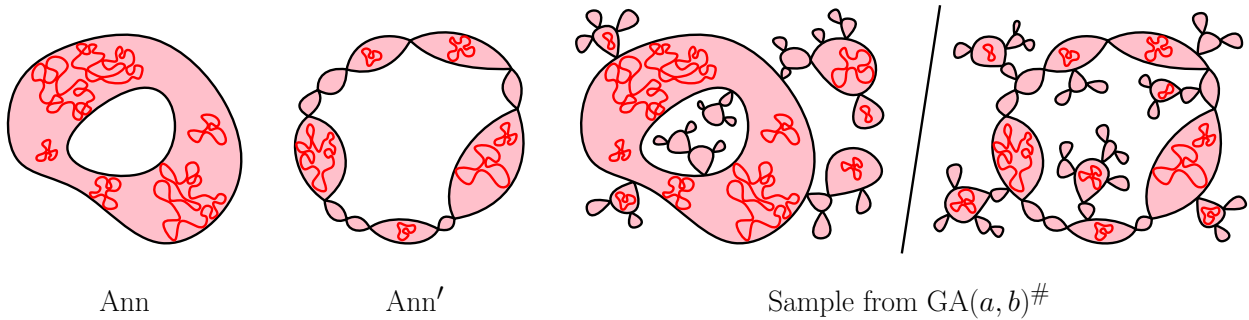}%
	\end{center}
	\caption{\label{fig-ann-types}  \textbf{Left:} An element of $\mathrm{Ann}$. \textbf{Middle}: An element of $\mathrm{Ann}'$. \textbf{Right:} A sample from $\mathrm{GA}^{\rm d}(a,b)^\#$. Note that both configurations are possible.}
\end{figure}

As explained below Proposition~\ref{prop:fr-ppp}, forested quantum surfaces arise from adding a Poisson point process of generalized quantum disks to the boundary of a quantum surface. If we further sample in each connected component of the added generalized quantum disks an independent $\CLE_{\kappa'}$, we call it $\CLE_{\kappa'}$-decorated foresting. 

\begin{proposition}\label{prop-ann-ns}
There exists a measure $\mathrm{GA}^\mathrm{u.f.}$ on $\mathrm{Ann} \cup \mathrm{Ann}'$ such that, if $\mathrm{GA}^{\rm d}(a,b)^\#$ denotes the law of a sample from $\mathrm{GA}^\mathrm{u.f.}$ with $\CLE_{\kappa'}$-decorated foresting conditioned on having outer (resp.\ inner) generalized boundary length $a$ (resp.\ $b$), then the following holds:

For $a,b>0$, the law of a sample from $\GQD_{1,0}(a)^{\#}\otimes\CLE_{\kappa'}$ conditioned on the length of the outermost loop surrounding the marked bulk point being $b$ is $\mathrm{Weld}(\mathrm{GA}^{\rm d}(a,b)^\#, \GQD(b)^\# \otimes \CLE_{\kappa'})$. 
\end{proposition}

\begin{remark}\label{rmk:GA-def}
In Proposition~\ref{prop-ann-ns}, because $\mathrm{GA}^{\rm d}(a,b)^\#$ is defined in terms of $\mathrm{GA}^\mathrm{u.f.}$, by definition $\mathrm{GA}^{\rm d}(a,b)^\#$ is the law of a forested quantum surface (i.e., the generalized quantum disks added to the boundary arise from a Poisson point process according to quantum length). Thus, the following is an informal way of rephrasing Proposition~\ref{prop-ann-ns}: 

For a sample from $\GQD_{1,0}(a)^\# \otimes \CLE_{\kappa'}$, condition on the length of the outermost loop $\eta$ around the bulk point being $b$. Then, cutting by $\eta$ gives a pair of conditionally independent forested quantum surfaces $(\cA, \cD)$ corresponding to the regions outside and inside $\eta$ respectively; the conditional law of $\cD$ is  $\GQD(b)^\# \otimes \CLE_{\kappa'}$, and we denote the conditional law of $\cA$ by $\mathrm{GA}^{\rm d}(a,b)^\#$. Furthermore, if we condition on $(\cA, \cD)$, the initial surface agrees in law with the uniform conformal welding of $\cA$ and $\cD$. 

We also note that the proof of Proposition~\ref{prop-ann-ns} yields an explicit description of $\mathrm{GA}^\mathrm{u.f.}$ via conformal weldings of \emph{forested quantum triangles}.
\end{remark}

Proposition~\ref{prop-ann-ns} is implicitly proved in \cite{msw-non-simple} (in the same way that Proposition~\ref{prop:single loop} is implicitly proved in \cite{msw-cle-lqg}). The explanation would be more complicated than that of Proposition~\ref{prop:single loop}, so we will instead give an alternative proof in Section~\ref{sec-nonsimple-disk-cle} based on \cite{nonsimple-welding}.

\subsection{Scaling limit of O(n)-loop-decorated planar maps: Proofs of Propositions~\ref{prop-ccm} and~\ref{prop-ccm ns}}\label{subsec:rpm}

In this section we explain why Propositions~\ref{prop-ccm} and~\ref{prop-ccm ns} hold; since both have the same proof we focus on the former, and briefly discuss the latter at the end.
The claim that $(x_i)_{i \geq 1}$ and $(\ell_i)_{i \geq 1}$ agree in law holds because they give two descriptions of the scaling limit of the same discrete model: loop lengths in the  $O(n)$-loop-decorated quadrangulation. This was pointed out at the end of Section 1 of \cite{ccm-perimeter-cascade}. 
We give a more detailed justification by assembling results in  \cite{ccm-perimeter-cascade, bbck-growth-frag, msw-cle-lqg}. 

We recall the loop-decorated quadrangulation from~\cite{ccm-perimeter-cascade}. 
A quadrangulation with boundary is a planar map where each face has degree four except  a distinguished face which we call the external face. (Others  faces are called internal faces.)
The degree of the external face is called the perimeter. 
A (rigid) loop configuration on a quadrangulation with a boundary is a set of disjoint undirected simple closed paths in the dual map which do not visit the external face, and with
the additional constraint that when a loop visits a face of q it must cross it through opposite edges.
Let $\cO_p$ be the set of pairs $(\mathbf q, \Gamma)$ such that $\mathbf q$  is a quadrangulation with boundary  whose perimeter is $2p$, and $\Gamma$ is a loop configuration on $\mathbf q$.  Similar as for CLE, for each $\Gamma$ on $\mathbf q$, there is a collection of outermost loops whose are not surrounded by any other loop in $\Gamma$.

We now recall a scaling limit result in~\cite{ccm-perimeter-cascade}. 
Recall $\beta=\frac{4}{\kappa}+\frac12$, let $n\in (0,2)$ be such that $\beta= \frac32 + \frac1\pi \arccos(n/2)$.
For each   $(\mathbf q, \Gamma)\in \cO_p$, we let  $|\mathbf q|$ be the number of faces of $\mathbf q$, let $|\Gamma|$ be the total length of all loops of $\Gamma$, and let $\#\Gamma$ be the number of loops in $\Gamma$.  For $h>0,g>0$, assign weight $w(\mathbf q, \Gamma) = g^{|\mathbf q| - |\Gamma|} h^{|\Gamma|} n^{\#\Gamma}$ to $(\mathbf q, \Gamma)$. 
For some choices of $(g,h)$, this gives a finite measure on $\cO_p$ which can be normalized into a probability measure. 
Let $\mathfrak M_p$ be a sample from this measure. Let $L_1^p \geq L_2^p \geq \dots$ be the lengths of the outermost loops of $\mathfrak M_p$ in decreasing order.
\begin{proposition}[{\cite[Proposition 3]{ccm-perimeter-cascade}}]\label{prop-ccm-scaling}
There exists $(g,h)$ such that 
as $p \to \infty$, the sequence $(\frac{a}{2p} L_i^p)_{i \geq 1}$ converges in law to {$(x_i)_{i \geq 1}$}, where {$(x_i)_{i \geq 1}$} and $a>0$ are as in Proposition~\ref{prop-ccm}.
\end{proposition} 
We refer to  \cite[Definition 1]{ccm-perimeter-cascade} for a description of $(g,h)$ such that Proposition~\ref{prop-ccm-scaling} holds.
In the rest of the section we explain why the following proposition follows from~\cite{bbck-growth-frag,msw-cle-lqg}.
\begin{proposition}\label{prop-bbck-scaling}
For  $(g,h)$ satisfying  Proposition~\ref{prop-ccm-scaling}, $(\frac{a}{2p} L_i^p)_{i \geq 1}$ converges in law to $(\ell_i)_{i \geq 1}$ as defined in Proposition~\ref{prop-ccm}.
\end{proposition}

We first describe $(\ell_i)_{i \geq 1}$ in terms of a growth fragmentation process considered in~\cite{msw-cle-lqg} and~\cite{bbck-growth-frag}.
We will not give the full description of the  growth fragmentation but only provide enough information to specify the law of $(\ell_i)_{i \geq 1}$.
Our presentation is based on the treatment in~\cite{bbck-growth-frag}. 
The description of growth fragmentation process in~\cite{msw-cle-lqg} is different in appearance. 
But as explained in~\cite[Remark 6.2]{msw-cle-lqg} they correspond to the same process.

For $\theta  = \frac4\kappa \in (1, \frac32)$, let $\nu_\theta$ be the measure on $(\frac12, \infty)$ defined by
\begin{equation}\label{eq-jump-measure}
\nu_\theta(dx) = \frac{\Gamma(\theta+1)}{\pi} \left( \frac{1_{1/2 < x < 1}}{(x(1-x))^{\theta+1}} + \sin(\pi(\theta-\frac12)) \frac{1_{x>1}}{(x(x-1))^{\theta+1}} \right) dx.
\end{equation}
Let $\Lambda_\theta$ be the pushforward of $\nu_\theta$ under the map $x \mapsto \log x$. 
For $\lambda>0$, let 
\[\Psi_\theta(\lambda) = \left(\frac{\Gamma(2-\theta)}{2\Gamma(2-2\theta)\sin(\pi \theta)} + \frac{\Gamma(\theta+1)B_{\frac12}(-\theta, 2-\theta)}{\pi} \right)\lambda + \int_\R (e^{\lambda y} -1 +\lambda(1-e^{y})) \Lambda_\theta(dy), \]
where $B_{\frac12}(a,b) = \int_0^{\frac12} t^{a-1}(1-t)^{b-1} \, dt$ is the incomplete Beta function; see  \cite[(28)]{bbck-growth-frag}. By  
\cite[Theorem 5.1, Proposition 5.2]{bbck-growth-frag},  there is a real-valued L\'evy process $(\xi(t))_{t \geq 0}$ whose law is described by 
$\E[e^{\lambda \xi(t)}] = e^{t \Psi_\theta(\lambda)}$ for all $\lambda,t>0$.  
For $t\ge 0$, let $\tau_t = \inf \{ r \ge 0 \: : \: \int_0^r e^{\theta \xi(s)} \, ds \geq t\}$. Then $\tau_t$ a.s.\ reaches $\infty$ in finite time. 
For $a>0$, define
\[Y_t^a := a \exp(\xi(\tau_{ta^{-\theta}}))  \quad \text{for }t \geq 0, \]
with the convention that $\xi(+\infty) = -\infty$. Then $(Y_t^a)_{t \geq 0}$ is nonnegative Markov process , with initial value $Y_0^a = a$ and a.s.\ hits $0$ in finite time. 
Since $\nu_\theta$ is supported on $(\frac12, \infty)$, the downward jumps $y \to y - \ell$ of $(Y_t^a)_{t \geq 0}$ satisfy $\ell < \frac12y$. 

We now relate $(Y_t^a)_{t \geq 0}$ to the CPI on quantum disks reviewed in Section~\ref{subsec:uniform-MSW}. Suppose $(D,h,\Gamma,x,y)$ are as in Proposition~\ref{prop-msw-future-disk1} such that the law of $(D,h,\Gamma,x)/{\sim_\gamma}$ is $\QD_{0,1}(a)^{\#} \otimes\CLE_\kappa$.
The chordal CPI from $x$ to $y$ can be viewed as an exploration of the carpet of $\Gamma$ such that 
at any splitting time it goes into the domain with the target $y$ on its boundary. Namely, it enters $D_t$ instead of $\cB_t$ to continue.
We can also alter the exploration rule at these splitting times, each of which defines a variant of CPI and corresponds to a branch in the so-called CPI branching exploration tree rooted at $x$.  One particular variant  of CPI is such that at any splitting time it enters the domain with the largest quantum boundary length. 
In the terminology of~\cite{msw-cle-lqg}, this is called the CPI with  ($q=\infty$)-exploration mechanism.
	 
Let $\wt Y^a_t$ be the boundary length of the to-be-explored  region at time $t$  CPI with  ($q=\infty$)-exploration mechanism.  Then by \cite[Theorem 1.2]{msw-cle-lqg},
for some constant $c>0$ not depending on $a$, we have  $(\wt Y^a_{ct})_{t \geq 0} \stackrel d= (Y^a_t)_{t \geq 0}$. Moreover, the upward jumps in $\wt Y_t$ correspond to times when the CPI discovers a loop, and downward jumps in $\wt Y_t$ correspond to times when the CPI splits the to-be-explored region into two. 
Iteratively applying this fact we get
 the following description the outermost loop lengths is as in Proposition~\ref{prop-ccm}.

\begin{proposition}[{\cite{msw-cle-lqg}}]\label{prop-msw-loop-lengths}
The lengths $(\ell_i)_{i \geq 1}$ agree in law with $(L_i)_{i\geq1}$ sampled as follows. 
First sample $(Y^a_t)_{t \geq 0}$, and let $U_1$ and $D_1$ be the sets of the sizes of upward and downward jumps of $(Y^a_t)_{t \geq 0}$. Given $D_1$, sample a collection of independent processes $S_2 = \{(Y^x_t)_{t \geq 0} \: : \: x \in D_1 \}$, and let $U_2$ and $D_2$ be the sets of the sizes of all upward and downward jumps of processes in $S_2$. Iteratively define $S_i, U_i, D_i$ for all $i$, and finally set $U = \bigcup_{i \geq 1} U_i$. Finally, rank the elements of $U$ as $L_1 \geq L_2 \geq \cdots$. 
\end{proposition}
\begin{proof}
The  quantum lengths of loops discovered by   CPI with  ($q=\infty$)-exploration mechanism   
correspond to the sizes of the upward jumps in $\wt Y^a_t$, which has the same law as the upward jumps in $Y^a_t$. 
The analogous Markov properties in Propositions~\ref{prop-msw-future-disk1} and~\ref{prop:strong-Markov} still hold for this CPI. 
Now we continue this exploration mechanism  to explore the rest of CLE carpet. Iteratively applying  this relation the quantum length of the discovered loops and the upward jumps, we get Proposition~\ref{prop-msw-loop-lengths}. 
\end{proof}
We now explain how  Proposition~\ref{prop-bbck-scaling} follows from the scaling limit results for $\mathfrak M_p$ proved in \cite[Section 6]{bbck-growth-frag}.

\begin{proof}[Proof of Proposition~\ref{prop-bbck-scaling}]
Let ${\mathfrak M'_p}$  be the planar map obtained from  $\mathfrak M_p$ by removing all the regions surrounded by outermost loops on $\mathfrak M_p$. 
For $(g,h)$ satisfying Proposition~\ref{prop-ccm-scaling},  ${\mathfrak M'_p}$ is the so-called critical non-generic Boltzmann planar map considered  in~\cite[Section 6]{bbck-growth-frag}. The map ${\mathfrak M'_p}$  is the discrete analog of the CLE carpet on the LQG background. The discrete analog of CPI exploration for CLE carpet is considered in~\cite[Section 6.3]{bbck-growth-frag} which is called the branching peeling exploration. The exact analog of the CPI with ($q=\infty$)-exploration mechanism is considered in~\cite[Section 6.4]{bbck-growth-frag}, where the exploration is always towards the component with the largest perimeter when there is splitting.
It was shown in  \cite[Proposition 6.6]{bbck-growth-frag} that the rescaled lengths of the loops discovered by this peeling process converge in law to
the sizes of the upward jumps in $Y^a_t$. Iterating the exploration in both discrete and continuum, we get the desired convergence.
\end{proof}

\begin{proof}[Proof of Proposition~\ref{prop-ccm}]
Combining Propositions~\ref{prop-ccm-scaling} and~\ref{prop-bbck-scaling}, we conclude the proof.
\end{proof}

\begin{proof}[Proof of Proposition~\ref{prop-ccm ns}]
The argument is exactly the same as that of Proposition~\ref{prop-ccm}. All results taken from \cite{ccm-perimeter-cascade} and \cite{bbck-growth-frag} still hold for $\beta  = \frac4{\kappa'} + \frac12 \in (1, \frac32)$ and $\theta = \frac4{\kappa'} \in (\frac12, 1)$. The results from \cite{msw-cle-lqg} each have non-simple CLE counterparts in \cite{msw-non-simple}; in particular, to see $(\wt Y^a_{ct})_{t \geq 0} \stackrel d= (Y^a_t)_{t \geq 0}$,  \cite[Theorem 5.1]{msw-non-simple} gives $(\wt Y^a_{ct})_{t \geq 0} \stackrel d= (Y^a_t)_{t \geq 0}$ when $\sin(\pi (\theta - \frac12))$ in~\eqref{eq-jump-measure} is replaced by a quantity they call $A_+(p)/A_-(p)$, and \cite[Remark 5.2]{msw-non-simple} identifies this quantity as $-\cos( \pi \theta )$ ($=\sin(\pi (\theta - \frac12))$).
\end{proof}

\subsection{Proof of {Proposition~\ref{prop:single loop}}}\label{subsec:uniform-MSW}

We now explain how  Proposition~\ref{prop:single loop} follows from \cite{msw-cle-lqg}. It reduces to the following:

\begin{proposition}\label{prop:uniform}
In the setting of Proposition~\ref{prop:single loop}, we can find a random point $p\in \eta$ such that, conditioning on $(A_\eta, h, \Gamma|_{A_\eta}, p)/{\sim_\gamma}$, 
the conditional law of $(D_\eta, h, z, p)/{\sim_\gamma}$   is $\QD_{1,1}(\ell_h(\eta))^\#$. 
\end{proposition}
\begin{proof}[Proof of Proposition~\ref{prop:single loop} given Proposition~\ref{prop:uniform}]
Suppose $(h, \Gamma, p,z)$ satisfies Proposition~\ref{prop:uniform}. 
Conditioning on $(h, \Gamma, p,z)$, let $U$ be a uniform random variable on $(0,1)$. Let $w$ be the point of $\eta$ such that the counterclockwise arc on $\eta$ from $p$ to $w$ 
is of $\ell_h$-length $U\ell_h(\eta)$. By definition, $w$ is sampled from the probability measure proportional to $\ell_h(\eta)$. By Proposition~\ref{prop:uniform} and the re-rooting invariance of   $\QD_{1,1}(\ell_h(\eta))$, conditioning on $(A_\eta, h, \Gamma|_{A_\eta}, p)/{\sim_\gamma}$ and $U$, the conditional law of $(D_\eta, h, z, w)/{\sim_\gamma}$  is $\QD_{1,1}(\ell_h(\eta))$. Since  $(A_\eta, h, \Gamma|_{A_\eta}, w)/{\sim_\gamma}$ is determined by $(A_\eta, h, \Gamma|_{A_\eta}, p)/{\sim_\gamma}$ and $U$, we are done. 
\end{proof}

To find the desired $p$ in Proposition~\ref{prop:uniform}, we use the 
\emph{conformal percolation interface} (CPI) within a CLE carpet introduced by Miller, Sheffield and Werner~\cite{cle-percolations}.
Suppose $\Gamma$ is a $\CLE_\kappa$  on a Jordan domain $D$ (i.e. $\bdy D$ is a simple curve) for some $\kappa\in (\frac83,4)$.
Given  two boundary points $x,y$, a (chordal) CPI for $\Gamma$ from $x$ to $y$  is a random curve from $x$ to $y$  
coupled with $\Gamma$ that does not enter the interior of any region surrounded by a loop of $\Gamma$  (but it can touch the loops).
We also need to specify how a CPI proceeds upon hitting a loop of $\Gamma$ on its way from $x$ to $y$. We require that it always leaves the loop to its right.
In the terminology of~\cite{cle-percolations}, this corresponds to CPI with $\beta=1$.
The marginal law of a CPI is $\SLE_{\kappa'}(\kappa'-6)$  on $D$ from $x$  to  $y$ where $\kappa'=16/\kappa$ \cite[Theorem 7.7 (ii)]{cle-percolations} and the force point is on the right of $x$. In particular, a CPI is a non-simple curve.
Intuitively, the chordal CPI describes the chordal interface of a percolation on a CLE carpet.
It is characterized by certain conformal invariance and Markov properties which are consistent with this intuition; see~\cite[Definition 2.1]{cle-percolations}.  
We will not review the full details but will rely on an analogous Markov property that CPI satisfies on a quantum disk background.
This was established in~\cite{msw-cle-lqg}, which we review now.

Fix $\kappa\in (8/3,4)$ and $\gamma=\sqrt{\kappa}$.  For $L_0,R_0>0$,
suppose $(D,h,\Gamma, x)$ be an embedding of a  sample from $\QD_{0,1}^{\#} (L_0+R_0)\otimes \CLE_\kappa$. 
Let $y$ be the point on $\bdy D$ such that the quantum length of the counterclockwise arc from $x$ to $y$ is $R_0$.
Conditioning on $(h,\Gamma, x,y)$, sample a CPI  $\eta'$ within the carpet of $\Gamma$ from $x$ to $y$. Since the law of $\eta'$ is a $\SLE_{\kappa'}(\kappa'-6)$, 
there is a quantum natural time parametrization for $\eta'$ with respect to $h$~\cite[Section 6]{wedges}, which we use throughout. 
Under this parametrization, $\eta'$ has a finite duration $T$.
For a fixed time $t>0$, on the event that $t\le T$, let $\wt\eta'_t$ be the union of $\eta'[0,t]$ and all the loops of $\Gamma$
touched by $\eta'[0,t]$. If $t<T$, let $D_t$ be the simply connected component of $D\setminus \wt\eta'_t$ which contains $y$ on its boundary.
For a fixed $t$, both $D_t$ and the interior of $D\setminus D_t$ are Jordan domains a.s. We write the interior $D\setminus D_t$ as $U_t$. The interface between $D_t$ and $U_t$ are $\SLE_\kappa$ types curves, on which there is a well-defined quantum length measure. 

\begin{figure}[ht!]
	\begin{center}
		\includegraphics[scale=0.7]{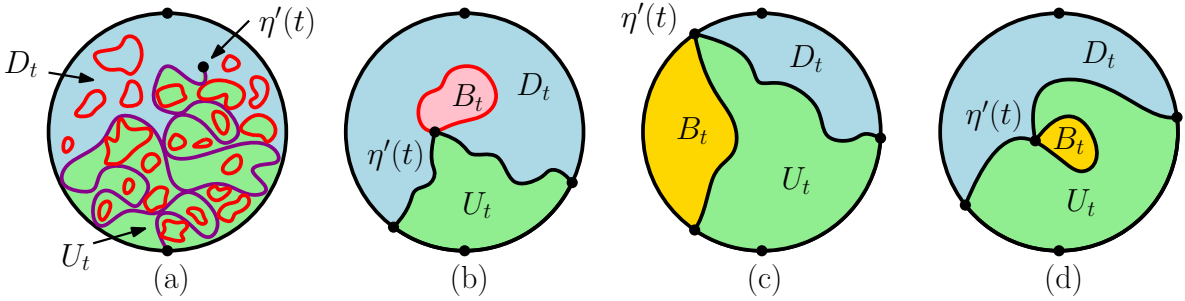}%
	\end{center}
	\caption{\label{fig-CPI-regions}  We color $D_t$ blue and $U_t$ green. \textbf{(a):} At all but countably many times $t$, we have $D_t = D_{t^-}$. To simplify (b)--(d) we omit their loops. \textbf{(b)}: A loop discovery time. \textbf{(c), (d):} A splitting time. At times when the left boundary length $L_t$ has a downward jump, there are two possible topologies; we do not illustrate the similar cases when the split is to the right.}
\end{figure}
The following   Markov property of CPI on quantum disks was proved in \cite{msw-cle-lqg}, although it  was not  explicitly stated. 

\begin{proposition}[{\cite{msw-cle-lqg}}]\label{prop-msw-future-disk1}
For a fixed $t>0$, on the even that $t<T$, let $L_t$ and $R_t$ be the quantum lengths of the clockwise  and  counterclockwise arcs from $\eta' (t)$ to $y$ of $\partial D_t$. 
Conditioning on the decorated quantum surface $(U_t, h,\Gamma|_{U_t}, \eta'|_{[0,t]})/{\sim_\gamma}$ and $(L_t,R_t)$, 
the conditional law of 
$(D_t, h, \Gamma|_{D_t},y)/{\sim_\gamma}$  is $\QD_{0,1}(L_t+R_t)^\# \otimes \CLE_\kappa$. 
\end{proposition}
\begin{proof}
	This proposition is essentially \cite[Proposition 5.1]{msw-cle-lqg}, except that we condition on more information than they explicitly stated. But the argument carries over directly to our setting. 
\end{proof}

The point $p$ in Proposition~\ref{prop:uniform} that we will find is a point where a CPI hits a loop. Therefore we need a stronger variant  of the Markov property in Proposition~\ref{prop-msw-future-disk1} at certain random times which we now define. 
For each $t\in (0,T)$, let $D_{t^-}$ be the interior of $\cap_{s<t} D_s$. 
According to~\cite{msw-cle-lqg}, for each fixed time $t$,  on the event that $t<T$, almost surely $D_{t^-}=D_t$; see Figure~\ref{fig-CPI-regions} (a).
But there exist countably many times where $D_{t^-}\neq D_t$.  In this case, there are two scenarios:
\begin{enumerate}
\item  The point $\eta'(t)$ is on a loop of $\Gamma$. In this case the interior of $D_{t^-}\setminus D_t$ is the Jordan domain enclosed by the  this loop.  But $D_t$ is not a Jordan domain since  $\eta'(t)$ corresponds to two points on $\bdy D_t$.
See Figure~\ref{fig-CPI-regions} (b). We call  $t$ a loop discovery time.
\item The point $\eta'(t)$ is not on a loop of $\Gamma$. In this case, both $D_t$ and  $D_{t^-}\setminus D_t$ are Jordan domains, and their boundaries intersect at the single point $\eta'(t)$.
See Figure~\ref{fig-CPI-regions} (c)--(d).  We call $t$ a splitting time. 
\end{enumerate}
In both cases, we let $\cB_t$ be the interior of $D_{t^-}\setminus D_t$ and $U_{t^-}$ be the interior of $D\setminus D_{t^{-}}$. Then $\bdy \cB_t\setminus \bdy D$ is an $\SLE_\kappa$  type curve. By definition, $\bdy \cB_t$ is a loop in $\Gamma$ if and only if $t$ is a loop discovery time. 
Recall $(L_t,R_t)$ from Proposition~\ref{prop-msw-future-disk1}.    If $t$ if it is a loop discovery time, then $R_t$ has an upward jump. If $t$ is a splitting time, then either $L_t$ or $R_t$ has a downward jump. In both cases, the size of the jump equals the quantum length of $\bdy \cB_t$, which we denote by  $ X_t$.
We now state the stronger version of Proposition~\ref{prop-msw-future-disk1}.

\begin{proposition}\label{prop:strong-Markov}
Fix $\eps>0$ and a positive integer $n$. Let $\tau$ be the $n$-th time such that  $D_{t^-}\neq D_t$ and the quantum length  $ X_t$ of 
$\bdy \cB_t$ is larger than $\eps$. If this time never occurs, set $\tau=\infty$.  Conditioning on $\tau<\infty$, the decorated quantum surface $(U_{\tau^-}, h,\Gamma|_{U_{\tau^-}}, \eta'|_{[0,\tau]})/{\sim_\gamma}$, the indicator $1_{\{\bdy \cB_\tau \textrm{ is a loop} \}}$, and the quantum lengths $X_\tau$ of $\bdy \cB_\tau$ and $L_\tau,R_\tau$ of the two arcs on $D_\tau$, the conditional law of   $(D_\tau, h, \Gamma|_{D_\tau}, y)/{\sim_\gamma}$ and $(\cB_\tau, h, \Gamma|_{\cB_\tau}, \eta'(\tau))/{\sim_\gamma}$ is  given by independent  samples from $\QD_{0,1}(L_\tau+R_\tau)^\# \otimes \CLE_\kappa$ and $\QD_{0,1}( X_\tau)^\# \otimes \CLE_\kappa$, respectively. 
\end{proposition}

\begin{proof}
For a fixed $t>0$, on the event that $t \le T$, consider the ordered collection of decorated quantum surfaces
$\{(\cB_s, h, \Gamma|_{\cB_s} ,\eta'(s))/{\sim_\gamma} : s\le t \textrm{ and } D_{s^{-}}\neq D_s \}$. 
It was proved in~\cite{msw-cle-lqg} that conditioning on $(D_t, h, \Gamma|_{D_t}, \eta'(t), y)/{\sim_\gamma}$, and the ordered information of the quantum lengths of their boundaries   and whether their times are loop discovery or splitting, 
the conditional law of these decorated quantum surfaces are independent $\CLE_\kappa$ decorated quantum disks with given boundary lengths. 
To see why this assertion follows from~\cite{msw-cle-lqg}, we note that Propositions~3.1 and 3.5 of \cite{msw-cle-lqg} yield the corresponding  assertion  for the analogous case of CLE on the quantum half-plane.  The pinching argument of \cite[Proposition 5.1]{msw-cle-lqg} then gives this assertion.

We claim that conditioning on $\tau<\infty$, $(U_{\tau^-}, h,\Gamma|_{U_{\tau^-}}, \eta'|_{[0,\tau]})/{\sim_\gamma}$, $1_{\{\bdy \cB_\tau \textrm{ is a loop} \}}$, and $X_\tau$,  $L_\tau,R_\tau$, the conditional law of  $(\cB_\tau, h, \Gamma|_{\cB_\tau}, \eta'(\tau))/{\sim_\gamma}$ is $\QD_{0,1}( X_\tau)^\# \otimes \CLE_\kappa$.	
Fix a large $k>0$. Let $s_k$ be the largest integer multiple of $2^{-k}$ smaller than $\tau$. 
 Let $\cU_t=(U_t, h,\Gamma|_{U_t}, \eta'|_{[0,t]})/{\sim_\gamma}$   and $\cD_t=(D_t, h|_{D_t}, \Gamma|_{D_t},y)/{\sim_\gamma}$.
For a fixed $j$, by Proposition~\ref{prop-msw-future-disk1}, conditioning on  $\cU_{j2^{-k}}$ and $(L_{j2^{-k}},R_{j2^{-k}})$, 
the conditional law of  $\cD_{j2^{-k}}$ is $\QD_{0,1}(L_{j2^{-k}}+R_{j2^{-k}})^\# \otimes \CLE_\kappa$. 
Note that $\{s_k=j2^{-k}\}$ is determined by $\cU_{j2^{-k}}$ and the quantum lengths of the boundaries of elements in 
$\{(\cB_s, h, \Gamma|_{\cB_s} ,\eta'(s))/{\sim_\gamma} : j2^{-k} \le s\le (j+1)2^{-k} \textrm{ and } D_{s^{-}}\neq D_s \}$.  
Applying the assertion of the first paragraph to $\cD_{j2^{-k}}$ with $T=2^{-k}$, 
we see that conditioning on $\tau < \infty$, $\cU_{s_k}$, $\{s_k=j2^{-k}\}$, $1_{\{\bdy \cB_{\tau} \textrm{ is a loop} \}}$, $X_{\tau}$,
$L_{s_k+2^{-k}}$ and $R_{s_k+2^{-k}}$,
the conditional law of  $(\cB_{\tau}, h, \Gamma|_{\cB_{\tau}}, \eta'(\tau))/{\sim_\gamma}$ is $\QD_{0,1}( X_{\tau})^\# \otimes \CLE_\kappa$.	
Varying $j$, we can remove the condition $\{s_k=j2^{-k}\}$. Since almost surely 
$\cU_{s_k}\rta (U_{\tau^-}, h,\Gamma|_{U_{\tau^-}}, \eta'|_{[0,\tau]})/{\sim_\gamma}$ and $(L_{s_k+2^{-k}},R_{s_k+2^{-k}})\rta (L_\tau,R_\tau)$ as $k\to \infty$, we have proved the desired claim. 

It remains to show  that conditioning on $\tau<\infty$, $(U_{\tau^-}, h,\Gamma|_{U_{\tau^-}}, \eta'|_{[0,\tau]})/{\sim_\gamma}$, $1_{\{\bdy \cB_\tau \textrm{ is a loop} \}}$, $X_\tau$, $L_\tau,R_\tau$ and  $(\cB_\tau, h, \Gamma|_{\cB_\tau}, \eta'(\tau))/{\sim_\gamma}$, the conditional law of $(D_\tau, h, \Gamma|_{D_\tau}, y)/{\sim_\gamma}$  is $\QD_{0,1}(L_\tau+R_\tau)^\# \otimes \CLE_\kappa$. This follows from  a similar but easier argument: we consider the smallest multiple of $2^{-k}$ larger than $\tau$ and use the Markov property in Proposition~\ref{prop-msw-future-disk1} at this time. 
We omit the details. \end{proof}

\begin{proof}[Proof of Proposition~\ref{prop:uniform}]
For $a>0$, let $(D,h,\Gamma,x)$ be an embedding of a sample from $\QD_{0,1} (a)^{\#}\otimes \CLE_\kappa$. 
Now we reweight $\QD_{0,1} (a)^{\#}\otimes \CLE_\kappa$ by $\mu_h(D)$ and sample a point $z$ according to the probability measure proportional to $\mu_h$.
This way, the law of $(D,h,\Gamma,z,x)$ is $\QD_{1,1} (a)^{\#}\otimes \CLE_\kappa$ as in Propositions~\ref{prop:single loop} and~\ref{prop:uniform}. Let $y$ be the point on $\bdy D$ such that both  the two arcs between $x$ to $y$ have quantum length $a/2$, and sample a CPI $\eta'$ from $x$ to $y$, parametrized by quantum natural time. Let $t_0$  be the time such $z\in \cB_{t_0}$ and set $p_0=\eta'(t_0)$. 

Consider $\tau$ and $\cB_\tau$ as defined in Proposition~\ref{prop:strong-Markov} with this choice of $(D,h,\Gamma, x,y)$. 
Then on the event that $t_0=\tau$, namely $z\in \cB_\tau$, conditioning on $(U_{\tau^-},h,\Gamma|_{U_{\tau^-}}, \eta'|_{[0,\tau]})/{\sim_\gamma}$,  $1_{\{\bdy \cB_t \textrm{ is a loop} \}}$, the quantum length  $ X_\tau$ of $\bdy\cB_\tau$
and $(D_\tau, h, \Gamma|_{D_\tau}, y)/{\sim_\gamma}$,  the conditional law of  
$(\cB_\tau, h, \Gamma|_{\cB_\tau}, z, \eta'(\tau))/{\sim_\gamma}$ is  $\QD_{1,1} (  X_\tau)^{\#}\otimes \CLE_\kappa$, where we have $\QD_{1,1}$ instead of $\QD_{0,1}$ because of area weighting. 
This means that conditioning on the quantum intrinsic information on  $D\setminus \cB_\tau$, 
the conditional law of $(\cB_\tau, h, \Gamma|_{\cB_\tau}, z, \eta'(\tau))/{\sim_\gamma}$ is a CLE decorated marked quantum disk with given boundary length. 
By varying $\eps$ and $n$ in  Proposition~\ref{prop:strong-Markov}, the same holds with $(\cB_\tau, h|_{\cB_\tau}, \Gamma|_{\cB_\tau}, z, \eta'(\tau))/{\sim_\gamma}$  replaced by 
$(\cB_{t_0}, h, \Gamma|_{\cB_{t_0}}, z, \eta'(t_0))/{\sim_\gamma}$. 

If $t_0$ is a loop discovery time,  then $\bdy \cB_{t_0}$ is the loop $\eta$ surrounding $z$ and  $p_0=\eta'(t_0)$ is  the desired point we need for Proposition~\ref{prop:uniform}. Otherwise, we set $D_1=\cB_{t_0}$ and construct $t_1$, $\cB_{t_1}$ and $ p_1$ as above  with $(D,h, \Gamma, x)$ replaced by $(D_1,h, \Gamma|_{D_1}, p_0)$. If $p_1\in \eta$ then we are done by setting $p=p_1$. Otherwise, we iterate this procedure and construct  $p_2,p_3,\cdots$. 
We claim that there must be a finite $k$ such that $p_k\in \eta$, hence we can set $p=p_k$ and conclude the proof.

To see the finiteness of the iteration, recall the set $(\wt \eta'_t)_{t\ge [0,T]}$ defined from a CPI $\eta'$. We now require that once $\eta'$ hits a loop, it first finishes tracing the loop counterclockwise and then proceeds in its own track.  This turns $\wt\eta'_T$ into the trace of a non-self-crossing  curve sharing the same endpoints as $\eta'$. According to~\cite[Theorem 7.7 (ii)]{cle-percolations}, viewed as a curve the law of $\wt \eta'$ is a chordal $\SLE_\kappa(\kappa-6)$ as  defined in~\cite{shef-cle}. The curve $\eta'$ is the so-called  trunk of $\wt \eta'$.  By the target invariance of $\SLE_\kappa(\kappa-6)$, if we iterate the above chordal CPI exploration towards $z$ and keep track of the chordal $\SLE_\kappa(\kappa-6)$'s along the way, we get a radial $\SLE_\kappa(\kappa-6)$ on $D$ from $x$ to $z$. From the relation between the  $\SLE_\kappa(\kappa-6)$ exploration tree and $\CLE_\kappa$ in~\cite{shef-cle},  after finite many iterations we must reach the loop $\eta$ at a point $p$, the place where the radial $\SLE_\kappa(\kappa-6)$ starts exploring $\eta$.
\end{proof}

\subsection{Proof of Proposition~\ref{prop-ann-ns}}\label{sec-nonsimple-disk-cle}

Our proof will depend on conformal welding results from \cite{asyz-nested-path}.
We first give a decomposition of CLE via the \emph{continuum exploration tree}. See Figure~\ref{fig-exploration-cle}. 

\begin{lemma}\label{lem-cle-decomp}
    Let $D$ be a Jordan domain and let $z \in D$, $x \in \partial D$. Let $\wt \eta$ be a radial $\SLE_{\kappa'}(\rho)$ curve in $D$ from $x$ to $z$, with $\rho = \kappa' - 6$ and force point located infinitesimally clockwise from $x$. Let $\sigma$ be the time that the connected component of $\C \backslash\wt \eta([0,\sigma])$ containing $z$ is bounded; 
    this is the first time $\wt \eta$ completes a loop around $z$. Consider each connected component $U$ of $D \backslash \wt \eta([0,\sigma])$ such that $\partial U$ intersects both $\partial D$ and $\wt \eta([0,\sigma])$, and $\wt \eta$ fills some boundary segment of $\partial U \cap \wt \eta([0,\sigma])$ in the counterclockwise direction. Let $s(U) < t(U)$ be the first and last times that $\wt \eta$ hits $\partial U$, and let $\eta_U$ be the concatenation of $\wt \eta|_{[s(U), t(U)]}$ with an independent $\SLE_{\kappa'}$ curve in $D'$ from $\wt \eta(t(U))$ to $\wt \eta(s(U))$, so $\eta_U$ is a loop with counterclockwise orientation. Let $\Gamma_0$ be the collection of all such $\eta_U$. In each connected component of $D \backslash \bigcup_{\eta_U \in \Gamma_0} \eta_U$ 
    sample an independent $\CLE_{\kappa'}$, and let $\Gamma$ be the union of $\Gamma_0$ and these $\CLE_{\kappa'}$s. Then $\Gamma$ has the law of $\CLE_{\kappa'}$ in $D$. 
\end{lemma}
\begin{proof}
This follows immediately from the continuum exploration tree construction of $\CLE_{\kappa'}$ of Sheffield \cite{shef-cle}.
\end{proof}

\begin{figure}[ht!]
	\begin{center}
		\includegraphics[scale=0.5]{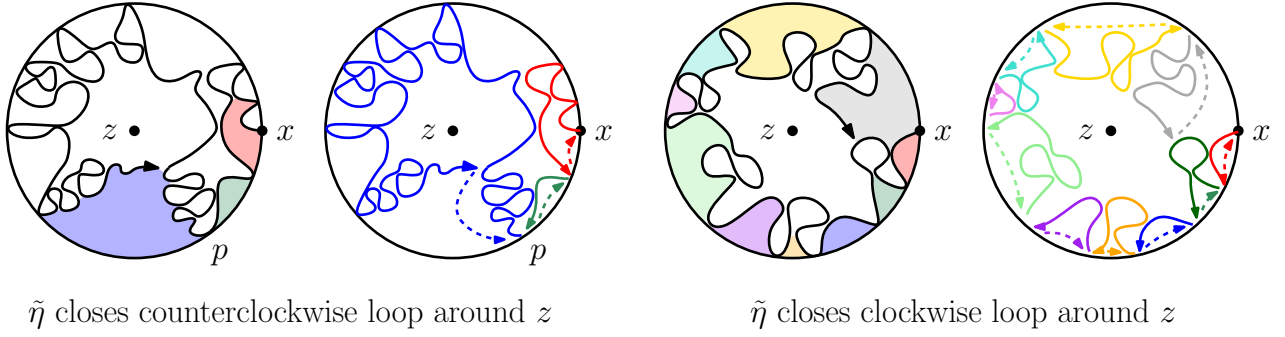}%
	\end{center}
	\caption{\label{fig-exploration-cle}  Figure for Lemma~\ref{lem-cle-decomp}. The curve $\wt \eta$ (black) is radial $\SLE_\kappa(\kappa-6)$ stopped at the first time it closes a loop around the target point $z$. Some connected components of $D \backslash \wt \eta([0,\sigma])$ have a boundary arc lying in $\partial D$ and a boundary arc traced in the counterclockwise direction by $\wt \eta$; these are shown in color. In each such region, draw an independent $\SLE_\kappa$ curve (dashed) from the last point hit by $\wt \eta$ to the first point hit by $\wt \eta$. Concatenating these curves with segments of $\wt \eta$ gives a collection of loops $\Gamma_0$. This, together with independent $\CLE_\kappa$ in each connected component of $D$ in the complement of $\Gamma_0$, gives $\CLE_\kappa$ in $D$. \textbf{Left:} If $\wt \eta$ closes a counterclockwise loop around $z$, then $z$ is surrounded by a loop in $\Gamma_0$ (shown in blue). \textbf{Right:} If $\wt \eta$ closes a clockwise loop around $z$, then $z$ is not surrounded by a loop in $\Gamma_0$. (Note that the drawing is only schematic; in both cases, due to the fractal nature of SLE, there is a.s.\ no loop touching $x$.)
}
\end{figure}

In Lemmas~\ref{lem-cut-ccw} and~\ref{lem-cut-cw}, we give decompositions of $\cM_{1,0,1}^\mathrm{f.d.}(\gamma)$ decorated by $(\wt \eta, \Gamma)$ (from Lemma~\ref{lem-cle-decomp})  depending on whether $\wt \eta$ first closes a counterclockwise or clockwise loop around $z$. Combining these will yield Proposition~\ref{prop-ann-ns}. 

Recall the space $\mathrm{Ann}'$ and the notion of $\CLE_{\kappa'}$-decorated foresting defined immediately above Proposition~\ref{prop-ann-ns}. Let $\mathrm{Ann}'_1$ denote the space of quantum surfaces in $\mathrm{Ann}'$ with an additional marked point on the outer boundary.
\begin{lemma}\label{lem-cut-ccw}
There exists a measure $M_\mathrm{ccw}$ on $\mathrm{Ann}_1'$ such that if $M_\mathrm{ccw}^f$ denotes the law of a sample from $M_{\mathrm{ccw}}$ with $\CLE_{\kappa'}$-decorated foresting, and $M_\mathrm{ccw}^f(b)^\#$ denotes the law of a sample from $M_\mathrm{ccw}^f$ conditioned on having inner generalized boundary length $b$, then the following holds:

Let $b>0$. Let $(\wt D, h, z,x)$ be an embedding of a sample from $\cM_{1,1,0}^\mathrm{f.d.}(\gamma)$, let $D$ be the connected component of $\mathrm{int}(\wt D)$ containing $w$, and independently sample $\wt \eta$ and $\Gamma$ in $D$ as in Lemma~\ref{lem-cle-decomp}. 
	 Condition on the event that $\wt \eta$ closes a counterclockwise loop around $z$ at time $\sigma$, and on the loop of $\Gamma$ surrounding $z$ having quantum length $b$. Then the conditional law of $(\wt D, h, \Gamma, \wt \eta|_{[0,\sigma]}, z, w)/{\sim_\gamma}$ is $\mathrm{Weld}(M_\mathrm{ccw}^f(b)^\#, \GQD_{1,0}(b)^\#\otimes \CLE_{\kappa'})$.
\end{lemma}
The $\mathrm{Weld}$ notation in Lemma~\ref{lem-cut-ccw} refers to the law of a sample from $M_{\mathrm{ccw}}^f(b)^\#\times  (\GQD(b)^\#\otimes \CLE_{\kappa'})$ after the inner boundary of the first surface is uniformly conformally welded to the boundary of the second surface. Similarly to Remark~\ref{rmk:GA-def}, 
Lemma~\ref{lem-cut-ccw} can be informally stated as follows. In the setup and conditioning of Lemma~\ref{lem-cut-ccw}, 
cutting by the loop of $\Gamma$ surrounding $z$ gives a pair of conditionally independent decorated forested quantum surfaces $(\cA, \cD)$ corresponding to the regions outside and inside $\eta$ respectively; the conditional law of $\cD$ is  $\GQD(b)^\# \otimes \CLE_{\kappa'}$, and we denote the conditional law of $\cA$ by $M_\mathrm{ccw}^f(b)^\#$. Furthermore, if we condition on $(\cA, \cD)$, the initial surface agrees in law with the uniform conformal welding of $\cA$ and $\cD$.  
\begin{proof}
See Figure~\ref{fig-radial-cut}. 
Cut $(\wt D, h, \wt \eta, z, x)/{\sim_\gamma}$ by $\wt \eta|_{[0,\sigma]}$ to obtain a pair of forested quantum surfaces $(\cT_1, \cD_1)$ where $\cD_1$ contains the bulk marked point. 
The first claim of 
\cite[Proposition 4.3]{asyz-nested-path} (with parameter $\alpha = \gamma$) states that the law of $(\cT_1, \cD_1)$ is
\eqb\label{eq-weld-nonsimple-ii}
C \iint_0^\infty  {\rm QT}^f (\frac{3\gamma^2}{2}-2,2-\frac{\gamma^2}{2},\gamma^2-2;t, a+t) \times  \cM_{1,1,0}^\mathrm{f.d.}(\gamma; a)\, da \, dt
\eqe
for some constant $C$. Here, $\mathrm{QT}^f(W_1,W_2,W_3)$ is the law of the \emph{forested quantum triangle with weights $W_1, W_2, W_3$}, and $\mathrm{QT}^f(W_1,W_2,W_3; t, a+t)$ is the law of the quantum triangle whose boundary arc between the vertices of weights $W_1$ and $W_2$ (resp.\ $W_3$) has generalized boundary length $t$ (resp.\ $a+t$). Since the vertex with weight $2-\frac{\gamma^2}2$ is thin (since $2-\frac{\gamma^2}2 < \frac{\gamma^2}2$), by \cite[Definition 2.18]{asy-triangle} $\cT_1$ decomposes into a pair $(\cT_2, \cD_2)$ where  $\cT_2$ is a forested quantum triangle with weights $(\frac{3\gamma^2}{2}-2,2-\frac{\gamma^2}{2},\gamma^2-2)$ and $\cD_2$ is a forested quantum disk of weight $2-\frac{\gamma^2}2$. Precisely, the joint law of $(\cT_2, \cD_2, \cD_1)$ is 
\eqb\label{eq-weld-nonsimple-iii}
(2- \frac 4{\gamma^2})C \iiint_0^\infty  {\rm QT}^f (\frac{3\gamma^2}{2}-2,\frac{3\gamma^2}{2} - 2,\gamma^2-2;b, a+b+c) \times  \cM_{0,2}^\mathrm{f.d.}(2-\frac{\gamma^2}2;c) \times  \cM_{1,1,0}^\mathrm{f.d.}(\gamma; a)\, da \, db\, dc.
\eqe

Now, consider the connected component $U_0$ of $D \backslash \wt \eta([0,\sigma])$ such that $\partial U_0$ intersects $\partial D$ and contains $\wt \eta(\sigma)$, so $\eta_{U_0} \in \Gamma$ is the loop  surrounding $z$. Let $p$ be the first point on $\partial U_0$ hit by $\wt \eta$, i.e., $p = \wt \eta(s(U_0))$. 
Recall that $\eta_{U_0}$ is the concatenation of a segment of $\wt \eta$ and an $\SLE_\kappa$ curve in $U_0$ from $\wt \eta (\sigma)$ to $p$; this latter curve corresponds to an $\SLE_\kappa$ curve in $\cT_2$ between the two vertices of weight $\frac{3\gamma^2}2 - 2$. 
By \cite[Lemma 4.1]{nonsimple-welding} and \cite[Theorem 1.4]{nonsimple-welding}, cutting by this curve gives a pair $(\cT_3, \cD_3)$ where $\cT_3$ is a forested quantum triangle with weights $(\frac{3\gamma^2}2 - 2, \frac{3\gamma^2}2 - 2, \gamma^2-2)$ and $\cD_3$ is a forested quantum disk with weight $\gamma^2-2$. Precisely, the joint law of $(\cT_3, \cD_3, \cD_2, \cD_1)$ is 
\eqb \label{eq-weld-nonsimple-iv}
C' \iiint_0^\infty  {\rm QT}^f (\gamma^2-2, \gamma^2-2, \gamma^2-2;e, a+b+c) \times  \cM_{0,2}^\mathrm{f.d.}(\gamma^2-2; e, b)\times \cM_{0,2}^\mathrm{f.d.}(2-\frac{\gamma^2}2;c) \times  \cM_{1,1,0}^\mathrm{f.d.}(\gamma; a)\, da \, db\, dc\, de
\eqe
where $C'$ is a constant. 

Now, let $\cT_4$ be $\cT_3$ with its first marked point shifted counterclockwise by $(a+b)$ units of generalized boundary length, so its boundary arcs counterclockwise and clockwise from the first vertex have generalized boundary lengths $c$ and $a+b+e$ respectively. Since a sample from $\mathrm{QT}^f(\gamma^2-2, \gamma^2-2, \gamma^2-2)$ is invariant in law under the operation of forgetting a boundary marked point and resampling it according to generalized boundary length measure \cite[Lemma 4.1]{asyz-nested-path}, we conclude that $\cT_4$ is also a quantum triangle with all three weights $\gamma^2-2$. Let $\cD_4$ be the concatenation of $\cD_1$ and $\cD_3$ where the boundary marked point of $\cD_1$ is identified with the first marked point of $\cD_3$, and this point is then unmarked. By definition $\cD_3$ is a generalized quantum disk with two boundary points, so by Propositions~\ref{prop:gqd-fd} and~\ref{prop:101=gqd11} $\cD_4$ is a generalized quantum disk with one bulk and one boundary point. In summary, the joint law of $(\cT_4, \cD_4, \cD_2)$ is 
\eqb \label{eq-weld-nonsimple-v}
C'' \iint_0^\infty  {\rm QT}^f (\gamma^2-2, \gamma^2-2, \gamma^2-2;f,c) \times \GQD_{1,1}(f) \times \cM_{0,2}^\mathrm{f.d.}(2-\frac{\gamma^2}2;c) \, dc\, df
\eqe
where $C''$ is a constant.

For $c>0$, let $\widehat {\rm QT}^f (\gamma^2-2, \gamma^2-2, \gamma^2-2;-,c)$ be the law of a sample from ${\rm QT}(\gamma^2-2, \gamma^2-2, \gamma^2-2)$ with forested boundary between the first and third vertices, disintegrating on this forested boundary arc having generalized boundary length $c$, and similarly define $\wh \cM_{0,2}^\mathrm{f.d.}(2-\frac{\gamma^2}2;c)$.  
Let $\wh M_\mathrm{ccw}$ be the law of the (non-forested) quantum surface obtained from a sample from $C'' \int_0^\infty  \widehat {\rm QT}^f (\gamma^2-2, \gamma^2-2, \gamma^2-2;-,c) \times \wh \cM_{0,2}^\mathrm{f.d.}(2-\frac{\gamma^2}2;c) \, dc$ by conformally welding the pair of boundary arcs with generalized boundary length $c$, and identifying the first and second vertices of the quantum triangle. A sample from $\wh M_\mathrm{ccw}$ has two marked points; its first (resp.\ second) marked point corresponds to the third (resp.\ first) marked point of the quantum triangle. Its inner boundary corresponds to the boundary arc of the quantum triangle between its first and second vertices. 
Let $\wh M_\mathrm{ccw}^f(b)^\#$ be the law of a sample from $\wh M_\mathrm{ccw}$ with $\CLE_{\kappa'}$-decorated foresting and conditioned on having inner generalized boundary length $b$. By~\eqref{eq-weld-nonsimple-v},  $(\wt D, h, \eta|_{[0,\sigma]}, \eta_{U_0}, z, w, p)/{\sim_\gamma}$ conditioned on the quantum length of $\eta_{U_0}$ being $b$ agrees in law with the conformal welding of a pair of quantum surfaces $(\cA, \cD) \sim \wh M_{\mathrm{ccw}}^f(b)^\# \times \GQD_{1,1}(b)^\#$ in which the second marked point of $\cA$ is identified with the boundary marked point of $\cD$. Since forgetting the boundary marked point of $\cD$ and resampling it according to generalized boundary length leaves $\cD$ invariant in law, the conditional law of $(\wt D, h, \eta|_{[0,\sigma]}, \eta_{U_0}, z, w)/{\sim_\gamma}$ is $\mathrm{Weld}(\wh M_{\mathrm{ccw}}^f(b)^\# , \GQD_{1,0}(b)^\#)$ (where the second marked point of $\cA$ is forgotten).

Finally, the desired $M_\mathrm{ccw}$ can be obtained from $\wh M_\mathrm{ccw}$ by sampling independent chordal $\SLE_{\kappa'}$ curves via the process described in Lemma~\ref{lem-cle-decomp}. By Lemma~\ref{lem-cle-decomp} the claim holds. 
\end{proof}

\begin{remark}
Equation~\eqref{eq-weld-nonsimple-v} 
gives a stronger result than Lemma~\ref{lem-cut-ccw}, which we will not need in the present work. We state it here:

Let $M_\mathrm{ccw}$ be defined as in the above proof of Lemma~\ref{lem-cut-ccw}, let $M_\mathrm{ccw}^f$ be the law of a sample from $M_\mathrm{ccw}$ with $\CLE_{\kappa'}$-decorated foresting, and let $\{M_\mathrm{ccw}^f(b)\}_{b > 0}$ be the disintegration of $M_\mathrm{ccw}^f$ with respect to inner generalized boundary length. 
For a sample $(\cA, \cD)$ from $\int_0^\infty M_\mathrm{ccw}^f(b) \times (\GQD_{1,1}(b) \otimes \CLE_{\kappa'})\, db$, forget the marked boundary point of $\cD$, and let $\cD'$ be the uniform conformal welding of $\cA$ and $\cD$. There is a constant $C$ such that the law of $\cD'$ is $C \cM_{1,1,0}^\mathrm{f.d.} (\gamma) \otimes \CLE_{\kappa'}$ restricted to the event that the outermost loop around the marked bulk point touches the boundary. 
\end{remark}

\begin{figure}[ht!]
	\begin{center}
		\includegraphics[scale=0.37]{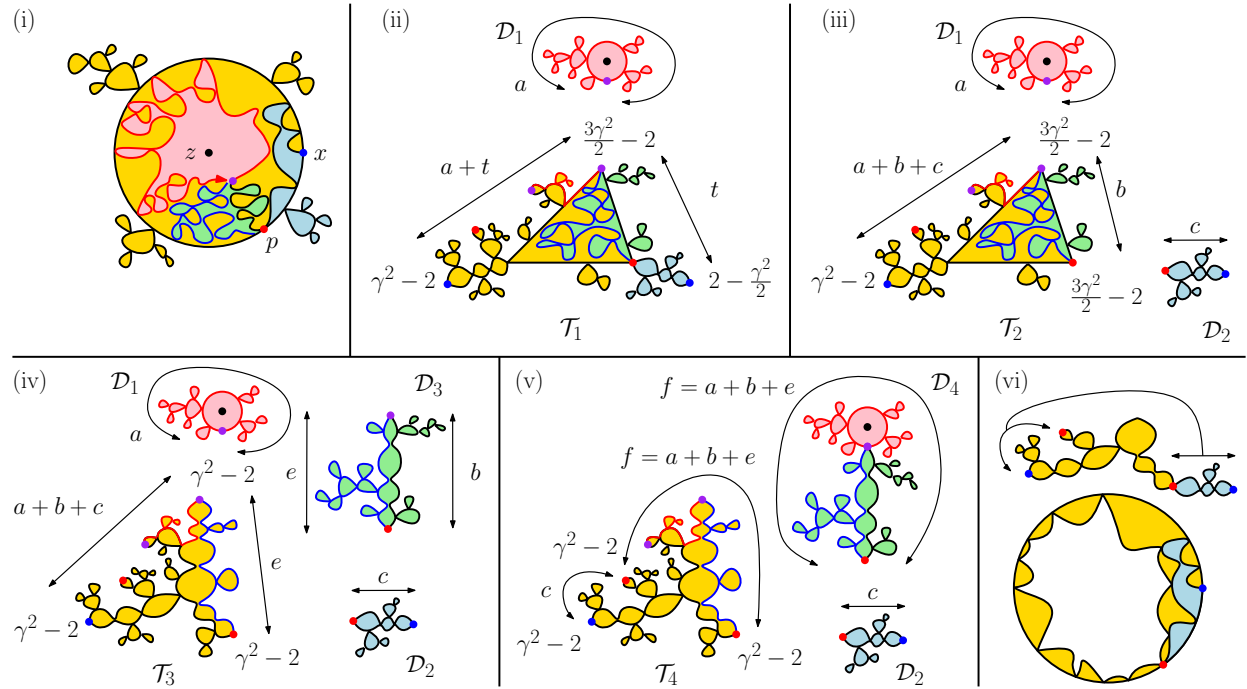}%
	\end{center}
	\caption{\label{fig-radial-cut} 
 Diagram for Lemma~\ref{lem-cut-ccw}. \textbf{(i)} A sample from $\cM_{1,1,0}^\mathrm{f.d.}(\gamma, \gamma)$ decorated by an independent pair $(\wt \eta, \Gamma)$ and restricted to the event $\{\wt\eta\text{ closes counterclockwise loop around }z\}$. The loop closure time is called $\sigma$, and the curve $\wt \eta|_{[0,\sigma]}$ is shown in black and red. The SLE curve in $U_0$ is colored in blue; concatenating it with a segment of $\wt \eta$ gives the loop of $\Gamma$ surrounding $z$. \textbf{(ii)} Cutting by $\wt \eta|_{[0,\sigma]}$ gives the surfaces $\cD_1, \cT_1$ whose joint law is described in~\eqref{eq-weld-nonsimple-ii}. \textbf{(iii)} The surfaces $\cD_1, \cD_2, \cT_2$ from~\eqref{eq-weld-nonsimple-iii}.
 \textbf{(iv)} Cutting by the blue curve gives $\cD_1, \cD_2, \cD_3, \cT_3$ whose joint law is described in~\eqref{eq-weld-nonsimple-iv}. \textbf{(v)} Concatenating $\cD_1$ with $\cD_3$ gives $\cD_4$, and shifting the first marked point of $\cT_3$ counterclockwise by $(a+b)$ units of generalized boundary length gives $\cT_4$. The law of $\cD_4, \cD_2, \cT_4$ is described by~\eqref{eq-weld-nonsimple-v}. \textbf{(vi)} Definition of $\wh M_{\mathrm{ccw}}$ via conformal welding of forested quantum surfaces. Only the pair of boundary arcs being welded are forested boundary arcs, so the resulting quantum surface is not forested. 
	}
\end{figure}

Let $\mathrm{Ann}'_2$ denote the space of quantum surfaces in $\mathrm{Ann}'$ with two additional marked points, one on each boundary component.
\begin{lemma}\label{lem-cut-cw}
There exists a measure $M_\mathrm{cw}$ on $\mathrm{Ann}'_2$ such that if $M_\mathrm{cw}^f$ denotes the law of a sample from $M_{\mathrm{cw}}$ with $\CLE_{\kappa'}$-decorated foresting, and $M_\mathrm{cw}^f(b)^\#$ denotes the law of a sample from $M_\mathrm{cw}^f$ conditioned on having inner generalized boundary length $b$, then the following holds:

Let $b>0$. Let $(\wt D, h, z,x)$ be an embedding of a sample from $\cM_{1,1,0}^\mathrm{f.d.}(\gamma, \gamma)$, let $D$ be the connected component of $\mathrm{int}(\wt D)$ containing $w$, and independently sample $\wt \eta$ and $\Gamma$ in $D$ as in Lemma~\ref{lem-cle-decomp}. 
	 Condition on the event that $\wt \eta$ closes a clockwise loop around $z$ at time $\sigma$. Let $\tau$ be the first time that $\wt \eta$ hits $\wt \eta(\sigma)$, and condition on the  loop $ \wt \eta|_{[\tau, \sigma]}$ having quantum length $b$.
      Then the conditional law of $(\wt D, h, \Gamma, \wt \eta|_{[0,\sigma]}, z, \wt \eta(\sigma), w)/{\sim_\gamma}$ is $\mathrm{Weld}(M_\mathrm{cw}^f(b)^\#, \cM_{1,1,0}^\mathrm{f.d.}(\gamma; b)^\#\otimes \CLE_{\kappa'})$.
\end{lemma}

As with Remark~\ref{rmk:GA-def}, 
Lemma~\ref{lem-cut-cw} can be informally rephrased as follows. In the setup and conditioning of Lemma~\ref{lem-cut-cw}, cutting by $\wt \eta|_{[\tau, \sigma]}$ gives a pair of conditionally independent decorated forested quantum surfaces $(\cA, \cD)$ corresponding to the regions outside and inside $\eta$ respectively; the conditional law of $\cD$ is  $\cM_{1,1,0}^\mathrm{f.d.}(\gamma; b)^\# \otimes \CLE_{\kappa'}$, and we denote the conditional law of $\cA$ by $M_\mathrm{cw}^f(b)^\#$. Furthermore, if we condition on $(\cA, \cD)$, the initial surface agrees in law with the uniform conformal welding of $\cA$ and $\cD$.  

\begin{proof}
    As explained in the second claim of \cite[Proposition 4.3]{asyz-nested-path} (with their parameter $\alpha$ set to $\gamma$), the curve $\wt \eta|_{[0,\sigma]}$ cuts the forested quantum surface into a pair whose joint law is 
    \[\left(\int \mathrm{QT}^f(\frac{3\gamma^2}2-2, 2-\frac{\gamma^2}2, \gamma^2-2; \ell + b, \ell) \times \cM_{1,1,0}^\mathrm{f.d.}(\gamma; b)\, d\ell\right)^\#.\]
    Here, $\mathrm{QT}^f(W_1,W_2,W_3)$ is the law of the forested quantum triangle with weights $W_1, W_2, W_3$, and $\mathrm{QT}^f(W_1,W_2,W_3; \ell + b, \ell)$ is the law of the quantum triangle whose boundary arc between the vertices of weights $W_1$ and $W_2$ (resp.\ $W_3$) has generalized boundary length $\ell+b$ (resp.\ $\ell$). 

    Thus, we can choose $M_\mathrm{cw}^f$ to be the law of a sample from $\iint_0^\infty\mathrm{QT}^f(\frac{3\gamma^2}2 - 2, 2 - \frac{\gamma^2}2, \gamma^2-2; \ell + a, \ell) \, d\ell \, da$ conformally welded to itself (and $M_\mathrm{cw}$ to be the law after forgetting the foresting). 
\end{proof}

\begin{proof}[Proof of Proposition~\ref{prop-ann-ns}]
Let $(\wt D, h, z, w)$ be an embedding of a sample from $\cM_{1,1,0}^\mathrm{f.d.}(\gamma)$, let $D$ be the connected component of $\mathrm{int}(\wt D)$ containing $z$, and independently sample $\wt \eta$ and $\Gamma$ in $D$ as in Lemma~\ref{lem-cle-decomp}. We first prove the analogous statement in this setting. On the event that $\wt \eta|$ closes a counterclockwise loop at time $\sigma$, the desired independence holds by Lemma~\ref{lem-cut-ccw}. If instead $\wt \eta$ closes a clockwise loop at time $\sigma$, 
by Lemma~\ref{lem-cut-ccw}, conditioned on the length of the loop $\ell$, the decorated quantum surface is obtained by conformally welding samples $(\cA, \cD) \sim M_\mathrm{cw}^f(\ell)^\#\times (\cM_{1,1,0}^\mathrm{f.d.}(\gamma; \ell)^\# \otimes \CLE_{\kappa'})$. We may then decompose $\cD$ in exactly the same way by coupling radial $\SLE_{\kappa'}(\kappa'-6)$ with its $\CLE_{\kappa'}$, and checking if the first loop closed is clockwise or counterclockwise, iterating until a loop is closed in the counterclockwise direction. The conclusion is that if $n \geq 0$ clockwise loops are closed before the first counterclockwise loop, then conditioned on $n$, the generalized boundary length $\ell_0$ of $(\wt D, h, z, w)$, and on the lengths $\ell_1, \dots, \ell_n$ of these clockwise loops, the conditional law of $(\wt D, h, z, w)/{\sim_\gamma}$ agrees with that of the conformal welding of a sample from 
\[ \int_0^\infty  \cdots \int_0^\infty  \left( \prod_{i=1}^n  M_\mathrm{cw}^f(\ell_{i-1}, \ell_i)^\# \times M_\mathrm{ccw}^f(\ell_n, b)^\# \right) \prod_{i=0}^n d\ell_i \, db.\]
Here, $M_\mathrm{cw}^f(\ell_{i-1}, \ell_i)^\#$ denotes the law of a sample from $M_\mathrm{cw}^f$ conditioned to have inner and outer generalized boundary lengths $(\ell_{i-1}, \ell_i)$, and similarly for  $M_\mathrm{ccw}^f(\ell_n, b)^\#$. We conclude that there exists a measure $\wt {\mathrm{GA}}^\mathrm{u.f.}$ on $\mathrm{Ann}_1 \cup \mathrm{Ann}'_1$ such that, if $\wt {\mathrm{GA}}^{\rm d}(b)^\#$ denotes the law of a sample from $\wt {\mathrm{GA}}^\mathrm{u.f.}$ with $\CLE_{\kappa'}$-decorated foresting conditioned on having inner generalized boundary length $b$, then the law of a sample from $\cM_{1,1,0}^\mathrm{f.d.}(\gamma)\otimes \CLE_{\kappa'}$ conditioned on the length of the loop surrounding the marked bulk point being $b$ is $\mathrm{Weld}(\wt {\mathrm{GA}}^{\rm d}(b)^\#, \GQD(b)^\# \otimes \CLE_{\kappa'})$.

Next, using Propositions~\ref{prop:gqd-fd} and~\ref{prop:101=gqd11}, an analogous result holds when $\cM_{1,1,0}^\mathrm{f.d.}(\gamma)$ is replaced by $\GQD_{1,1}$ (since we can concatenate an independent sample from $\GQD_{0,2}$ to the marked boundary point of a sample from $\cM_{1,1,0}^\mathrm{f.d.}(\gamma)$). The desired claim then follows by forgetting the boundary point of $\GQD_{1,1}$ and weighting by $L^{-1}$ where $L$ is the generalized boundary length.
\end{proof}

 \section{Quantum annulus }\label{sec:annulus}
 In this section, we introduce the \emph{quantum annulus} and its generalized variant, which will be crucial to our proof of Theorem~\ref{cor-loop-equiv-main}. Throughout this section, we assume   $\kappa\in (\frac83,4)$ and  $\kappa'\in (4,8)$.

Fix $a>0$.  For $\kappa\in (\frac83,4)$ and $\gamma=\sqrt{\kappa}$,
suppose $(D, h, z)$ is an embedding of a sample from $\QD_{1,0}(a)$. Let $\Gamma$ be a $\CLE_\kappa$ on $D$ independent of $h$. 
Recall that $\QD_{1,0}\otimes \CLE_\kappa$ is the law of the decorated quantum surface  $(D, h,  \Gamma, z)/{\sim_\gamma}$. 
 Let  $\eta$ be the outermost loop of $\Gamma$ surrounding $z$. Let $\frk l_\eta$ be the quantum length of $\eta$. 
 To ensure the existence of the disintegration of $\QD_{1,0}(a)\otimes\CLE_\kappa$ over $\frk l_\eta$, we check the following fact.
 \begin{lemma}\label{lem:1-pt-length}
 For a Borel set $E\subset \R$ with zero Lebesgue measure, $\QD_{1,0}(a)\otimes \CLE_\kappa[\frk l_\eta\in E]=0$.
 \end{lemma}
 \begin{proof}
 Let $(\ell_i)_{i \geq 1}$ be the quantum lengths of the outermost loops in a sample from $\QD_{1,0}(a) \otimes \CLE_\kappa$, ordered such that $\ell_1 > \ell_2 > \cdots$. By the explicit law of $(\ell_i)_{i\ge 1}$ from Proposition~\ref{prop-ccm}, for each $i > 0$ we have $\QD_{1,0}(a) \otimes \CLE_\kappa[\ell_i \in E] = 0$.  Since $\{\frk l_\eta\in E\} \subset \cup_i \{\ell_i \in E\}$, we conclude the proof.
 \end{proof}
 Given Lemma~\ref{lem:1-pt-length}, we can define  the disintegration of $\QD_{1,0}(a)\otimes\CLE_\kappa$ over $\frk l_\eta$, which we denote by   
 \(\{\QD_{1,0}(a)\otimes \CLE_\kappa(\ell): \ell\in (0,\infty)\}\). We now define the quantum annulus. 
 \begin{definition}[Quantum annulus]\label{def-QA}
 Given $(D,h,\Gamma, \eta, z)$ defined right above, let $A_{\eta}$ be the non-simply-connected component of $D\setminus \eta$. 
  For $a, b>0$, let $\wt \QA(a,b)$ be the law of the quantum surface  $(A, h)/{\sim_\gamma}$ under the measure $\QD_{1,0}(a)\otimes \CLE_\kappa(b)$. 
 Let $\QA(a,b)$ be such that 
 	\begin{equation}\label{eq:QA}
 	b|\QD_{1,0}(b)|\QA(a,b)=	 \wt \QA(a,b)
 	\end{equation}
 	We call a sample from $\QA(a,b)$ a \emph{quantum annulus}.
\end{definition}

 \begin{remark}[No need to say ``for almost every $b$'']\label{rem-ae-QA}
 	Using the general theory of regular conditional probability, for each $a>0$ the measure $\QA(a,b)$ is only well defined for almost every $b>0$.
	The ambiguity does not affect any application of this concept in this paper because we will take integrations over $b$; see e.g., Proposition~\ref{prop-QA2-pt} below. 
    On the other hand, one can give an equivalent definition of $\QA(a,b)$  for all $a,b>0$ in terms of the Liouville field on the annulus \cite[Theorem 1.4]{ars-annuli}\footnote{\cite{ars-annuli} builds on the present work but does not use the existence of a canonical definition of $\QA(a,b)$ for all $b>0$, so there is no circular dependence between the works.}
    For these reasons, we omit the phrase ``for almost every $b$" in statements concerning $\QA(a,b)$. 
\end{remark}

 Given a pair of independent samples from  $\QA(a,b)$ and $\QD_{1,0}(b)$,
 we can uniformly conformally weld them along the boundary component with length $b$ to get a loop-decorated quantum surface with one marked point.
We write  $\mathrm{Weld}(\QA(a,b),\QD_{1,0}(b))$ for the law of the resulting decorated quantum surface.
 Then Proposition~\ref{prop:single loop} can be reformulated as follows. 
\begin{proposition}\label{prop-QA2-pt}
	For $a>0$, let $(D,h,\Gamma,z)$ be an embedding of a sample from $\QD_{1,0}(a)\otimes\CLE_\kappa$. Let $\eta$ be the outermost loop of $\Gamma$ surrounding $z$.
	Then the law of the decorated quantum surface $(D,h,\eta,z)/{\sim_\gamma}$ equals 
	\(\int_0^\infty  b \mathrm{Weld}(\QA(a,b),\QD_{1,0}(b)) \, db\).
\end{proposition}
 \begin{proof}
	From Proposition~\ref{prop:single loop} and the definitions of $\wt \QA$ and uniform welding, the law of 
	$(D,h,\eta,z)/{\sim_\gamma}$ under $\QD_{1,0}(a)\otimes \CLE_\kappa(b)$ is 
	\( \mathrm{Weld}(\wt \QA(a,b),\QD_{1,0}(b)^{\#})\). By~\eqref{eq:QA}, this measure equals 
	\[
	\mathrm{Weld}(b|\QD_{1,0}(b)|\QA(a,b),\QD_{1,0}(b)^{\#})%
	=b\mathrm{Weld}(\QA(a,b), \QD_{1,0}(b)). \qedhere. 	\]
 \end{proof}

 The reason we consider $\QA$ instead of $\wt\QA$ in Definition~\ref{def-QA} is that it is in some sense more canonical.
 In particular, its total measure has the following simple and symmetric form. 
 \begin{proposition}\label{prop-QA-partition}
The total mass of $\QA(a,b)$ is 
 \[|\QA(a,b)|= \frac{\cos(\pi(\frac{4}{\gamma^2}-1))}{\pi\sqrt{ab} (a+b)}.\] 
 \end{proposition}
 More strongly, we will see in Proposition~\ref{prop-qa-symmetry} that $\QA(a,b) = \QA(b,a)$.

When defining the $\kappa' \in (4,8)$ variant of the quantum annulus, which we denote by $\GA(a,b)$, there are some subtleties if one follows the same fashion as in Definition~\ref{def-QA}, as we will explain in Section~\ref{section: GA}. We use an alternative treatment. Recalling that Proposition~\ref{prop-ann-ns} already fixes the definition of the probability measure $\GA(a,b)^\#$, we simply define $\GA(a,b)$ by specifying its total mass. Then we obtain the counterpart of Definition~\ref{def-QA} as a property of $\GA(a,b)$.

The rest of the section is organized as follows, we first prove  Proposition~\ref{prop-QA-partition} in Sections~\ref{subsec:change} and~\ref{subsec:annulus-length}. In Section~\ref{section: GA}, we treat  $\GA(a,b)$ in detail.

 \subsection{Reduction of Proposition~\ref{prop-QA-partition} to the setting of Proposition~\ref{prop-ccm}}\label{subsec:change}

 In this section, we reduce Proposition~\ref{prop-QA-partition} to the setting of Proposition~\ref{prop-ccm} to use the Levy process defined there. 
 Recall that the setting where  $(D, h,\Gamma )$ is an embedding of a sample of $\QD(\outleng)^\#\otimes\CLE_\kappa$ for some $a>0$.
  Sample a loop $\eta$ from  the counting measure on $\Gamma$, and let $\mathbb M_a$ be the law of the decorated quantum surface $(D, h, \Gamma, \eta)$. 
In other words, consider $((D, h, \Gamma)/{\sim_\gamma}, \mathbf n)$ sampled from the product measure 
 $(\QD(\outleng)^\#\otimes\CLE_\kappa) \times \mathrm{Count}_{\mathbb N}$ where $\mathrm{Count}_{\mathbb N}$ is the counting measure on $\mathbb N$. 
 Let $\eta\in \Gamma$ be the outermost loop with the $\bfn$th largest quantum length. Then    $\mathbb M_a$ is the law of $(D, h, \Gamma, \eta)/{\sim_\gamma}$. 
 The following proposition is the  analog of Proposition~\ref{prop-QA2-pt} for  $\mathbb M_a$.

 \begin{proposition}\label{prop-QA2-unpointed}
 	Under $\mathbb M_a$, the law of  $(D, h, \eta)/{\sim_\gamma}$ is 
 	\[\frac1{|\QD(a)|}\int_0^\infty b \mathrm{Weld}(\QA(a,b), \QD(b)) \, db.\] 
 \end{proposition}

 	\begin{proof} 
 		Let $\mathbb F$ be a measure on $H^{-1}(\D)$ such that the law of $(\D, h)/{\sim_\gamma}$  is $\QD(a)$ if $h$ is a sample  from $\mathbb M$. 
 		We write $\CLE_\kappa(d\Gamma)$ as the probability measure for $\CLE_\kappa$ on $\D$.  Write $\mathrm{Count}_{\Gamma^o}(d\eta)$ as the counting  measure on the set of outermost loops in $\Gamma$. Then we have the following equality on measures:
 		\eqb\label{eq-fubini-QA}
 		(1_E \mathrm{Count}_{\Gamma^o}(d\eta)) \mu_h(d^2z)\mathbb F(dh) \CLE_\kappa(d\Gamma) = \mu_h|_{D_\eta}(d^2z) \mathrm{Count}_{\Gamma^o}(d\eta)\mathbb F(dh) \CLE_\kappa(d\Gamma),
 		\eqe
		where $D_\eta$ is the simply-connected component of $\D \backslash \eta$, and $E = \{ z \in D_\eta\}$.

 		If $(h, \Gamma, \eta, z)$ is sampled according to the left hand side of~\eqref{eq-fubini-QA}, then the law of $(\D, h, \Gamma, z)/{\sim_\gamma}$ is $\QD_{1,0}(a) \otimes \CLE_\kappa$ and $\eta$ is the outermost loop around $z$. Therefore we are in the setting of Proposition~\ref{prop-QA2-pt} hence the law of $(\D, h, \eta, z)/{\sim_\gamma}$ is  $\int_0^\infty b \mathrm{Weld}(\QA(a,b), \QD_{1,0}(b)) \, db$.

 		If $(h, \Gamma, \eta)/{\sim_\gamma}$ is sampled from $\mathrm{Count}_{\Gamma^o}(d\eta) \mathbb F(dh) \CLE_\kappa(d\Gamma)$,  
		then the law of  $(D, h, \Gamma, \eta)/{\sim_\gamma}$ is $|\QD(a)|\cdot \mathbb M_a$ by the definition of $\mathbb M_a$. If we then weight by $\mu_h(D_\eta)$ and sample a point $z$ from $(\mu_h|_{D_\eta})^\#$, then the law of $(h, \Gamma, \eta, z)$ is given by the right hand side of~\eqref{eq-fubini-QA},  so the law  of $(D, h, \eta, z)/{\sim_\gamma}$ is $\int_0^\infty b \mathrm{Weld}(\QA(a,b), \QD_{1,0}(b)) \, db$. Unweighting by $\mu_h(D_\eta)$ and forgetting $z$,
 		we see that the law of 	$(D, h, \eta)/{\sim_\gamma}$ under $|\QD(a)|\cdot \mathbb M_a$ is $\int_0^\infty b \mathrm{Weld}(\QA(a,b), \QD(b)) \, db$.
 	Dividing by $|\QD(a)|$, we conclude the proof.
 	 	\end{proof}

 Proposition~\ref{prop-QA-partition} immediately follows from its counterpart under  $\mathbb M_a$. 
 \begin{proposition}\label{prop-annulus1}
 	The  law of  the quantum length of $\eta$ under $\mathbb M_a$ is 
	\begin{equation}\label{eq:length-dis}
 	\frac {\fC 1_{b>0} b   |\QD(b)|db}{\sqrt{ab} (a+b)|\QD(a)|} 	\quad \textrm{for constant } \fC =\frac{\cos(\pi(\frac{4}{\gamma^2}-1))}{\pi}.   
 	\end{equation}
\end{proposition}
 We now explain how Proposition~\ref{prop-annulus1} yields Proposition~\ref{prop-QA-partition} and then prove it in Section~\ref{subsec:annulus-length}
\begin{proof}[Proof of Proposition~\ref{prop-QA-partition} assuming Proposition~\ref{prop-annulus1}]
	By Proposition~\ref{prop-QA2-unpointed}, the law of $\ell_h(\eta)$ under $\mathbb M_a$ is $\frac1{|\QD(a)|} b |\QA(a,b)| |\QD(b)| \, db$. Comparing this to Proposition~\ref{prop-annulus1} yields $|\QA(a,b)| = \frac{\mathfrak C}{\sqrt{ab}(a+b)}$. 
 \end{proof}

 \subsection{Distribution of the loop length: proof of Proposition~\ref{prop-annulus1}}  \label{subsec:annulus-length}
 We first reduce Proposition~\ref{prop-annulus1} to a problem on the Levy process  in Proposition~\ref{prop-ccm}. Let  
 \begin{equation}\label{eq:lexp}
 \lexp= \frac4{\kappa} + \frac12=\frac{4}{\gamma^2}+\frac12.
 \end{equation}
 Let $\P^\lexp$ be the probability measure on 
 c\`adl\`ag processes on $[0,\infty)$ describing the law of a $\lexp$-stable L\'evy process  
 with L\'evy measure $1_{x>0} x^{-\beta-1} \, dx$.
 Let $(\zeta_t)_{t\geq 0}$ be a sample from $\P^\lexp$.
 Let $J:=\{(x,t): t\ge 0 \textrm{ and }\zeta_t-\zeta_{t^-}=x>0 \}$ be the set of jumps of $\zeta$. 
 Given  $J$, let $(\bfb, \bft)\in J$ be sampled from the counting measure on $J$. 
 Namely $(\bfb, \bft)$ is chosen uniformly randomly from $J$ similarly  as $\eta$ from the outermost loops of $\Gamma$.
 Let $M^\beta$ be the law of  $(\zeta, \bfb,\bft)$, which is an infinite measure. 
 For each $a>0$, let $\tau_{-a}=\inf\{ t: \zeta_t=-\outleng  \}$ and $J_a=\{(x,t)\in J: t\le \tau_{-a}\}$.
 Let $M^\lexp_a$ be the restriction of the measure $M^\lexp$ to the event $\{\bft\le \tau_{-a} \}$. 
 Let $\wt M^\beta_a = \frac{\tau_{-a}^{-1}}{\E^\lexp[\tau_{-a}^{-1}]} M^\lexp_a$, where $\E^\lexp$ is the expectation with respect to $\P^\lexp$.
 Then by Proposition~\ref{prop-ccm} we have the following. 
 \begin{lemma}\label{lem-Levy-transfer}
 	Suppose $(D,h,\Gamma,\eta)$ and $\mathbb M_a$ are as in Proposition~\ref{prop-QA2-unpointed}. 
 	Let $\ell_1 > \ell_2 > \dots$ be the quantum lengths of the outermost loops in $\Gamma$ and $\ell_h(\eta)$ be the quantum length of the distinguished loop $\eta$.
 	Then the joint law of  $\{ \ell_i: i\ge 1  \}$ and $\ell_h(\eta)$ under $\mathbb M_a$ 
 	equals the  joint law of  $\{ x: (x,t)\in J   \}$ and $\bfb$ under the measure $\wt M^\beta_a$ defined right above. 
\end{lemma}
 \begin{proof}
 	Under the measure $\wt M_a^\lexp$, the jump $(\bfb,\bft)$ is  chosen uniformly from $J_a$.
 	Now Lemma~\ref{lem-Levy-transfer} follows from Proposition~\ref{prop-ccm} and the fact that $\eta$  is chosen
 	uniformly among all outermost loops of $\Gamma$.
 \end{proof}
 Given~Lemma~\ref{lem-Levy-transfer},  Proposition~\ref{prop-annulus1} immediately  follows from the  proposition below. 
 \begin{proposition}\label{prop-Levy-jump}
 The law of $\bfb$ under $\wt M^\beta_a$ is 
 	\begin{equation}\label{eq:Levy-jump}
 	\frac {\fC 1_{b>0} }{(a+b)} \left( \frac{a}{b} \right)^{\lexp+1}	db  \quad \textrm{for constant }\fC =\frac{\cos(\pi(\frac{4}{\gamma^2}-1))}{\pi}   
 	\end{equation}
 \end{proposition}
 \begin{proof}[Proof of Proposition~\ref{prop-annulus1} given Proposition~\ref{prop-Levy-jump}]
 	By Proposition~\ref{prop-QD}, we have 
 	\[
 	|\QD(b)|/|\QD(a)|  =(a/b)^{\frac{4}{\gamma^2}+2}= (a/b)^{\lexp+\frac32}.
 	\]
 	By Lemma~\ref{lem-Levy-transfer} and Proposition~\ref{prop-Levy-jump}, the $\mathbb M_a$-law of the quantum length  of $\eta$ is 
 	\[  
 	\frac {\fC}{(a+b)} \left( \frac{a}{b} \right)^{\lexp+1} db=	\frac {\fC b   |\QD(b)| db}{\sqrt{ab} (a+b)|\QD(a)|} \quad  \textrm{ with $\fC$  in~\eqref{eq:Levy-jump}}. \qedhere
 	\]
 \end{proof}
 \begin{figure}[ht!]
 	\begin{center}
 		\includegraphics[scale=0.6]{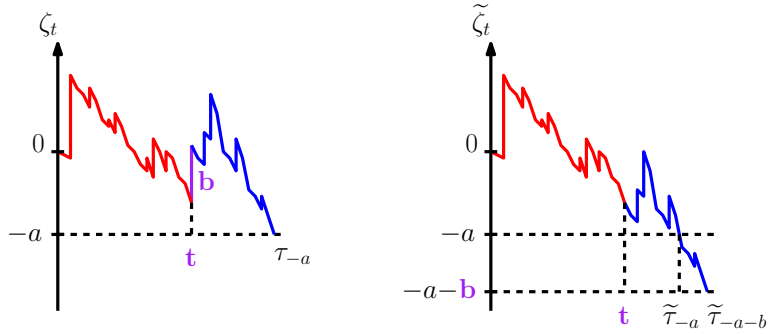}
 	\end{center}
 	\caption{\label{fig-levy-QA2}   Illustration of Lemma~\ref{lem-Palm} and~\ref{lem:new-measure}. The measure $M^\beta$ is $\P^\beta$ with a uniformly chosen jump $(\mathbf b, \mathbf t)$. Conditioned on $(\bfb, \bft)$, removing $(\bfb,\bft)$ and concatenating the two pieces of $(\zeta_t)$ before and after $\bft$ gives a sample $(\wt\zeta_t)_{t\ge 0}$ of   $\P^\beta$. The event  $\{ \bft <\tau_{-a}  \}$ for $\zeta$ becomes  $\{ \bft <\wt\tau_{-a}  \}$ for $\wt \zeta$.
 	} 
 \end{figure}

 We now prove Proposition~\ref{prop-Levy-jump}. 
 We first use  Palm's Theorem for Poisson point process to give an alternative description of the measure $M^\beta$. See Figure~\ref{fig-levy-QA2} for an illustration.
 \begin{lemma}\label{lem-Palm}
 	Let $\wt\zeta_t=\zeta_t- 1_{t\ge \bft} \bfb$. Then the $M^\beta$-law of $(\bfb,\bft)$ is $1_{b>0,t>0} b^{-\beta-1} \, db \, dt$.  
 	Conditioning on $(\bfb,\bft)$, the conditional law of $\wt\zeta_t$ under $M^\beta$ is $\P^\lexp$. 
 	Equivalently, the joint law of $(\bfb,\bft)$ and $(\wt\zeta_t)_{t\ge 0}$ is the product measure $(1_{b>0,t>0} b^{-\beta-1} \, db \, dt) \times \P^\lexp$.  
 \end{lemma}
 \begin{proof}
 	By  the definition of Levy measure, the jump set $J$ of $\zeta$ is a Poisson point process on $(0,\infty)^2$  with intensity measure $1_{x>0,t>0} x^{-\beta-1} \, dx \, dt$.  
 	Since $(\bfb, \bft)$ is chosen from the counting measure on $J$, by Palm's theorem  (see e.g.\ \cite[Page 5]{kallenberg-random-measures}),  
 	the $M^\lexp$-law of $(\bfb, \bft)$ is the same as the intensity measure of $J$, which is  $1_{b>0,t>0} b^{-\beta-1} \, db \, dt$.   
 	Moreover, conditioning on  $(\bfb,\bft)$, the conditional law of $J\setminus \{ (\bfb, \bft) \} $ is given by the original Poisson point process, which is the $\P^\lexp$-law of $J$. Note that  $(\wt\zeta_t)_{0\le t< \bft}=(\zeta_t)_{0\le t<\bft}$ is measurable with respect to the jump set $\{(x,t)\in J: t<\bft \}$. 
 	Therefore, the conditional law of $(\wt\zeta_t)_{0\le t< \bft}$ conditioning on $(\bfb, \bft)$ is  the $\P^\lexp$-law of $(\zeta_t)_{0\le t<\bft}$.
 	Similarly, conditioning on $(\bfb, \bft)$ and $(\wt\zeta_t)_{0\le t< \bft}$,  the conditional law of $(\wt\zeta_{t+\bft}-\wt\zeta_{\bft})_{t\ge 0}$ under $M^\beta$  is  $\P^\lexp$.
 	Concatenating $(\wt\zeta_t)_{0\le t< \bft}$ and $(\wt\zeta_{t+\bft}-\wt\zeta_{\bft})_{t\ge 0}$  we see that  the conditional law of $\wt\zeta_t$ under $M^\beta$ is $\P^\lexp$. 
 \end{proof}

 Proposition~\ref{prop-Levy-jump} is an immediate consequence of the following two lemmas.
 \begin{lemma}\label{lem:new-measure}
 	Recall the expectation  $\E^\lexp$  for $\P^\lexp$.	The law of $\bfb$ under $\wt M^\beta_a$ in Proposition~\ref{prop-Levy-jump} is  
 	$$\frac{\E^\lexp \left[ {\tau_{-a}}/{\tau_{-a-b}} \right]}{\E^\lexp[\tau^{-1}_{-a}]}  b^{-\lexp-1}  1_{b>0}db.  $$
 \end{lemma}
 \begin{proof}
 	We start from $(\zeta_t)_{t\ge 0}$ and  $(\bfb, \bft)$ under $M^\lexp$. 
 	Let $\wt\zeta_t=\zeta_t- 1_{t\ge \bft} \bfb$ as in~Lemma~\ref{lem-Palm}.   
   	Let $\wt\tau_{-\ell} = \inf\{t: \wt\zeta_t=-\ell  \}$ for each $\ell>0$. Then $\tau_{-a}= \wt \tau_{-a+\bfb}$. Moreover, 
 	both the events $\{ \bft <\tau_{-a} \}$ and $\{ \bft <\wt \tau_{-a} \}$ are the same as $ \inf\{ \zeta_t: t\in (0,\bft)\}>-a$, hence are equal.
 	Now the measure $\wt M^\beta_a$ can be described  as 
 	\begin{equation}\label{eq:Palm}
 	\frac{\tau^{-1}_{-a}}{\E^\lexp[\tau^{-1}_{-a}]}    M^\lexp_a=	\frac{\tau^{-1}_{-a}1_{\bft < \tau_{-a}}}{\E^\lexp[\tau^{-1}_{-a}]}   M^\lexp= \frac{ \wt \tau^{-1}_{-a-\bfb}1_{\bft < \wt \tau_{-a}}}{\E^\lexp[\tau^{-1}_{-a}]}   M^\lexp.
 	\end{equation}
 	Integrating out $\bft$ and $(\wt\zeta_t)_{t\ge 0}$ on the  right side of~\eqref{eq:Palm} and using the joint law of $(\bfb,\bft)$ and $(\wt\zeta_t)_{t\ge 0}$ from Lemma~\ref{lem-Palm}, we see that the  $\wt M^\beta_a$-law of $\bfb$ is $\frac{\E^\lexp \left[ {\tau_{-a}}/{\tau_{-a-b}} \right]}{\E^\lexp[\tau^{-1}_{-a}]}  b^{-\lexp-1}  1_{b>0}db$ as desired. 
 \end{proof}

 \begin{lemma}\label{lem-ratio}
	Let $(\zeta_t)_{t\ge 0}$ be sampled from $\P^\lexp$ and $\tau_{-a}=\inf\{t: \zeta_t=-a \}$ for each $a>0$.  
 	Then 
 	\[
 	\E^\lexp \left[ {\tau_{-a}}/{\tau_{-a-b}} \right]= \frac{a}{a+b} \textrm{ for each }  a,b>0.
 	\] 
  Moreover, $\E[\tau_{-1}^{-1}]=\frac{\pi}{\sin(-\pi\beta)}$.
 \end{lemma}

 \begin{proof}
The process $(\tau_{-s})_{s\ge 0}$ is a stable subordinator of index $1/\beta$. Since $(\tau_{-a}, \tau_{-2a})$ has the distribution as $(\tau_{-2a}-\tau_{-a}, \tau_{-2a})$, we have $\E^\lexp \left[ {\tau_{-a}}/{\tau_{-2a}} \right]=\frac12$. Similary, we have $\E^\lexp \left[ {\tau_{-a}}/{\tau_{-b}} \right]=\frac{a}{a+b}$ if $\frac{a}{b}$ is rational. By continuity we can extend this to all $a,b>0$. For the second equality, let $(Y_t)_{t\ge 0}$ be the Levy process with Levy measure $\Pi(dx)=\Gamma(-\beta)^{-1} 1_{x>0} x^{-\beta-1} dx$ so that $\E[e^{-\lambda Y_t}]=e^{t \lambda^{\beta}}$. Let $\sigma_{-s}=\inf\{t\ge 0: Y_t = -s \}$ for $s>0$. Then $\E[e^{-\lambda \sigma_{-s}}]=e^{s \lambda^{1/\beta}}$ and $\E[\sigma_{-1}^{-1}]=\int_0^\infty \E[e^{-\lambda \sigma_{-1} }] d\lambda=\Gamma(1+\beta)$.
  By scaling, we have  $\E[\tau_{-1}^{-1}]=\Gamma(-\beta) \E[\sigma^{-1}_{-1}]=\Gamma(-\beta)\Gamma (1+\beta)=\frac{\pi}{\sin(-\pi\beta)}$.
\end{proof}
 \begin{proof}[Proof of Proposition~\ref{prop-Levy-jump}]
 	With respect to $\P^\beta$, we have $\tau_{-a}\stackrel d= a^\beta \tau_{-1}$ by the scaling property of $\lexp$-stable L\'evy processes. Therefore $\E^\lexp[\tau^{-1}_{-a}] = \E^\lexp[\tau^{-1}_{-1}] a^{-\lexp}$. Combined with Lemmas~\ref{lem:new-measure} and~\ref{lem-ratio} we get~\eqref{eq:Levy-jump} with  $\fC= 1/\E^{\lexp}[\tau^{-1}_{-1}]=\frac{\cos(\pi(\frac{4}{\gamma^2}-1))}{\pi}$.  
 \end{proof}

\subsection{Generalized quantum annulus}\label{section: GA}
In this section, we introduce the {\it generalized quantum annulus}. It is a quantum surface with annulus topology \footnote{We allow the random annulus to possibly have pinch points.}.
 We assume $\kappa'\in (4,8)$ and $\gamma=\frac{4}{\sqrt{\kappa}'}\in (\sqrt{2},2)$ in this subsection.
 
Let $(D, h , \Gamma, z)$ be an embedding of a sample from $\mathrm{GQD}_{1,0}(a)^\# \otimes \CLE_{\kappa'}$ where $D$ is a bounded domain. Recall that a loop \emph{surrounds} a point if they are disjoint and the loop has nonzero winding number with respect to the point, and a loop $\eta \in \Gamma$ is \emph{outermost} if no point in $\eta$ is surrounded by a loop in $\Gamma$. Let $\eta$ be the outermost loop of $\Gamma$ surrounding $z$ and $\frk l_\eta$ be the generalized quantum length of $\eta$. To ensure the existence of disintegration of  $\mathrm{GQD}_{1,0}(a)^\# \otimes \CLE_{\kappa'}$ over $\frk l_\eta$, we need similar argument as Lemma \ref{lem:1-pt-length}. This is directly from Proposition \ref{prop-ccm ns}, which describes the joint law of generalized quantum length as the jump size of L\'evy process. So the disintegration of  $\mathrm{GQD}_{1,0}(a)^\# \otimes \CLE_{\kappa'}$ over $\frk l_\eta$ exists, which we denoted by $\{\mathrm{GQD}_{1,0}(a)^\# \otimes \CLE_{\kappa'}(\ell);\ell\in (0,\infty)\}.$

We first try to define the generalized quantum annulus $\GA(a,b)$ following the same approach as $\QA(a,b)$ in Definition~\ref{def-QA}, then explain that this ``definition'' requires more care to be correct.
\begin{fakedefinition}\label{fakedef-GA}
    Given $a>0$ and $(D, h , \Gamma, z)$ defined right above, let $\eta$ be the outermost loop of $\Gamma$ surrounding $z$. Let $D_\eta$ be the region surrounded by $\eta$ and $A_\eta=D\setminus D_\eta$. For $b>0$, let $\wt \GA(a,b)$ be the law of the quantum surface parametrized by $A_\eta$ 
    under the measure $\GQD_{1,0}(a)\otimes \CLE_{\kappa'}(b)$. 
 Let $\GA(a,b)$ be such that \begin{equation}\label{eq:QA ns}
 	b|\GQD_{1,0}(b)|\GA(a,b)=	 \wt \GA(a,b)
 	\end{equation}
\end{fakedefinition}
Under this ``definition'' one can obtain the analog of Proposition~\ref{prop-QA2-pt} via exactly the same method in Sections~\ref{subsec:change} and~\ref{subsec:annulus-length}:
\eqb\label{eq-ga-lengths}
|\GA(a,b)|=\frac{\cos(\pi(\frac{\gamma^2}{4}-1))}{\pi\sqrt{ab}(a+b)}. \eqe
The reason we will not actually use ``Definition''~\ref{fakedef-GA} is that the quantum surface $(A_\eta, h)/{\sim_\gamma}$ has too many self-intersection points. Indeed, in the conformal welding of a generalized quantum annulus and a generalized quantum disk (Proposition~\ref{prop-GA2-pt} below), each double point of the interface $\eta$ will correspond to a local cut point of exactly one of the generalized quantum surfaces. ``Definition''~\ref{fakedef-GA} can be fixed by removing the self-intersection points of $(A_\eta, h)/{\sim_\gamma}$ which correspond to local cut points of the generalized quantum surface surrounded by $\eta$. This could be done similarly as in \cite{msw-non-simple}.  However,
we instead give an alternative treatment which is simpler to present and follow. What is important for us is that~\eqref{eq-ga-lengths} and Proposition~\ref{prop-GA2-pt} (the analog of Proposition~\ref{prop-QA2-pt}) both hold for this definition. 

\begin{definition}[Generalized quantum annulus]\label{def-ga}
For $a,b>0$, define
\[\mathrm{GA}(a,b) := \frac{\cos(\pi(\frac{\gamma^2}{4}-1))}{\pi\sqrt{ab}(a+b)} \mathrm{GA}(a,b)^\# \]
where $\mathrm{GA}(a,b)^\#$ is the probability measure from Proposition~\ref{prop-ann-ns}. 	We call a sample from $\GA(a,b)$ a \emph{generalized quantum annulus}.
\end{definition}

\begin{proposition}\label{prop-GA2-pt}
	For $a>0$, let $(D,h,\Gamma,z)$ be an embedding of a sample from $\GQD_{1,0}(a)\otimes\CLE_{\kappa'}$. Let $\eta$ be the outermost loop of $\Gamma$ surrounding $z$.
	Then the law of the decorated quantum surface $(D,h,\eta,z)/{\sim_\gamma}$ equals 
	\(\int_0^\infty  b \mathrm{Weld}(\GA(a,b),\GQD_{1,0}(b)) \, db\), as well as the law of the decorated quantum surface $(D,h,\Gamma,\eta,z)/{\sim_\gamma}$ equals \(\int_0^\infty  b \mathrm{Weld}(\GA^{\rm d}(a,b),\GQD_{1,0}(b)\otimes\CLE_{\kappa'}) \, db\).
\end{proposition}
\begin{proof}
We only prove the second claim since the first claim will then follow by just forgetting the loop configuration.
By Proposition~\ref{prop-ann-ns}, 
conditioned on the length of $\eta$ being $b$, the conditional law of $(D, h, \eta, \Gamma, z)/{\sim_\gamma}$ is $\mathrm{Weld}(\mathrm{GA}^{\rm d}(a,b)^\#, \GQD_{1,0}(b)^\#\otimes\CLE_{\kappa'})$. It thus suffices to show that the law of the length of $\eta$ is $b \frac{\cos(\pi (\frac{\gamma^2}4-1))}{\pi \sqrt{ab}(a+b)} |\GQD_{1,0}(b)|\, db$. 
Exactly as in Section~\ref{subsec:change}, this reduces to showing that for a sample from $\GQD(a)^\# \otimes \CLE_{\kappa'}$ and an outermost loop chosen from counting measure, the law of the loop length is $\frac{\cos(\pi (\frac{\gamma^2}4-1))}{\pi \sqrt{ab}(a+b)} \frac{b|\GQD(b)|}{|\GQD(a)|}\, db $ (the analog of Proposition~\ref{prop-annulus1}), and by Lemma~\ref{length gqd} this latter measure can be written as  $\frac{\cos(\pi (\beta' - \frac32))}{\pi(a+b)} (\frac ab)^{\beta' +1}\, db$  where $\beta' = \frac4{\kappa'}+ \frac12$. 

Thus, by Proposition~\ref{prop-ccm ns}, it suffices to prove Proposition~\ref{prop-Levy-jump} when the L\'evy measure is $1_{x>0} x^{-\beta'-1}\,dx$ with $\beta' = \frac4{\kappa'}+ \frac12 \in (1, \frac32)$ rather than $1_{x>0} x^{-\beta-1} \, dx$ with $\beta = \frac4\kappa+\frac12 \in (\frac32, 2)$. The proof of Proposition~\ref{prop-Levy-jump} directly carries over to the parameter range $(1, \frac32)$ so we are done. \qedhere

Finally, we give a corrected version of ``Definition''~\ref{fakedef-GA}, where instead of attempting to define the forested boundary by removing the extra self-intersection points of $(A_\eta, h)/{\sim_\gamma}$, we simply forget and resample the forested boundary using the notion of a forested line segment (Definition~\ref{def-forested-line-segment}). We will not use this description in the present work.
\begin{proposition} \label{prop-alternative-defn-GA}
    Let $(D, h, \Gamma, z)$ be an embedding of a sample from $\GQD_{1,0} \otimes \CLE_{\kappa'}$ and let $\eta \in \Gamma$ be the outermost loop around $z$. Let $\eta^\mathrm{out}$ be the boundary of the unbounded connected component of $\C\backslash \eta$, and let $D(\eta^\mathrm{out})$ be the bounded connected component of $\C \backslash \eta^\mathrm{out}$. Choose $p \in \eta^\mathrm{out}$ such that  $(D \backslash D(\eta^\mathrm{out}), \Gamma, z, p)/{\sim_\gamma}$ is measurable with respect to  $(D \backslash D(\eta^\mathrm{out}), \Gamma, z)/{\sim_\gamma}$. 
     Conditioned on 
    $(D \backslash D(\eta^\mathrm{out}), \Gamma, z,p)/{\sim_\gamma}$, sample a forested line segment $\cL$ conditioned on having quantum length (resp.\ generalized boundary length) equal to that of $\eta^\mathrm{out}$ (resp.\ $\eta$), and in each connected component of $\cL$ sample an independent $\CLE_{\kappa'}$. Glue $\cL$ to the boundary loop $\eta^\mathrm{out}$ of  $(D \backslash D(\eta^\mathrm{out}), \Gamma, z, p)/{\sim_\gamma}$, identifying the endpoints of $\cL$ with $p$. Let $\wh \GA^{\rm d}$ be the law of the resulting generalized quantum surface, and let $\wh \GA^{\rm d}(a,b)$ be the disintegration of $\wh \GA^{\rm d}$ according to the outer and inner generalized boundary lengths. Then $\GA^{\rm d}(a,b) = (b\GQD_{1,0}(b))^{-1} \wh \GA^{\rm d}(a,b)$.
\end{proposition}
\begin{proof}
    This is immediate from Proposition~\ref{prop-GA2-pt} and the fact that for a forested quantum surface (such as a sample from $\GA^{\rm d}$), conditioned on the unforested quantum surface, the foresting is described by a forested line segment with specified quantum length. 
\end{proof}

\end{proof}

\section{Equivalence of two definitions of  the SLE loop}\label{sec:msw-sphere}

In this section we prove Theorem~\ref{cor-loop-equiv-main}, namely  for $\kappa\in (\frac{8}{3},8)$, $\wt\SLE_\kappa^{\mathrm{loop}}=C\SLE_\kappa^{\mathrm{loop}}$ for some $C\in (0,\infty)$, where 
$\SLE_\kappa^{\mathrm{loop}}$ is Zhan's loop measure (Definition~\ref{def:loop}), and $\wt\SLE_\kappa^{\mathrm{loop}}$ is the loop measure obtained from
choosing a loop in a full-plane $\CLE^{\C}_\kappa$ from the counting measure. 
For $\kappa \in (\frac83, 4]$, this was already implicitly shown in \cite{werner-sphere-cle}; see \cite[Theorem 2.18]{ang2021integrability} for details.

Recall  from Section~\ref{subsec:KW-intro} the probability measures $\cL_\kappa$ and $\wt \cL_\kappa$ that describe the shapes of  $\SLE_\kappa^{\mathrm{loop}}$ and $\wt\SLE_\kappa^{\mathrm{loop}}$, respectively. Namely, given a loop $\eta$ on $\C$ surrounding 0,  we have  $R_\eta=\inf\{ |z|:z\in \eta \}$ and $\hat\eta= \{z: R_\eta z\in \eta \}$.
Suppose  $\eta$ is a sample from  $\SLE_\kappa^{\mathrm{loop}}$ restricted to the event that $\eta$ surrounds $0$. 
Then the law of $(\hat \eta, \log R_\eta)$ is a constant multiple of the product measure $\cL_\kappa\times dt$, where $dt$ is the Lebesgue measure on $\R$.
The same holds for $\wt\SLE_\kappa^{\mathrm{loop}}$ and $\wt \cL_\kappa$ when $\kappa\in (\frac83,8)$.
By definition, Theorem~\ref{cor-loop-equiv-main} is equivalent to  $\cL_\kappa=\wt \cL_\kappa$.

The main device for the proof of $\cL_\kappa=\wt \cL_\kappa$ is   a natural Markov chain with stationary distribution $\wt\cL_\kappa$.
We describe it in the cylinder coordinate for convenience.
Consider the horizontal cylinder $\cC = \R \times [0,2\pi]$ with $x\in \R$  identified  with $x + 2\pi i$. We also include $\pm\infty$
in $\cC$ so that $\cC$ is conformally equivalent to a Riemann sphere. 
Let $\psi(z)=e^{-z}$ be a conformal map  from $\cC$ to $\C\cup \{\infty\}$ with $\psi(+\infty)=0$ and $\psi(-\infty)=\infty$.  
\begin{definition}
	Fix $\kappa\in (0,8)$. Let  $\eta$ be a sample from  $\SLE_\kappa^{\mathrm{loop}}$ restricted to the event that $\eta$ surrounds $0$, hence $\psi^{-1}(\eta)$ is a loop on $\cC$ separating $\pm\infty$.	Let $\SLE^{\sep}_\kappa(\cC)$ be the law of $\psi^{-1}(\eta)$. We also write $\cL_\kappa(\cC) := (\psi^{-1})_* \cL_\kappa$ and $\wt\cL_\kappa(\cC) := (\psi^{-1})_* \wt\cL_\kappa$.
\end{definition}
Let $\mathrm{Loop}_0(\cC)$ be the set of simple loops $\eta$ on $\cC$  separating $\pm\infty$ such that 
$\max\{ \Re z: z\in \eta \}= 0$. The relation between $\SLE^{\sep}_\kappa(\cC)$ and $\cL_\kappa(\cC)$ is the following. 
\begin{lemma}\label{lem-translation}
	Let  $\eta$ be a sample from  $\SLE^{\sep}_\kappa(\cC)$. 
	Then $\eta$ can be uniquely written as $\eta^0+\bft$ where $\eta^0\in \mathrm{Loop}_0(\cC)$ and $\bft\in \R$.
	The law of $(\eta^0,\bft)$ is $C\cL_\kappa(\cC)\times dt$ for a constant $C\in (0,\infty)$, where  $dt$ is the Lebesgue measure on $\R$. 
\end{lemma}
\begin{proof}
	If $\eta$ is a sample from  $\SLE_\kappa^{\mathrm{loop}}$ restricted to the event that $\eta$ surrounds $0$, the  law of $(\hat \eta, \log R_\eta)$ is a constant multiple of $\cL_\kappa\times dt$. Therefore Lemma~\ref{lem-translation} follows from mapping $\C$ to $\cC$.
\end{proof}

In Sections~\ref{sec-thm1-markov}--\ref{subsec:tv} we prove Theorem~\ref{cor-loop-equiv-main} for $\kappa \leq 4$; the arguments will mainly focus on $\kappa < 4$, while the case $\kappa = 4$ is obtained by a limiting argument at the end. A proof outline is given in Section~\ref{sec-thm1-markov}. In Section~\ref{eq ns} we carry out the same argument again for $\kappa \in (4,8)$. 
\subsection{The case \texorpdfstring{$\kappa\in (\frac{8}{3},4]$}{g0}:  setup and proof overview}\label{sec-thm1-markov}
The goal of Sections~\ref{sec-thm1-markov}--\ref{subsec:tv} is to prove the following. 
\begin{proposition}\label{cor-loop-equiv}
	For each $\kappa\in (\frac83,4]$,  	there exists  $C\in (0,\infty)$ such that $\wt\SLE_\kappa^{\mathrm{loop}}=C\SLE_\kappa^{\mathrm{loop}}$.
\end{proposition}

Our proof uses the following Markov chain. 
By definition, $\wt \cL_\kappa(\cC)$ is a probability measure on $\mathrm{Loop}_0(\cC)$. 
Given $\eta^0\in \mathrm{Loop}_0(\cC)$, let $\cC^+_{\eta^0}$ be  the connected component of $\cC\setminus \eta^0$ containing $+\infty$.
Sample a $\CLE_\kappa$ on $\cC^+_{\eta^0}$ and  translate its outermost loop surrounding $+\infty$ to an element $\eta^1\in \mathrm{Loop}_0(\cC)$. 
Then $\eta^0 \rta \eta^1$ defines a Markov transition kernel on $\mathrm{Loop}_0(\cC)$.

Recall that the total variational distance between two measures $\mu$ and $\nu$ on a measurable space $(\Omega,\cF)$ is $d_\tv(\mu,\nu) := \sup_{A\in \cF} |\mu(A)-\nu(A)|$. 
\begin{lemma}\label{lem-Markov}
	Fix  $\eta^0\in \mathrm{Loop}_0(\cC)$ and let $(\eta^i)_{i\ge 1}$ be the Markov chain starting from $\eta^0$.
	Then the law of $\eta^n$ tends to $\wt\cL_\kappa(\cC)$ in total variational distance. In particular,	
	$\wt \cL_\kappa(\cC)$  is the unique stationary probability measure of the Markov chain.
\end{lemma}

\begin{proof}
	This is an immediate consequence of \cite{werner-sphere-cle}; see Proposition~\ref{prop-kemp-werner} for details.
\end{proof}

The starting point of our proof of $\cL_\kappa(\cC)=\wt \cL_\kappa(\cC)$ is the following conformal welding result  for $\SLE^{\sep}_\kappa(\cC)$.
Fix  $a>0$. A pair of quantum surfaces sampled from $\QD_{1,0}(a)\times \QD_{1,0}(a)$ can be uniformly conformally welded to get a loop-decorated quantum surface with two marked points (corresponding to the interior marked points of each surface). 
We write  $\mathrm{Weld}(\QD_{1,0}(a),\QD_{1,0}(a))$  for the law of this  loop-decorated quantum surface with two marked points.
\begin{proposition}\label{prop-loop-zipper}
	Fix $\kappa\in (\frac{8}{3},4)$ and $\gamma=\sqrt{\kappa}$.
	Let $\mathbb F$ be the law of $h$ as in Definition~\ref{def-sphere}, so that  the law of $(\cC, h,+\infty,-\infty)/{\sim_\gamma}$  is the two-pointed quantum sphere $\QS_2$.
	Now sample $(h,\eta)$ from $\mathbb F \times \SLE^{\sep}_\kappa(\cC)$
	and write $\QS_2\otimes \SLE_\kappa^{\sep}$ as 	the law of $(\cC, h, \eta, +\infty,-\infty)/{\sim_\gamma}$. 
	Then 
	\begin{equation}\label{eq-loop-marked}
	\QS_2\otimes \SLE_\kappa^{\sep}= 
	C  \int_0^\infty a\cdot \mathrm{Weld}(\QD_{1,0}(a), \QD_{1,0}(a)) \, da \quad \textrm{ for some }C>0.
	\end{equation}
\end{proposition}
\begin{proof}
	The proof is similar to that of Proposition~\ref{prop-QA2-unpointed}.
	Let $\mathbb F_0$ be a measure on $H^{-1}_\mathrm{loc}(\C)$ such that the law of $(\C, h)/{\sim_\gamma}$  is the unmarked quantum sphere $\QS$ if $h$ is a sample  from $\mathbb F_0$.
	Let $(h,\eta, z,w)$ be sampled from $\mu_h(d^2z)\,\mu_h(d^2w)\,\mathbb F_0(dh) \SLE^{\mathrm{loop}}_\kappa (d\eta)$. 
	Let $E_{\sep}$ be the event that $\eta$ separates $z,w$. Then the law of $(h,\eta, z,w)$ restricted to $E_{\sep}$ can be obtained  in two ways:
	\begin{enumerate}
		\item Sample $(h,z,w)$ from $\mu_h(d^2z)\,\mu_h(d^2w)\,\mathbb F_0(dh)$; then sample $\eta$ from $\SLE^{\mathrm{loop}}_\kappa$ on $\C$ and restrict  to the event that $\eta$ separates $z,w$. 
		\item Sample $(h,\eta)$ from $2\mu_h(D_\eta)\mu_h(D'_\eta)\,\mathbb F_0(dh) \SLE^{\mathrm{loop}}_\kappa (d\eta)$, where $D_\eta, D_\eta'$ are the connected components of $\C \backslash \eta$; then sample $(z,w)$ from the probability measure proportional to 
		$$\mu_h|_{D_\eta}(d^2z) \mu_h|_{D'_\eta}(d^2w) +  \mu_h|_{D'_\eta}(d^2z) \mu_h|_{D_\eta}(d^2w).$$
	\end{enumerate}
	Recall Theorem~\ref{thm-loop2} and the definitions of $\QS_2$ and  $\QD_{1,0}$ from Section~\ref{subsub:quantum-surface}.
	The law of $(h,\eta, z,w)/{\sim_\gamma}$ restricted to $E_{\sep}$, from the first sampling,  equals $\QS_2\otimes \SLE_\kappa^{\sep}$.
	From the second sampling,  the same law equals  $C  \int_0^\infty a\cdot \mathrm{Weld}(\QD_{1,0}(a), \QD_{1,0}(a)) \, da$ for some $C\in (0,\infty)$.
\end{proof}

In the rest of this section we consider the setting where Lemma~\ref{lem-translation} and Proposition~\ref{prop-loop-zipper} apply.
For $\kappa\in (8/3,4)$, let $(h,\eta)$ be a sample from  $\mathbb F \times \SLE^{\sep}_\kappa(\cC)$. 
Then $\eta$ can be uniquely written as $\eta^0+\bft$  where $\eta^0\in \mathrm{Loop}_0(\cC)$ and $\bft\in \R$. By Lemma~\ref{lem-translation}, the law of
$(\eta^0,\bft)$ is  $C\cL_\kappa(\cC)\times dt$. 
Conditioning on $(h,\eta)$, sample a $\CLE_\kappa$ $\Gamma^+$ in $\cC^+_\eta$, where $\cC^+_\eta$ is  the component of $\cC\setminus\eta$ containing $+\infty$.
Let   $(\eta_i)_{i \geq 1}$ be  the set of  loops in $\Gamma^+$ separating $\pm\infty$ ordered from left to right.
We write $\eta_i=\eta^i+t_i$ where $\eta^i\in \mathrm{Loop}_0(\cC)$ and $t_i\in \R$. 
See Figure~\ref{fig-loop-cylinder}  for an illustration.

\begin{figure}[ht!]
	\begin{center}
		\includegraphics[scale=0.7	]{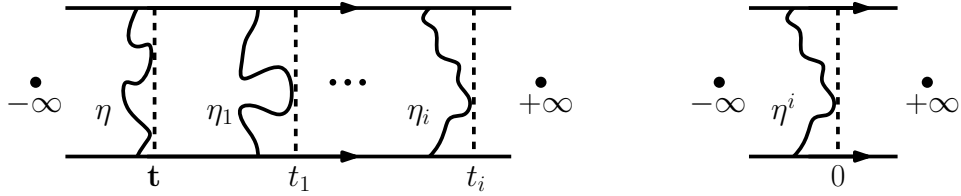}%
	\end{center}
	\caption{\label{fig-loop-cylinder}    \textbf{Left:} 	Illustration of $\eta_i$ and $t_i$. 
		\textbf{Right:} 	Illustration of $\eta^i$.} 
\end{figure}
To prove $\cL_\kappa(\cC)=\wt\cL_\kappa(\cC)$,
consider the decorated quantum surface $S_0=(\cC,h,\eta, \pm \infty)/{\sim_\gamma}$ and $S_i=(\cC,h,\eta_i, \pm \infty)/{\sim_\gamma}$ for $i\ge 1$. 
If $\cL_\kappa(\cC)=\wt\cL_\kappa(\cC)$ holds,  then the stationarity of $\wt\cL_\kappa(\cC)$ gives the stationarity of $(S_i)_{i\ge 0}$.
Although we cannot show this directly before proving   $\cL_\kappa(\cC)=\wt\cL_\kappa(\cC)$,  we will prove using Proposition~\ref{prop-QA2-pt}  
that the law of the subsurface of $S_i$ on the right side of $\eta_i$  indeed does not depend on $i$.
This is the content of Section~\ref{subsec:markov}. 
In Section~\ref{subsec:loop-proof} we use Lemma~\ref{lem-Markov} to show that as $i \to \infty$ the  law of $S_i$
converges to $\QS_2\otimes \wt\SLE^{\sep}_\kappa$ in a suitable sense, where $\QS_2\otimes \wt\SLE^{\sep}_\kappa$ is defined as
$\QS_2\otimes\SLE^{\sep}_\kappa$ in Proposition~\ref{prop-loop-zipper} with $\wt\SLE^{\mathrm{loop}}_\kappa$ in place of 
$\SLE^{\mathrm{loop}}_\kappa$. This convergence would be immediate in the total variational sense if the law of $S_i$ were a probability measure
instead of being \emph{infinite}. To handle this subtlety we will prove some intermediate results  that will also be useful in Section~\ref{sec:KW}.

Once the two steps in the previous paragraph are achieved,  we will have that the right subsurface of  $\QS_2\otimes \wt\SLE^{\sep}_\kappa$
has the same law as that of $\QS_2\otimes\SLE^{\sep}_\kappa$. Then by  left/right symmetry we can conclude that $\QS_2\otimes \wt\SLE^{\sep}_\kappa=\QS_2\otimes\SLE^{\sep}_\kappa$, hence $\wt\cL_\kappa(\cC)=\cL_\kappa(\cC)$.
We carry out these final steps and complete the proof of Proposition~\ref{cor-loop-equiv} in Section~\ref{subsec:loop-proof}.
In Section~\ref{subsec:tv} we supply a technical ingredient in the proof coming from the fact that $\QS_2\otimes\SLE^{\sep}_\kappa$ is an infinite measure.

\subsection{Stationarity of the subsurfaces on the right side of the loops}\label{subsec:markov}
Consider $(h,\eta)$ and $(\eta^i)_{i\ge 0}$ as defined in and below Proposition~\ref{prop-loop-zipper}.
Let $S_0=(\cC,h,\eta, \pm \infty)/{\sim_\gamma}$ and $S_i=(\cC,h,\eta_i, \pm \infty)/{\sim_\gamma}$ for $i\ge 1$ as discussed right above Section~\ref{subsec:markov}.
For $i\ge 0$, let $\ell_h(\eta_i)$ be the quantum length of $\eta_i$. 
\begin{lemma}\label{lem:stationarity-length}
The law of  $\ell_h(\eta_i)$  is the same as that of $\ell_h(\eta)$ in Proposition~\ref{prop-loop-zipper} for all $i\ge 0$. 
\end{lemma}
\begin{proof}
	By definition the law of $\ell_h(\eta_0)$ is the same as that of $\ell_h(\eta)$ in Proposition~\ref{prop-loop-zipper}.
	It remains to show that the law of $\ell_h(\eta_i)$ does not depends on $i\ge 0$. By  Propositions~\ref{prop-QA2-pt} and~\ref{prop-loop-zipper},
	the joint law of $\ell_h(\eta_0)$ and $\ell_h(\eta_1)$ is given by 
    \[
Cab|\QD_{1,0}(a)| |\QA(a,b)| |\QD_{1,0}(b)| \,da\,db \quad \textrm{for some }C\in(0,\infty).\]
	Since $|\QA(a,b)|=|\QA(b,a)|$ by Proposition~\ref{prop-QA-partition}, the laws  of $\ell_h(\eta_0)$ and $\ell_h(\eta_1)$  are the same. 
	By  Proposition~\ref{prop-QA2-pt}, and the fact that $\Gamma^+$ restricted to $\cC^+_{\eta_i}$ is still a $\CLE_\kappa$,
	the conditional law of $\ell_h(\eta_i)$ given $\ell_h(\eta_{i-1})$ does not depend on $i$. 
	Since $\ell(\eta_0)$ and $\ell_h(\eta_1)$  have the same law, so do   $\ell_h(\eta_i)$ and $\ell_h(\eta_{i+1})$ for all $i\ge 1$.
\end{proof}

\begin{lemma}\label{lem:stationarity}
	For $i\ge 0$, conditioning on $(h,\eta_i)$, let $z_i$ be a point on $\eta_i$ sampled from the probability measure proportional to the quantum length measure. 
	Let $S^+_i= (\cC^+_{\eta_i}, h,+\infty, z_i)/{\sim_\gamma}$.  
	Let  $\cC^-_{\eta_i}$ be the component of $\cC\setminus\eta_i$ containing $-\infty$ and $S_i^-=(\cC^-_{\eta_i},h, -\infty, z_i)/{\sim_\gamma}$. 
	Then the conditional law of $S_i^+$ given $S_i^-$ is $\QD_{1,1}(\ell_{h}(\eta_i))^{\#}$ for all $i\ge 0$.
\end{lemma}
\begin{proof}
	When $i=0$, this follows from Theorem~\ref{thm-loop2}.
	The $i\ge 1$ case follows iteratively using Proposition~\ref{prop-QA2-pt}.
\end{proof}
From Lemmas~\ref{lem:stationarity-length} and~\ref{lem:stationarity}, we see that the law of $S_i^+$ is a quantum disk whose law does not depend on $i$. 
Moreover, conditional on $\ell_h(\eta_i)$, $S^+_i$ and $S^-_i$ are conditionally independent and $S$ is obtained from the conformal welding of $S^+_i$ and $S^-_i$.
\subsection{Convergence of \texorpdfstring{$S_i$}{S}  and the proof of Proposition\texorpdfstring{~\ref{cor-loop-equiv}}{g0}} \label{subsec:loop-proof}
We retain the setting in Section~\ref{subsec:markov}.
For $z\in \cC$, let $\phi^0(z )=h(z + \bft)$ and $\phi^i (z )=h(z + t_i)$ for $i\ge 1$.
By definition, $(\cC,\phi^i, \eta^i)/{\sim_\gamma} $ is an embedding of the decorated quantum surface $S_i$ for $i\ge 0$. 
By Theorem~\ref{thm-sph-field}, we immediately have the following description of the law of  $(\phi^i,\eta^i)$.
\begin{lemma}\label{lem:coupling}
	Recall the Markov chain in Lemma~\ref{lem-Markov} and the Liouville field $\LF_\cC^{(\gamma, \pm\infty)}$ from Definition~\ref{def-LFC}.
	Let $P_0=\cL_\kappa(\cC)$. For $i\ge 1$, let $P_i$ be the distribution of the $i$-th step of this Markov chain on $\mathrm{Loop}_0(\cC)$ with the initial distribution $P_0$.  
	Then the law of $(\phi^i,\eta^i)$ is $C\LF_\cC^{(\gamma, \pm\infty)}\times P_i$ for all $i\ge 0$, where $C\in (0,\infty)$ is a constant that does not depend on $i$.
\end{lemma}
\begin{proof}
	By definition, the law of $(h, \eta^0, \mathbf t)$ is $\mathbb F \times P_0 \times [C\, dt]$, where $\mathbb F$ is the law of $h$ in Definition~\ref{def-sphere}. By the definition of $P_i$ and the translation invariance of $[C\, dt]$, the law of $(h, \eta^i, t_i)$ is $\mathbb F \times P_i \times [C\, dt]$. Finally, Theorem~\ref{thm-sph-field} yields the result. 
\end{proof}
From now on we write the measure $C\LF_\cC^{(\gamma, \pm\infty)}$ in Lemma~\ref{lem:coupling} as $\LF_\cC$ for simplicity so that the law of $(\phi^i,\eta^i)$ is $\LF_\cC\times P_i$.

From Lemma~\ref{lem-Markov} we know that $d_\tv(P_i,\wt\cL_\kappa(\cC))\rta 0$. 
We would like to say that $d_\tv(\LF_\cC\times P_i,\LF_\cC\times \wt\cL_\kappa)\rta 0$ as well, but since  
$\LF_\cC$ is an infinite measure we need a truncation.  Recall that the quantum length  $\ell_h(\eta)$ is a measurable function of  $(\phi^0,\eta^0)$~\cite{shef-zipper}. We write this as $\ell_h(\eta)=\ell(\phi^0,\eta^0)$.
Then $\ell_h(\eta_i) = \ell(\phi^i,\eta^i)$ a.s.\ for  each $i\ge 0$.   For $\eps>0$, consider the event $E_\eps = \{ \ell(\phi,\eta)> \eps\}$. 
Then by Lemma~\ref{lem:stationarity-length} we have  
\begin{equation}\label{eq:LF-trunc}
\LF_\cC\times P_i(E_\eps)=\LF_\cC\times P_0(E_\eps)\quad  \textrm{ for each } i\ge 0 \textrm{ and } \eps>0.
\end{equation}
We claim that the total variational convergence holds after restriction to $E_\eps$.
\begin{lemma}\label{lem:tv}
	For each $\eps>0$, $\lim_{i\to \infty} d_{\tv}(\LF_\cC\times P_i |_{E_\eps}, \LF_\cC\times \wt \cL_\kappa(\cC) |_{E_\eps}) =0$.
\end{lemma}
The proof of Lemma~\ref{lem:tv} is not hard but technical so we postpone it to  Section~\ref{subsec:tv}. 
We proceed to finish the proof of Proposition~\ref{cor-loop-equiv}. 
\begin{proof}[Proof of Proposition~\ref{cor-loop-equiv} given Lemma~\ref{lem:tv}]
	Let $(\phi, \wt\eta)$ be a sample from $\LF_\cC\times \wt \cL_\kappa(\cC)$.
	Conditioning on $(\phi,\wt\eta)$, let $z\in \wt\eta$ 
	be sampled from the probability measure proportional to the quantum length measure of $\wt \eta$.
	Recall the decorated quantum surfaces $S_i^+$ and $S_i^-$ from Lemma~\ref{lem:stationarity}. 
	We similarly define $\wt S^+=(\cC_{\wt\eta}^+, \phi,+\infty, z)/{\sim_\gamma}$ and $\wt S^-=(\cC_{\wt\eta}^-, \phi,-\infty, z)/{\sim_\gamma}$, where $\cC_{\wt\eta}^+$ (resp. $\cC_{\wt\eta}^-$) is the component of $\cC\setminus \wt\eta$ to the right (resp., left) of $\wt\eta$. 
	By Lemma~\ref{lem:tv}, restricted to the event  $E_\eps$, the law of $(S_i^+,S_i^-)$ converge in total variational distance to that of $(\wt S^+,\wt S^-)$. 
	Since $\eps$ can be arbitrary, 
	by Lemmas~\ref{lem:stationarity-length} and~\ref{lem:stationarity} and Equation~\eqref{eq:LF-trunc}, the joint  law of $\ell(\phi, \wt\eta)$ 
	and $\wt S^+$ 	under $\LF_\cC\times \wt \cL_\kappa(\cC)$
	is the same as that of $\ell(\phi,\eta)$ and $S^+_0$ in Lemma~\ref{lem:stationarity}. 
	Moreover, conditioning on $\ell(\phi,\wt\eta)$,  the decorated quantum surfaces $\wt S^+$ and $\wt S^-$ are conditionally independent. 
	
	Now we use the additional observation that both $\LF_\cC$ and $\wt\cL_\kappa(\cC)$ are invariant under the mapping $z\mapsto-z$ from $\cC$ to $\cC$ (in the case of $\wt \cL_\kappa$ we also translate the reflected loop so it lies in $\mathrm{Loop}_0(\cC)$). 
	Therefore $(\wt S^+,\wt S^-)$ must agree in law with $(\wt S^-,\wt S^+)$. Hence $(\wt S^+,\wt S^-)$ agrees in law with $(S_0^+,\wt S_0^-)$ in Lemma~\ref{lem:stationarity}.
	Namely, if we uniformly conformal weld $\wt S^+$ and $\wt S^-$, the resulting decorated quantum surface is  $\QS_2\otimes \SLE_\kappa^{\sep}$ from Proposition~\ref{prop-loop-zipper}. 

	The conditional law of $(\phi, \wt \eta)$ given $(\wt S^+, \wt S^-)$ is obtained by conformally welding $\wt S^+, \wt S^-$ then embedding the decorated quantum surface in $(\cC, -\infty, +\infty)$ in a rotationally invariant way around the axis of the cylinder. The same holds for $(\phi^0, \eta^0)$ and $(S^+_0, S^-_0)$. Consequently $(\phi, \wt \eta)$ and $(\phi^0, \eta^0)$ agree in distribution, i.e.\  $\LF_\cC\times \wt \cL_\kappa(\cC)=\LF_\cC\times \cL_\kappa(\cC)$ where  $\LF_\cC\times \cL_\kappa(\cC)$ is the law of $(\phi^0,\eta^0)$ by Lemma~\ref{lem:coupling}.
	Therefore $\wt \cL_\kappa=\cL_\kappa$ as desired, hence $\wt\SLE_\kappa^{\mathrm{loop}}=C \SLE_\kappa^{\mathrm{loop}}$ for some constant $C$. This proves Proposition~\ref{cor-loop-equiv} for $\kappa\in (8/3,4)$.	For $\kappa=4$, we can take the $\kappa\uparrow 4$ limit as explained in Lemmas~\ref{lem-kappa-shape} and~\ref{lem-kappa-shape-2} to get $\wt \cL_4=\cL_4$ and conclude the proof.
	\end{proof}

\subsection{Convergence in total variation: proof of Lemma~\ref{lem:tv}}\label{subsec:tv}
We first make a simple observation.
\begin{lemma}\label{lem-trunc-field}
	For $\eps>0$,  we have  $(\LF_\cC\times \wt \cL_\kappa(\cC))[E_\eps]\le (\LF_\cC\times \cL_\kappa(\cC))[E_\eps]<\infty$.
\end{lemma}
\begin{proof}
	We consider a coupling $\{ \eta^i:i\ge 0\}$ of $(P_i)_{i\ge 0}$ such that for each $i\ge 1$, $\P[\eta^i\neq \eta^0]$ achieves the minimum among all couplings of $P_i$ and $P_0$.
	Then in this coupling  we can find a subsequence $i_k$ such that $\P[\eta^{i_k}\neq \eta^0]\le 1/k^2$ hence by Borel-Cantelli lemma $\eta^{i_k}= \eta^0$ a.s.\ for large enough $k$.
	Now we take the product measure of $\LF_\cC$ and this coupling. Then by   Fatou's lemma we have 
	$ (\LF_\cC\times  \wt \cL_\kappa(\cC))[E_\eps]\le \liminf_{k \to \infty}(\LF_\cC\times P_{i_k})[E_\eps]$. On the other hand, by Lemma~\ref{lem:stationarity-length},
	\((\LF_\cC\times P_{i})[E_\eps] =(\LF_\cC\times P_0)[E_\eps]=(\LF_\cC\times \cL_\kappa(\cC))[E_\eps]\)	for all $i\ge 0$. This concludes the proof.
\end{proof}

Given $\eta \in \mathrm{Loop}_0(\cC)$, let $\cC^+_\eta$ be the connected component of $\cC\setminus \eta$ containing $+\infty$. 
Let $f: \cC_+\to \cC^+_\eta$ be a conformal map such that $f(+\infty)=\infty$, where $\cC_+ := \{ z \in \cC \: : \: \Re z > 0\}$ is the half-cylinder. 
By standard conformal distortion estimates (e.g.\ \cite[Lemma 2.4]{sphere-constructions}), there exists a positive constant $C_0$ not depending on $\eta$ such that $|f(z) - z| < \frac {C_0}3$ and $|f''(z)| < \frac12 < |f'(z)| < 2$ for  $\Re z > \frac {C_0}3$ (these are quantitative versions of the facts that $\mathrm{CR}(\exp(-\eta),0) \approx 1$ and $\lim_{z \to +\infty} f'(z) = 1$). 
We need the following estimate for embeddings of $\QD_{1,0}(\ell)^\#$ on $\cC^+_\eta$ that is uniform in $\eta \in \mathrm{Loop}_0(\cC)$.

\begin{lemma}\label{lem-disk-field-av}
	Suppose $\eta\in \mathrm{Loop}_0(\cC)$ and  $\phi$ is such that the law of  $(\cC^+_\eta,\phi,+\infty)/{\sim_\gamma}$ is $\QD_{1,0}(\ell)^\#$ for some $\ell > 0$. 
	Suppose $g$ is a deterministic smooth function on $\cC$ supported on $ \{ z\in \cC: \Re z \in [C_0,C_0+1]  \} $ such that $\int g(z) \,d^2z = 1$. 
	Then $$\P[(\phi, g) \notin (-K + \frac2\gamma \log \ell, K + \frac2\gamma \log \ell)] \rta 0 \textrm{ as }K\to\infty$$
	where the convergence rate only depends on $g$,  but not on $\eta$, $\ell$ or the precise law of $\phi$.
\end{lemma}
\begin{proof}
	If  $( \cC^+_\eta, \phi, +\infty)/{\sim_\gamma}$  is a sample from $\QD_{1,0}(\ell)^\#$
	then $( \cC^+_\eta, \phi-\frac{2}{\gamma}\log \ell, +\infty)/{\sim_\gamma}$  is a sample from $\QD_{1,0}(1)^\#$. Therefore 
	it suffices to prove that if  the law of $( \cC^+_\eta, \phi , +\infty )/{\sim_\gamma}$ is $\QD_{1,0}(1)^\#$, then 
	\begin{equation}\label{eq:g-phi}
	\P[(\phi, g) \notin (-K, K)] \rta 0 \textrm{ as }K\to\infty
	\end{equation}
	where the convergence rate only depends on $g$. 
	
	To prove~\eqref{eq:g-phi}, let $f: \cC_+\to \cC^+_\eta$ be a conformal map such that $f(+\infty)=+\infty$ and set $\phi_0= \phi\circ f  + Q \log |f'|$. 
	Then $(\cC_+, \phi_0, +\infty)/{\sim_\gamma}= (\cC^+_\eta, \phi, +\infty)/{\sim_\gamma}$,  hence its law is $\QD_{1,0}(1)^{\#}$. 	
	Let $S$ be the collection of smooth functions $\xi$ that are supported on $\{ \Re z \in [\frac23 C_0, \frac43C_0 + 1]\} \subset \cC_+$ and satisfy $\|\xi\|_\infty \leq 4\|g\|_\infty$ and $\|\nabla\xi\|_\infty \leq 8(\|g\|_\infty +\|\nabla g\|_\infty)$.	By the definitions of $C_0$ and $f$, we see that  $|f'|^2 g \circ f \in S$.
	Let $Y = \sup_{\xi \in S} |(\phi_0, \xi)|$.	Since $\phi_0$ is almost surely in the local Sobolev space of index $-1$, we have $Y<\infty$ a.s. 
	Note that $\|f'(z)1_{\Re z \in [C_0, C_0 + 1]}\|_\infty \leq 2$ so
	\[|(\phi, g)| = |(\phi_0 \circ f^{-1} + Q \log |(f^{-1})'|, g)| \leq |(\phi_0, |f'|^2 g\circ f)| + Q \log 2 \leq Y + Q \log 2. \]
	Since the law of $\phi_0$ is unique modulo rotations around the axis of $\cC$, the law of $Y$ does not depend on  $\eta$ or the precise law of $\phi$.
	Therefore, $\P[(\phi, g) \not \in (-K, K)] \leq \P[Y + Q\log 2 > K] \to 0$ as $K \to \infty$, as desired. 
\end{proof}

\begin{proof}[Proof of Lemma~\ref{lem:tv}]
	Let $g$ be a smooth function on $\cC$ as in Lemma~\ref{lem-disk-field-av}. For $K>0$,
	define
 \begin{equation}\label{eq:Sk}
 	S_K=E_\eps\cap\{(\phi,\eta)\in H^{-1}(\cC)\times \mathrm{Loop}_0(\cC): (\phi,g)\in  (-K + \frac2\gamma \log \ell(\phi,\eta), K + \frac2\gamma \log \ell(\phi,\eta) \}.
 \end{equation}
	Then $	d_{\tv}(\LF_\cC\times P_i |_{E_\eps}, \LF_\cC\times \wt \cL_\kappa |_{E_\eps})$ is bounded from above  by %
	\begin{equation}\label{eq:tv-bound}
	d_{\tv}(\LF_\cC\times P_i |_{S_K}, \LF_\cC\times \wt \cL_\kappa |_{S_K})+(\LF_\cC \times P_i) [ E_\eps \setminus S_{K}] +
	(\LF_\cC \times \wt\cL_\kappa) [ E_\eps \setminus S_{K}].
	\end{equation}
	Since $(\LF_\cC \times \wt\cL_\kappa) [ E_\eps]<\infty$ by Lemma~\ref{lem-trunc-field} and 
	$\cap_{K>\infty} (E_\eps \setminus S_{K})=\emptyset$, we see that 
	\begin{equation}\label{eq:tv-bound2}
	\lim_{K\to\infty}  (\LF_\cC \times \wt\cL_\kappa) [ E_\eps \setminus S_{K}]=0.
	\end{equation}
	On the other hand, conditioning on the boundary length being $\ell$, the conditional law of $(\cC_{\eta^i}^+, \phi^i, \infty)/{\sim_\gamma}$ is $\QD_{1,0}(\ell)^\#$. 
	Since $\eta^i\in \mathrm{Loop}_0(\cC)$, by Lemma~\ref{lem-disk-field-av}, we have
	\[
	(\LF_\cC \times P_i) [ E_\eps \setminus S_K] \le (\LF_\cC \times P_i) [ E_\eps] \times o_K(1)
	\]
	where $o_K(1)$ is a function converging to $0$ as $K\to 0$ uniform in $i\ge 0$. Since $ (\LF_\cC \times P_i) [ E_\eps]$ does not depend on $i$ by~\eqref{eq:LF-trunc}, we have 
	\begin{equation}\label{eq:tv-bound3}
	 \lim_{K\to \infty} \max_{i\ge 0}(\LF_\cC \times P_i) [ E_\eps \setminus S_K] =0.   
	\end{equation}
	It remains to  handle the first term in~\eqref{eq:tv-bound}. Let $F_{K}=\{  \phi\in H^{-1}(\cC):  (\phi, g) > -K + \frac2\gamma \log \eps\}$. Then 
	$S_K \subset F_{\eps, K} \times \mathrm{Loop}_0$. We claim that $\LF_\cC[F_{\eps, K}] < \infty$.  Assuming this, since  $\lim_{i \to \infty}d_\mathrm{tv}(P_i, \wt \cL_\kappa) = 0$, we have $\lim_{i \to \infty}d_\mathrm{tv}(\LF_\cC|_{F_{\eps, K}} \times P_i, \LF_\cC|_{F_{\eps, K}} \times  \wt \cL_\kappa) = 0$ hence $\lim_{i \to \infty} d_\mathrm{tv}(\LF_\cC \times P_i |_{S_K}, \LF_\cC \times \wt \cL_\kappa|_{S_K}) = 0$. Thus the quantity~\eqref{eq:tv-bound} tends to 0 as $i \to \infty$ then $K \to \infty$, as desired.
	
	By the definition of $\LF_\cC^{(\gamma, \pm\infty)}$  from Definition~\ref{def-LFC}, our claim  $\LF_\cC[F_{\eps, K}] < \infty$ follows from
	the fact  that $\int_{-\infty}^\infty \P[G > -c] e^{(2\gamma - 2Q)c}\,dc < \infty$ if $G$ is a Gaussian random variable. 
\end{proof}

\subsection{The \texorpdfstring{$\kappa'\in (4,8)$}{g} case}\label{eq ns}

In this section, we prove the following proposition.
\begin{proposition}\label{cor-loop-equiv ns}
For each $\kappa'\in (4,8)$,  	there exists  $C\in (0,\infty)$ such that $\wt\SLE_{\kappa'}^{\mathrm{loop}}=C\SLE_{\kappa'}^{\mathrm{loop}}$.
\end{proposition}

The proof here follows the framework of previous sections line by line, where the equivalence of $\SLE_\kappa$ loop measure for $\kappa\in(\frac{8}{3},4)$ is proved. Therefore, we just list the main steps for the case $\kappa'\in(4,8)$ and briefly mention the differences with the arguments in the proof of Proposition~\ref{cor-loop-equiv}.

Similar to the simple case, we let $\Loop_0(\mathcal{C})$ be the set of non-crossing loops on the cylinder $\mathcal{C}$ separating $\pm\infty$ with $\max\{\Re z:z\in\eta\}=0$. Define $\mathcal{L}_{\kappa'}(\mathcal{C})$ to be the shape measure of Zhan's $\SLE_{\kappa'}$ loop measure $\SLE_{\kappa'}^{\lp}$ (see Definition \ref{def:loop}), and define $\widetilde{\mathcal{L}}_{\kappa'}(\mathcal{C})$ to be the shape measure of $\wt\SLE_{\kappa'}^{\lp}$ (the counting measures of the full-plane $\CLE_{\kappa'}$). Finally, we say a loop $\eta$ in $\cC$ \emph{surrounds} $+\infty$ if the loop $\exp(-\eta)$ in $\C$ surrounds $0$ (has nonzero winding number around $0$).

Let $\mathcal{C}_{\eta^0}^+$ be the connected component of $\mathcal{C}\backslash\eta^0$ that contains $+\infty$. As the simple case, if one samples $\CLE_{\kappa'}$ configuration on $\mathcal{C}^+_{\eta^0}$, and translates its outermost loop surrounding $+\infty$ to be an element $\eta^1\in\mathrm{Loop}_0(\cC)$, then we define a Markov transition kernel $\eta^0\to\eta^1$ on $\Loop_0(\mathcal{C})$. The following lemma gives the $\kappa'\in(4,8)$ counterpart of  Lemma \ref{lem-Markov} (also extends the result of Proposition 4 in \cite{werner-sphere-cle} to the nonsimple case).
\begin{lemma}\label{markov ns}
Let $(\eta^i)_{i\ge1}$ be the Markov chain starting from $\eta^0$. Then $\eta^n$ converges in the total variation distance to $\widetilde{\mathcal{L}}_{\kappa'}(\mathcal{C})$. Moreover, $\widetilde{\mathcal{L}}_{\kappa'}(\mathcal{C})$ is the unique stationary measure of the Markov chain.
\end{lemma}

\begin{lemma}\label{lem:ns-cle-coupling}
Suppose $0\in D\subset \wt D$ are two Jordan domains in $\C$. Let $\Gamma$ and $\wt\Gamma$ be $\CLE_{\kappa'}$ configurations in $D$ and $\wt D$ respectively. Then there is a coupling of $\Gamma$ and $\wt \Gamma$ such that with a positive probability $c>0$, their loops surrounding the origin coincide.
\end{lemma}
\begin{proof}
Let $\wh D$ be a domain such that $D\subset \wh D\subset \wt D$ and both of $\partial D\cap\partial \wh D$ and $\partial \wt D\cap\partial \wh D$ have connected subsets $\alpha,\beta$ more than one point. Let $0\in A\subset D$ be such that $\alpha\subset\partial A$ and $\partial A$ does not intersect $\partial D\Delta \partial\wh D$. Similarly, let $0\in B\subset \wh D$ be such that $\beta\subset\partial B$ and $\partial B$ does not intersect $\partial \wh D\Delta \partial\wt D$. Fix $a\in\alpha$, $b\in \partial D\backslash\partial A$, $c\in\beta$ and $d\in\partial \wt D\backslash\partial B$. Let $\eta, \wh \eta^{ac}, \wh \eta^{ca}, \wt \eta$ be the chordal $\SLE_{\kappa'}(\kappa'-6)$ from $a$ to $b$ in $D$, from $a$ to $c$ and from $c$ to $a$ in $\wh D$, and from $c$ to $d$ in $\wt D$ (with the force points being $a+$, $a+$, $c-$, $c-$ respectively, where $+$ stands for the couterclockwise direction).

In the following we will couple $\eta$ and $\wt \eta$ such that with a positive probability $0$ is \textit{disconnected}, i.e. not encircled by $\eta$ (resp. $\wt\eta$) and clockwise boundary arc $\ol{ba}$ (resp. $\ol{cd}$) in $D$ (resp. $\wt D$), and the connected components of $D\backslash\eta$ and $\wt D\backslash \wt \eta$ containing $0$ are the same.
Once this coupling successes, according to the construction of $\CLE_{\kappa'}$ by $\SLE_{\kappa'}(\kappa'-6)$ exploration trees in Lemma \ref{lem-cle-decomp} (which is based on \cite{shef-cle}), when $0$ is disconnected in the above sense, the $\CLE_{\kappa'}$ loops surrounding $0$ are from the conditionally independent $\CLE_{\kappa'}$ in the remaining connected component containing $0$.

We will establish the coupling by using $\wh \eta^{ac}$ and $\wh \eta^{ca}$ as intermediaries. Note that $\eta$ and $\wh \eta^{ac}$ can be viewed as the counterflow line of GFFs in $D$ and $\wh D$ with some boundary data, which coincide on $\partial D\cap\partial \wh D\subset A$. By the absolute continuity of GFF (see e.g. \cite[Proposition 3.4]{ig1}), the law of restriction $h|_A$ and $\wh h|_A$ are mutually absolutely continuous. Since the counterflow line is determined by the field \cite[Proposition 5.13]{ig1}, we find that the laws of $\eta$ and $\wh\eta^{ac}$ before exiting $A$ are mutually absolutely continuous. 
For a non self-touching curve $\gamma$ from $a$ and disconnecting $0$, let $\tau$ be its first time of that $0$ is disconnected, and $\sigma<\tau$ be such that $\gamma(\sigma)=\gamma(\tau)$. We call $\gamma[\sigma,\tau]$ to be the \textit{disconnecting loop} of $0$. Now define the event $E$ to be that $0$ is disconnected before exiting $A$, and the disconnecting loop of $0$ is in $A\cap B$. Since $E$ is measurable w.r.t. the curve before exiting $A$, once $\P[E]>0$, we can deduce that conditioned on the event $E$, the laws of the connected components of $D\backslash\eta$ and $\wh D\backslash\wh\eta^{ac}$ containing $0$ are mutually absolutely continuous.
Similarly, we can define the disconnecting loop of $0$ for curves from $c$ in parallel, and define the event $\wt E$ to be that $0$ is disconnected before exiting $B$, and the disconnecting loop of $0$ is in $A\cap B$. For the same reason, once $\P[\wt E]>0$, we find conditioned on $\wt E$, the laws of the connected components of $\wh D\backslash\wh\eta^{ca}$ and $\wt D\backslash\wt\eta$ containing $0$ are mutually absolutely continuous.

By the reversibility of the chordal $\SLE_{\kappa'}(\kappa'-6)$ \cite[Theorem 1.2]{ig3}, $\wh \eta^{ca}$ and the time-reversal of $\wh \eta^{ac}$ are equal in law, hence in the following we set $\wh \eta^{ca}$ to be the time-reversal of $\wh \eta^{ac}$. Define the event $\wh E:=\{\wh\eta^{ac}\text{ is in }E\}\cap\{\text{the time-reversal }\wh\eta^{ca} \text{ of }\wh\eta^{ac}\text{ is in }\wt E\}$. Once $\P[\wh E]>0$, by the domain Markov property, we find that the laws of $\wh\eta^{ac}[0,\tau]$ ($\tau$ is the time that $\wh\eta^{ac}$ disconnects $0$), conditioned on $E$ and $\wh E$ respectively, are mutually absolutely continuous. Indeed, if $\wh\eta^{ac}[0,\tau]$ satisfies the event $E$, then with a positive probability the remaining part of $\wh\eta^{ac}$ will stay in $B$, which achieves the event $\wh E$. Hence we deduce that conditioned on $E$ and $\wh E$, the laws of the connected components containing $0$ of $\wh D\backslash\wh\eta^{ac}$ are mutually absolutely continuous. For the same reason, conditioned on $\wt E$ and $\wh E$ respectively, the laws of the connected components containing $0$ of $\wh D\backslash\wh\eta^{ca}$ are mutually absolutely continuous. Since we have set $\wh\eta^{ca}$ to be the time-reversal of $\wh\eta^{ac}$, we have $\wh D\backslash\wh\eta^{ca}=\wh D\backslash\wh\eta^{ac}$. Combining with the previous paragraph, we conclude that the laws of the connected components containing $0$ of $D\backslash\eta$ and $\wt D\backslash\wt\eta$, conditioned on $E$ and $\wt E$ respectively, are mutually absolute continuous.
Since then, by the maximal coupling, one can couple $\eta$ (restricted on $E$) and $\wt\eta$ (restricted on $\wt E$) such that the connected components of $D\backslash\eta$ and $\wt D\backslash\wt\eta$ containing $0$ coincide with a positive probability.

It remains to show that the probabilities for events $E,\wt E,\wh E$ are all positive. In the following we show $\P[\wh E]>0$; the cases for other two events can be done similarly.
Fix three tubes $T_1,T_2,T_3$ as the figure below. According to \cite[Lemma A.1]{gmq-cle-inversion}, the following three (conditional) probabilities are all positive: the event $E_1$ that $\eta$ hits the marked boundary of $T_2$ before exiting $T_1$; conditioned on $E_1$, the event $E_2$ that $\eta$ disconnects $0$ and then hits $T_3$ before exiting $T_2$; and conditioned on $E_1$ and $E_2$, the event $E_3$ that $\eta$ ends at $c$ without leaving $T_3$. Thus $\P[E_1\cap E_2\cap E_3]>0$. Since $E_1\cap E_2\cap E_3\subset \wh E$, we have $\P[\wh E]>0$.

\end{proof}
\begin{figure}[ht!]
	\begin{center}
		\includegraphics[scale=0.55]{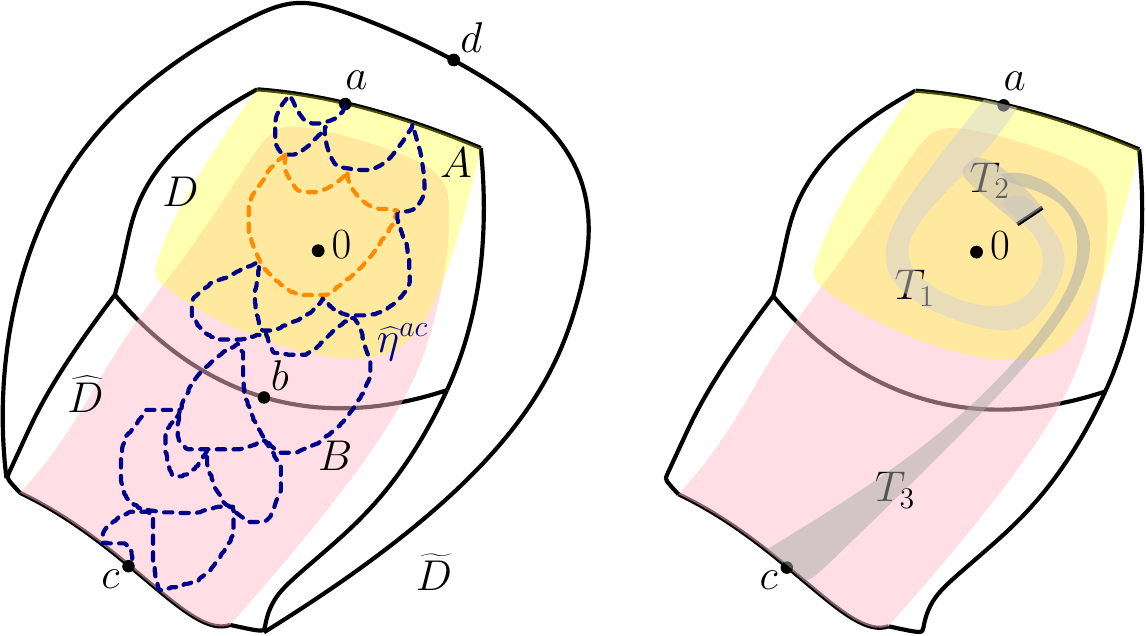}
	\end{center}
	\caption{\textbf{Left}: The configuration that the event $E$ occurs, where the orange loop corresponds to the segment $\wh\eta^{ac}[\sigma,\tau]$. $A$ and $B$ are the yellow and pink regions. \textbf{Right}: Three tubes $T_1,T_2,T_3$ are colored grey light to dark, such that $T_1\subset A$, $T_2\subset A\cap B$ and $T_3\subset B$. The marked boundary of $T_2$ is colored black.
	}
\end{figure}

\begin{proof}[Proof of Lemma \ref{markov ns}]
Let $0\in D\subset\C$ be a Jordan domain and consider a full-plane $\CLE_{\kappa'}$ configuration $\Gamma$. Denote $\gamma$ to be the inner-most loop in $\Gamma$ surrounding $D$, and let $G$ be the connected component of $\C\backslash\gamma$ containing $0$ (and also $D)$. Conditionally on $\Gamma$, independently take $\CLE_{\kappa'}$ configurations $(\Gamma_D,\Gamma_G)$ on $D$ and $G$ from the coupling in Lemma \ref{lem:ns-cle-coupling}, and denote its the success probability to be $c(G)$. Note that the resampling property for full-plane $\CLE_{\kappa'}$ gives that the collection of loops $\wt\Gamma:=\Gamma|_{\C\backslash G}\cup\{\gamma\}\cup \Gamma_G$ also has the law of full-plane $\CLE_{\kappa'}$. Hence we obtain the desired coupling $(\Gamma_D,\wt\Gamma)$ for $\CLE_{\kappa'}$ on $D$ and $\C$ such that with probability $\E c(G)$ their loops surrounding $0$ in $D$ coincide ($\E c(G)>0$ since $c(G)>0$ a.s.).
\end{proof}

The following proposition is the counterpart of Proposition \ref{prop-loop-zipper}. Its proof is just by the same argument as in the proof of Proposition~\ref{prop-loop-zipper}, with the input Theorem \ref{loop weld ns}.
\begin{prop}\label{prop-loop-zipper ns}
Fix $\kappa'\in(4,8)$. For some $C>0$, we have
\begin{equation*}
\QS_2\otimes\SLE_{\kappa'}^{\rm{sep}}=C\int_0^\infty  a\Weld(\GQD_{1,0}(a),\GQD_{1,0}(a))da.
\end{equation*}
\end{prop}
For $\kappa'\in (4,8)$, let $(h,\eta)$ be a sample from  $\mathbb F \times \SLE^{\sep}_{\kappa'}(\cC)$. 
Then $\eta$ can be uniquely written as $\eta^0+\bft$  where $\eta^0\in \mathrm{Loop}_0(\cC)$ and $\bft\in \R$. The law of
$(\eta^0,\bft)$ is  $C\cL_{\kappa'}(\cC)\times dt$. 
Conditioning on $(h,\eta)$, sample a $\CLE_{\kappa'}$ $\Gamma^+$ in $\cC^+_\eta$.
Let   $(\eta^i)_{i \geq 1}$ be  the set of  loops in $\Gamma^+$ separating $\pm\infty$ ordered from left to right (note that $(\eta^i)_{i \geq 1}$ is just the Markov chain with the transition kernel defined above).
Let $\ell_h(\eta^i)$ be the quantum length of $\eta^i$. The following is the counterpart of Lemma \ref{lem:stationarity-length}.
\begin{lemma}\label{lem:stationarity-length ns}
The law of $\ell_h(\eta^i)$ is the same as the law of $\ell_h(\eta^0)$.
\end{lemma}
\begin{proof}
By Propositions~\ref{prop-GA2-pt} and~\ref{prop-loop-zipper ns}
the joint law of $\ell_h(\eta^0)$ and $\ell_h(\eta^1)$ is \[C ab|\GQD_{1,0}(a)||\GA(a,b)||\GQD_{1,0}(b)|\,da\,db,\] and by Definition~\ref{def-ga} $|\GA(a,b)| = |\GA(b,a)|$ so this expression is symmetric in $a$ and $b$. Thus the lemma holds for $i = 1$, and inductively for general $i$.
\end{proof}

We now state Lemma~\ref{lem:stationarity ns}, which is the $\kappa' \in (4,8)$ analog of Lemma~\ref{lem:stationarity}. In Lemma~\ref{lem:stationarity} we defined quantum surfaces $S_i^+$ and $S_i^-$ parametrized by the regions $C_{\eta_i}^+$ and $C_{\eta_i}^-$ to the right and left of the (simple) loop $\eta_i$. In the present setting $\eta_i$ is nonsimple, so to state Lemma~\ref{lem:stationarity ns}, we will define the analogous \emph{forested} quantum surfaces $\mathfrak S_i^+$ and $\mathfrak S_i^-$ parametrized by regions $\mathfrak C_{\eta_i}^+$ and $\mathfrak C_{\eta_i}^-$; these regions have to be defined via \emph{index} with resepct to the loop, rather than  ``right'' and ``left''. 

For $\eta^0\in\Loop_0(\mathcal{C})$, we define $\frak{C}_{\eta^0}^+$ as follows. Let $f(z) = e^{-z}$ be the conformal map from $\mathcal{C}$ to $\C$, and let $D_{\eta^0}$ be the set of points on $\C$ whose index w.r.t.\ $f\circ\eta^0$ is non-zero (in particular, $0\in D_{\eta^0}$). We then set $\frak{C}_{\eta^0}^+ = f^{-1}(D_{\eta^0})$ (so $+\infty\in\frak{C}_{\eta^0}^+$). One can similarly define $\frak{C}_{\eta^0}^-$. 
By Proposition~\ref{prop-loop-zipper ns} the quantum surfaces parametrized by $\mathfrak C_{\eta_i}^+$ and $\mathfrak C_{\eta_i}^-$ are forested quantum surfaces when $i = 0$, and inductively by Proposition~\ref{prop-GA2-pt} they are forested quantum surfaces when $i > 0$.

\begin{lemma}\label{lem:stationarity ns}
	For $i\ge 0$, conditioning on $(h,\eta_i)$, let $z_i$ be a point on $\eta_i$ sampled from the probability measure proportional to the generalized quantum length measure. Let $\mathfrak S_i^+$ (resp.\ 
$\mathfrak S_i^-$) be the forested quantum surface parametrized by $\mathfrak C_{\eta_i}^+$ (resp.\ $\mathfrak C_{\eta_i}^-$) having marked points points $z_i$ and $+\infty$ (resp.\ $-\infty$). Then the conditional law of $\frak{S}_i^+$ given $\frak{S}_i^-$ is $\GQD_{1,1}(\ell_{h}(\eta_i))^{\#}$ for all $i\ge 0$.
\end{lemma}
\begin{proof}
When $i=0$, this follows from Theorem~\ref{loop weld ns}.
	The $i\ge 1$ case follows iteratively using Proposition~\ref{prop-GA2-pt}.
\end{proof}

Let $E_\varepsilon$ be the event that the generalized quantum length of $\eta$ is larger than $\varepsilon$. The following is the counterpart of Lemma \ref{lem:tv}.

\begin{lemma}\label{lem-trunc-field ns}
For each $\varepsilon>0$, $\lim_{i\to\infty}d_{\rm tv}(\LF_{\mathcal{C}}\times P_i|_{E_\varepsilon},\LF_{\mathcal{C}}\times \widetilde{\mathcal{L}}_{\kappa'}(\mathcal{C})|_{E_\varepsilon})=0$.
\end{lemma}

\begin{proof}[Proof of Proposition \ref{cor-loop-equiv ns} given Lemma \ref{lem-trunc-field ns}]
According to \cite[Theorem 1.1]{gmq-cle-inversion}, the inversion invariance also holds for $\CLE^{\C}_{\kappa'}$, i.e.
$\wt\cL_{\kappa'}(\cC)$ is invariant under the mapping $z\mapsto-z$ from $\cC$ to $\cC$.
Then the result just follows line by line in the proof of Proposition~\ref{cor-loop-equiv}, with the $\kappa < 4$ inputs replaced by Lemmas \ref{markov ns} and~\ref{lem:stationarity ns} and Proposition \ref{prop-loop-zipper ns}.
\end{proof}

Now it remains to prove Lemma \ref{lem-trunc-field ns}. To this end, we repeat the argument in Section \ref{subsec:tv}. The key input is the following counterpart of Lemma \ref{lem-disk-field-av}. Recall the universal constant $C_0$ defined above Lemma~\ref{lem-disk-field-av}.

\begin{lemma}\label{tight ns}
Suppose $\eta\in\Loop_0(\mathcal{C})$ and $\phi$ are such that the law of $(\frak{C}_\eta^+,\phi,+\infty)/{\sim_\gamma}$ is  $\GQD_{1,0}(\ell)^{\#}$. Suppose $g$ is a deterministic smooth function on $[C_0,C_0+1]\times[0,2\pi]$ with $\int g(z)d^2z=1$. Then
\begin{equation*}
\mathbb{P}[(\phi,g)\notin(-K+\frac{\gamma}{2}\log \ell, K+\frac{\gamma}{2}\log \ell)]\to0 \text{ as } K\to\infty,
\end{equation*}
where the convergence rate only depends on $g$.
\end{lemma}
\begin{proof}
Similarly as in Lemma~\ref{lem-disk-field-av}, if the law of $(\frak{C}_\eta^+,\phi,\infty)/{\sim_\gamma}$ is $\GQD_{1,0}(\ell)^{\#}$, then the law of  $(\frak{C}_\eta^+,\phi-\frac{\gamma}{2}\log\ell,\infty)/{\sim_\gamma}$ is $\GQD_{1,0}(1)^{\#}$. Thus it suffices to show that if  the law of $(\frak{C}_\eta^+,\phi,\infty)/{\sim_\gamma}$ is $\GQD_{1,0}(1)^\#$, then 
\begin{equation}\label{field estimate ns}
\mathbb{P}[(\phi,g)\notin(-K, K)]\to0\text{ as } K\to\infty,
\end{equation}
where again the convergence rate depends only on $g$.
Let $\cC^+_{\eta}$ be the connected component of $\cC\backslash\eta$ containing $+\infty$. Sample a point $z$ on $\eta^{\rm in} = \partial \cC^+_\eta$ from the probability measure proportional to quantum length, and let $S^+= (\cC^+_{\eta}, \phi,+\infty, z)/{\sim_\gamma}$. Since the law of $(\frak{C}_\eta^+,\phi,\infty)/{\sim_\gamma}$ is $\GQD_{1,0}(1)^\#$, the conditional law of $S^+$ given the quantum length $\ell_\phi(\eta^{\rm in})$ is $\QD_{1,1}(\ell_\phi(\eta^{\rm in}))^{\#}$. Since $C_0>0$, we have $\frak{C}_\eta^+\cap [C_0,C_0+1]=\cC_\eta^+\cap[C_0,C_0+1]$. Thus, by Lemma \ref{lem-disk-field-av} applied to $\eta^\mathrm{in}$,
\begin{equation}\label{eq:cond-conv}
\lim_{K\to\infty}\mathbb{P}[(\phi,g)\notin(-K, K)|\ell_\phi(\eta^{\rm in})]=0.
\end{equation}
Since
$\mathbb{P}[(\phi,g)\notin(-K, K)]=\E\left[\mathbb{P}[(\phi,g)\notin(-K, K)|\ell_\phi(\eta^{\rm in})]\right]$, the dominated convergence theorem and~\eqref{eq:cond-conv} imply \eqref{field estimate ns}.
\end{proof}

\begin{proof}[Proof of Lemma \ref{lem-trunc-field ns}]
Recall $E_\varepsilon$ is the event that the generalized quantum length of $\eta$ is larger than $\varepsilon$. Following line by line as in Lemma \ref{lem-trunc-field}, one can first show that $(\LF_\cC\times \wt \cL_\kappa(\cC))[E_\eps]\le (\LF_\cC\times \cL_\kappa(\cC))[E_\eps]<\infty$.
Then the proof is the same as the proof of Lemma \ref{lem:tv}, except that we need to replace $S_K$ in \eqref{eq:Sk} with
\begin{equation*}
S_K=E_\varepsilon\cap\{(\phi,\eta)\in H^{-1}(\mathcal{C})\times\Loop_0(\mathcal{C}):(\phi,g)\in(-K+\frac{\gamma}{2}\log \ell_\phi(\eta), K+\frac{\gamma}{2}\log \ell_\phi(\eta))\},
\end{equation*}
and use Lemma \ref{tight ns} in the place of Lemma \ref{lem-disk-field-av}.
\end{proof}

\begin{proof}[Proof of Theorem \ref{cor-loop-equiv-main}]
Propositions~\ref{cor-loop-equiv} and~\ref{cor-loop-equiv ns} prove the claim for  $\kappa \in (0,4]$ and $\kappa \in (4,8)$.
\end{proof}

\section{Applications of equivalence of loop measures}\label{sec-loop-equiv-applications}

In this section, we present some applications of Theorem~\ref{cor-loop-equiv-main}. 
In Section~\ref{sec-duality}, we prove Theorem~\ref{loop duality}. In Section~\ref{subsec-application-cle}, we combine Theorem~\ref{cor-loop-equiv-main} with the conformal welding results for the SLE loop (Theorems~\ref{thm-loop2} and~\ref{loop weld ns}) to obtain conformal welding results for CLE; these will be used in our subsequent work \cite{ACSW24}. Finally, in Section~\ref{sec-symmetry-qa} we prove that $\QA$ and $\GA$ are symmetric, in the sense that they are invariant in law under the interchange of their boundary components.

\subsection{Duality of the SLE loop measure}\label{sec-duality}
In this section we prove Theorem~\ref{loop duality}: For $\kappa\in (2,4)$ and $\eta$ sampled from $\SLE^{\rm loop}_{16/\kappa}$ on $\hat\C=\C\cup \{\infty\}$, letting $\eta^{\rm out}$ be the boundary of the unbounded connected component of $\C \backslash \eta$, the law of $\eta^{\rm out}$ equals  $C\SLE^{\rm loop}_\kappa$ for some constant $C$.

\begin{proof}[{Proof of Theorem~\ref{loop duality}}]
In this proof we let $C$ be a constant depending only on $\gamma$ that  may change from line to line. 

Let $(\hat \C, h, \eta, \infty)$ be an embedding of a sample from $\QS_1\otimes \SLE_{16/\kappa}^{\rm loop}$, and let $\eta^\mathrm{out}$ be the boundary of the connected component of $\hat \C \backslash \eta$ containing $\infty$. Let $M_1$ be the law of $(\hat \C, h, \eta^\mathrm{out}, \infty)/{\sim_\gamma}$. Let $p,q$ be independently sampled from the probability measure proportional to the quantum length measure on $\eta^\mathrm{out}$, and let $M_3$ be the law of $(\hat \C, h, \eta^\mathrm{out}, p, q, \infty)/{\sim_\gamma}$ weighted by the square of the quantum length of $\eta^\mathrm{out}$. We will show that 
\eqb \label{eq-eta-in-law}
M_3 = C \int_0^\infty    \mathrm{Weld}(\QD_{0,2}(b,a),\QD_{1,2}(a,b)) \, da \, db.
\eqe 
Given~\eqref{eq-eta-in-law}, by forgetting the points on $\eta^\mathrm{out}$ and deweighting, we get $M_1 = C\int_0^\infty \ell   \mathrm{Weld}(\QD(\ell),\QD_{1,0}(\ell)) \, d\ell$, so Theorem~\ref{thm-loop2} tells us that $M_1 = C\QS_{1}\otimes \SLE_\kappa^\mathrm{loop}$. The desired result follows. The rest of the proof is devoted to showing~\eqref{eq-eta-in-law}. 

By Theorem~\ref{loop weld ns} $\eta$ cuts $(\hat \C, h, \eta, p, q, \infty)/{\sim_\gamma}$ into a pair of forested quantum surfaces; let $\mathfrak S^+$ be the forested quantum surface containing marked points $(p,q,\infty)$, and let $\mathfrak S^-$ be the other forested quantum surface (having marked points $(p,q)$). Similarly, $\eta^\mathrm{out}$ cuts $(\hat \C, h, \eta^\mathrm{out}, p, q, \infty)/{\sim_\gamma}$ into a pair of quantum surfaces $S^+$ and $S^-$, where $S^+$ has marked points $(p,q,\infty)$ and $S^-$ has marked points $(p,q)$. See Figure~\ref{fig-duality}.

By Theorem~\ref{loop weld ns} and Proposition~\ref{prop:101=gqd11}, we have $\QS_1 \otimes \SLE_{16/\kappa}^\mathrm{loop} = C \int_0^\infty \ell \mathrm{Weld}( \GQD(\ell), \cM_{1,0,0}^\mathrm{f.d.}(\gamma; \ell))\, d\ell$. Thus, 
writing  $\cM_{2, \bullet}^\mathrm{f.d.} (2; \ell_1, \ell_2)$ to denote the law of a sample from $\cM_{2}^\mathrm{f.d.}(2; \ell_1, \ell_2)$ with a bulk point sampled from the probability measure proportional to the quantum area measure on the thick quantum disk (without foresting) and with law weighted by the quantum area of the thick quantum disk (without foresting), by Definition \ref{def-qd}, the law of $(\mathfrak S^-, \mathfrak S^+)$ is
$\iint_0^\infty \GQD_{0,2}(\ell_2, \ell_1) \times \cM_{2, \bullet}^\mathrm{f.d.}(2; \ell_1, \ell_2)\, d\ell_1\,d\ell_2$. Recall from Definition~\ref{def-GQD} that $\GQD_{0,2}(\ell_2, \ell_1) = \cM_{2}^\mathrm{f.d.}(\gamma^2-2; \ell_2, \ell_1)$, and that two-pointed forested quantum disks arise from welding forested line segments to two-pointed quantum disks (Definition~\ref{def-forested-2disk}).  Thus, the joint law of $(S^-, S^+)$ is 
\[C\int\mathrm{Weld}\left(\cM_2^\mathrm{f.l.}(b, \ell_2), \cM_2^\mathrm{f.l.}(t_2,\ell_2), \cM_2^\mathrm{disk}(\gamma^2-2; t_2, t_1), \cM_2^\mathrm{f.l.}(t_1, \ell_1), \cM_2^\mathrm{f.l.}(\ell_1, a)\right) \times \cM_{2,\bullet}^\mathrm{disk}(2;a,b)   \]
with integral taken over $a,b,\ell_1, \ell_2, t_1, t_2 >0$ and welding interfaces forgotten. By  Proposition~\ref{sy}, the first term in the integrand integrated over $\ell_1,\ell_2>0$ simplifies to $C \mathrm{Weld}(\cM_2^\mathrm{disk}(2-\frac{\gamma^2}2; b, t_2), \cM_2^\mathrm{disk}(\gamma^2-2;t_2,t_1), \cM_2^\mathrm{disk}(2-\frac{\gamma^2}2; t_1, a))$, and by Proposition~\ref{prop:disk-welding} integrating this term over $t_1,t_2>0$ gives $C \cM_2^\mathrm{disk}(2;b,a)$. We conclude that the joint law of $(S^-, S^+)$ is $C \iint_0^\infty \cM_2^\mathrm{disk}(2;b,a) \times \cM_{2,\bullet}^\mathrm{disk}(2; a,b)\,da\,db$, so by Definition~\ref{def-qd} we obtain~\eqref{eq-eta-in-law}.
\end{proof}

\begin{figure}[ht!]
	\begin{center} \includegraphics[scale=0.5]{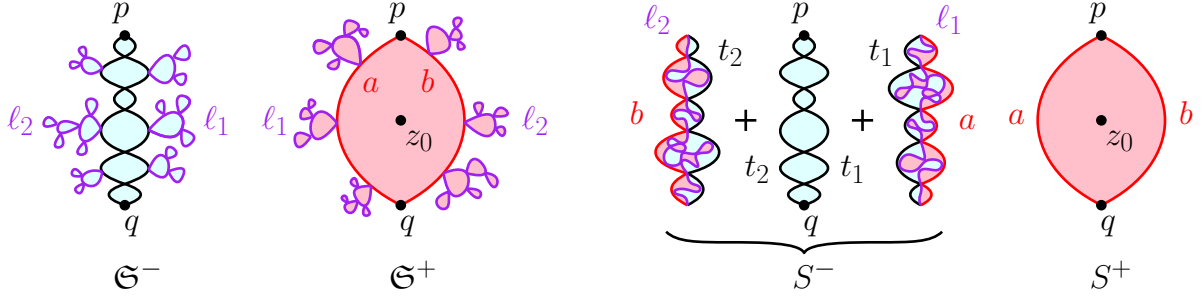}
	\end{center}
	\caption{\label{fig-duality}
 \textbf{Left}: Forested quantum surfaces $\mathfrak S^-$ and $\mathfrak S^+$. \textbf{Right}: By Proposition~\ref{prop:101=gqd11}, $S^+$ is a sample from $\QD_{1,2}$. On the other hand, splitting the forested line segments off of $\mathfrak S^\pm$ and welding them together (Proposition~\ref{sy}), we can express $S^-$ in terms of the conformal welding of two-pointed quantum disks with weights $(2-\frac{\gamma^2}2, \gamma^2-2, 2-\frac{\gamma^2}2)$, hence it is a two-pointed quantum disk with weight $2$ (Proposition~\ref{prop:disk-welding}). Thus we can identify the joint law of $S^-$ and $S^+$.}
\end{figure}

\subsection{Full-plane CLE via conformal welding}\label{subsec-application-cle}
Theorem~\ref{cor-loop-equiv-main} yields the following variant of  Theorem~\ref{thm-loop2}. 
\begin{theorem}\label{thm-loop3}
	Let $\mathbb F$ be a measure on $H^{-1}_\mathrm{loc}(\C)$ such that the law of $(\C, h)/{\sim_\gamma}$ is $\QS$
	if $h$ is sampled from $\mathbb F$. For $\kappa\in (\frac{8}{3},4)$, we sample $(h,\Gamma,\eta)$ from 
	$\mathbb F(dh)\mathrm{Count}_{\Gamma}(d \eta)\CLE_\kappa^\C(d\Gamma)$, where $\CLE_\kappa^\C(d\Gamma)$ is the law of a full-plane $\CLE_\kappa$ and $\mathrm{Count}_{\Gamma}(d \eta)$ is the counting measure on a sample $\Gamma$ from $\CLE_\kappa^\C(d\Gamma)$. For $\ell>0$, let $ \mathrm{Weld}(\QD(\ell)\otimes \CLE_\kappa,\QD(\ell)\otimes \CLE_\kappa)$ be the quantum surface decorated by a loop ensemble and a distinct loop obtained from uniformly welding a pair of $\CLE$-decorated quantum disks sampled from $[\QD(\ell)\otimes \CLE_\kappa] \times [\QD(\ell)\otimes \CLE_\kappa]$. 
	Then the law of $(\C,h,\Gamma, \eta)/{\sim_\gamma}$ is
	\[
C\int_0^\infty \ell\cdot  \mathrm{Weld}\big(\QD(\ell)\otimes \CLE_\kappa,\QD(\ell)\otimes \CLE_\kappa \big)   \, d\ell \quad \textrm{for some }C\in (0,\infty).
	\]
\end{theorem}
\begin{proof}[Proof of Theorem~\ref{thm-loop3}] This immediately follows from Theorems~\ref{cor-loop-equiv-main} and~\ref{thm-loop2}, combined with the following property of full-plane $\CLE^{\C}_\kappa$ from~\cite{werner-sphere-cle}.  
	\begin{proposition}[{\cite{werner-sphere-cle}}]\label{prop-CLE-markov}
	Sample $(\Gamma,\eta)$ from $\mathrm{Count}_{\Gamma}(d \eta)\CLE_\kappa^\C(d\Gamma)$. Conditioning on $\eta$, the conditional  law  of $\Gamma\setminus \{\eta \}$  is given by two independent $\CLE_\kappa$'s on the two components of $\wh \C\backslash \eta$. \qedhere
	\end{proposition}
\end{proof}
The analog of Proposition~\ref{prop-CLE-markov} holds for $\kappa'\in (4,8)$ as well.
	\begin{proposition}\label{prop-CLE-markov ns}
	Sample $(\Gamma,\eta)$ from $\mathrm{Count}_{\Gamma}(d \eta)\CLE_{\kappa'}^\C(d\Gamma)$. Conditioning on $\eta$, the conditional law  of $\Gamma\setminus \{\eta \}$  is given by sampling independent $\CLE_{\kappa'}$'s on the each connected  components of $\wh \C\backslash \eta$. 
	\end{proposition}
\begin{proof}
   From \cite[Lemma 2.9]{gmq-cle-inversion}, let $\eta_n$ be the sequence of loops of $\Gamma$ surrounding $0$. For each $n$, $\gamma_n$ splits $\C$ into countably many connected components. If we condition on 
$\eta_n$ and the set of loops
in $\Gamma$ which are contained in the unbounded connected component of $\C\setminus\eta_n$, then the conditional law
of the rest of $\Gamma$ is that of an independent $\CLE_{\kappa'}$ in each bounded connected $\C\setminus \gamma_n$. By conformal invariance of full-plane CLE (\cite[Theorem 1.1]{gmq-cle-inversion}), we can swap the roles of the connected component containing $0$ and the unbounded component, hence the result follows.
\end{proof}
We have the following analog of  Theorem \ref{thm-loop3} for $\kappa'\in (4,8)$.
	\begin{theorem}\label{thm-loop3 ns}
	Let $\mathbb F$ be a measure on $H^{-1}_\mathrm{loc}(\C)$ such that the law of $(\C, h)/{\sim_\gamma}$ is $\QS$
	if $h$ is sampled from $\mathbb F$. For $\kappa'\in (4,8)$, we sample $(h,\Gamma,\eta)$ from 
	$\mathbb F(dh)\mathrm{Count}_{\Gamma}(d \eta)\CLE_{\kappa'}^\C(d\Gamma)$, where $\CLE_{\kappa'}^\C(d\Gamma)$ is the law of a full-plane $\CLE_{\kappa'}$ and $\mathrm{Count}_{\Gamma}(d \eta)$ is the counting measure on a sample $\Gamma$ from $\CLE_{\kappa'}^\C(d\Gamma)$. For $\ell>0$, let $ \mathrm{Weld}(\GQD(\ell)\otimes \CLE_{\kappa'},\GQD(\ell)\otimes \CLE_{\kappa'})$ be the generalized quantum surface decorated by a loop ensemble and a distinct loop obtained from uniformly welding a pair of CLE-decorated quantum disks sampled from $[\GQD(\ell)\otimes \CLE_{\kappa'}] \times [\GQD(\ell)\otimes \CLE_{\kappa'}]$. 
	Then the law of $(\C,h,\Gamma, \eta)/{\sim_\gamma}$ is
	\[
C\int_0^\infty \ell\cdot  \mathrm{Weld}\big(\GQD(\ell)\otimes \CLE_{\kappa'},\GQD(\ell)\otimes \CLE_{\kappa'} \big)   \, d\ell \quad \textrm{for some }C\in (0,\infty).
	\]
\end{theorem}
\begin{proof}
    Combine Proposition~\ref{prop-CLE-markov ns}  and Theorem \ref{cor-loop-equiv-main}.
\end{proof}

\subsection{Symmetry of $\QA$ and $\GA$}\label{sec-symmetry-qa}
In this section, we show that the quantum annulus and generalized quantum annulus are each invariant in law when their boundaries are interchanged. The proof depends on the invariance in law of $\QS_2$ under interchange of its marked points, and the inversion invariance of full-plane CLE. 

\begin{proposition}\label{prop-qa-symmetry}
For $a,b>0$ we have $\QA(a,b) = \QA(b,a)$. In other words, the quantum annulus is invariant in law under the operation of reordering its boundary components. 
\end{proposition}
\begin{proof}
Let $\kappa = \sqrt\gamma \in (\frac{8}{3},4)$.
	Let $\mathbb F$ be the law of $h$ as in Definition~\ref{def-sphere}, so that  the law of $(\cC, h,+\infty,-\infty)/{\sim_\gamma}$  is the two-pointed quantum sphere $\QS_2$. Let $\CLE^\cC_\kappa$ denote the law of $\CLE_\kappa$ in $\cC$. 
	Sample $(h, \Gamma, \eta_0)$ from $\mathbb F(dh) \mathrm{Count}_\Gamma(d\eta)\CLE^\cC_\kappa(d\Gamma)$ and restrict to the event that $\eta_0$ separates $\pm \infty$. Let $\cC_{\eta_0}^+$ be the connected component of $\cC \backslash \eta_0$ containing $+\infty$, and let $\eta_1$ be the outermost loop of $\Gamma$ in $\cC_{\eta_0}^+$ surrounding $+\infty$. 
    
    By Theorem~\ref{cor-loop-equiv-main}, the law of $(h, \eta_0)$ is $C \mathbb F \times \SLE_\kappa^\mathrm{sep}(\cC)$ for some constant $C$, 
    so by Propositions~\ref{prop-loop-zipper} and~\ref{prop-CLE-markov} the law of $(\cC_{\eta_0}^+, h, 0, \Gamma|_{\cC_{\eta_0}^+})/{\sim_\gamma}$ is $C\int_0^\infty a|\QD_{1,0}(a)|\QD_{1,0}(a) \otimes \CLE_\kappa \, da$ for some constant $C$. Thus, by Proposition~\ref{prop-QA2-pt}, writing $A_{\eta_0, \eta_1}$ to denote the annular connected component of $\cC \backslash (\eta_0 \cup \eta_1)$, the law of $(A_{\eta_0, \eta_1}, h)/{\sim_\gamma}$ is $C \iint_0^\infty ab|\QD_{1,0}(a)| |\QD_{1,0}(b)| \QA(a,b) \, da\,db$ for some constant $C$. 

    On the other hand, since for $h \sim \mathbb F$ the distribution $h(-z)$ agrees in law with $h(z)$, and $\Gamma \sim \CLE_\kappa^\cC$ is invariant in law under $z \mapsto -z$ \cite{werner-sphere-cle}, we conclude that $(A_{\eta_0, \eta_1}, h, \eta_0, \eta_1)/{\sim_\gamma}$ agrees in law with $(A_{\eta_0, \eta_1}, h, \eta_1, \eta_0)/{\sim_\gamma}$, as desired. 
\end{proof}
\begin{proposition}
For $a,b>0$ we have $\GA(a,b) = \GA(b,a)$. In other words, the generalized quantum annulus is invariant in law under the operation of reordering its boundary components. 
\end{proposition}
\begin{proof}
We first claim that if $\cA^\mathrm{u.f.}$ is a sample from $\GA^\mathrm{u.f.}$ as in Proposition~\ref{prop-ann-ns}, then $\cA^\mathrm{u.f}$ is invariant in law under the operation of reordering its boundary components. The proof of this claim is identical to that of Proposition~\ref{prop-qa-symmetry}, except it uses Propositions~\ref{prop-loop-zipper ns} and~\ref{prop-CLE-markov ns} in place of Propositions~\ref{prop-loop-zipper} and~\ref{prop-CLE-markov}, uses Proposition~\ref{prop-GA2-pt} in place of Proposition~\ref{prop-QA2-pt}, and uses the inversion invariance of non-simple CLE \cite{gmq-cle-inversion} in place of that of simple CLE \cite{werner-sphere-cle}. 

Now, since $\GA(a,b)^\#$ is obtained from $\cA^\mathrm{u.f.}$ by foresting both boundaries and conditioning on boundaries having generalized quantum lengths $a$ and $b$, we see that $\GA(a,b)^\# = \GA(b,a)^\#$. Using Definition~\ref{def-ga}, we conclude that $\GA(a,b) = \GA(b,a)$. 
\end{proof}

\section{Electrical thickness of the SLE loop via conformal welding}\label{sec:KW}

In this section we prove Theorem~\ref{thm-KW}. 
Recall from Section~\ref{sec:msw-sphere} the cylinder $\cC = \R \times [0,2\pi]$ with $x \in \R$ identified with $x + 2\pi i$, 
and that $\cL_\kappa(\cC)$ is the pullback of  the loop shape measure $\cL_\kappa$ under the map $z \mapsto e^{-z}$. Thus $\cL_\kappa(\cC)$ is a probability measure on loops $\eta$ in $\cC$ separating $\pm\infty$ and satisfying $\max_{z \in \eta} \Re z = 0$. 
	For a loop $\eta$ sampled from $\cL_\kappa(\cC)$,    we write $\vartheta(\eta)$ for the electrical thickness of $\exp(\eta)$. 
	Then Theorem~\ref{thm-KW} is equivalent to: For $\kappa\in (0,8)$
	\begin{equation}\label{eq:KW-cylinder}
	\E[e^{\lambda \vartheta(\eta)}] 
	= \left\{
	\begin{array}{ll}
	\frac{\sin(\pi (1-\kappa/4))}{\pi (1-\kappa/4)} 
	\frac{\pi \sqrt{(1-\kappa/4)^2+\lambda \kappa/2}}{\sin(\pi \sqrt{ (1-\kappa/4)^2+\lambda \kappa/2})}  & \mbox{if } \lambda < 1-\frac{\kappa}8. \\
	\infty & \mbox{if } \lambda \geq 1-\frac{\kappa}8
	\end{array}
	\right.
	\end{equation}
Consider $\alpha < Q$ and set $\lambda = \frac{\alpha^2}2 - Q\alpha+2$. Let $\cL_\kappa^\alpha$ be defined by the following reweighting of $\cL_\kappa(\cC)$.
\eqb\label{eq-L-alpha}
\frac{d \cL_\kappa^\alpha}{d\cL_\kappa(\cC)}(\eta) =  (\frac14 \CR(\exp(\eta),0) \CR(\exp(-\eta),0))^{-\frac{\alpha^2}2 + Q\alpha - 2} = 2^{2\lambda}e^{\lambda \vartheta(\eta)}
\eqe

Thus proving Theorem~\ref{thm-KW} amounts to computing the total mass $|\cL_\kappa^\alpha|$. We will obtain the following two results by conformal welding of quantum surfaces and forested quantum surfaces:
\begin{proposition}\label{prop-KW}
	Let $\kappa \in (0,4)$, $\gamma = \sqrt\kappa$ and $Q = \frac\gamma2+\frac2\gamma$. For some constant $C = C(\kappa)$ and all $\alpha \in (\frac\gamma2, Q)$ we have
	\[2^{-\alpha^2 + 2Q\alpha}|\cL_\kappa^\alpha| = C\frac{ \frac\gamma2 (Q-\alpha)}{\sin (\frac{\gamma\pi}2 (Q-\alpha))}. \]
\end{proposition}
\begin{proposition}\label{prop-KW-ns}
Let $\kappa' \in (4,8)$, $\gamma = 16/\sqrt{\kappa'}$ and $Q = \frac\gamma2+\frac2\gamma$. For some constant $C = C(\kappa')$ and all $\alpha\in(\frac{\gamma}{2},Q)$, we have
\begin{equation*}
2^{-\alpha^2+2Q\alpha}|\mathcal{L}_{\kappa'}^\alpha|=C\frac{\frac{2}{\gamma}(Q-\alpha)}{\sin(\frac{2\pi}{\gamma}(Q-\alpha))}.
\end{equation*}
\end{proposition}
Given these, we complete the proof. 

\begin{proof}[Proof of Theorem~\ref{thm-KW}]
We break the proof of~\eqref{eq:KW-cylinder} (hence Theorem~\ref{thm-KW}) into several cases.
	\medskip 
	
	\noindent \textbf{Case I: $\kappa <4$ and $\lambda \in (1 - \frac\kappa8 - \frac2\kappa, 1-\frac\kappa8)$.} Let $\alpha = Q - \sqrt{Q^2 - 4 + 2\lambda} \in (\frac\gamma2, Q)$, so $\lambda = \frac{\alpha^2}2-Q\alpha + 2$. 
	By~\eqref{eq-L-alpha} and  Proposition~\ref{prop-KW} we have 
	\[
	\E[e^{\lambda \vartheta(\eta)}] = 2^{-2\lambda} |\cL_\kappa^\alpha|  = C\frac{ \frac\gamma2 (Q-\alpha)}{\sin (\frac{\gamma\pi}2 (Q-\alpha))}= C\frac{  \sqrt{(1-\frac\kappa4)^2+\frac{\lambda\kappa}2}}{\sin(\pi \sqrt{(1-\frac\kappa4)^2+\frac{\lambda\kappa}2})} 
	\]
for some constant $C = C(\gamma)$.	Since $\kappa \in (0,4)$, we have $0 \in (1-\frac\kappa8-\frac2\kappa, 1-\frac\kappa8)$. Thus we can  obtain  the value of $C$
 by considering $1 = \E[e^0] = \frac{ C (1-\kappa/4)}{\sin (\pi (1-\kappa/4))}$. This yields~\eqref{eq:KW-cylinder}  in this case. 
	\medskip
	
	\noindent \textbf{Case II: $\kappa < 4$ and $\lambda \in \R$.}
	Since $\vartheta(\eta) \geq 0$ a.s.\ the function $\lambda \mapsto \E[e^{\lambda \vartheta(\eta)}]$ is increasing. Thus for $\lambda < 0$ we have $\E[e^{\lambda \vartheta(\eta)}] \leq \E[e^{0\cdot  \vartheta(\eta)}] = 1$, and since $1 - \frac\kappa8 -\frac2\kappa < 0$ we can use analytic continuation to extend the result from $(1-\frac\kappa8 -\frac2\kappa, 1-\frac\kappa8)$ to $(-\infty, 1-\frac\kappa8)$. Finally, taking a limit from below, we have for any $\lambda \geq 1-\frac{\kappa}8$ that $\E[e^{\lambda \vartheta(\eta)}] \geq \lim_{\lambda' \uparrow 1-\frac\kappa8} \E[e^{\lambda' \vartheta(\eta)}] = \infty$, where the limit follows from the explicit formula obtained in Case I. 
	\medskip
	
	\noindent \textbf{Case III: $\kappa = 4$.} For $\lambda < \frac12$, we obtain the result by taking $\kappa \uparrow 4$ as follows.  
	Let $\eta_\kappa$ be sampled from $\cL_\kappa$, then $\vartheta(\eta_\kappa) \to \vartheta(\eta_4)$ in law as $\kappa \uparrow 4$ by Lemma~\ref{lem-kappa-shape-2}.
	Fix $\lambda < \lambda' < 1 - \frac48$. For $\kappa$ sufficiently close to $4$ the family $\{ e^{\lambda \vartheta(\eta_\kappa)}\}$ is uniformly integrable, since Theorem~\ref{thm-KW} gives a uniform bound on $\E[e^{\lambda' \vartheta(\eta_\kappa)}]$ for $\kappa$ close to $4$. Therefore $\lim_{\kappa\uparrow4} \E[e^{\lambda \vartheta(\eta_\kappa)}] = \E[e^{\lambda \vartheta(\eta_4)}]$. Now, for $\lambda \geq \frac12$, the monotonicity argument of Case II gives $\E[e^{\lambda \vartheta(\eta)}] = \infty$. 
\medskip 

 \noindent \textbf{Case IV: $\kappa \in (4,8)$.} For $\lambda \in (1 - \frac{\kappa}8 - \frac2{\kappa}, 1-\frac{\kappa}8)$ the argument is identical to that of Case I, using Proposition~\ref{prop-KW-ns} instead of Proposition~\ref{prop-KW}. The extension to general $\lambda \in \R$ follows the same monotonicity argument of Case II. 
\end{proof}

We now introduce the key objects in our proofs of Propositions~\ref{prop-KW} and~\ref{prop-KW-ns}. Let $\gamma \in (0,2)$ and $\kappa \in \{ \gamma^2, 16/\gamma^2\} \cap (0,8)$ (in subsequent sections we will specialize to a fixed value of $\kappa$).

Sample a pair $(\eta, \mathbf t)$ from $\cL_\kappa^\alpha \times dt$ (where $dt$ is the Lebesgue measure on $\R$), and let $\SLE_\kappa^{\mathrm{sep},\alpha}$ be the law of the translated loop $\eta + \mathbf t$. Then $\SLE_\kappa^{\mathrm{sep},\alpha}$ is an infinite measure on loops on $\cC$ separating $\pm\infty$.
Note that $\cL_\kappa^\gamma =\cL_\kappa(\cC)$. According to  Lemma~\ref{lem-translation}, $\SLE_\kappa^{\mathrm{sep},\gamma}$ is a constant multiple of the measure $\SLE^{\sep}_\kappa(\cC)$ from Section~\ref{sec:msw-sphere}.

Let $\alpha < Q$ and $W = 2\gamma(Q-\alpha)$. Let $\mathbb F$ be the law of $h$ as in Definition~\ref{def-sphere}, 
so that  the law of $(\cC, h,-\infty,+\infty)/{\sim_\gamma}$  is $\cM_2^\sph(W)$.	Now sample $(h,\eta)$ from 
$\mathbb F \times \SLE_\kappa^{\mathrm{sep}, \alpha}$ and write $ \cM_2^\sph(W) \otimes \SLE_\kappa^{\mathrm{sep},\alpha}$ 	
for the law of $(\cC, h, \eta, -\infty,+\infty)/{\sim_\gamma}$.  

This object is related to $\cL^\alpha_\kappa$ as follows. Recall $\LF_\cC^{(\alpha, \pm\infty)}$ on $\cC$ from Definition~\ref{def-LFC}.
\begin{proposition}\label{prop:KW-LCFT}
If $(\phi, \eta)$ is sampled from $\LF_\cC^{(\alpha, \pm\infty)}\times  \cL_\kappa^\alpha$, the law of $(\cC, \phi, \eta, -\infty,+\infty)/{\sim_\gamma}$ is 
\eqb\label{eq-shift-leb}
C (Q-\alpha)^2 \cM_2^\sph(W) \otimes \SLE_\kappa^{\mathrm{sep}, \alpha} \quad \text{for some constant }C=C(\kappa)\in (0,\infty).
\eqe
\end{proposition}
\begin{proof}
This is an immediate consequence of   the definition of $\SLE_\kappa^{\mathrm{sep},\alpha}$ and \cite[Theorem 1.2]{AHS-SLE-integrability}, which says the following: 	let $h$ be the field  in Definition~\ref{def-sphere}, so the law of $(\cC, h, -\infty,+\infty)/{\sim}_\gamma$ is $\cM_2^\sph(W)$. Let $T \in \R$ be sampled from Lebesgue measure independently of $h$, and set $\phi := h( \cdot +T)$. Then $\phi$ has law $\frac{\gamma}{4 (Q-\alpha)^2} \LF_\cC^{(\alpha, \pm\infty)}$.
\end{proof}

In Section~\ref{sec:elect-thickness-simple} we will prove Proposition~\ref{prop-KW} ($\kappa < 4$). The proof uses a conformal welding result for $\cM_2^\sph(W) \otimes \SLE_\kappa^{\mathrm{sep},\alpha}$ (Proposition~\ref{prop-KW-weld}) shown in Section~\ref{subsec:KW-alpha}; the conformal welding gives two different descriptions of $\cM_2^\sph(W) \otimes \SLE_\kappa^{\mathrm{sep},\alpha}$, allowing us to compute the size of an event in two different ways, and a computation deferred to Section~\ref{subsec-proof-KW-LHS}. Next, for $\kappa' \in (4,8)$, we review the area statistics of generalized quantum disks with an $\alpha$-insertion in Section~\ref{subsec:gqd}, and in Section~\ref{section nonsimple kw} we carry out an argument parallel to that of Section~\ref{sec:elect-thickness-simple} to prove Proposition~\ref{prop-KW-ns}.

\subsection{The case \texorpdfstring{$\kappa\in (0,4)$}{g0}}\label{sec:elect-thickness-simple}

We now present the conformal welding result needed for the proof of Proposition~\ref{prop-KW}.

\begin{proposition}\label{prop-KW-weld}
	For $\alpha \in (\frac\gamma2, Q)$ and for some constant $C = C(\gamma)$ we have 
	\[C  (Q-\alpha)^2\cM_2^\sph(W) \otimes \SLE_\kappa^{\mathrm{sep}, \alpha} = \int_0^\infty \ell\cdot  \mathrm{Weld}(\cM_1^\disk(\alpha; \ell), \cM_1^\disk(\alpha; \ell)) \, d\ell. \]
\end{proposition}
We postpone the proof of Proposition~\ref{prop-KW-weld} to Section~\ref{subsec:KW-alpha} and proceed to the proof of Proposition~\ref{prop-KW}.
We would like to compare the area of a sample from $\cM_2^\sph(W) \otimes \SLE_\kappa^{\mathrm{sep}, \alpha} $ 
using Propositions~\ref{prop:KW-LCFT} and~\ref{prop-KW-weld} to obtain $|\cL^\alpha_\kappa|$ and hence prove Theorem~\ref{thm-KW}, but  $\cM_2^\sph(W) \otimes \SLE_\kappa^{\mathrm{sep}, \alpha}[e^{- A}]=\infty$. We will find a different computable observable that is finite. 
Note that $\cM_2^\sph(W) \otimes \SLE_\kappa^{\mathrm{sep}, \alpha} $ is a measure on quantum surfaces decorated by  two marked points and a loop separating them. The loop separates the quantum surface into two connected components. For $0<\eps<\delta$, let $E_{\delta, \eps}$ be the event that the connected component containing the first marked point has quantum area at least 1 and the loop has quantum length in $(\eps, \delta)$. The size of $E_{\delta, \eps}$ is easy to compute using Proposition~\ref{prop-KW-weld}: 
\begin{lemma}\label{lem-KW-right-E}
	Let $\alpha \in (\frac\gamma2, Q)$.  With $C$ from Proposition~\ref{prop-KW-weld},  $\cM_2^\sph(W) \otimes \SLE_\kappa^{\mathrm{sep}, \alpha}[E_{\delta,\eps}]$   equals
	\[
\frac{1}{C  (Q-\alpha)^2}\times \frac{(1+o_{\delta,\eps}(1))\log \eps^{-1}}{\frac2\gamma(Q-\alpha)\Gamma(\frac2\gamma(Q-\alpha))}\left(\frac2\gamma 2^{-\frac{\alpha^2}2}\ol U(\alpha) \right)^2  
	\left(4 \sin \frac{\pi\gamma^2}4\right)^{-\frac2\gamma(Q-\alpha)},
	\]
	where the error  term $o_{\delta,\eps}(1)$ satisfies $\lim_{\delta\to 0} \lim_{\eps\to 0} o_{\delta,\eps}(1)=0$.  %
\end{lemma}
\begin{proof}	Let $A$ be the quantum area of a sample from $\cM_1^\disk(\alpha; \ell)$. 
By Proposition~\ref{prop-KW-weld}, to prove Lemma~\ref{lem-KW-right-E}, it suffices to prove
		\begin{equation}\label{eq:KW-twodisk}
		\int_\eps^\delta \ell \cdot \left| \cM_{1}^\disk(\alpha; \ell)\right|\cM_{1}^\disk(\alpha; \ell)[A > 1] \, d\ell =		\frac{(1+o_{\delta,\eps}(1))\log \eps^{-1}}{\frac2\gamma(Q-\alpha)\Gamma(\frac2\gamma(Q-\alpha))}\left(\frac2\gamma 2^{-\frac{\alpha^2}2}\ol U(\alpha) \right)^2  
		\left(4 \sin \frac{\pi\gamma^2}4\right)^{-\frac2\gamma(Q-\alpha)}.
		\end{equation}
since the left side of~\eqref{eq:KW-twodisk} is the mass of $E_{\delta, \eps}$ under  $\int_0^\infty \ell\cdot  \mathrm{Weld}(\cM_1^\disk(\alpha; \ell), \cM_1^\disk(\alpha; \ell)) \, d\ell$.
	By the scaling of quantum area and boundary length, we have \(\cM_1^\disk(\alpha; \ell)^\# [A > 1] = \cM_1^\disk(\alpha; 1)^\# [A > \ell^{-2}].\)
	By Theorem~\ref{thm-FZZ} the quantum area law of $\cM_1^\disk(\alpha;1)^\#$ is inverse gamma with shape $a = \frac2\gamma(Q-\alpha)$ and scale $b = (4 \sin \frac{\pi \gamma^2}4)^{-1}$.	
	Let $\ul \Gamma$ be the lower incomplete gamma function; this satisfies $\lim_{y \to 0} \frac{\ul \Gamma(a; y)}{y^a} = \frac1a$.
	By the tail asymptotic property of the inverse gamma distribution, as $\ell\to 0$,
	\[\cM_1^\disk(\alpha; \ell)^\# [A > 1] = \frac{\ul\Gamma(\frac2\gamma(Q-\alpha); \ell^2/4\sin \frac{\pi\gamma^2}4)}{\Gamma(\frac2\gamma(Q-\alpha))} =  \frac{1+o_\ell(1)}{\frac2\gamma (Q-\alpha)\Gamma(\frac2\gamma (Q-\alpha))}\left(\frac{\ell^2}{4 \sin \frac{\pi\gamma^2}4}\right)^{\frac2\gamma(Q-\alpha)}.\]
 By Proposition~\ref{prop-remy-U} we have $|\cM_1^\disk(\alpha;\ell)| = \frac2\gamma 2^{-\frac{\alpha^2}2}\ol U(\alpha) \ell^{\frac2\gamma(\alpha-Q)-1}$. Therefore
 \begin{gather*}
 \int_\eps^\delta \ell \cdot \left| \cM_{1}^\disk(\alpha; \ell)\right|\cM_{1}^\disk(\alpha; \ell)[A > 1] \, d\ell = \int_\eps^\delta \ell \left(\frac2\gamma 2^{-\frac{\alpha^2}2}\ol U(\alpha) \ell^{\frac2\gamma(\alpha-Q)-1} \right)^2 \cM_1^\disk(\alpha; \ell)^\#  [A> 1] \\
 = \frac{1+o_{\delta,\eps}(1)}{\frac2\gamma(Q-\alpha)\Gamma(\frac2\gamma(Q-\alpha))}\left(\frac2\gamma 2^{-\frac{\alpha^2}2}\ol U(\alpha) \right)^2  
 \left(4 \sin \frac{\pi\gamma^2}4\right)^{-\frac2\gamma(Q-\alpha)}
 \int_\eps^\delta \ell^{-1} \, d\ell. \qedhere
 \end{gather*}
\end{proof}
We can also compute the size of $E_{\delta,\eps}$ using Proposition~\ref{prop:KW-LCFT} in terms of $|\cL_\kappa^\alpha|$ and $\ol R(\alpha)$.
\begin{proposition}\label{prop-KW-left-E} 
For $\alpha \in (\frac\gamma2, Q)$, with the error term in the same sense as in Lemma~\ref{lem-KW-right-E},
	\[(\cM_2^\sph(W) \otimes \SLE_\kappa^{\mathrm{sep},\alpha})[E_{\delta, \eps}] = (1+o_{\delta,\eps}(1))  \frac{\ol R(\alpha)}{2(Q-\alpha)^2} |\cL_\kappa^\alpha| \log \eps^{-1}. \]
\end{proposition}
We postpone the proof of Proposition~\ref{prop-KW-left-E} to Section~\ref{subsec-proof-KW-LHS} and proceed to the proof of Proposition~\ref{prop-KW}
using Lemma~\ref{lem-KW-right-E} and Proposition~\ref{prop-KW-left-E}.

\begin{proof}[Proof of Proposition~\ref{prop-KW}]
	In this proof, we write $C$ for a $\kappa$-dependent constant that may change from line to line. 
	By Lemma~\ref{lem-KW-right-E} and Proposition~\ref{prop-KW-left-E} we get
	\[ (Q-\alpha)^2  \frac{\ol R(\alpha)}{2(Q-\alpha)^2}|\cL_\kappa^\alpha|
	=
	\frac{C}{\frac2\gamma(Q-\alpha)\Gamma(\frac2\gamma(Q-\alpha))}\left(\frac2\gamma 2^{-\frac{\alpha^2}2}\ol U(\alpha) \right)^2  
	\left(4 \sin \frac{\pi\gamma^2}4\right)^{-\frac2\gamma(Q-\alpha)} .
	\]
	Using the definitions of $\ol R(\alpha)$ and $\ol U(\alpha)$ in~\eqref{eq-R} and~\eqref{eq:U0-explicit} and cancelling equal terms gives
	\[ \left(\frac{\pi\Gamma(\frac{\gamma^2}4)}{\Gamma(1-\frac{\gamma^2}4)}\right)^{\frac2\gamma (Q-\alpha)} \frac{\Gamma(\frac\gamma2(\alpha-Q))}{\Gamma(\frac\gamma2(Q-\alpha))} |\cL_\kappa^\alpha| = 
	\left(\frac{\pi^2}{\Gamma(1-\frac{\gamma^2}4)^2 \sin(\frac{\pi\gamma^2}4)}\right)^{\frac2\gamma(Q-\alpha)} C 2^{\alpha^2 - 2Q \alpha} 
	\Gamma(1 + \frac\gamma2(\alpha - Q))^2 \]
	The identity $\Gamma(z)\Gamma(1-z) = \frac\pi{\sin(\pi z)}$ gives equality of the first terms on the left and right hand sides, so rearranging and using $\Gamma(z+1) = z\Gamma(z)$ and $\Gamma(1-z)\Gamma(z) = \frac\pi{\sin(\pi z)}$ gives the desired identity:
	\[2^{-\alpha^2 + 2Q\alpha} |\cL_\kappa^\alpha|  = C \frac{\Gamma(1+\frac\gamma2(\alpha-Q))}{\Gamma(\frac\gamma2(\alpha-Q))}  \Gamma(1- \frac\gamma2(Q-\alpha))\Gamma(\frac\gamma2(Q-\alpha)) = C \frac\gamma2 (\alpha - Q) \cdot \frac{\pi}{\sin( \frac{\gamma\pi}2 (Q-\alpha))}.\]
\end{proof}

We now proceed to prove  Propositions~\ref{prop-KW-weld} and~\ref{prop-KW-left-E} in Sections~\ref{subsec:KW-alpha} and~\ref{subsec-proof-KW-LHS}, respectively.

\subsection{Conformal welding of two quantum disks with generic insertions}\label{subsec:KW-alpha}

The goal of this section is to prove Proposition~\ref{prop-KW-weld}. We start from the $\gamma$-insertion case.
\begin{lemma}\label{lem:KW-gamma}
    If $(\phi,\eta) $ is sampled form $\LF_\cC^{(\gamma,\pm\infty)}\times \cL_\kappa(\cC)$, then the law of $(\cC,\phi,\eta,-\infty,+\infty)/{\sim_\gamma}$ is $C\int_0^\infty \ell\cdot \Weld(\cM^{\rm disk}_{1,0}(\gamma;\ell),\cM^{\rm disk}_{1,0}(\gamma;\ell))d\ell$.
\end{lemma}
\begin{proof}
   Recall $ \SLE_\kappa^{\mathrm{sep}}$ from Proposition~\ref{prop-loop-zipper},
which  is $\SLE_\kappa^\mathrm{loop}$ restricted to loops separating $0$ and $\infty$.
By  the definition of $\cL_\kappa$, for some $\gamma$-dependent constant $C'\in (0,\infty)$, the pull back of $\SLE_\kappa^{\mathrm{sep}}$ via the map $z \mapsto e^{-z}$ is 
	$C'\SLE_\kappa^{\mathrm{sep},\gamma}$. Therefore, $  \QS_2\otimes \SLE_\kappa^{\mathrm{sep}}=C'\cM_2^\sph(\gamma) \otimes \SLE_\kappa^{\mathrm{sep,\gamma}}$.
Now   Propositions~\ref{prop:KW-LCFT} and~\ref{prop-loop-zipper} yield Lemma~\ref{lem:KW-gamma}. 
\end{proof}
By the definition of $\cM^{\disk}_1(\alpha;\ell)$,  if we sample a field $X$ from $\LF_{\bbH}^{(\alpha,i)}(\ell)$, then the law of $(\bbH, X,i)/{\sim_\gamma}$
is $\cM^{\disk}_1(\alpha;\ell)$.
We now recall a  fact from~\cite{ARS-FZZ} about the reweighting of $\LF_{\bbH}^{(\gamma,i)}(\ell)$ by ``$e^{(\alpha-\gamma)X}$''.
\begin{lemma}[{\cite[Lemma 4.6]{ARS-FZZ}}]\label{lem-disk-reweight}
For any $\ell>0,\eps\in (0,1)$ and for any nonnegative  measurable function $f$ of $X$ that depends only on $X|_{\bbH\setminus B_\eps(i)}$,  we have 
	\begin{equation*}
		\int f(X|_{\bbH\setminus B_\eps(i)}) \times   \eps^{\frac12(\alpha^2 - \gamma^2)} e^{(\alpha - \gamma)X_\eps(i)}  \,d   {\LF_\bbH^{(\gamma, i)}(\ell)}= \int f(X|_{\bbH\setminus B_\eps(i)}) \, d\LF_{\bbH}^{(\alpha, i)}(\ell),
	\end{equation*}
	where $X$ is a sample in $H^{-1}(\bbH)$ and $X_\eps(i)$ means the average of $X$ on the boundary of the ball $B_\eps(i)=\{z: |z-i|< \eps\}$. 
\end{lemma}

We now state a version for the sphere; the proof is very similar to that of Lemma~\ref{lem-disk-reweight} but we include it here for completeness.
\begin{lemma}\label{lem:reweight} Let $\eta_1$ be a  simple loop in $\C$ separating $0$ and $1$. 
	Let $D_{\eta_1}$ be the connected component of   $\C\setminus\eta_1$ containing $0$.  
	Let $p$ be a point on $\eta_1$ and  let $\psi: \bbH\to D_{\eta_1}$ be the conformal map with $\psi(i) = 0$ and $\psi(0) = p$.
	For $\eps\in (0,\frac14)$, let  $\C_{\eta_1,p,\eps} =\C\setminus \psi (B_\eps(i))$. 
	For $\phi$   sampled from $\LF_\C^{(\gamma, 0), (\gamma, 1)}$,
	let  $X=\phi\circ \psi  +Q \log |\psi'|$  so that $(\bbH,X,i,0)/{\sim_\gamma}=(D_{\eta_1}, \phi, 0, p)/{\sim_\gamma}$.
	Then for a fixed $\alpha\in (Q-\frac{\gamma}4,Q)$ and 
	for  any nonnegative  measurable function $f$ of $\phi$ that depends only on  $\phi|_{\C_{\eta_1,p,\eps}}$ we have 
	\begin{align*}
		\int   f (\phi) \times  \eps^{\frac12(\alpha^2 - \gamma^2)} e^{(\alpha - \gamma)X_\eps(i)} \, d
		\LF_\C^{(\gamma, 0), (\gamma, 1)}
		= \int f (\phi) \left(\frac{\CR(\eta_1, 0)}2\right)^{-\frac{\alpha^2}2 + Q \alpha  - 2}  d\LF_\C^{(\alpha,0), (\gamma, 1)}.
	\end{align*}
\end{lemma}
\begin{proof}
	Let $\theta_\eps$ be the uniform probability measure on $\partial B_\eps(i)$ and $\wh \theta_\eps := \psi_* \theta_\eps$.  Recall notations  in Section~\ref{subsubsec:GFF}, where $\P_\C$ is the probability measure for the GFF on $\C$. Let $\E$ be the expectation for $\P_\C$, then $G_\C(z,w)=\E[h(x)h(y)] = -\log|z-w| + \log|z|_+ + \log |w|_+$. For a sample $h$ from $\P_\C$, we set $\wt h := h - 2Q \log |\cdot|_+ + \sum_{j=1}^2 \gamma G_\C(\cdot, z_j)$.  Lemma~\ref{lem:reweight} will follow from three identities.
		\eqb\label{eq-gir-sph-1}
	e^{(\alpha - \gamma)X_\eps(i)}  = \left(\frac{\mathrm{CR}(\eta_1, 0)}{2} \right)^{(\alpha - \gamma)Q} e^{(\alpha - \gamma)(\phi, \wh \theta_\eps)}.
	\eqe
	\eqb\label{eq-gir-sph-2}
	\int   f (\phi) \times e^{(\alpha - \gamma)(\phi, \wh \theta_\eps)} \, d
	\LF_\C^{(\gamma, 0), (\gamma, 1)} = \left(\frac{\eps\mathrm{CR}(\eta_1, 0)}{2} \right)^{-(\alpha - \gamma) \gamma}\int \E[ e^{(\alpha - \gamma)(h, \wh \theta_\eps)} f(\wt h + c) ] e^{(\alpha +\gamma - 2Q)c}\,dc. 
	\eqe
	\eqb\label{eq-gir-sph-3}
	\int \E[ e^{(\alpha - \gamma)(h, \wh \theta_\eps)} f(\wt h + c) ] e^{(\alpha + \gamma - 2Q)c}\,dc = \left(\frac{\eps \mathrm{CR}(\eta_1, 0)}{2} \right)^{- \frac12 (\alpha - \gamma)^2}\int f(\phi) \LF_\C^{(\alpha, 0), (\gamma, 1)}. 
	\eqe
	
To prove~\eqref{eq-gir-sph-1}, note that $X_\eps(i)= (X, \theta_\eps) = (\phi \circ \psi + Q \log |\psi'|, \theta_\eps) $.   Since $\psi'$ is holomorphic and $\log |\psi'|$ is harmonic, we have $  (  \log |\psi'|, \theta_\eps) =  \log |\psi'(i)|$ hence $(X, \theta_\eps)  = (\phi, \wh\theta_\eps) + Q \log |\psi'(i)|$.  Since $|\psi'(i)| = \frac12 \CR(\eta_1, 0)$, we get~\eqref{eq-gir-sph-1}. 

To prove~\eqref{eq-gir-sph-2}, let $\eta^\eps:= \psi(\partial B_\eps(i))$.
Since $\wh \theta_\eps$ is the harmonic measure on $\eta^\eps$ viewed from $0$, we have  $(\log|\cdot|, \wh \theta_\eps) =  \log \mathrm{CR}(\eta^\eps,0) = \log \frac{\eps \mathrm{CR}(\eta_1, 0)}2$. 
Since $\eps < \frac14$, the curve $\eta^\eps$ is contained in the unit disk hence $(\log |\cdot|_+, \wh \theta_\eps) = 0$. Let $G_\C^\eps(0, z) := (G_\C(\cdot,z) , \wh \theta_\eps)$. Then 
$G_\C^\eps(0,z) = G_\C(0,z)$ for $z\in \C_{\eta_1, p, \eps}$ and $G_\C^\eps(0,0) =   -\log \frac{\eps \mathrm{CR}(\eta_1, 0)}{2}$.
Thus,
$(-2Q \log |\cdot|_+ + \gamma G_\C(\cdot, 0) + \gamma G_\C(\cdot, 1)
, \wh \theta_\eps) = - \gamma \log \frac{\eps \mathrm{CR}(\eta_1, 0)}2$. Therefore $(\wt h+c, \wh \theta_\eps)=(h,\wh \theta_\eps) - \gamma \log \frac{\eps \mathrm{CR}(\eta_1, 0)}2+c$.
Recall from  Definition~\ref{def-RV-sph} that $	\LF_\C^{(\gamma, 0), (\gamma, 1)}$ is the law of $\wt h+c$ under $e^{(2\gamma-2Q) c } dc \P_\C(d h) $. This gives~\eqref{eq-gir-sph-2}.
	
To prove~\eqref{eq-gir-sph-3}, note that $\E[h(z) (h, \wh \theta_\eps)]=(G_\C(\cdot, z), \wh\theta_\eps )=G_\C^\eps(0,z)$, which equals $G_\C(z,0)$ for $z\in \C_{\eta_1, p, \eps}$.
	By Girsanov's theorem and the fact that $f(\phi)$ depends only on $\phi|_{\C_{\eta, p, \eps}}$,  we have 
	\alb
	\int \E[ e^{(\alpha - \gamma)(h, \wh \theta_\eps)} f(\wt h + c) ] e^{(\alpha + \gamma - 2Q)c}\,dc &= \E[e^{(\alpha - \gamma)(h, \wh \theta_\eps)}]\int \E[ f(\wt h + (\alpha - \gamma)G^\eps_\C(\cdot, 0)+ c) ] e^{(\alpha + \gamma - 2Q)c}\,dc \\
	&= \E[e^{(\alpha - \gamma)(h, \wh \theta_\eps)}]\int \E[ f(\wt h + (\alpha - \gamma)G_\C(\cdot, 0)+ c) ] e^{(\alpha + \gamma - 2Q)c}\,dc.
	\ale 
	Since $\Var ((h, \wh \theta_\eps)) = (G_\C^\eps(\cdot, 0), \wh \theta_\eps) = -\log \frac{\eps \mathrm{CR}(\eta_1, 0)}{2}$, we have $\E[e^{(\alpha-\gamma)(h, \wh \theta_\eps)}] = \left(\frac{\eps \mathrm{CR}(\eta_1, 0)}{2} \right)^{- \frac12 (\alpha - \gamma)^2}$. 
	Since the law of $\wt h+c$ under $e^{(\alpha+\gamma-2Q) c } dc \P_\C(d h) $ is $\LF_\C^{(\alpha, 0), (\gamma, 1)} $, we get~\eqref{eq-gir-sph-3}.
	
	Combining~\eqref{eq-gir-sph-1},~\eqref{eq-gir-sph-2} and~\eqref{eq-gir-sph-3}, and collecting the prefactors via
	\[ \left(\frac{\mathrm{CR}(\eta_1, 0)}{2} \right)^{(\alpha - \gamma)Q}   \left(\frac{\eps\mathrm{CR}(\eta_1, 0)}{2} \right)^{-(\alpha - \gamma) \gamma}   \left(\frac{\eps \mathrm{CR}(\eta_1, 0)}{2} \right)^{- \frac12 (\alpha - \gamma)^2} = \eps^{-\frac12(\alpha^2-\gamma^2)}\left(\frac{\mathrm{CR}(\eta_1, 0)}{2} \right)^{-\frac{\alpha^2}2 + Q\alpha - 2},\]
	we conclude the proof of Lemma~\ref{lem:reweight}.
\end{proof}

Suppose $\eta$ is a simple curve in $\cC$ separating $\pm\infty$ with two marked points $p^-,p^+ \in \eta$. Let $D^\pm_\eta$  be the connected components of $\cC \backslash \eta$ containing $\pm\infty$, and let $\psi^\pm_\eta: \bbH\to D^\pm_\eta$ be the conformal maps sending $(i, 0)$ to $(\pm\infty, p^\pm)$.
We need the following lemma,  which is essentially the variant of Lemma~\ref{lem:reweight} on $\cC$ where $\gamma$-insertions at $\pm\infty$ are changed to $\alpha$-insertions.
\begin{lemma}\label{lem:reweight-2}
	Let $\eta$ be a simple curve in $\cC$ that separates $\pm\infty$ with two marked points $p^-,p^+ \in \eta$. 
	For $\eps\in (0,\frac14)$, let $\cC_{\eta,p^\pm,\eps} =\cC\setminus (\psi^- _\eta(B_\eps(i)) \cup \psi^+_\eta(B_\eps(i)))$.
	For $\phi$ sampled from $\LF_\cC^{(\gamma, \pm\infty)}$, let $X^\pm =\phi\circ \psi^\pm_\eta  +Q \log |(\psi^\pm_\eta)'|$.
	Then for a fixed $\alpha\in (\frac\gamma2,Q)$ and 
	for any $\eps\in (0,\frac14)$ and any nonnegative  measurable function $f$ of  $\phi|_{\cC_{\eta,p^\pm,\eps}}$ we have 
	\begin{align*}
		&\int   f (\phi|_{\cC_{\eta,p^\pm,\eps}}) \times  \eps^{\alpha^2 - \gamma^2} e^{(\alpha - \gamma)(X^-_\eps(i) + X^+_\eps(i))}\, d
		\LF_\cC^{(\gamma,\pm\infty)}
		\\
		&= \int f (\phi|_{\cC_{\eta,p^\pm,\eps}}) \left(\frac14\CR(\exp(\eta), 0)\CR(\exp(-\eta), 0)\right)^{-\frac{\alpha^2}2 + Q \alpha  - 2} \, d\LF_\cC^{(\alpha, \pm\infty)}.
	\end{align*}
\end{lemma}
\begin{proof}
		Let $g: \C \to \C$ be given by $g(z) = \frac{z}{z-1}$ and let $G : \cC\to \C$ be given by $G = g \circ \exp$. By \cite[Lemma 2.13]{AHS-SLE-integrability}, if $\phi$ is sampled from $\LF_\cC^{(\gamma, \pm\infty)}$ then $\hat \phi := \phi \circ G^{-1} + Q \log |(G^{-1})'|$ has law $\LF_\C^{(\gamma, 0), (\gamma, +1)}$, and the same is true when $\gamma$ is replaced by $\alpha$. Let $(\hat \eta, \hat p^-, \hat p^+) = (G(\eta), G(p^-), G(p^+))$. Since $g'(0) = -1$ and $\frac{d}{dz} (g(\frac1z))|_{z=0} = 1$, we see that $\CR(\exp(\eta), 0) = \CR( \hat \eta, 0)$ and $\CR( \exp(-\eta), 0) = \CR(\hat \eta, +1)$. Let $\C_{\hat \eta, \hat p^\pm , \eps} := \C \backslash (G(\psi^-(B_\eps(i)))\cup G(\psi^+( B_\eps(i))) )$. Then Lemma~\ref{lem:reweight-2}  is equivalent to the following:
	for any nonnegative  measurable function $\hat f$ of $\hat \phi$ that depends only on  $\hat \phi|_{\C_{\hat \eta, \hat p^\pm , \eps}}$, we have 
		\begin{align}
		&\int   \hat f (\hat \phi) \times  \eps^{\alpha^2 - \gamma^2} e^{(\alpha - \gamma)(X^-_\eps(i) + X^+_\eps(i))} \, d
		\LF_\C^{(\gamma,0), (\gamma, +1)} \nonumber
		\\
		&= \int \hat f (\hat \phi) \left(\frac14\CR(\hat \eta, 0)\CR(\hat \eta, +1)\right)^{-\frac{\alpha^2}2 + Q \alpha  - 2}  d\LF_\C^{(\alpha, 0), (\alpha, + 1)}.
		\label{eq:KW-alpha}
	\end{align}
Now we apply Lemma~\ref{lem:reweight},we get
\begin{align*}
	\int   \hat f (\hat\phi) \times  \eps^{\frac12(\alpha^2 - \gamma^2)} e^{(\alpha - \gamma)X^-_\eps(i)} \, d
	\LF_\C^{(\gamma,0), (\gamma, +1)}
	= \int \hat f (\hat \phi) \left(\frac12\CR(\hat \eta, 0)\right)^{-\frac{\alpha^2}2 + Q \alpha  - 2}  d\LF_\C^{(\alpha, 0), (\gamma, + 1)}.
\end{align*}
Applying the argument of Lemma~\ref{lem:reweight} again to change the insertion at $+1$, we get~\eqref{eq:KW-alpha}. 
\end{proof}

For a curve $\eta$ in $\cC$ that separates $\pm\infty$, we let  $\mathrm{Harm}_{-\infty, \eta}$ (resp.  $\mathrm{Harm}_{+\infty, \eta}$)  be the harmonic measure on $\eta$ viewed from $-\infty$ (resp., $+\infty$). 
\begin{lemma}\label{lem-KW-reweighted-fields}
	There is a constant $C = C(\gamma)$ such that the following holds. 
	Suppose $\alpha \in (\frac\gamma2, Q)$. 
	Sample $(\phi, \eta, p^-, p^+)$ from the measure \[C \cdot \LF_\cC^{(\alpha, \pm\infty)}(d\phi)\, \cL_\kappa^\alpha(d\eta)\, \mathrm{Harm}_{-\infty, \eta}(dp^-)\, \mathrm{Harm}_{+\infty, \eta}(dp^+).\]
	Let $X_\pm = \phi \circ \psi_\eta^\pm + Q \log |(\psi_\eta^\pm)'|$. Let $\tau$ be the quantum length of the clockwise arc from $p^-$ to $p^+$ in $D^+_\eta$. Then the law of $(X^-, X^+, \tau)$ is
	\[\int_0^\infty \LF_\bbH^{(\alpha, i)}(\ell) \times \LF_\bbH^{(\alpha, i)}(\ell) \times  [1_{\tau \in (0, \ell)}  d\tau] \, d\ell.\]
\end{lemma}
\begin{proof}
		We first prove the case $\alpha = \gamma$, namely 
	\eqb \label{eq-sph-gamma}
	C \cdot \LF_\cC^{(\gamma, \pm\infty)}(d\phi)\, \cL_\kappa(d\eta)\, \mathrm{Harm}_{-\infty, \eta}(dp^-)\, \mathrm{Harm}_{+\infty, \eta}(dp^+) = \int_0^\infty \LF_\bbH^{(\gamma, i)}(\ell) \times \LF_\bbH^{(\gamma, i)}(\ell) \times  [1_{\tau \in (0, \ell)} d\tau] \, d\ell,
	\eqe
	where with abuse of notation we view the left-hand side as a measure on triples $(X^-, X^+, \tau) \in H^{-1}(\bbH)\times H^{-1}(\bbH)\times[0,\infty)$. Indeed,~\eqref{eq-sph-gamma} is an immediate consequence of Proposition~\ref{prop-loop-zipper} and Lemma~\ref{lem:har}.

	Now, let $\eps \in (0, \frac14)$ and $\bbH_\eps := \bbH \backslash B_\eps(i)$. 
	Let $f$ be a nonnegative measurable function of $(X^-|_{\bbH_\eps}, X^+|_{\bbH_\eps}, \tau)$, then  reweighting~\eqref{eq-sph-gamma} gives 
	\alb
	C\int f(X^-|_{\bbH_\eps}, X^+|_{\bbH_\eps}, \tau) \eps^{\alpha^2 - \gamma^2} e^{(\alpha - \gamma) (X^-_\eps(i) + X^+_\eps(i))}\LF_\cC^{(\gamma,\pm\infty)}(d\phi) \cL_\kappa(d\eta) \mathrm{Harm}_{-\infty, \eta}(dp^+)\mathrm{Harm}_{-\infty, \eta}(dp^+) \\
	= \int_0^\infty \left(\int f(X^-|_{\bbH_\eps}, X^+|_{\bbH_\eps}, \tau) \eps^{\alpha^2 - \gamma^2} e^{(\alpha - \gamma) (X^-_\eps(i) + X^+_\eps(i))}  \ell \LF_\bbH^{(\gamma, i)}(\ell) \times \LF_\bbH^{(\gamma, i)}(\ell) \times  [1_{\tau \in (0, \ell)} d\tau] \right) \, d\ell.
	\ale
	By Lemma~\ref{lem:reweight-2}, the left hand side equals 
	\[ C\int f(X^-|_{\bbH_\eps}, X^+|_{\bbH_\eps}, \tau) \LF_\cC^{(\alpha,\pm\infty)}(d\phi) \cL_\kappa^\alpha(d\eta) \mathrm{Harm}_{-\infty, \eta}(dp^-)\mathrm{Harm}_{+\infty, \eta}(dp^+).\]
	By Lemma~\ref{lem-disk-reweight}, the right hand side equals
	\[ \int_0^\infty \left(\int f(X^-|_{\bbH_\eps}, X^+|_{\bbH_\eps}, \tau) \LF_\bbH^{(\alpha, i)}(\ell) \times \LF_\bbH^{(\alpha, i)}(\ell) \times  [1_{\tau \in (0, \ell)} d\tau] \right) \, d\ell. \]
	Since the above two expressions agree for every $\eps$ and $f$, we obtain the result.
\end{proof}

\begin{proof}[Proof of Proposition~\ref{prop-KW-weld}]
	In Lemma~\ref{lem-KW-reweighted-fields}, the law of $(\cC, \phi, \eta, \pm\infty)/{\sim_\gamma}$ is $C (Q-\alpha)^2 \cM_2^\sph(W) \otimes \SLE_\kappa^{\mathrm{sep}, \alpha}$ by~\eqref{eq-shift-leb}, and the joint law of $((\bbH, X^-, i)/{\sim_\gamma}, (\bbH, X^+, i)/{\sim_\gamma}, \tau)$ is $\int_0^\infty \ell \cM_1^\disk(\alpha; \ell) \times \cM_1^\disk(\alpha; \ell) \times [1_{\tau \in (0,\ell)} \ell^{-1} d\tau]\, d\ell$. Since $[1_{\tau \in (0,\ell)} \ell^{-1} d\tau]$ corresponds to uniform conformal welding, the law of $(\cC, \phi, \eta, \pm\infty)/{\sim_\gamma}$ is $\int_0^\infty \ell \mathrm{Weld}(\cM_1^\disk(\alpha; \ell), \cM_1^\disk(\alpha; \ell))\, d\ell$.
\end{proof}

\subsection{The appearance of \texorpdfstring{$\ol R(\alpha) $}{g} and  \texorpdfstring{$|\cL_\kappa^\alpha|$}{g}: proof of Proposition\texorpdfstring{~\ref{prop-KW-left-E}}{g}}\label{subsec-proof-KW-LHS}
We will prove  Proposition~\ref{prop-KW-left-E} via a particular embedding  of $\cM_2^\sph(W)$.
Let $h$ be  the field $\hat h+\mathbf c$ in Definition~\ref{def-sphere} so that the law of $(\cC, h , -\infty, +\infty)/{\sim_\gamma}$ is $\cM_2^\sph(W)$. 
Now we restrict to the event $\{ \mu_h(\cC) > 1\}$ and set $\phi := h(\cdot - a)$, where $a \in \R$ is such that $\mu_h((-\infty, a) \times [0,2\pi]) = 1$. 
Namely, we shift $h$ horizontally such that $\mu_\phi(\{ z\in \cC: \Re z\le 0  \})=1$.
Let $M$ be the law of $\phi$ under this restriction. Then we can represent $\cM_2^\sph(W) \otimes \SLE_\kappa^{\mathrm{sep}, \alpha}[E_{\delta,\eps}] $ in Proposition~\ref{prop-KW-left-E} as follows.
\begin{lemma}\label{lem:start}
Given a simple closed curve $\eta$ on $\cC$ separating $\pm\infty$,
let $\cC_\eta^+$ (resp. $\cC_\eta^-$) be the connected component of $\cC\backslash \eta$ containing $+\infty$ (resp. $-\infty$).
For $t\in \R$, let $\eta+t$ be the curve on $\cC$ obtained by shifting $\eta$  by $t$.
Now sample $(\phi, \mathbf t, \eta^0)$ from $M\times dt \times \cL_\kappa^\alpha$ and set $\eta = \eta^0 + \mathbf t$.
Let 
\begin{equation}\label{eq:event-ob}
E'_{\delta, \eps} := \{(\phi, \eta): \eps < \ell_\phi(\eta)< \delta  \text{ and } \mu_\phi(\cC_\eta^-) >1 \}
\end{equation}
where $\ell_\phi(\eta)$ is the quantum length of $\eta$.
Then \(\cM_2^\sph(W) \otimes \SLE_\kappa^{\mathrm{sep}, \alpha}[E_{\delta,\eps}] = (M \times dt \times \cL_\kappa^\alpha)[E'_{\delta, \eps}].\)
\end{lemma}
\begin{proof}
Since the measure $\SLE_\kappa^{\mathrm{sep},\alpha}$ is invariant under  translations along the cylinder,
the law of  $(\cC, \phi, \eta, +\infty,-\infty)/{\sim_\gamma}$ is the restriction of  $\cM_2^\sph(W) \otimes \SLE_\kappa^{\mathrm{sep}, \alpha}$ to 
the event that the  total quantum area is larger than $1$.
Now Lemma~\ref{lem:start} follows from the definition of $E_{\delta,\eps}$.
\end{proof}

The next lemma explains how the reflection coefficient $\ol R (\alpha)$ shows up in Proposition~\ref{prop-KW-left-E}. 
\begin{lemma}\label{lem-sph-area-law-M}
	 \(|M| = \int_1^\infty \frac12 \ol R(\alpha) a^{\frac2\gamma(\alpha-Q)-1} \, da = \frac{\gamma \ol R(\alpha)}{4 (Q-\alpha)}\).
\end{lemma}
\begin{proof} 
	This simply comes from Lemma \ref{lem-sph-area-law}.
\end{proof}

By Lemmas~\ref{eq:event-ob} and~\ref{lem-sph-area-law-M},   Proposition~\ref{prop-KW-left-E} is reduced to  the following proposition.

\begin{proposition}\label{prop:product}
Let $\alpha\in (\frac{\gamma}{2},Q)$.	With an  error term $o_{\delta,\eps}(1)$ satisfying   $\lim_{\delta\to 0} \limsup_{\eps\to 0} |o_{\delta,\eps}(1)|=0$, we have
\begin{equation}\label{eq:KW-LCFT-key}
	(M \times dt \times \cL_\kappa^\alpha)[E'_{\delta, \eps}] =
	(1+o_{\delta,\eps}(1)) |M| |\cL_\kappa^\alpha| \frac{2\log \eps^{-1}}{\gamma(Q-\alpha)}. 
	\end{equation}
\end{proposition}	

\begin{proof}[Proof of Proposition~\ref{prop-KW-left-E} given  Proposition~\ref{prop:product}]
	By Lemma~\ref{lem:start} and~\ref{lem-sph-area-law-M}, we have 
	\[(\cM_2^\sph(W) \otimes \SLE_\kappa^{\mathrm{sep},\alpha})[E_{\delta, \eps}]= (1+o_{\delta,\eps}(1))  |M| |\cL_\kappa^\alpha| \frac{2 \log \eps^{-1}}{\gamma(Q-\alpha)}= (1+o_{\delta,\eps}(1))  \frac{\ol R(\alpha)}{2(Q-\alpha)^2} |\cL_\kappa^\alpha| \log \eps^{-1}. \qedhere\]
\end{proof}

\begin{figure}[ht!]
	\begin{center}
		\includegraphics[scale=0.55]{figures/cylinder-loops-2}
	\end{center}
	\caption{\label{fig-2pt}  Illustration of the proof of Proposition~\ref{prop:product}. On the left $(\phi, \mathbf{t}, \eta^0)$ is sampled from $M\times dt \times \cL^\alpha_\kappa$ and $\eta=\eta^0+\mathbf{t}$. The curve $\eta^0$ satisfies $\sup_{z \in \eta_0}\Re z=0$ and the field $\phi$ satisfies $\mu_\phi(\{z\in\cC: \Re z\le 0 \})=1$. The domain $D^-_\eta$ is colored grey. The event $E'_{\delta,\eps}$ is $\{ \ell_\phi(\eta)\in (\eps,\delta) \textrm{ and } \mu_\phi(\cC_\eta^-) >1 \}$. If  $E'_{\delta,\eps}$  occurs, we have $\ell_\phi(\eta) \approx e^{\frac{\gamma}2 X_{\bf t}}$. Therefore $\mathbf t$ is in an interval close to the red one on the right figure, which is approxiately $(\frac2{\gamma(Q-\alpha)}\log \delta^{-1},  \frac2{\gamma(Q-\alpha)}\log \eps^{-1})$. The proofs of Lemmas~\ref{prop-KW-left-E-lower} and~\ref{lem-KW-left-F|E} are done by quantifying this statement in two directions. 
	}
\end{figure}
The high level idea for proving Proposition~\ref{prop:product}  is the following.  Suppose  $(\phi, \mathbf t, \eta^0)$ is sampled  from $M\times dt \times \cL_\kappa^\alpha$.
For $s \geq 0$, let $X_s$ be the average of $\phi$ on $[s,s+2\pi i]/{\sim}$.
For most realizations of $(\phi, \eta^0)$, the occurrence of $E'_{\delta, \eps}$ is equivalent to  
the event that $\mathbf t$ lies in some interval of length  $(1+o_{\delta,\eps}(1)) \frac{2\log \eps^{-1}}{\gamma(Q-\alpha)}$ determined by  $X$. (See Figure~\ref{fig-2pt}.)
Hence the mass of $E'_{\delta, \eps}$ is this length times $|M||\cL_\kappa^\alpha|$.
In the rest of this section, we first prove a few properties for $X$ in Section~\ref{subsub:field-average} 
and then prove Proposition~\ref{prop:product} in Section~\ref{subsub:event-ob}.

\subsubsection{The field average process}\label{subsub:field-average}
We  need  the following description of  the law of $\phi|_{\cC_+}$ and $(X_s)_{s\ge 0}$, where $\cC_+ := \{ z \in \cC \: : \: \Re z > 0\}$ is the right half cylinder.
\begin{lemma}\label{lem-markov}
Let $\phi$ be a sample from $M$. 
Conditioned on $\phi|_{\cC_-}$, the conditional law of $\phi|_{\cC_+}$ is the law of $\phi_0 + \hat{\mathfrak h} - (Q-\alpha) \Re (\cdot)$, where $\phi_0$ is a zero boundary GFF on $\cC_+$, and  $\hat{\mathfrak h}$ is a harmonic function determined by $\phi|_{\cC_-}$ whose average on the circle $[s, s+2\pi i]/{\sim}$ does not depend on  $s$.
Moreover, let $X_s$ be the average of $\phi$ on $[s,s+2\pi i]/{\sim}$. Conditioned on $X_0$, the conditional law of $(X_s-X_0)_{s \geq 0}$ is the law of $(B_s-(Q-\alpha)s)_{s\ge 0}$ where $B_s$ is a Brownian motion.
\end{lemma}
\begin{proof}
	The first statement is the sphere analog of \cite[Lemmas 2.10 and 2.11]{ag-disk} based on the domain Markov property of Gaussian free field. 
	The proof is identical so we omit it.  The second statement on $(X_s)_{s\ge 0}$ follows from the first statement.  
\end{proof}

Proposition~\ref{prop:product} essentially follows from the fact that the field average process $(X_s)_{s\ge 0}$ looks like a line of slope $-(Q-\alpha)$. We now introduce two random times to quantify this. For $y>0$, let
\begin{equation}\label{eq:BM-times}
\sigma_y =  \inf\{ s > 0 \: : \: X_s < \frac2\gamma \log y\} \quad \textrm{and}\quad   \tau_y = \sup\{ s > 0 \: : \: X_s > \frac2\gamma \log y\}.
\end{equation}
\begin{lemma}\label{lem-BM-hitting}
	$M$-a.e.\ the field $\phi$ satisfies $\lim_{y \to 0} \frac{\sigma_y}{\log y^{-1}} = \lim_{y \to 0} \frac{\tau_y}{\log y^{-1}} = \frac2{\gamma(Q-\alpha)}$.
\end{lemma}
\begin{proof}
	Given Lemma~\ref{lem-markov}, this is a straightforward   fact about drifted Brownian motion.
\end{proof}

The proof of the upper bound in  Proposition~\ref{prop:product} will rely on the following lemma.
\begin{lemma}\label{lem:average}
Sample $(\phi, \mathbf t, \eta^0)$ from $M\times dt \times \cL_\kappa^\alpha$ and set $\eta = \eta^0 + \mathbf t$. Let $\mathfrak l= \ell_\phi(\eta)$, which is the boundary length of the quantum surface $(\cC_\eta^-, \phi)/{\sim_\gamma}$.
Fix $\zeta\in (0,1)$. 
Then there exists  a constant $C>0$ and a function $\mathrm{err}(\ell)$ such that  $\lim_{\ell\downarrow 0} \mathrm{err}(\ell)=0$,  and conditioned on
$(\cC_\eta^-, \phi)/{\sim_\gamma}$, 
the conditional probability of $\{\mathbf t \in (\sigma_{{\mathfrak l}^{1-\zeta}} - C , \tau_{{\mathfrak l}^{1+\zeta}})\}$ is at least $1-\mathrm{err}(\mathfrak l)$.
\end{lemma}
\begin{proof}
We first introduce $C$ and $\mathrm{err},$ and then show that they satisfy Lemma~\ref{lem:average}.
Recall that $\eta^0$ is a loop in $\cC$ separating $\pm\infty$ and with $\sup_{z \in \eta^0} \Re z = 0$.
Let $\cC_{\eta^0}^+$ be the connected component of $\cC \backslash \eta^0$ containing $+\infty$. 
Let  $f: \cC_+\to \cC_{\eta^0}^+$ be the unique  conformal map such that $f(+\infty)=+\infty$ and $\lim_{z\to +\infty} (f(z)-z) \in \R$. 
By standard conformal distortion estimates (e.g.\ \cite[Lemma 2.4]{sphere-constructions}), for $C=e^{10}$
we have 
\begin{equation}\label{eq:C0-def}
|f(z) - z| < \frac {C}3  \quad \textrm{and}\quad    |f''(z)| < \frac12 < |f'(z)| < 2   \quad\textrm{for } \Re z > \frac {C}3.
\end{equation}
Let $g$ be a fixed smooth function on $\cC$ supported on $\{\Re z \in (C-1, C)\}$ such that $\int g(z) \, dz = 1$ and $g$ is invariant under rotations against the axis of $\cC$. Let $S$ be the collection of smooth functions $\xi$ that are supported on $\{ \Re z \in [\frac23 C-1, \frac43C]\} \subset \cC_+$ and satisfy $\|\xi\|_\infty \leq 4\|g\|_\infty$ and $\|\nabla\xi\|_\infty \leq 8(\|g\|_\infty +\|\nabla g\|_\infty)$.	
Let 
\begin{equation}\label{eq:err-def}
\mathrm{err} (\ell) :=\cM_1^{\mathrm{disk}}  (\alpha; 1)^{\#} \Big[   \sup_{\xi \in S} |(h, \xi)|+ Q\log 2 \ge -\frac{2\zeta}{\gamma} \log \ell \Big]
\end{equation}

where $(\cC_+,h,+\infty)$ is an embedding of a sample from $\cM_1^{\mathrm{disk}}  (\alpha; 1)^{\#}$, the probablity measure proportional to $\cM_1^{\mathrm{disk}}  (\alpha; 1)$. Since the space $S$ is invariant under rotations against the axis of $\cC$, the probability in~\eqref{eq:err-def} does not depend on the choice of the embedding. Moreover, since $h$ is  a.s.\ in the Sobolev space of index $-1$, we see that 
$\sup_{\xi \in S} |(h, \xi)|<\infty$ a.s.\ hence $\lim_{\ell\to 0} \mathrm{err} (\ell)=0$.
 
We now show that $C$ and $\mathrm{err}$  satisfy Lemma~\ref{lem:average}. Set $\phi^0(\cdot)=\phi(\cdot+\mathbf t)$ and $\hat \phi= \phi^0\circ f  + Q \log |f'|-\frac{2}{\gamma} \log \mathfrak l$. 
Then  $(\cC_{\eta}^+, \phi , +\infty)/{\sim_\gamma}= (D_{\eta_0}^+, \phi^0, +\infty)/{\sim_\gamma}$ hence $(\cC_+, \hat\phi, +\infty)/{\sim_\gamma}= (\cC_{\eta}^+, \phi-\frac{2}{\gamma}\log\mathfrak l, +\infty)/{\sim_\gamma}$. 
Moreover,   
\begin{equation}\label{eq:coordinate}
(\phi^0, g) -\frac{2}{\gamma} \log \mathfrak l= (\hat\phi \circ f^{-1} + Q \log |(f^{-1})'|, g)=  (\hat\phi, |f'|^2 g\circ f) + (Q \log |(f^{-1})'|, g).
\end{equation}
By~\eqref{eq:C0-def} and the definition of $S$, we have  $|f'|^2 g \circ f \in S$. Then by~\eqref{eq:C0-def} and~\eqref{eq:coordinate}, we have
\[
\big|(\phi^0, g) -\frac{2}{\gamma} \log \mathfrak l\big|  \le |(\hat\phi, |f'|^2 g\circ f)| + | (Q \log |(f^{-1})'|, g)|
\leq  \sup_{\xi \in S} |(\hat \phi, \xi)| + Q \log 2. 
\] 
	
Recall from the proof of Lemma~\ref{lem:start} that the law of  $(\cC, \phi, \eta, +\infty,-\infty)/{\sim_\gamma}$ is the restriction of  $\cM_2^\sph(W) \otimes \SLE_\kappa^{\mathrm{sep}, \alpha}$ to  the event that the  total quantum area is larger than $1$. 
By Proposition~\ref{prop-KW-weld}, conditioning on
$(\cC_\eta^-, \phi)/{\sim_\gamma}$, the conditional law of $(\cC_{\eta}^+, \phi, +\infty)/{\sim_\gamma}$ is $\cM_1^\disk(\alpha; \mathfrak l)^\#$, hence the conditional law of $(\cC_+, \hat\phi, +\infty)/{\sim_\gamma}$ is $\cM_1^\disk(\alpha; 1)^\#$. 
Therefore, by~\eqref{eq:err-def} and~\eqref{eq:coordinate},  conditioning on $(\cC_\eta^-, \phi)/{\sim_\gamma}$,
the conditional probability of the event $\left|(\phi^0, g) -\frac{2}{\gamma} \log \mathfrak l\right|< -\frac{2\zeta}{\gamma}\log \mathfrak l$ is at least $1-\mathrm{err}(\mathfrak l)$. Since $g$ is supported on  $\{\Re z \in (C-1, C )\}$, if $\left|(\phi^0, g) -\frac{2}{\gamma} \log \mathfrak l\right| <  -\frac{2\zeta}{\gamma}\log \mathfrak l$, 
 then there exists $s \in [C-1,C]$ such that the average of $\phi^0$ on $[s, s+2\pi i]$ lies in $((1+\zeta) \frac2\gamma \log \mathfrak l, (1-\zeta) \frac2\gamma \log \mathfrak l)$. This gives $X_{\mathbf t + s} \in ((1+\zeta) \frac2\gamma \log \mathfrak l, (1-\zeta) \frac2\gamma \log \mathfrak l)$ hence $\mathbf t + s \in (\sigma_{{\mathfrak l}^{1-\zeta}}, \tau_{{\mathfrak l}^{1+\zeta}})$  for some $s\in [C-1,C]$. Therefore $\mathbf t \in (\sigma_{{\mathfrak l}^{1-\zeta}} - C, \tau_{{\mathfrak l}^{1+\zeta}})$, which gives Lemma~\ref{lem:average} with our choice of $C$ and $\mathrm{err}$.
\end{proof}

 \subsubsection{Proof of Proposition~\ref{prop:product}}\label{subsub:event-ob} 	

We refer to Figure~\ref{fig-2pt} and the paragraph above Section~\ref{subsub:field-average}   for an illustration of our proof ideas.
We will write $o_{\delta,\eps}(1)$ as  an error term  satisfying  $\lim_{\delta\to 0} \limsup_{\eps\to 0} |o_{\delta,\eps}(1)|=0$ which can change from place to place.
We first prove the lower bound for $(M \times dt \times \cL_\kappa^\alpha)[E'_{\delta, \eps}]$.

\begin{lemma}\label{prop-KW-left-E-lower}
We have \((M \times  dt \times \cL_\kappa^\alpha)[E'_{\delta, \eps}] \ge (1+o_{\delta,\eps}(1)) |M| |\cL_\kappa^\alpha| \frac{2\log \eps^{-1}}{\gamma(Q-\alpha)}. \)	
\end{lemma}

We will prove Lemma~\ref{prop-KW-left-E-lower}   
by a coupling of the probability measure $M^\#=M/|M|$ and cylindrical GFF measure  $P_\cC$.
Recall that for a sample $h$ from $P_\cC$, it can be written as 
$h=h_1+ h_2$, where $h_1$ is constant on vertical  circles  and  $h_2$ is the  lateral component that has mean zero on all such
circles. 
\begin{lemma}\label{lem:coupling kw}
There exists a coupling of $h$ sampled from $P_\cC$ and $\phi$ sampled from $M^\#$ 
such that $h_2$ is independent of $(X_s)_{s \geq 0}$, and moreover, $\sup_{\Re z > 1} | {\mathfrak h}(z)| < \infty$ where  \(\mathfrak h(z) = \phi(z)-  h_2(z) -  X_{\Re z}\).

\end{lemma}
\begin{proof}
By \cite[Proposition 2.8]{ig4}, on $\cC_+$ the field  $h$ can be written as $h_0 + \tilde {\mathfrak h}$ where $h_0$ is a zero boundary GFF on $\cC_+$ and $\tilde {\mathfrak h}$ is an independent harmonic function on $\cC_+$ such that $\sup_{\Re z > 1} |  \tilde {\mathfrak h}(z)| < \infty$. 
Similarly, by Lemma~\ref{lem-markov},  $\phi(z) +  (Q-\alpha) \Re z$ on $\cC_+$ can be written as $ \phi_0 + \hat{\mathfrak h} $ with the same  properties. 
Coupling $h$ and $\phi$ such that $h_0 = \phi_0$ and $ \tilde {\mathfrak h}$ is independent of $\phi$,  we are done.
\end{proof}

\begin{proof}[Proof of Lemma~\ref{prop-KW-left-E-lower}]
Let $\P$ be the probability measure corresponding to the law of $(\phi, h)$ as coupled in Lemma~\ref{lem:coupling kw}.
Let  $(\phi, h,\mathbf t, \eta^0)$ be a sample from $\P\times dt \times \cL_\kappa^\alpha$.
Then the law of $(\phi, \mathbf t, \eta^0)$ is 
$M^\# \times dt \times \cL_\kappa^\alpha$, where $M^\#=M/|M|$.  Let $h_2$ and $\mathfrak h$ be defined 
as in Lemma~\ref{lem:coupling kw}  so that  \(\phi= h_2+ X_{\Re \cdot} + \mathfrak h\).
Fix  $\zeta \in(0,0.1)$ which will be sent to zero later.  
Let $I_{\delta,\eps}$ be the interval $((1+3\zeta)\frac2{\gamma(Q-\alpha)} \log \delta^{-1} + \delta^{-1}, (1-3\zeta) \frac2{\gamma(Q-\alpha)} \log \eps^{-1})$. 
Let 
	\[
 G_{\delta,\eps} : = \{I_{\delta,\eps} \subset ( \tau_{\delta^{1+2\zeta}} + \delta^{-1} , \sigma_{\eps^{1-2\zeta}}) \quad \textrm{and}\quad \sup_{\Re z > 1} | {\mathfrak h}(z)| < \zeta \cdot \frac2\gamma \log \delta^{-1}\},
	\]where the random times $\sigma_y,\tau_y$ are as in~\eqref{eq:BM-times}.
By Lemma~\ref{lem-BM-hitting} we have $\P [I_{\delta,\eps} \subset ( \tau_{\delta^{1+2\zeta}} + \delta^{-1} , \sigma_{\eps^{1-2\zeta}})]=1-o_{\delta,\eps}(1)$. 
Since $\sup_{\Re z > 1} | {\mathfrak h}(z)| < \infty$ almost surely, we have  $\P[G_{\delta,\eps}] = 1-o_{\delta,\eps}(1)$.

Fix $n>0$. Let $A_n =\{\inf_{z \in \eta^0} \Re z > -n\}$. Then $\cL^\alpha_\kappa[A_n]<\infty$. Let $\eta=\eta^0+\mathbf t$ and $ \ell_{h_2}(\eta) $
be the quantum length of $\eta$ with respect to $h_2$. Define the event 
\[
E'_{\delta, \eps}(n) := \{ A_n \textrm{ occurs},  \quad \mathbf t\in I_{\delta,\eps}, \quad \textrm{and}\quad \ell_{h_2}(\eta) \in (\eps^{\zeta}, \delta^{-\zeta})\}.
\]
Suppose $n<\delta^{-1} < \frac{2\zeta}{\gamma(Q-\alpha)} \log \eps^{-1}$ and the event $E'_{\delta, \eps}(n) \cap G_{\delta,\eps}$ occurs.
Since  $\mathbf t \in I_{\delta,\eps}$ and $G_{\delta,\eps} \cap A_n$ occurs,  we have $ \Re z \in (\tau_{\delta^{1+2\zeta}}  ,\sigma_{\eps^{1-2\zeta}} ) $ for each $z\in \eta$, hence $\cC_-:=\{z\in \cC: \Re z\le 0 \} \subset D^-_\eta$. 
 By the definition of $\phi$ we have 
 \begin{equation}\label{area inequality}
      \mu_\phi (\cC_\eta^-)\ge \mu_\phi (\cC_-)\geq1
 \end{equation}

Moreover,  from the bound on $\sup_{\Re z > 1} |   {\mathfrak h}(z)|$ we have $\frac{\gamma}{2}(X_{\Re z} + \mathfrak h(z)) \in ((1-\zeta)\log \eps, (1+\zeta)\log \delta)$ for each $z\in \eta$.  
Now  $\ell_{h_2}(\eta) \in (\eps^{\zeta}, \delta^{-\zeta})$ yields $\ell_\phi(\eta) \in (\eps, \delta)$. 
This gives 
\begin{equation}\label{inclusion of event}
  E'_{\delta, \eps}(n) \cap G_{\delta,\eps} \subset E'_{\delta, \eps}=\{ \mu_\phi (\cC_\eta^-)\ge 1,  \ell_\phi(\eta) \in (\eps, \delta) \}  
\end{equation}
 Therefore 
\begin{equation}\label{eq:lower1}
(\P\times dt \times \cL_\kappa^\alpha) [E'_{\delta, \eps}] \ge( \P\times dt \times \cL_\kappa^\alpha)[E'_{\delta, \eps}(n)] - ( \P\times dt \times \cL_\kappa^\alpha)[E'_{\delta, \eps}(n)\setminus G_{\delta,\eps}].
\end{equation}
We claim that 
\begin{equation}\label{eq:lower2}
(\P\times dt \times \cL_\kappa^\alpha)[E'_{\delta, \eps}(n)] = (1-o_{\delta,\eps}(1)) |I_{\delta,\eps}| \cL^\alpha_\kappa[A_n]  \textrm{ and }
(\P\times dt \times \cL_\kappa^\alpha)[E'_{\delta, \eps}(n)\setminus G_{\delta,\eps}]=o_{\delta,\eps}(1) |I_{\delta,\eps}|. 
\end{equation}
Since  the law of $h_2$ is translation invariant, namely $h_2\stackrel d= h_2(\cdot - t)$ for each $t \in \R$, we have 
\begin{align*}
&(\P\times dt \times \cL_\kappa^\alpha)[E'_{\delta, \eps}(n)] = |I_{\delta,\eps}| (\P\times \cL^\alpha_\kappa)[\ell_{h_2}(\eta^0) \in (\eps^{\zeta}, \delta^{-\zeta}), 
A_n]\\
&=  |I_{\delta,\eps}| (\P\times \cL^\alpha_\kappa)[A_n]- |I_{\delta,\eps}| (\P\times \cL^\alpha_\kappa)[\ell_{h_2}(\eta^0) \notin (\eps^{\zeta}, \delta^{-\zeta}),
A_n]=(1-o_{\delta,\eps}(1)) |I_{\delta,\eps}| \cL^\alpha_\kappa[A_n].
\end{align*}
In the last line we used  $(\P\times \cL^\alpha_\kappa)[A_n]= \cL^\alpha_\kappa[A_n]<\infty$ and $ (\P\times \cL^\alpha_\kappa)[\ell_{h_2}(\eta^0) \notin (\eps^{\zeta}, \delta^{-\zeta}) |A_n] = o_{\delta,\eps} (1)$, the latter of which holds because $\ell_{h_2}(\eta^0)<\infty$ a.e.\  in $\P\times \cL^\alpha_\kappa$.
This gives  the first identity in~\eqref{eq:lower2}.

The second identity in~\eqref{eq:lower2} is proved by as follows:
\begin{align*}
&(\P\times dt \times \cL_\kappa^\alpha)[E'_{\delta, \eps}(n)\setminus G_{\delta,\eps}] \le (\P\times dt \times \cL_\kappa^\alpha)[A_n\setminus G_{\delta,\eps} \textrm{ and } \mathbf{t}\in I_{\delta, \eps}]\\
&= \P[G_{\delta,\eps} \textrm{ does not occur} ] \times |I_{\delta,\eps}|  \times \cL^\alpha_\kappa[A_n]= o_{\delta,\eps}(1) |I_{\delta,\eps}|  \cL^\alpha_\kappa[A_n].
\end{align*}

By~\eqref{eq:lower1} and~\eqref{eq:lower2},  for a fixed $n$,
	we have \(	(\P \times dt \times \cL_\kappa^\alpha)[E'_{\delta, \eps}] \ge  (1-o_{\delta,\eps}(1))|I_{\delta,\eps}|| \cL^\alpha_\kappa[A_n] \) hence 
	\[	\frac1{|M|}(M \times dt \times \cL_\kappa^\alpha)[E'_{\delta, \eps}] \geq (1-o_{\delta,\eps}(1)) |I_{\delta,\eps}|  \cL^\alpha_\kappa[A_n] \geq (1-o_{\delta,\eps}(1)) \cL^\alpha_\kappa[A_n]  (1-4\zeta) \frac{2\log\eps^{-1}}{\gamma(Q-\alpha)}  .\]
	Sending $\eps\to 0$, $\delta\to 0$, and $\zeta \to 0$ in order, we see that	
	\[  
	\lim_{\delta\to0} \liminf_{\eps\to 0} \frac{( M \times dt \times \cL_\kappa^\alpha)[E'_{\delta, \eps}] \times  \gamma(Q-\alpha)}{2\log\eps^{-1}}  	
	\geq |M| \cL^\alpha_\kappa[A_n].\]
	Since $\cL^\alpha_\kappa[A_n]\rta |\cL^\alpha_\kappa|$, sending $n\to \infty$ we conclude the proof.
\end{proof}

It remains to prove the upper  bound for $(M \times  dt \times \cL_\kappa^\alpha)[E'_{\delta, \eps}]$.
\begin{lemma}\label{lem-KW-left-F|E}
We have \((M \times  dt \times \cL_\kappa^\alpha)[E'_{\delta, \eps}] \leq (1+o_{\delta,\eps}(1)) |M| |\cL_\kappa^\alpha| \frac{2\log \eps^{-1}}{\gamma(Q-\alpha)}. \)
\end{lemma}
\begin{proof}
Recall  $\sigma$ and $\tau$ defined in \eqref{eq:BM-times}. For $C>0$, consider the event
\eqb
F_{x,y,C} = \{(\phi, \mathbf t, \eta^0) \: : \:  \mathbf t \in (\sigma_x - C, \tau_y) \}.
\eqe
We will prove Lemma~\ref{lem-KW-left-F|E} by first bounding $(M \times dt \times \cL_\kappa^\alpha)[F_{x, y, C}]$  and then  comparing $E'_{\delta, \eps}$ to  $F_{\delta^{1-\zeta} ,\eps^{1+\zeta}, C}$. 
More precisely, we will prove two estimates. First, for each $x>0$, we have
\begin{equation}\label{eq:F}
(M \times dt \times \cL_\kappa^\alpha)[F_{x, y, C}] \le  (1+o_y(1)) |M| |\cL_\kappa^\alpha|  ( \frac{2 \log y^{-1}}{\gamma(Q-\alpha)}+C)\quad \textrm{as }y \to 0,
\end{equation}
where $o_y(1)\to 0$ uniformly in $C$.
Moreover, for a fixed $\zeta \in (0,1)$ there exists some $C>0$ such that
\eqb\label{eq-F|E}
(M \times dt \times \cL_\kappa^\alpha)[F_{\delta^{1-\zeta}, \eps^{1+\zeta},C} \mid E'_{\delta, \eps}] = 1-o_{\delta,\eps}(1).
\eqe
Given~\eqref{eq:F} and~\eqref{eq-F|E}, we see that for any fixed $\zeta>0$ there exists some $C>0$ such that 
\[(M \times  dt \times \cL_\kappa^\alpha)[E'_{\delta, \eps}] \leq (1+o_{\delta,\eps}(1))(M \times dt \times \cL_\kappa^\alpha)[F_{\delta^{1-\zeta}, \eps^{1+\zeta},C}] 
\le (1+o_{\delta,\eps}(1)) |M| |\cL_\kappa^\alpha| \frac{(1+\zeta)2 \log \eps^{-1}}{\gamma(Q-\alpha)},\]
Sending $\zeta \to 0$ we  will get Lemma~\ref{lem-KW-left-F|E}.

We first prove~\eqref{eq:F}. Note that 
\((M \times dt \times \cL_\kappa^\alpha)[F_{x, y, C}]=M^\#[\tau_y - \sigma_x + C]  |M| |\cL^\alpha_\kappa| \) from the definition of $F_{x,y,C} $.
Hence~\eqref{eq:F} is equivalent to  $M^\#[\tau_y - \sigma_x ] \le (1+o_y(1)) \frac{2\log y^{-1}}{\gamma(Q-\alpha)}$.
By Lemma~\ref{lem-markov},  the process $X_s$ evolves as Brownian motion with drift $-(Q-\alpha)$. 
Let $T=\inf\{s\ge 0: B^x_s-(Q-\alpha)s=\frac2\gamma \log y\}$ where $(B^x_s)_{s\ge 0}$ is a  Brownian motion starting from $\frac2\gamma \log x$.
Then $M^\#[\tau_y - \sigma_x] \leq \E[T]$. Since $\E[T] = (1+o_y(1)) \frac{2\log y^{-1}}{\gamma (Q-\alpha)}$, we get \eqref{eq:F}.

Since the event $E'_{\delta, \eps}$ is determined by $(\cC_\eta^-, \phi)/{\sim_\gamma}$, and $\ell_\phi(\eta)\in (\eps,\delta)$ on $E'_{\delta, \eps}$,  Lemma~\ref{lem:average} yields~\eqref{eq-F|E}
with the constant $C$ from Lemma~\ref{lem:average}. This concludes the proof.
\end{proof}

\subsection{Area statistics of generalized quantum disks}\label{subsec:gqd}
We now transition from the $\kappa  < 4$ regime to the $\kappa' \in (4,8)$ setting, with the goal of proving Proposition~\ref{prop-KW-ns} in Section~\ref{section nonsimple kw}. 
We review several results on the law of forested disks, mainly from \cite{hl-lqg-cle}. In order to use the results there, we first need to clarify the equivalence between our $\cM^{\rm f.d.}$ as in Definition \ref{def:forested-disk} and the definitions of generalized quantum disks in \cite{hl-lqg-cle}.

We first recall the notion of looptree, which is originally defined in \cite{wedges}. 
Suppose $\nu\in (1,2)$. Let $(\zeta_t)_{t\geq 0}$ be a stable L\'evy process with exponent $\nu$ with no negative jumps started at $\zeta_0 = 0$. This only specifies the process up to a multiplicative constant, which we fix by requiring $\E[e^{-\lambda \zeta_t}]=e^{-t\lambda^\nu}$ for $\lambda, t>0$. Equivalently, $(\zeta_t)_{t\geq 0}$ is the $\nu$-stable L\'evy process whose jump measure is $\frac{1}{\Gamma(-\nu)}h^{-1-\nu}1_{h>0}dh$. One can make sense of the \emph{L\'evy excursion conditioned to have duration $\ell$} by conditioning on the event that $\zeta_t$ first hits $-\eps$ at time $\ell$, then sending $\eps \to 0$. See e.g. \cite[Section 1]{duquesne-legall} for further discussion.
\begin{definition}[\cite{msw-non-simple}]\label{def:looptree}
Let $e:[0,\ell]\to\R$ be a sample from the $\nu$-stable L\'evy excursion with no negative jumps conditioned on having duration $\ell$.
On the graph $\{(t,e(t)):t\in[0,\ell]\}$ we define an equivalence relation $s\sim t$ iff $e(s)=e(t)$ and the horizontal segment connecting $(s,e(s))$ and $(t,e(t))$ is below the graph of $e|_{s,t}$. We also set $(t,e(t))$ and $(t,e(t-))$ to be equivalent if $t$ is a jump time of $e$; this  naturally identifies a root for each loop. We call the quotient $\mathcal{T}_\ell$ of the graph $\{(t,e(t)):t\in[0,\ell]\}$ under the above equivalence relation the \emph{$\nu$-stable looptree} with length $\ell$.
\end{definition}

Note that the looptree $\mathcal{T}_\ell$ defined above is naturally rooted at the equivalence class of the origin $(0,0)$. We define the {\it generalized boundary length measure} on the looptree to be the pushforward of the Lebesgue measure on $[0,\ell]$.

\begin{figure}[ht!]
	\begin{center}
		\includegraphics[scale=0.8]{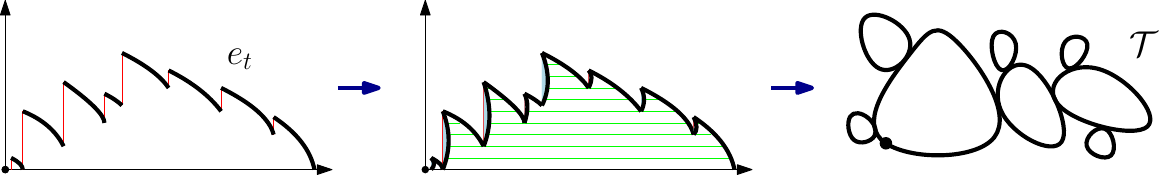}%
	\end{center}
	\caption{From a L\'evy excursion to a looptree. The root of $\mathcal{T}_\ell$ is marked as the dot.}\label{fig:looptree}
\end{figure}

Now we review the definitions of generalized quantum disks in \cite{hl-lqg-cle}. According to \cite[(1.8)]{hl-lqg-cle}, the authors define the generalized quantum disk with one $\alpha$-bulk marked point and one boundary marked point as
\begin{equation}\label{eq:def-fd101-hl}
\wt \cM_{1,0,1}^{\rm f.d.}(\alpha ; \ell)=\sum_{t:\Delta e_t\neq0} \frac{1}{|\QD_{0,1}(|\Delta e_t|)|}\cM_{1,1}^{\rm disk}(\alpha;|\Delta e_t|)\prod_{s\neq t:\Delta e_s\neq0} \QD_{0,1}^\#(|\Delta e_s|) \,    d\cT_\ell,
\end{equation}
where $ d\cT_\ell$ denotes for the probability measure of $\frac{\kappa'}{4}$-stable loop tree with length $\ell$. Here, the summation is over the set of jumps of the L\'evy excursion $e$ defining $\cT_\ell$. A sample from the RHS is a looptree $\cT_\ell$, a quantum disk with a $\alpha$-bulk insertion and a marked boundary point, and a collection of quantum disks each with a marked boundary point. Identifying the boundaries of the quantum surfaces with the corresponding loops according to boundary length gives a generalized quantum surface with a marked bulk point and a marked boundary point (given by the root of $\cT_\ell$); we denote its law by $\wt \cM_{1,0,1}^\mathrm{f.d.}(\alpha; \ell)$. 
We denote the law of a sample from $\ell^{-1}\wt \cM_{1,0,1}^\mathrm{f.d.}(\alpha; \ell)$ with boundary point forgotten by $\wt \cM_{1,0,0}^\mathrm{f.d.}(\alpha;\ell)$. 
\begin{lemma}\label{rem-GQD-alpha}
 We have $\wt\cM_{1,0,1}^{\rm f.d.}(\gamma;\ell)\ell^{-1-\frac{4}{\kappa'}}=R_\gamma'^{-1}\frac{2\pi (Q-\gamma)^2}{\gamma}\GQD_{1,1}(\ell)$.
\end{lemma}
\begin{proof}
As explained above Proposition \ref{prop:area-law-gqd}, our $\GQD_{0,1}(\ell)^\#$ is same as in \cite{hl-lqg-cle}, i.e. the probability measure $\prod_{s:\Delta e_s\neq0} \QD_{0,1}^\#(|\Delta e_s|) \,    d\cT_\ell$. Take the total mass into consideration, by Lemma \ref{length gqd} we deduce $\GQD_{0,1}(\ell)=R_\gamma'\ell^{-1-\frac{4}{\kappa'}}\cdot\prod_{s:\Delta e_s\neq0} \QD_{0,1}^\#(|\Delta e_s|) \,    d\cT_\ell$. Then on both sides we weigh the total quantum area and sample one point according to the probability measure proportional to the quantum area measure; the LHS is just $\GQD_{1,1}(\ell)$ and the RHS gives
\begin{equation}\label{eq:weight-gqd}
R_\gamma'\ell^{-1-\frac{4}{\kappa'}}\sum_{t:\Delta e_t\neq0} \frac{1}{|\QD_{0,1}(|\Delta e_t|)|}\QD_{1,0}(|\Delta e_t|)\prod_{s\neq t:\Delta e_s\neq0} \QD_{0,1}^\#(|\Delta e_s|) \,    d\cT_\ell.
\end{equation}
Since by Theorem \ref{thm:def-QD} one has $\QD_{1,0}(\ell) = \frac\gamma{2\pi (Q-\gamma)^2}\cM_{1,0}^\disk(\gamma;\ell )$, according to \eqref{eq:def-fd101-hl}, \eqref{eq:weight-gqd} above is just equal to $R_\gamma'\ell^{-1-\frac{4}{\kappa'}}\frac\gamma{2\pi (Q-\gamma)^2}\wt\cM_{1,0,1}^{\rm f.d.}(\gamma;\ell)$. 
\end{proof}
According to the above lemma, we can conclude
\begin{proposition}\label{prop:101-hl}
The definitions of $\wt \cM_{1,0,1}^{\rm f.d.}(\alpha ; \ell)$ and $\wt \cM_{1,0,0}^{\rm f.d.}(\alpha ; \ell)$ above are consistent with $\cM_{1,0,1}^{\rm f.d.}(\alpha;\ell)$ and $\cM_{1,0,0}^{\rm f.d.}(\alpha;\ell)$ in Definition \ref{def:forested-disk}, namely, we have
\begin{equation}\label{eq:fd101-hl}
\wt \cM_{1,0,1}^{\rm f.d.}(\alpha ; \ell)\ell^{-1-\frac{4}{\kappa'}}=C\cM_{1,0,1}^{\rm f.d.}(\alpha;\ell), \quad \wt \cM_{1,0,0}^{\rm f.d.}(\alpha ; \ell)\ell^{-2-\frac{4}{\kappa'}}=C\cM_{1,0,0}^{\rm f.d.}(\alpha;\ell),
\end{equation}
where $C$ only depends on $\gamma$.
\end{proposition}
\begin{proof}
The proof is by using a reweighting argument on Lemma \ref{rem-GQD-alpha}. That is, according to Lemma \ref{rem-GQD-alpha} and Proposition \ref{prop:101=gqd11}, we have
\begin{equation}\label{eq:wtcM}
\wt\cM_{1,0,1}^\mathrm{f.d.}(\gamma;\ell)\ell^{-1-\frac{4}{\kappa'}} = C \cM_{1,0,1}^\mathrm{f.d.}(\gamma;\ell)
\end{equation}
with a $\gamma$-dependent constant $C$. On both sides of \eqref{eq:wtcM}, if we reweight the bulk marked point of weight $\gamma$ to be of weight $\alpha$, then the LHS turns out to be $ \wt \cM_{1,0,1}^{\rm f.d.}(\alpha ; \ell)\ell^{-1-\frac{4}{\kappa'}}$, while the RHS to be exactly $C\cM_{1,0,1}^{\rm f.d.}(\alpha;\ell)$. Then the first claim follows; by reweighting $\ell^{-1}$ and forgetting the boundary marked point the second claim follows straightforwardly.
\end{proof}

According to the equivalence in Proposition \ref{prop:101-hl}, we can compute the total mass of $\mathcal{M}_{1,0,0}^{\rm f.d.}(\alpha;\ell)$ based on \cite{hl-lqg-cle}.

\begin{prop}\label{total mass gqd alpha}
For $\alpha>\frac{\gamma}{2}$, we have $|\mathcal{M}_{1,0,0}^{\rm f.d.}(\alpha;\ell)|=D(\alpha)|\mathcal{M}_{1,0,0}^{\rm f.d.}(\gamma;1)|\ell^{\frac{\gamma}{2}(\alpha-Q)-1}$, where

$$D(\alpha):=\frac{|\mathcal{M}_{1,0,0}^{\rm f.d.}(\alpha;1)|}{|\mathcal{M}_{1,0,0}^{\rm f.d.}(\gamma;1)|}=(2\pi)^{2-2\alpha/\gamma}2^{2-\alpha Q+\alpha^2/2}\frac{\Gamma\left(\frac{2\alpha}{\gamma}-\frac{4}{\gamma^2}\right)}{\Gamma\left(2-\frac{4}{\gamma^2}\right)}\Gamma\left(1-\frac{\gamma^2}{4}\right)^{2\alpha/\gamma-2}.$$
\end{prop}
\begin{proof}
According to Proposition \ref{prop:101-hl}, our definition of $\mathcal{M}_{1,0,0}^{\rm f.d.}(\alpha;1)$ only differs from \cite{hl-lqg-cle} by a $\gamma$-dependent constant, i.e. we have $\frac{|\mathcal{M}_{1,0,0}^{\rm f.d.}(\alpha;1)|}{|\mathcal{M}_{1,0,0}^{\rm f.d.}(\gamma;1)|}=\frac{|\wt \cM_{1,0,0}^{\rm f.d.}(\alpha;1)|}{|\wt \cM_{1,0,0}^{\rm f.d.}(\gamma;1)|}$. In the following
we will prove \begin{equation}\label{eq:fd-alpha-mass-0}
|\wt \cM_{1,0,0}^{\rm f.d.}(\alpha;1)|=\frac{2}{\gamma}2^{-\alpha^2/2}\bar U(\alpha)\frac{\Gamma\left(1-\frac{\gamma^2}{4}\right)\Gamma\left(\frac{2\alpha}{\gamma}-\frac{4}{\gamma^2}\right)}{\Gamma\left(-\frac{4}{\gamma^2}\right)\Gamma\left(\frac{\gamma\alpha}{2}-\frac{\gamma^2}{4}\right)}\frac{1}{\bar R(\gamma;1,1)}.
\end{equation}
Given this and the formulas for $\Bar{U}(\alpha)$ in Proposition~\ref{prop-remy-U} and $\Bar{R}(\gamma;1,1)$ in Proposition~\ref{prop-QD}, the result follows.
According to \eqref{eq:def-fd101-hl}, the total mass of $\wt\cM_{1,0,0}^{\rm f.d.}(\alpha;1)$ equals
\begin{equation*}
|\wt \cM_{1,0,0}^{\rm f.d.}(\alpha;1)|=\E\left[\sum_{(t,\Delta e_t)}\frac{1}{|\QD_{0,1}(\Delta e_t)|}\big|\cM_{1,1}^{\rm disk}(\alpha;\Delta e_t)\big|\right]
\end{equation*}
where $\E$ denotes expectation with respect to the probability measure of the $\frac{4}{\kappa'}$-stable L\'evy excursion with living time $1$. 
According to the expressions of $|\QD_{0,1}(\Delta e_t)|$ and $|\cM_{1,1}^{\rm disk}(\alpha;\Delta e_t)\big|$ in Propositions \ref{prop-QD} and \ref{prop-remy-U}, we deduce
\begin{equation}\label{eq:fd-alpha-mass}
|\wt \cM_{1,0,0}^{\rm f.d.}(\alpha;1)|=\frac{2}{\gamma}2^{-\alpha^2/2}\bar U(\alpha)\E\left[\sum_{t}|\Delta e_t|^{2\alpha/\gamma}\right]\frac{1}{\bar R(\gamma;1,1)}
\end{equation}
The expectation $\E\left[\sum_{t}|\Delta e_t|^{2\alpha/\gamma}\right]$ can be computed by applying \cite[Lemma 4.2]{hl-lqg-cle}. Namely, choose $G(x)=0$ in Proposition \cite[Lemma 4.2]{hl-lqg-cle}, then for $\lambda>0$ and $F:[0,\infty)\to[0,\infty)$ twice continuously differentable with $F(0)=F'(0)=0$, \cite[Formula (4.3)]{hl-lqg-cle} gives
\begin{equation}\label{eq:hl22-excursion}
		-\frac{d}{d\lambda}(e^{-\rho(\lambda)})\frac{1}{\Gamma(-\nu)}\int_0^\infty \frac{dh}{h^{1+\nu}} e^{-\rho(\lambda)h}F(h)= \frac{e^{-\rho(\lambda)}}{\nu \Gamma(1-1/\nu)} \int_0^\infty \frac{d\ell}{\ell^{1+1/\nu}}\E\left(\sum_{t\le1}F(\ell^{1/\nu}\Delta e_t) e^{-\lambda\ell} \right)\;
	\end{equation}
where $\rho(\lambda)=\lambda^{\1/\nu}$. Setting $F(x)=x^{2\alpha/\gamma}$ in \eqref{eq:hl22-excursion} and simplifying gives
\begin{equation*}
e^{-\rho(\lambda)}\frac{1}{\nu}\lambda^{\1/\nu-1}\frac{1}{\Gamma(-\nu)}\Gamma\left(-\nu+\frac{2\alpha}{\gamma}\right)\rho(\lambda)^{\nu-\frac{2\alpha}{\gamma}}=\frac{e^{-\rho(\lambda)}}{\nu\Gamma(1-1/\nu)}\lambda^{-\frac{2\alpha}{\gamma\nu}+\frac{1}{\nu}}\Gamma\left(\frac{2\alpha}{\gamma\nu}-\frac{1}{\nu}\right)\E\left[\sum_{t}|\Delta e_t|^{2\alpha/\gamma}\right],
\end{equation*}
i.e. we have
\begin{equation}\label{eq:jump-moment}
\E\left[\sum_{t}|\Delta e_t|^{2\alpha/\gamma}\right]=\frac{\Gamma(1-1/\nu)\Gamma\left(-\nu+\frac{2\alpha}{\gamma}\right)}{\Gamma(-\nu)\Gamma\left(\frac{2\alpha}{\gamma\nu}-\frac{1}{\nu}\right)}=\frac{\Gamma\left(1-\frac{\gamma^2}{4}\right)\Gamma\left(\frac{2\alpha}{\gamma}-\frac{4}{\gamma^2}\right)}{\Gamma\left(-\frac{4}{\gamma^2}\right)\Gamma\left(\frac{\gamma\alpha}{2}-\frac{\gamma^2}{4}\right)}.
\end{equation}
Plugging \eqref{eq:jump-moment} into \eqref{eq:fd-alpha-mass} gives \eqref{eq:fd-alpha-mass-0}, which gives the result.
\end{proof}
Finally, we record the Laplace transform of the quantum area of $\mathcal{M}_{1,0,0}^{\rm f.d.}(\alpha;\ell)^\#$ in \cite{hl-lqg-cle}.
\begin{proposition}[{\cite[Theorem 1.8]{hl-lqg-cle}}]\label{generic general lap}
Let $M':=2\left(\frac{\mu}{4\sin\frac{\pi\gamma^2}{4}}\right)^{\frac{\kappa'}{8}}$. For $\alpha\in (\frac{\gamma}{2},Q)$, we have
    \begin{align}
         &\mathcal{M}_{1,0,0}^{\rm f.d.}(\alpha;\ell)^{\#}[e^{-\mu A}]=\frac{2}{\Gamma(\frac{\gamma}{2}(Q-\alpha))}(\frac{M'\ell}{2})^{\frac{\gamma}{2}(Q-\alpha)}K_{\frac{\gamma}{2}(Q-\alpha)}(M'\ell)\\
         &\mathcal{M}_{1,0,0}^{\rm f.d.}(\alpha;\ell)^{\#}[Ae^{-\mu A}]=\frac{\kappa'}{2\mu\Gamma(\frac{\gamma}{2}(Q-\alpha))}(\frac{M'\ell}{2})^{\frac{\gamma}{2}(Q-\alpha)+1}K_{1-\frac{\gamma}{2}(Q-\alpha)}(M'\ell)
    \end{align}
    In particular, 
    \begin{equation*}
\GQD_{1,0}(\ell)^\#[e^{-\mu A}]=\bar K_{1-4/\kappa'}\left(2l\left(\frac{\mu}{4\sin\frac{\pi\gamma^2}{4}}\right)^{\kappa'/8}\right).
\end{equation*}
\end{proposition}

\subsection{The \texorpdfstring{$\kappa'\in (4,8)$}{g} case}\label{section nonsimple kw}
In this section, we prove Proposition~\ref{prop-KW-ns}. The proof of Proposition~\ref{prop-KW} in Section~\ref{sec:elect-thickness-simple} used Proposition~\ref{prop-KW-weld}, Lemma~\ref{lem-KW-right-E} and Proposition~\ref{prop-KW-left-E}. The analogous inputs in our setting are  Proposition~\ref{nonsimple alpha welding}, Lemma~\ref{lemma-kw-right-E ns}  and Proposition~\ref{prop-KW-left-E ns} below, and given these, the proof of Proposition~\ref{prop-KW-ns} is identical to that of Proposition~\ref{prop-KW}.
As usual, we assume $\kappa' \in (4,8)$ and $\gamma = 16/\sqrt{\kappa'}$.

Sample $\eta$ from $\cL_{\kappa'}$, we define the connected component that contains $+\infty$ as $\cD_{\eta^+}$, with $\eta^+$ as its boundary and we also define the connected component contains $-\infty$ as $\cD_{\eta^-}$, with $\eta^-$ as its boundary. Let $-\eta$ be the image of $\eta$ under reflection $z\mapsto -z$.  Under these conventions, we have 
$$ \vartheta(\eta) = -\log \CR(\exp( \eta^+),0) - \log \CR(\exp(-\eta^-),0)$$
 $$d\cL^\alpha_{\kappa'}(\cC)(\eta)=2^{2\lambda}e^{\lambda\vartheta(\eta)}{d\cL_{\kappa'}(\cC)}(\eta)$$

Then the goal of this section is to compute $ |\cL^\alpha_{\kappa'}(\cC)|$ for $\kappa'\in (4,8)$.

Recall for a sample $(\eta,{\bf t})$ from $\cL^\alpha_{\kappa'}\times dt$. $\SLE^{\rm sep,\alpha}_{\kappa'}$ is the law of the translated loop $\eta+{\bf t}$.  We have the following conformal welding result, we prove it in Section~\ref{nonsimple welding} later.
\begin{proposition}\label{nonsimple alpha welding} There is a constant $C = C(\gamma)$ such that for all $\alpha\in (\frac{\gamma}{2},Q)$ we have
    \begin{equation}\label{nonsimple alpha formula}
        C(Q-\alpha)^2\cM_2^{\rm sph}(\alpha)\otimes \SLE^{\rm sep,\alpha}_{\kappa'}=\int_0^\infty \ell\cdot \Weld (\cM_{1,0,0}^{\rm f.d.}(\alpha;\ell),\cM_{1,0,0}^{\rm f.d.}(\alpha;\ell))d\ell.
    \end{equation}
\end{proposition}

For $0<\eps<\delta$, let $E_{\delta,\eps}$ be the event that the generalized quantum surface containing the first marked point has quantum area at least $1$ and
the loop has generalized quantum length in $(\eps,\delta)$. As in the simple case, we will evaluate event $E_{\delta,\eps}$ on both sides of \eqref{nonsimple alpha formula} to deduce a formula for $|\mathcal{L}_{\kappa'}^\alpha|$. To evaluate $E_{\delta,\eps}$ on the right hand side of \ref{nonsimple alpha formula}, we need the following lemma. 
\begin{lemma}\label{alpha tail}
Writing $A$ to denote quantum area, we have the following tail asymptotics
\begin{align}
&\GQD_{1,0}(1)^\#[A>x]=\frac{\Gamma(4/\kappa')}{\Gamma(2-4/\kappa')\Gamma(2-\kappa'/4)}(4\sin\frac{\pi\gamma^2}{4})^{-\kappa'/4+1} x^{-\kappa'/4+1}(1+o(1))\\
&\mathcal{M}_{1,0,0}^{\rm f.d.}(\alpha ; 1)^{\#}\left[A>x\right]=\frac{\Gamma\left(\frac{\alpha\gamma}{2}-\frac{4}{\kappa'}\right)}{\Gamma\left(2-\frac{\alpha\gamma}{2}+\frac{4}{\kappa'}\right)\Gamma\left(\frac{2\alpha}{\gamma}-\frac{\kappa'}{4}\right)}\left(4\sin\frac{\pi\gamma^2}{4}\right)^{-1-\frac{\kappa'}{4}+\frac{2\alpha}{\gamma}}x^{-1-\frac{\kappa'}{4}+\frac{2\alpha}{\gamma}}(1+o(1)).
\end{align}
where $o(1)$ satisfies $\lim_{x\to0}o(1)=0$.
\end{lemma}
\begin{proof}
We only prove the first equation, since the argument for the second one is identical. According to Proposition \ref{prop:area-law-gqd}, we have $\GQD_{1,0}(1)^\#[e^{-\mu A}]=\bar K_{1-4/\kappa'}\left(M'\right)$ for $M'=2\left(\frac{\mu}{4\sin\frac{\pi\gamma^2}{4}}\right)^{\frac{\kappa'}{8}}$. Now we need to estimate its inverse Laplace transform. According to Post's inversion formula \cite{post1930generalized}, for a bounded continuous function $f:[0,\infty)\to\R$, its Laplace transform $F(s)=\int_0^\infty e^{-st}f(t)dt$ exists and is smooth for $s>0$, and we can uniquely recover $f$ as
\begin{equation}\label{eq:post}
\mathcal{L}^{-1}\{F\}(t)=\lim_{k\to\infty}\frac{(-1)^k}{k!}\left(\frac{k}{t}\right)^{k+1}F^{(k)}\left(\frac{k}{t}\right).
\end{equation}
Now we use \eqref{eq:post} to obtain the leading term of inverse Laplace transform of $\GQD_{1,0}(1)^\#[e^{-\mu A}]=\bar K_{1-4/\kappa'}\left(M'\right)$. Note that we have the absolutely convergent expansion (see e.g. \cite[(9.6.2), (9.6.10)]{abramowitz1948handbook}
\begin{equation*}
\bar K_\nu(z)=\Gamma(1-\nu)\left[\sum_{n=0}^\infty\frac{1}{n!\Gamma(n-\nu+1)}\left(\frac{z}{2}\right)^{2n}-\sum_{n=0}^\infty\frac{1}{n!\Gamma(n+\nu+1)}\left(\frac{z}{2}\right)^{2n+2\nu} \right],
\end{equation*}
hence we have
\begin{equation}\label{eq:expansion-gqd-area-law}
\GQD(1)^\#[e^{-\mu A}]=\Gamma\left(\frac{4}{\kappa'}\right)\left[\sum_{n=0}^\infty\frac{1}{n!\Gamma\left(n+\frac{4}{\kappa'}\right)}\left(\frac{\mu}{4\sin\frac{\pi\gamma^2}{4}}\right)^{n\cdot\frac{\kappa'}{4}}-\sum_{n=1}^\infty\frac{1}{(n-1)!\Gamma\left(n-\frac{4}{\kappa'}+1\right)}\left(\frac{\mu}{4\sin\frac{\pi\gamma^2}{4}}\right)^{n\cdot\frac{\kappa'}{4}-1}
\right],
\end{equation}
and we can evaluate its $k$-th derivative as (one can easily check the condition of term by term differentiation holds)
\begin{footnotesize}
\begin{equation}\label{eq:expansion-gqd-area-derivative}
\Gamma\left(\frac{4}{\kappa'}\right)\left[\sum_{n=1}^\infty\frac{\Gamma\left(n\frac{\kappa'}{4}+1\right)}{n!\Gamma\left(n+\frac{4}{\kappa'}\right)\Gamma\left(n\frac{\kappa'}{4}-k+1\right)}\frac{\mu^{n\cdot\frac{\kappa'}{4}-k}}{\left(4\sin\frac{\pi\gamma^2}{4}\right)^{n\cdot\frac{\kappa'}{4}}}-\sum_{n=1}^\infty\frac{\Gamma\left(n\frac{\kappa'}{4}\right)}{(n-1)!\Gamma\left(n-\frac{4}{\kappa'}+1\right)\Gamma\left(n\frac{\kappa'}{4}-k\right)}\frac{\mu^{n\cdot\frac{\kappa'}{4}-1-k}}{\left(4\sin\frac{\pi\gamma^2}{4}\right)^{n\cdot\frac{\kappa'}{4}-1}}
\right].
\end{equation}
\end{footnotesize}
We now take \eqref{eq:expansion-gqd-area-derivative} into the formula \eqref{eq:post}. Using the fact $\Gamma(z-n)\sim(-)^n\frac{\pi}{\sin\pi z}\frac{n^{z-1}}{\Gamma(n)}$ (see e.g. \cite[(6.1.46), (6.1.17)]{abramowitz1948handbook}), we find the limit in \eqref{eq:post} is given by
\begin{footnotesize}
\begin{equation}\label{eq:expansion-gqd-area-derivative-2}
-\Gamma\left(\frac{4}{\kappa'}\right)\left[\sum_{n=1}^\infty\frac{\Gamma\left(n\frac{\kappa'}{4}+1\right)}{n!\Gamma\left(n+\frac{4}{\kappa'}\right)}\frac{\sin\left(\pi\frac{n\kappa'}{4}\right)}{\pi}\frac{t^{-n\cdot\frac{\kappa'}{4}-1}}{\left(4\sin\frac{\pi\gamma^2}{4}\right)^{n\cdot\frac{\kappa'}{4}}}+\sum_{n=1}^\infty\frac{\Gamma\left(n\frac{\kappa'}{4}\right)}{(n-1)!\Gamma\left(n-\frac{4}{\kappa'}+1\right)}\frac{\sin\left(\pi\frac{n\kappa'}{4}\right)}{\pi}\frac{t^{-n\frac{\kappa'}{4}}}{\left(4\sin\frac{\pi\gamma^2}{4}\right)^{n\cdot\frac{\kappa'}{4}-1}}
\right].
\end{equation}
\end{footnotesize}
Note that the $n=1$ term in the second summation above gives the leading order of \eqref{eq:expansion-gqd-area-derivative-2} as $t\to\infty$. Therefore the inversion of $\GQD_{0,1}(1)^\#[e^{-\mu A}]$ has the asymptotic:
\begin{equation}\label{eq:inv-asym}
-\frac{\Gamma(4/\kappa')}{\pi\Gamma(2-4/\kappa')}(4\sin\frac{\pi\gamma^2}{4})^{-\kappa'/4+1} \sin\frac{\pi\kappa'}{4}\Gamma(\kappa'/4)t^{-\kappa'/4}+O(t^{-1-\kappa'/4}).
\end{equation}
Finally, the tail probability follows just by the integration of \eqref{eq:inv-asym} on $[x,\infty)$, which leads to
\begin{align*}
\GQD_{1,0}(1)^\#[A>x]&=\frac{\Gamma(4/\kappa')}{\pi\Gamma(2-4/\kappa')}(4\sin\frac{\pi\gamma^2}{4})^{-\kappa'/4+1} \Gamma(\kappa'/4)\frac{1}{-\kappa'/4+1}\sin\frac{\pi\kappa'}{4}x^{-\kappa'/4+1}+O(x^{-\kappa'/4})\\
&=\frac{\Gamma(4/\kappa')}{\Gamma(2-4/\kappa')\Gamma(2-\kappa'/4)}(4\sin\frac{\pi\gamma^2}{4})^{-\kappa'/4+1} x^{-\kappa'/4+1}+O(x^{-\kappa'/4}),
\end{align*}
which gives the result.
\end{proof}
We  now evaluate the  measure of the event $E_{\delta,\eps}$ with respect to the right hand side of \eqref{nonsimple alpha formula}:
\begin{lemma}\label{lemma-kw-right-E ns}
    With $D(\alpha)$ the constant  in Proposition \ref{total mass gqd alpha}, we have
\begin{align*}
    &(\cM_2^{\rm sph}(\alpha) \otimes \SLE_{\kappa'}^{\mathrm{sep},\alpha})[E_{\delta, \eps}]\\
    =&C'\times\frac{D(\alpha)^2}{(Q-\alpha)^2}\frac{\Gamma\left(\frac{\alpha\gamma}{2}-\frac{4}{\kappa'}\right)}{\Gamma\left(2-\frac{\alpha\gamma}{2}+\frac{4}{\kappa'}\right)\Gamma\left(\frac{2\alpha}{\gamma}-\frac{\kappa'}{4}\right)}\left(4\sin\frac{\pi\gamma^2}{4}\right)^{-1-\frac{\kappa'}{4}+\frac{2\alpha}{\gamma}}(1+o_{\delta,\eps}(1))\log \eps^{-1}
\end{align*}
where $C'$ is a constant only depending on $\gamma$.
\end{lemma}
\begin{proof}
Similar to the computation in Lemma \ref{lem-KW-right-E}, we have 
\begin{align*}
&(\cM_2^\sph(W) \otimes \SLE_{\kappa'}^{\mathrm{sep},\alpha})[E_{\delta, \eps}]=\frac{1}{C(Q-\alpha)^2} \int_\delta^\varepsilon \ell |\mathcal{M}_{1,0,0}^{\rm f.d.}(\alpha;\ell)|^2\mathcal{M}_{1,0,0}^{\rm f.d.}(\alpha ; 1)^{\#}\left[A>\ell^{-8/\kappa'}\right] d\ell\\
=&C'\times\frac{D(\alpha)^2}{(Q-\alpha)^2}\frac{\Gamma\left(\frac{\alpha\gamma}{2}-\frac{4}{\kappa'}\right)}{\Gamma\left(2-\frac{\alpha\gamma}{2}+\frac{4}{\kappa'}\right)\Gamma\left(\frac{2\alpha}{\gamma}-\frac{\kappa'}{4}\right)}\left(4\sin\frac{\pi\gamma^2}{4}\right)^{-1-\frac{\kappa'}{4}+\frac{2\alpha}{\gamma}}(1+o_{\delta,\eps}(1))\log \eps^{-1}
\end{align*}
where we use Lemma \ref{alpha tail} to obtain the last expression.
\end{proof}

Next, we evaluate the measure of $E_{\delta,\eps}$ with respect to the left hand side of \eqref{nonsimple alpha formula}:
\begin{proposition}\label{prop-KW-left-E ns} 
Let $\alpha\in (\frac{\gamma}{2},Q)$. We have
    \begin{equation*}
(\cM_2^\sph(W) \otimes \SLE_{\kappa'}^{\mathrm{sep},\alpha})[E_{\delta, \eps}] = (1+o_{\delta,\eps}(1))  \frac{ \gamma^2 \ol R(\alpha)}{8(Q-\alpha)^2} |\cL_{\kappa'}^\alpha| \log \eps^{-1}
\end{equation*}
where the reflection coefficient $\ol R(\alpha)$ is as in Lemma \ref{lem-sph-area-law},
\end{proposition}
\begin{proof}
As we explain, the argument is a minor modification of that of Section~\ref{subsec-proof-KW-LHS}.  
Keeping the notation in that section, we sample $(\phi,{\bf t}, \eta^0)$  from $M\times dt\times \cL_{\kappa'}^\alpha$ and set $\eta=\eta^0+{\bf t}$, let $\frak{C}_\eta^-$ be the region surrounded by $\eta $ that contains $-\infty$, as defined in the paragraph above Lemma~\ref{lem:stationarity ns}. Define the event $$E'_{\delta,\eps}=\{(\phi,\eta):\eps <\ell_\phi(\eta)<\delta \text{ and } \mu_\phi(\frak{C}_\eta^-)>1\}$$
where $\ell_\phi(\eta)$ is the generalized quantum length of $\eta$.
Then it suffices to prove 
\begin{align}\label{kw equality}
(M\times dt\times \cL^\alpha_{\kappa'})[E'_{\delta, \eps}] = (1+o_{\delta,\eps}(1))  |M| |\cL_{\kappa'}^\alpha| \frac{\gamma\log \eps^{-1}}{2(Q-\alpha)}.
\end{align}
Equation~\eqref{kw equality} is the $\kappa'\in (4,8)$ analog of Proposition \ref{prop:product}; indeed, the only difference in the right hand sides is that the factor $\frac{2}{\gamma}$ in Proposition \ref{prop:product} becomes $\frac{\gamma}{2}$ in~\eqref{kw equality}. The reason for the change $\frac2\gamma \to \frac\gamma2$ is that under $E'_{\delta,\eps}$, the generalized quantum length $\ell_\phi(\eta^0+{\bf t}) \approx e^{\frac{2}\gamma X_{\bf t}}$.

The proof of the upper bound in~\eqref{kw equality} is identical to the proof of Lemma~\ref{lem-KW-left-F|E} (upper bound of Proposition \ref{prop:product}), except that each factor $\frac\gamma2$ or $\frac2\gamma$ is replaced by its reciprocal, and the input Lemma \ref{lem:average} is replaced with the following:
\begin{lemma}\label{lem:average ns}
Sample $(\phi, \mathbf t, \eta^0)$ from $M\times dt \times \cL_{\kappa'}^\alpha$ and set $\eta = \eta^0 + \mathbf t$. Let $\mathfrak l= \ell_\phi(\eta)$, which is the boundary length of the beaded quantum surface $(\frak{C}_\eta^-, \phi)/{\sim_\gamma}$.
Fix $\zeta\in (0,1)$. 
Then there exists  a constant $C>0$ and a function $\mathrm{err}(\ell)$ such that  $\lim_{\ell\downarrow 0} \mathrm{err}(\ell)=0$,  and conditioned on
$(\frak{C}_\eta^-, \phi)/{\sim_\gamma}$, 
the conditional probability of $\{\mathbf t \in (\sigma_{{\mathfrak l}^{1-\zeta}} - C , \tau_{{\mathfrak l}^{1+\zeta}})\}$ is at least $1-\mathrm{err}(\mathfrak l)$. Where \begin{equation}\label{eq:BM-times ns}
\sigma_y =  \inf\{ s > 0 \: : \: X_s < \frac\gamma2 \log y\} \quad \textrm{and}\quad   \tau_y = \sup\{ s > 0 \: : \: X_s > \frac\gamma2 \log y\}.
\end{equation}
and $\mathrm{err}(\mathfrak l):=\cM^{\rm f.d.}_{1,0,0}  (\alpha; 1)^{\#} \Big[   \sup_{\xi \in S} |(h, \xi)|+ Q\log 2 \ge -\frac{\gamma\zeta}{2} \log \ell \Big]$.
\end{lemma}
\begin{proof}
    The argument is identical to Lemma \ref{lem:average}, except we use Lemma \ref{tight ns} in place of Lemma \ref{lem-disk-field-av} .
\end{proof}

The proof of the lower bound of~\eqref{kw equality} is identical to the proof of Lemma \ref{prop-KW-left-E-lower} (lower bound of Proposition \ref{prop:product}); the only changes needed are that each factor of $\frac\gamma2$ or $\frac2\gamma$ should be replaced by its reciprocal, and in~\eqref{area inequality} and~\eqref{inclusion of event} $\cC_\eta^-$ should be replaced by $\mathfrak C_\eta^-$.  This concludes the proof of~\eqref{kw equality}, so Proposition~\ref{prop-KW-left-E ns}  holds. 

\end{proof}

\begin{proof}[Proof of Proposition~\ref{prop-KW-ns}]
The result quickly follows by evaluating $E_{\delta,\eps}$ on both sides of \eqref{nonsimple alpha formula}, and plugging in the expressions in Lemma \ref{lemma-kw-right-E ns} and Proposition \ref{prop-KW-left-E ns}. 
\end{proof}

\subsection{Proof of Proposition \ref{nonsimple alpha welding} }\label{nonsimple welding}

We start with the $\gamma$ insertion case. The following lemma is the counterpart of Lemma~\ref{lem:KW-gamma}.
\begin{lemma}\label{lem:KW-gamma ns}
    For $(\phi,\eta)$ sampled from $\LF_\cC^{(\gamma,\pm\infty)}\times \cL_{\kappa'}(\cC)$,  the law of $(\cC,\phi,\eta,-\infty,+\infty)/{\sim_\gamma}$ is $C\int_0^\infty \ell \cdot \Weld(\cM_{1,0,0}^{\rm f.d.}(\gamma;\ell),\cM^{\rm f.d.}_{1,0,0}(\gamma;\ell))d\ell$ for some $\gamma$-dependent constant $C\in (0,\infty)$.
\end{lemma}
\begin{proof}
As in the proof of Lemma~\ref{lem:KW-gamma}, the result is immediate from Propositions~\ref{prop-loop-zipper ns} and~\ref{prop:KW-LCFT}.
\end{proof}

Fix $\eps\in (0,\frac{1}{4})$ and let $\eta$ be a non-self-crossing loop separating $\pm \infty$. As at the beginning of Section \ref{section nonsimple kw}, we can define $\eta^+$ and $\eta^-$. We mark a point $p^+$ on $\eta^+$ and $p^- $ on $\eta^-$ and let $\psi_{\eta^{\pm}}:\mathbb{H}\to \cC_{\eta^{\pm}}$, $\psi_{\eta^{\pm}}(i)=\pm \infty$ and $\psi_{\eta^{\pm}}(0)=p^{\pm}$. For $\phi$ sampled from $\LF_\cC^{(\gamma, \pm\infty)}$, let  $X^\pm =\phi\circ \psi_{\eta^{\pm}}  +Q \log |(\psi_{\eta^{\pm}})'|$ and $\cC_{\eta,p^\pm,\eps} =\cC\setminus (\psi_{\eta^{-}}(B_\eps(i)) \cup \psi_{\eta^{+}}(B_\eps(i)))$.
	
\begin{lemma}\label{lem:reweight-2 ns}
     For a fixed $\alpha\in (\frac\gamma2,Q)$ and 
	for any $\eps\in (0,\frac14)$ and any nonnegative  measurable function $f$ of  $\phi|_{\cC_{\eta,p^\pm,\eps}}$ we have 
	\begin{align*}
		&\int   f (\phi|_{\cC_{\eta,p^\pm,\eps}}) \times  \eps^{\alpha^2 - \gamma^2} e^{(\alpha - \gamma)(X^-_\eps(i) + X^+_\eps(i))}\, d
		\LF_\cC^{(\gamma,\pm\infty)}
		\\
		&= \int f (\phi|_{\cC_{\eta,p^\pm,\eps}}) \left(\frac14\CR(\exp(\eta^+), 0)\CR(\exp(-\eta^-), 0)\right)^{-\frac{\alpha^2}2 + Q \alpha  - 2} \, d\LF_\cC^{(\alpha, \pm\infty)}.
	\end{align*}
\end{lemma}
\begin{proof}
    The proof is identical to that of Lemma \ref{lem:reweight-2}.
\end{proof}

Now, we give the analog of Lemma~\ref{lem-KW-reweighted-fields}.  Let $\wt \cM_2^\mathrm{f.l.}(t, \ell)$ denote the law of a sample from $\cM_2^\mathrm{f.l.}(t, \ell)$ with a marked point sampled from the probability measure proportional to generalized boundary length. 
\begin{lemma}\label{lem-KW-reweighted-fields-ns}
	There is a constant $C = C(\gamma)$ such that the following holds. 
	Suppose $\alpha \in (\frac\gamma2, Q)$. 
	Sample $(\phi, \eta, p^-, p^+)$ from the measure \[C \cdot \LF_\cC^{(\alpha, \pm\infty)}(d\phi)\, \cL_{\kappa'}^\alpha(d\eta)\, \mathrm{Harm}_{-\infty, \eta^-}(dp^-)\, \mathrm{Harm}_{+\infty, \eta^+}(dp^+),\]
 and further sample a point $q \in \eta$ from the probability measure proportional to generalized quantum length. Then the curve $\eta$ cuts $(\cC, \phi, \eta, -\infty, +\infty, p^-, p^+, q)/{\sim_\gamma}$ into a pair of forested quantum surfaces $\cD^-$ and $\cD^+$ (respectively containing marked points  $(-\infty, p^-, q)$ and $(+\infty, p^+, q)$) with the following description. 

 Embed the connected component of $\cD^\pm$ containing the bulk marked point as $(\bbH, X^\pm, i, 0)$ (so $X^\pm = \phi \circ \psi_\eta^\pm + Q \log |(\psi_\eta^\pm)'|$), and let $\cL^\pm_\bullet$ be the forested line segment corresponding to the foresting of the boundary $\partial \bbH$ in counterclockwise direction starting from $0$, so $\cL^\pm_\bullet$ has a marked point on its forested boundary arc. Then the law of $(X^-, \cL^-_\bullet, \cL^+_\bullet, X^+)$ is 
 \eqb\label{eq-KW-reweighted-fields-ns}
 \iiint_0^\infty \LF_\bbH^{(\alpha, i)}(t^-) \times \wt \cM_2^\mathrm{f.l.}(t^-, \ell) \times \wt \cM_2^\mathrm{f.l.}(t^+, \ell) \times \LF_\bbH^{(\alpha, i)}(t^+)\, dt^- \, \cdot \ell  d\ell \cdot \, dt^+.
 \eqe
\end{lemma}
\begin{proof}
We first establish the $\alpha =\gamma$ case. 
Writing $\wt \cD^\pm$ for $\cD^\pm$ without the marked point $q$, Lemma~\ref{lem:KW-gamma ns} identifies the law of ($\wt\cD^-,\wt\cD^+$) as $C\int_0^\infty \cM_{1,0,0}^{\rm f.d.}(\gamma;\ell)\times \cM^{\rm f.d.}_{1,0,0}(\gamma;\ell)\ell d\ell$. Thus, if we let $\cL^\pm$ be $\cL^\pm_\bullet$ with marked point forgotten, by Lemma~\ref{lem:har} and the definition of $\cM_{1,0,0}^\mathrm{f.d.}(\gamma;\ell)$ (Definition~\ref{def:forested-disk}) the law of $(X^-, \cL^-, \cL^+, X^+)$ is~\eqref{eq-KW-reweighted-fields-ns} with the measures $\wt \cM_2^\mathrm{f.l.}(t^\pm, \ell)$ replaced by $ \cM_2^\mathrm{f.l.}(t^\pm, \ell)$. The desired $\alpha = \gamma$ result then holds by the uniform conformal welding claim in Lemma~\ref{lem:KW-gamma ns}.

 The extension to the general $\alpha$ case is identical to that in the proof of Lemma~\ref{lem-KW-reweighted-fields}, using Lemma \ref{lem:reweight-2 ns} (in place of Lemma~\ref{lem:reweight-2}) and Lemma \ref{lem-disk-reweight}.
 \end{proof}.

\begin{proof}[Proof of Lemma~\ref{nonsimple alpha welding}]
In Lemma~\ref{lem-KW-reweighted-fields-ns}, 
by Definition~\ref{def:forested-disk} the pair $(\wt \cD^-, \wt \cD^+)$ agree in law with a sample from 
$\int_0^\infty \cM_{1,0,0}^{\rm f.d.}(\alpha;\ell)\times \cM^{\rm f.d.}_{1,0,0}(\alpha;\ell)\ell d\ell$ after a boundary point is sampled from the probability measure proportional to generalized boundary length for each surface. Thus, for $(\phi, \eta) \sim C \LF_\cC^{(\alpha, \pm\infty)} \cL_{\kappa'}^\alpha$ the law of $(\cC, \phi, \eta, -\infty, +\infty)/{\sim_\gamma}$ is $\int_0^\infty \mathrm{Weld}(\cM_{1,0,0}^{\rm f.d.}(\alpha;\ell), \cM^{\rm f.d.}_{1,0,0}(\alpha;\ell))\ell d\ell$. The desired result then follows from Proposition~\ref{prop:KW-LCFT}. 
\end{proof}

\section{Evaluation of constants for the loop intensity measure}\label{sec:constant}
In this section, we evaluate the constants mentioned in Section~\ref{subsec:outlook} for the loop intensity measure.

\subsection{The dilation constant}\label{subsec:dilation}
The first constant concerns the decomposition of the loop measure into its shape measure and a dilation. For convenience, we map the measure to the infinite cylinder $\cC$ so that dilation becomes translation. Let $\wt\SLE_\kappa^\mathrm{sep}$ be the measure on the space of loops in $\cC$ separating $\pm \infty$ which is  obtained by sampling a loop from $\wt\SLE_\kappa^\lp$, mapping it to $\cC$ via $\log$, and restricting to $\{\text{loop separates } \pm\infty\}$. For a loop $\eta\sim\wt\SLE_\kappa^\mathrm{sep}$ we define a pair $(\eta_0, \textbf t)$, where $\textbf t = \sup_{z \in \eta} \Re z$ and $\eta_0 = \eta - \textbf  t$ is the translation of $\eta$ such that $\sup_{z \in \eta_0} \Re z = 0$. On the other hand, given a pair $(\eta_0, \textbf  t)$, we can recover $\eta = \eta_0 + \textbf t$. Thus we may view $\wt\SLE_\kappa^\mathrm{sep}$ as the law of the pair $(\eta_0, \textbf t)$.

\begin{proposition}\label{prop-SLE-lp-density}
Let $\kappa \in (8/3, 8)$.  There is a probability measure $\cL_\kappa$ such that  
\[ \wt\SLE_\kappa^\mathrm{sep} = \frac1\pi{(\frac\kappa4-1) \cot(\pi (1- \frac4\kappa)) } \cL_\kappa(d\eta_0) \,dt. \]
\end{proposition}

To prove Proposition~\ref{prop-SLE-lp-density}, we need the following asymptotic result. 

\begin{lemma}\label{lem-N_C-count}
Let $\kappa \in (8/3, 8)$ and $C>0$. Sample a full-plane $\CLE_\kappa$ and let $N_C \geq 0$ be the number of loops $\eta$ which separate $0$ and $\infty$ satisfying $e^{-C} \leq \mathrm{dist}(\eta, 0) \leq 1$. Then 
\eqb\label{eq-rw-moment}
\lim_{C \to \infty} \frac1C\E[N_C]  = \frac1\pi(\frac\kappa4-1) \cot(\pi (1 - \frac4\kappa)).
\eqe
\end{lemma}
\begin{proof}
Let $\dots, \eta_{-1}, \eta_0, \eta_1, \dots$ be the doubly-infinite sequence of loops separating $0$ and $\infty$ such that each successive loop is nested in the previous loop, and $\eta_0$ is the first loop which intersects $\D$. Let $B_n := \log \CR (\eta_{n-1}, 0) - \log \CR(\eta_n, 0)$, so the sequence $B_1, B_2, \dots$ is i.i.d., and \cite[Proposition 1, Equation (2)]{ssw-radii} gives
	\[\frac1{\E[B_1]} = \frac1\pi{(\frac\kappa4-1) \cot(\pi (1 - \frac4\kappa)) }.\]
 For $n \geq 0$ let $S_n := B_1 + \dots + B_n$ and let $R := \CR(\eta_0,0)$, so $\CR(\eta_n, 0) = e^{-S_n}R$. Let $\tau_C \geq 0$ be the first index such that $S_{\tau_C} > C - \log 4 + \log R$, and let $\sigma_C \geq 0$ be the first index such that $S_{\sigma_C} >C + \log R$. Then by the Koebe quarter theorem
 \[\mathrm{dist}(\eta_{\tau_C-1}, 0) \geq \frac14 \CR(\eta_{\tau_C-1}, 0) \geq e^{-C}, \qquad \mathrm{dist}(\eta_{\sigma_C},0) \leq \CR(\eta_{\sigma_C}, 0) \leq e^{-C},\]
 hence
 \eqb\label{eq-N_C}
\tau_C \leq N_C \leq \sigma_C+1.
 \eqe
 We will show that $\lim_{C \to \infty}\E[\tau_C]/C = \lim_{C \to \infty} \E[\sigma_C]/C = 1/\E[B_1]$. This and~\eqref{eq-N_C} imply the desired~\eqref{eq-rw-moment}.
 
 First, 
 the explicit density for $B_1$ computed in \cite[(4)]{ssw-radii} implies that for some constants $a,b>0$ we have  $\P[B_1 \in dx] = (1+o_x(1))a e^{-bx}$ as $x \to \infty$, so there is a constant $K$ such that for any $x>0$, we have $\E[B_1 \mid B_1 > x]< x+K$. We conclude that $\E[S_{\sigma_C} \mid R] \leq C + \log R + K$. By Wald's identity, we have $\E[\sigma_C \mid R] = \E[S_{\sigma_C} \mid R]/\E[B_1] \leq  (C + \log R + K)/\E[B_1]$. Since $R = \CR(\eta_0, 0) \leq 4$ by the Koebe quarter theorem, we conclude $\limsup_{C \to \infty} \E[S_{\sigma_C}]/C \leq  1/\E[B_1]$.

 Next, by Wald's identity $\E[\tau_C \mid R] = \E[S_{\tau_C}\mid R]/\E[B_1] \geq (C - \log 4 + \log R)/\E[B_1]$. Thus for any constant $x>0$ we have $\E[\tau_C] \geq \E[\tau_C 1_{R >x}] \geq \P[R>x] (C - \log 4 + \log x)/\E[B_1]$, and $\liminf_{C \to \infty} \E[\tau_C]/C \geq  \P[R>x]/\E[B_1]$. Since $x$ is arbitrary we have $\liminf_{C \to \infty} \E[\tau_C]/C \geq 1/\E[B_1]$. Combining with the previous paragraph gives the desired result. 
\end{proof}

\begin{proof}[Proof of Proposition~\ref{prop-SLE-lp-density}]
By~\eqref{eq-rw-moment} we have $\wt \SLE_\kappa^\mathrm{sep} [0 \leq \textbf t \leq C ] = (1+o(1))\frac1\pi{(\frac\kappa4-1) \cot(\pi (1- \frac4\kappa))} C$, so the conformal invariance of $\wt\SLE_\kappa^\mathrm{loop}$ implies the marginal law of $\mathbf t$ is $\frac1\pi{(\frac\kappa4-1) \cot(\pi (1- \frac4\kappa))}\,dt$. Conformal invariance also implies the conditional law of $\eta_0$ given $\mathbf t$ does not depend on $\mathbf t$; this conditional law is $\cL_\kappa$.
\end{proof}

\subsection{The  welding constant for \texorpdfstring{$\kappa\in(\frac{8}{3},4)$}{g} }
 For $\kappa\in(\frac{8}{3},4)$ and $\gamma=\sqrt{\kappa}$, Theorems~\ref{thm-loop2} and~\ref{cor-loop-equiv-main} imply that
\begin{equation}\label{eq:weld-constant}
\Weld(\QD,\QD)=C_\gamma\QS\otimes\wt\SLE_{\kappa}^{\rm loop}
\end{equation}
For some constant $C_\gamma$. We now evaluate this constant.
\begin{proposition}\label{weldingconstant simple}
    For $\kappa\in(\frac{8}{3},4)$ and $\gamma=\sqrt{\kappa}$, we have 
\begin{equation}\label{eq:C-gamma}
       C_\gamma=\frac{1}{4\pi} \frac{\Gamma(\frac{\gamma^2}4) \Gamma(1-\frac{\gamma^2}4)}{(Q-\gamma)^2} \tan(\pi(\frac4{\gamma^2}-1)),
    \end{equation} 
    \end{proposition}

As same in the proof of Proposition \ref{prop-loop-zipper}, we sample two points from quantum area measure and restrict to the event that the loop separates the two points. Then it suffices to prove (we incur a factor of 2 since there are two choices for which quantum disk the first point lies on)
$$2 \Weld(\QD_{1,0},\QD_{1,0})=C_\gamma\cdot\QS_2\otimes\wt\SLE_{\kappa}^{\rm sep}$$
with the same $C_\gamma$ as \eqref{eq:C-gamma}. This can be done by the following lemma.
    \begin{lemma}\label{lem:K0}
    We have $\Weld(\QD_{1,0},\QD_{1,0})=\wt K_\gamma\QS_2\otimes(\cL_\kappa\times dt)$, with the welding constant
    $$\wt K_\gamma =  \frac{\gamma}{16\pi^2(Q-\gamma)}  \Gamma(\frac{\gamma^2}4) \Gamma(1-\frac{\gamma^2}4).$$
\end{lemma}
\begin{proof}\label{weldingconstant shape simple}
	By Proposition \ref{prop-KW-weld}, there exists a constant $C$ such that 
	\[ \int_0^\infty \ell \mathrm{Weld}(\cM_1^\disk(\gamma; \ell), \cM_1^\disk(\gamma; \ell))\, d\ell = C (Q-\gamma)^2 \cM_2^\sph(\gamma) \otimes (\cL \times dt). \]
 We evaluate the event $E_{\delta,\varepsilon}$ (i.e. the total quantum area is larger than 1 and the quantum length of the $\SLE_\kappa$ loop is in $(\delta,\varepsilon)$; see the paragraph before Lemma \ref{lem-KW-right-E} on both sides of the above welding equation).
By Lemma \ref{lem-KW-right-E} and Proposition \ref{prop-KW-left-E} , we have 
	\[\frac1{C(Q-\gamma)^2} \frac1{\frac2\gamma(Q-\gamma) \Gamma(\frac2\gamma(Q-\gamma))}\left(\frac2\gamma 2^{-\gamma^2/2} \ol U(\gamma) \right)^2 \left(4 \sin \frac{\pi \gamma^2}4 \right)^{-\frac2\gamma (Q-\gamma)} = \frac{\ol R(\gamma)}{2(Q-\gamma)^2},\]
	hence 
	\[ C = \frac{2^{3-\gamma^2}}{ \gamma^2\Gamma(\frac4{\gamma^2})} \left(4 \sin \frac{\pi \gamma^2}4\right)^{1- \frac4{\gamma^2}} \frac{\ol U(\gamma)^2}{\ol R(\gamma)}
	= \frac{Q-\gamma}{4\gamma}  \Gamma(\frac{\gamma^2}4) \Gamma(1-\frac{\gamma^2}4) .\]
	Here, we are using the following computation which uses $\Gamma(x)\Gamma(1-x) = \frac\pi{\sin(\pi x)}$ and $\Gamma(x+1) = x\Gamma(x)$:
		\alb
	\frac{\ol U(\gamma)^2}{\ol R(\gamma)} &= \left. \left( \frac{2^{2-\gamma^2} \pi^2 }{\Gamma(1-\frac{\gamma^2}4)}\right)^{\frac4{\gamma^2}-1} \Gamma(\frac{\gamma^2}4)^2 \right/ \left( \left( \frac{\pi\Gamma(\frac{\gamma^2}4)}{\Gamma(1-\frac{\gamma^2}4)^2}\right)^{\frac4{\gamma^2}-1} \frac{-\frac\gamma2\Gamma(\frac{\gamma^2}4 -1)}{ (Q-\gamma) \Gamma(1-\frac{\gamma^2}4) \Gamma(\frac4{\gamma^2}-1) } \right) \\
	&= 2^{\gamma^2 - 4}\left(\frac{4\pi}{\Gamma(\frac{\gamma^2}4) \Gamma(1-\frac{\gamma^2}4)}\right)^{\frac4{\gamma^2}-1} \frac{\Gamma(\frac{\gamma^2}4)^2}{-\frac\gamma2 \Gamma(\frac{\gamma^2}4-1)}(Q-\gamma) \Gamma(1-\frac{\gamma^2}4) \Gamma(\frac4{\gamma^2}-1) \\
	&= 2^{\gamma^2 - 5}\gamma\left(4\sin(\frac{\pi\gamma^2}4)\right)^{\frac4{\gamma^2}-1} \Gamma(\frac{\gamma^2}4) (Q-\gamma) \Gamma(1-\frac{\gamma^2}4) \Gamma(\frac4{\gamma^2})
	\ale
	Also, by Theorem \ref{thm:def-QD}, we have $\cM_{1,0}^\disk(\gamma; \ell) = \frac{2\pi}\gamma(Q-\gamma)^2 \QD_{1,0}(\ell)$. And by definition we have $\cM_2^\sph(\gamma) = \QS_2$. This gives $\wt K_\gamma = {C(Q-\gamma)^2}/{(\frac{2\pi}\gamma(Q-\gamma)^2)^2}$ as desired.
\end{proof}

\begin{proof}[Proof of Proposition~\ref{weldingconstant simple}] Combining Lemma~\ref{lem:K0} and Proposition \ref{prop-SLE-lp-density}, the result follows.
\end{proof}
Another natural formulation of conformal welding  result for the SLE loop measure is under the uniform embedding  $\mathbf{m}_{\wh \C}$. For the loop intensity measure, we have
\begin{equation}\label{eq:constant-K}
\mathbf m_{\wh \C} \ltimes \left( \int_0^\infty \ell \mathrm{Weld}(\QD(\ell), \QD(\ell)) \, d\ell \right) = K_\gamma \LF_\C \times \wt\SLE_\kappa^\lp
\end{equation}
where $K_\gamma$ is a constant depends on $\gamma$. We now evaluate this constant as a corollary of Proposition \ref{weldingconstant simple}.

\begin{corollary}\label{welding liouville}
    For $\kappa\in(\frac{8}{3},4)$ and $\gamma=\sqrt{\kappa}$, we have
		\[  K_\gamma =\frac{\gamma}{ 8} \frac{\Gamma(\frac{\gamma^2}4) \Gamma(1-\frac{\gamma^2}4)}{(Q-\gamma)^4} \tan(\pi(\frac4{\gamma^2}-1)). \]
\end{corollary}
\begin{proof}
	Applying the uniform embedding $\mathbf m_{\wh \C}$ in Proposition \ref{weldingconstant simple} gives
	\[\mathbf m_{\wh \C} \ltimes \left( \int_0^\infty \ell \mathrm{Weld}(\QD(\ell), \QD(\ell)) \, d\ell \right) = C_\gamma \mathbf m_{\wh \C} \ltimes \left( \QS \times \wt\SLE_\kappa^\lp\right).\]
	Since $\mathbf m_{\wh \C} \ltimes \QS=\frac{\pi\gamma}{2(Q-\gamma)^2}\LF_\C$ (Proposition \ref{prop:QS-field}), we conclude that $K_\gamma=C_\gamma\frac{\pi\gamma}{2(Q-\gamma)^2}$. Plugging \eqref{eq:C-gamma} yields the result.
\end{proof}

\subsection{The  welding constant for \texorpdfstring{$\kappa'\in(4,8)$}{g}}
We first prove the counterpart of Proposition \ref{weldingconstant simple}.
\begin{proposition}\label{weldingconstant non-simple}
    For $\kappa'\in(4,8)$ and  $\gamma=\frac{4}{\sqrt{\kappa'}}$, we have $\Weld(\GQD,\GQD)=C_\gamma'\QS\otimes\wt\SLE_{\kappa'}^{\rm loop}$ (by Theorems~\ref{loop weld ns} and~\ref{cor-loop-equiv-main}) with 
    \begin{equation}\label{eq:C'-gamma}
            C_\gamma'={R_\gamma'}^2\left(\frac{2}{\gamma}\right)^5\pi^{-\frac{8}{\gamma^2}+2}2^{-\frac{8}{\gamma^2}}\frac{\Gamma(\frac{4}{\gamma^2}-1)}{\Gamma(2-\frac{4}{\gamma^2})}\Gamma(1-\frac{\gamma^2}{4})^{\frac{8}{\gamma^2}+2}\tan\pi\left(\frac{\gamma^2}{4}-1\right)\frac{2(Q-\gamma)^2}{\pi\gamma},
    \end{equation}
    where  $R_\gamma'$ is specified by $|\GQD(\ell)|=R_\gamma'\cdot \ell^{-\frac{\gamma^2}{4}-2}$ as in Lemma \ref{length gqd}.
\end{proposition}

Similary as in the proof of Proposition \ref{weldingconstant non-simple} we need the following lemma. 
\begin{lemma}\label{w0}
The constant $\wt K'_\gamma$ in the welding equation $\int_0^\infty \ell\Weld(\GQD_{1,0}(\ell),\GQD_{1,0}(\ell))d\ell=\wt K'_\gamma\QS_2\times\mathcal{L}_{\kappa'}(d\eta_0)dt$ satisfies
\begin{equation*}
\wt K'_\gamma\frac{\gamma^2\bar R(\gamma)}{8(Q-\gamma)^2}=\frac{{R_\gamma'}^2\Gamma(4/\kappa')}{\Gamma(2-4/\kappa')\Gamma(2-\kappa'/4)}(4\sin\frac{\pi\gamma^2}{4})^{-\kappa'/4+1}\left(\frac{\kappa'}{16\sin\frac{\pi\gamma^2}{4}}\frac{\Gamma(1-\frac{4}{\kappa'})}{\Gamma(\frac{4}{\kappa'})}\right)^2
\end{equation*}
\end{lemma}
\begin{proof}
Recall $E_{\delta,\varepsilon}$ is the event that the generalized quantum disk containing the first marked point has quantum area 
t least 1 and the loop has boundary length in $(\varepsilon,\delta)$. We evaluate $E_{\delta,\varepsilon}$ on both sides of the welding equation. By Proposition \ref{prop-KW-left-E ns} and taking $\alpha=\gamma$, we have
\begin{equation}\label{eq:RHS-ede-ns}
(\QS_2\otimes(\mathcal{L}_{\kappa'}\times dt))[E_{\delta,\varepsilon}]=(1+o_{\delta,\varepsilon}(1))\frac{\gamma^2\bar R(\gamma)}{8(Q-\gamma)^2}\log\varepsilon^{-1},
\end{equation}
while evaluating $E_{\delta,\varepsilon}$ on the measure $\int_0^\infty \ell\Weld(\GQD_{1,0}(\ell),\GQD_{1,0}(\ell))d\ell$ gives
\begin{equation}\label{eq:LHS-ede-ns}
\int_\delta^\varepsilon \ell|\GQD_{1,0}(\ell)|^2\GQD_{1,0}(\ell)^\#[A>1]d\ell.
\end{equation}
Note that $\GQD_{1,0}(\ell)^\#[A>1]=\GQD_1(1)^\#[A>\ell^{-8/\kappa'}]$. According to Lemma \ref{alpha tail}, we see
\begin{equation*}
\GQD_1(1)^\#[A>\ell^{-8/\kappa'}]=\frac{\Gamma(4/\kappa')}{\Gamma(2-4/\kappa')\Gamma(2-\kappa'/4)}(4\sin\frac{\pi\gamma^2}{4})^{-\kappa'/4+1} \ell^{2-8/\kappa'}+O(\ell^2).
\end{equation*}
Then \eqref{eq:LHS-ede-ns} equals
\begin{equation}\label{eq:LHS-ede-ns-2}
\begin{aligned}
&\int_\varepsilon^\delta \ell|\GQD_{1,0}(\ell)|^2\GQD_{1,0}(\ell)^\#[A>1]d\ell\\
&=\int_\varepsilon^\delta\left[\frac{\Gamma(4/\kappa')}{\Gamma(2-4/\kappa')\Gamma(2-\kappa'/4)}(4\sin\frac{\pi\gamma^2}{4})^{-\kappa'/4+1} \ell^{2-8/\kappa'}+O(\ell^2)\right]l\left(R_\gamma'\frac{\kappa'}{16\sin\frac{\pi\gamma^2}{4}}\frac{\Gamma(1-\frac{4}{\kappa'})}{\Gamma(\frac{4}{\kappa'})}\ell^{\frac{4}{\kappa'}-2}\right)^2d\ell\\
&=\frac{{R_\gamma'}^2\Gamma(4/\kappa')}{\Gamma(2-4/\kappa')\Gamma(2-\kappa'/4)}(4\sin\frac{\pi\gamma^2}{4})^{-\kappa'/4+1}\left(\frac{\kappa'}{16\sin\frac{\pi\gamma^2}{4}}\frac{\Gamma(1-\frac{4}{\kappa'})}{\Gamma(\frac{4}{\kappa'})}\right)^2(1+o_{\delta,\varepsilon}(1))\log\varepsilon^{-1}.
\end{aligned}
\end{equation}
Let \eqref{eq:RHS-ede-ns} equal the last line of \eqref{eq:LHS-ede-ns-2}, the result then follows.
\end{proof}
\begin{proof}[Proof of Proposition \ref{weldingconstant non-simple}]
Combining Lemma~\ref{w0} and Proposition \ref{prop-SLE-lp-density}, we get the result.
\end{proof}

As in Corollary \ref{welding liouville}, in the
the uniform embedding setting we have:
\begin{corollary}\label{welding liouville ns}
We have
\begin{equation*}
\mathbf m_{\wh \C}\ltimes\left(\int_0^\infty \ell\Weld(\GQD(\ell),\GQD(\ell))d\ell\right)=K'_\gamma\LF_\C\times\wt\SLE_{\kappa'}^{\mathrm{loop}}
\end{equation*}
where the constant $K'_\gamma$ is given  by
\begin{equation}\label{wc}
\frac{K'_\gamma}{{R_\gamma'}^2}=-\left(\frac{2}{\gamma}\right)^5\pi^{-\frac{8}{\gamma^2}+2}2^{-\frac{8}{\gamma^2}}\frac{\Gamma(\frac{4}{\gamma^2}-1)}{\Gamma(2-\frac{4}{\gamma^2})}\Gamma(1-\frac{\gamma^2}{4})^{\frac{8}{\gamma^2}+2}\tan\pi\left(\frac{\gamma^2}{4}-1\right)
\end{equation}
\end{corollary}
\begin{proof}
We apply the uniform embedding $\mathbf{m}_{\wh\C}$ to Proposition \ref{weldingconstant non-simple} and use $\mathbf m_{\wh \C} \ltimes \QS=\frac{\pi\gamma}{2(Q-\gamma)^2}\LF_\C$ (Proposition \ref{prop:QS-field}). Then we conclude $K'_\gamma=C_\gamma'\frac{\pi\gamma}{2(Q-\gamma)^2}$ for $C_\gamma'$ in \eqref{eq:C'-gamma}; and the result follows by plugging into the expression of the reflection coefficient $\bar R(\gamma)$ in Lemma \ref{lem-sph-area-law}.
\end{proof}

\appendix

\section{Continuity of \texorpdfstring{$\CLE_\kappa$}{g} as \texorpdfstring{$\kappa\uparrow 4$}{g}}\label{sec:app-4}
In this section we supply the continuity in $\kappa$ needed to extend  Theorems~\ref{cor-loop-equiv-main} and~\ref{thm-KW}
from $\kappa<4$ to $\kappa=4$.  We start with a monotonicity statement for $\CLE_\kappa$ proved in \cite{shef-werner-cle}.
\begin{lemma}[\cite{shef-werner-cle}]\label{lem:monotone}
There exists a coupling of $\CLE_\kappa$ on the unit disk $\D$ for $\kappa\in (\frac83,4]$   such that a.s.\  each outermost loop of $\CLE_{\kappa_1}$ is surrounded by an outermost loop of $\CLE_{\kappa_2}$ if $\kappa_1<\kappa_2\le 4$.
\end{lemma}
\begin{proof}
 By \cite[Theorems 1.5, 1.6]{shef-werner-cle}, the law of the  boundaries of outermost loop clusters in a Brownian loop soup with intensity 
 $c_\kappa = (3\kappa - 8)(6-\kappa)/2\kappa$ is given by the  outermost loops of $\CLE_\kappa$. Now the monotonicity of $c_\kappa$ in $\kappa\in (\frac83,4]$ 
 yields the desired monotonicity in Lemma~\ref{lem:monotone}. 
\end{proof}
We recall  the formula from~\cite{ssw-radii} for  the conformal radii of  CLE.
\begin{theorem}[\cite{ssw-radii}]\label{thm-ssw}
	For $\kappa\in (8/3,8)$, let $\eta_\kappa$ be the outermost loop surrounding $0$ of a $\CLE_\kappa$ on $\D$. 	
	Let $\CR(\eta_\kappa, 0)$ be the conformal radius of the region surrounded  by $\eta_\kappa$ viewed from 0. Then 
	\[\E[ \CR(\eta_\kappa, 0)^\lambda ] = \frac{- \cos (\frac{4\pi}\kappa)}{\cos \left( \pi \sqrt{(1-\frac4\kappa)^2 -\frac{8\lambda}\kappa} \right)}
	\quad \textrm{for } \lambda > -1 + \frac\kappa8.
	\]
\end{theorem}
Recall that the  Hausdorff distance between two closed sets $E_1,E_2$ on a metric space $(X,d)$ is given by 
$\max\{ \sup_{x\in E_1}\{d(x,E_2)\} , \sup_{x\in E_2}\{d(x,E_1) \}\}$.
Lemma~\ref{lem:monotone} and Theorem~\ref{thm-ssw} yield the following continuity result.  
\begin{lemma}\label{lem-kappa-D}
Suppose we are in the coupling in Lemma~\ref{lem:monotone}. For each $z\in \D$, let $\eta_\kappa(z)$ be  the outermost loop around $z$ of the $\CLE_\kappa$. 
For any fixed $z$, viewed as closed sets,  $\eta_\kappa(z)$ converges almost surely to $\eta_4(z)$  in the Hausdorff metric as $\kappa \uparrow 4$.
\end{lemma}
\begin{proof}
By the conformal invariance of $\CLE_\kappa$ we assume $z=0$ because the same argument will work for a general $z$.  In this case we simply write $\eta_\kappa(0)$ as $\eta_\kappa$. Since $\eta_{\kappa_1}$ is surrounded by  $\eta_{\kappa_2}$ if $\kappa_1<\kappa_2\le 4$.
we see that $\CR(\eta_\kappa, 0)$ is increasing in $\kappa$. 
By the explicit formula in Theorem~\ref{thm-ssw}, we have $\lim_{\kappa \uparrow 4} \E[ \CR(\eta_\kappa, 0)] = \E[\CR(\eta_\kappa, 0)]$.
Thus we must have $\lim_{\kappa \uparrow 4} \CR(\eta_\kappa, 0)=\CR(\eta_4, 0)$ a.s. 

Let $D_\kappa$ be the region surrounded by $\eta_\kappa$ and $D_{4^-}=\cup_{\kappa<4} D_\kappa$. The conformal radius of $D_{4^-}$ must be between
$\lim_{\kappa \uparrow 4} \CR(\eta_\kappa, 0)$ and $\CR(\eta_4, 0)$, hence equals  $\CR(\eta_4, 0)$ a.s. 
This means that $D_\kappa\uparrow D_4$  a.s. hence  $\eta_\kappa \to \eta_4$ a.s.\ in the Hausdorff  metric in the coupling in Lemma~\ref{lem:monotone}.
\end{proof}

Recall the loop measures $ \SLE_\kappa^\mathrm{loop}$, $\wt \SLE_\kappa^\mathrm{loop}$, $ \cL_\kappa$  and $\wt \cL_\kappa$
from Section~\ref{sec:msw-sphere}.
We first give a quantitative version of Lemma~\ref{lem-Markov} and then prove the continuity in $\kappa$ for these measures.
\begin{proposition}\label{prop-kemp-werner}
Lemma~\ref{lem-Markov} holds. Moreover, the convergence is exponential with a uniform rate near $\kappa=4$. 
That is,  there exists a constant $a(\kappa)\in (0,1)$ depending on $\kappa$ such that 
 the total variation distance between the law of $\eta^n$ and $\wt\cL_\kappa(\cC)$ is at most $a(\kappa)^n$, and in addition,  
$\sup\{ a(\kappa) \: : \: \kappa \in [\kappa_0, 4]\} < 1$ for $\kappa_0 \in (8/3, 4]$.
\end{proposition}
\begin{proof}
	Fix $\kappa \in (\frac83, 4]$. \cite[Section 3.1, Fact 4]{werner-sphere-cle} shows that 
	there exists $a(\kappa)\in (0,1)$ such that regardless of the initial states, the Markov chain in Lemma~\ref{lem-Markov} couples in one step with probability $1-a(\kappa)$. 
	Moreover, inspecting that argument,  we see that $\sup\{ a(\kappa) \: : \: \kappa \in [\kappa_0, 4]\} < 1$ for $\kappa_0 \in (8/3, 4]$. This gives Proposition~\ref{prop-kemp-werner}. 
\end{proof}

\begin{lemma}\label{lem-kappa-shape}
We have $\lim_{\kappa \uparrow 4} \wt \cL_\kappa = \wt \cL_4$  weakly with respect to the Hausdorff metric.
\end{lemma}
\begin{proof}
By the domain Markov property of $\CLE_\kappa$ and iteratively applying Lemma~\ref{lem-kappa-D}, we see that for each $n\in \N$, the law of $\eta^n$ as $\kappa\to 4$
converge weakly to the law of $\eta^n$ for $\kappa=4$. Now the uniformly exponential convergence in Proposition~\ref{prop-kemp-werner}
gives $\lim_{\kappa \uparrow 4} \wt \cL_\kappa(\cC) = \wt \cL_4(\cC)$. Transferring from the cylinder to the disk gives  $\lim_{\kappa \uparrow 4} \wt \cL_\kappa = \wt \cL_4$.
\end{proof}

	\begin{lemma}\label{lem-kappa-shape-2}
We have $\lim_{\kappa \uparrow 4} \cL_\kappa =\cL_4$  weakly with respect to the Hausdorff metric.
	\end{lemma}
	\begin{proof}

We only explain that under the natural parameterization, chordal $\SLE_\kappa$ converges to $\SLE_4$ as $\kappa \uparrow 4$. 
Once this is done, the same convergence holds for $\SLE_\kappa^{p\rightleftharpoons q} $. Note that the whole-plane $\SLE_\kappa$ from $p$ to $q$ is continuous as $\kappa\uparrow4$ with Hausdorff topology. Conditioned on this, the law of the other is a chordal $\SLE_\kappa$ from $p$ to $q$ in the complementary domain, which is continuous as $\kappa\uparrow4$ under natural parametrization. By symmetry, we get two-sided whole plane $\SLE_\kappa$ curve  $\SLE_\kappa^{p\rightleftharpoons q} $ converges to $\SLE_4^{p\rightleftharpoons q}$ under natural parametrization as well. From this and  the definition of $\SLE_\kappa^{\mathrm{loop}}$ (Definition~\ref{def:loop}), we get the convergence of $ \cL_\kappa$.

We now show that the law of chordal $\SLE_\kappa$ on $\bbH$ from $0$ to $\infty$ 
 under natural parametrization is continuous as $\kappa\uparrow 4$. We first recall that this family of measures is tight in the local uniform topology  of parametrized curves thanks to their H\"older regularity established by Zhan~\cite{Zhan-holder}. On the other hand the natural parametrization of $\SLE_\kappa$ is characterized by a conformal invariance and domain Markov property considered by Lawler and Sheffield~\cite{LS-natural}. Therefore all subsequential limits of the chordal $\SLE_\kappa$ measure agree with $\SLE_4$.
	\end{proof}

\bibliographystyle{alpha}
\def\cprime{$'$}

\end{document}